\documentclass[10pt]{book}
\usepackage{index}
\makeindex

\newindex{notation}{adx}{and}{Index of symbols}
\newindex{notion}{bdx}{bnd}{Index of notions}
\input{package.tex}
\newcommand\notion[1]{\emph{#1}\index[notion]{#1}}
\newcommand\wcnotion[2]{\emph{#1}\index[notion]{#2}}
\newcommand\wcnotionsym[3]{\emph{#1}\index[notation]{#2}\index[notion]{#3}}
\newcommand\wcsnotion[3]{\emph{#1}\index[notion]{#2!\textit{#3}}}
\newcommand\snotion[2]{\emph{#1}\index[notion]{#1!\textit{#2}}}
\newcommand\snotionsym[3]{\emph{#1}\index[notion]{#1!\textit{#3}}\index[notation]{#2!\textit{#3}}}
\newcommand\notionsym[2]{\emph{#1}\index[notion]{#1!}\index[notation]{#2}}
\newcommand\wcsnotionsym[4]{\emph{#1}\index[notation]{#2!\textit{#4}}\index[notion]{#3!\textit{#4}}}

\newcommand\wcnotation[2]{\emph{#1}\index[notation]{#2}}
\newcommand\wcsnotation[3]{\emph{#1}\index[notation]{#2!\textit{#3}}}

\newcommand\sym[1]{\index[notation]{#1}}
\newcommand\ssym[2]{\index[notation]{#1!\textit{#2}}}

\newcommand{\nc}{{\neg\mathrm{c}}}
\newcommand{\comp}{\mathrm{c}}

\newcommand{\seg}{\mathrm{seg}}
\newcommand{\cseg}{\mathrm{cseg}}

\newcommand{\Zb}{\mathbb{Z}} 
 
\newcommand{\Nb}{\mathbb{N}}

\newcommand{\Ib}{\mathbb{I}}

\def\-{\raisebox{.75pt}{-}}

\newcommand{\uvar}{\_}

\newcommand{\sat}{\mathrm{sat}}

\newcommand{\Db}{\mathbf{D}}

\newcommand{\Fb}{\mathbf{F}} 
\DeclareMathOperator{\Gb}{G} 
  
\DeclareMathOperator{\N}{N}

\newcommand{\cell}{\mathrm{cell}}
\newcommand{\W}{\mathrm{W}}
\newcommand{\M}{\mathrm{M}}

\renewcommand{\O}{\mathrm{O}}

\DeclareMathOperator*{\CDA}{ADC}
\DeclareMathOperator*{\CDAB}{ADC_B}

\newcommand\omegacat{\omega\mbox{-$\cat$}}
\DeclareMathOperator\Set{Set}
\DeclareMathOperator\Sp{Sp}

\DeclareMathOperator{\Hom}{Hom}

\DeclareMathOperator*{\Arr}{Arr}

\newcommand{\colim}{\operatornamewithlimits{colim}}

\newcommand\iun{(\infty,1)}

\newcommand\io{(\infty,\omega)}

\newcommand\zo{(0,\omega)}

\DeclareMathOperator{\hstar}{\hat{\star}}

\newcommand{\costar}{\mathbin{\overset{co}{\star}}}
\newcommand{\fwedge}{\mathbin{\rotatebox[origin=c]{270}{$\gtrdot$}}}

\newcommand{\invamalg}{\mathbin{\rotatebox[origin=c]{180}{$\amalg$}}}
\DeclareMathOperator{\botimes}{\bar{\otimes}}

\DeclareMathOperator\Operatormark{mk}
\newcommand{\mk}{\Operatormark}

\DeclareMathOperator\Fun{Fun}

\DeclareMathOperator\End{End}

\DeclareMathOperator\cat{cat}

\DeclareMathOperator\R{R}

\newcommand\ocat{\omega\mbox{-$\cat$}}

\DeclareMathOperator\ocatB{\ocat_B}
\newcommand\iocat{(\infty,\omega)\mbox{-$\cat$}}

\newcommand\ncat[1]{#1\mbox{-$\cat$}}
\newcommand\incat[1]{(\infty,#1)\mbox{-$\cat$}}

\newcommand\zncat[1]{(0, #1)\mbox{-$\cat$}}

\DeclareMathOperator{\OperatorinfiniPsh}{Psh^\infty}

\DeclareMathOperator{\OperatorPsh}{Psh}

\DeclareMathOperator{\OperatormPsh}{mPsh}
\DeclareMathOperator{\OperatortPsh}{tPsh}
\newcommand\iPsh[1]{\OperatorinfiniPsh({#1})}

\newcommand\Psh[1]{\OperatorPsh({#1})}

\newcommand\tPsh[1]{\OperatortPsh({#1})}
\newcommand\tPshM[1]{{\OperatortPsh}_M({#1})}
\newcommand\mPsh[1]{\OperatormPsh({#1})}
\newcommand\mPshM[1]{{\OperatormPsh}_M({#1})}

\DeclareMathOperator{\OperatorSeg}{Seg}

\DeclareMathOperator{\OperatortSeg}{tSeg}
\DeclareMathOperator{\OperatormSeg}{mSeg}
\newcommand\Seg{\OperatorSeg}
\newcommand\mSeg{\OperatormSeg}
\newcommand\stratSeg{\OperatortSeg}

\DeclareMathOperator{\Sset}{\Psh{\Delta}}
\newcommand{\mSset}{\mPsh{\Delta}}
\newcommand{\stratSset}{\tPsh{\Delta}}

\usepackage{fancyhdr}
\usepackage{titlesec}
\usepackage{textcase}

\title{\Huge{The complicial model of  $(\infty,\omega)$-categories}}
\author{Félix Loubaton}
\date{}
\fancyhfoffset[RO,LE]{0.5cm}
\fancyhfoffset[LE,RO]{0.5cm}

\begin{document}

\maketitle
\dominitoc
\tableofcontents

\cleardoublepage
\phantomsection
\addcontentsline{toc}{chapter}{Introduction} 
\chapter*{Introduction}
%
%
%
%
%
%
%

A \textit{category} consists of a set of objects, and for any pair of objects $a, b$, a set of morphisms $\hom_C(a, b)$ equipped with composition operations satisfying associativity laws.

\textit{$\iun$-Categories} are a homotopical generalization of categories. Intuitively, they are defined similarly to categories, except that we replace sets of objects and morphisms with spaces of objects and morphisms, and the associativity and unit laws are no longer satisfied strictly but homotopically.

Thanks in particular to the work of Joyal (\cite{Joyal_Quasi-categories_and_Kan_complexes}) and Lurie (\cite{Lurie_Htt}), most of the important concepts and theorems of category theory now have their $\iun$-categorical analogues. These objects have become important tools in many areas of mathematics, including algebraic geometry, algebraic topology, and representation theory.

Another generalization of the notion of category is obtained by replacing the set of morphisms between two objects $a$ and $b$ with a category of morphisms between $a$ and $b$. These new objects are called \textit{$2$-categories}. By replacing sets of morphisms between two objects $a$ and $b$ with $(n-1)$-categories of morphisms between $a$ and $b$, one can define by induction the notion of \textit{$n$-category} for any integer $n$, and by limit, the notion of $\omega$-category.

For $n\in \Nb\cup \{\omega\}$, the notion of \textit{$(\infty,n)$-category} is obtained by making both of these generalizations simultaneously. These objects are now found in many areas, including derived algebraic geometry, where the $6$-functors formalism is expressed and manipulated using the theory of $(\infty,2)$-categories (\cite{gaitsgory2019study}), and in topological quantum field theory, where $(\infty,n)$-categories are essential for formulating and proving the cobordism hypothesis (\cite{Baez_Higher-dimensional_algebra_and_topological_quantum_field_theory}, \cite{Lurie_on_the_classification_of_topological_field_theories}, \cite{Grady_the_geometric_cobordism_hypothesis}, \cite{Calaque_a_note_on_the_category_of_cobordism}).
\vspace{1cm}

This text is devoted to the study of \textit{models of $(\infty,n)$-categories}. By this, we mean any model category whose associated $\iun$-category is $\incat{n}$. Among the known models of $(\infty,n)$-categories, we have for example Rezk's complete Segal $\Theta_n$-spaces, $n$-fold Segal spaces, and Segal $n$-categories (we refer to \cite{Barwick_on_the_unicity_of_the_theory_of_higher_categories} for a comprehensive presentation of these models and their equivalence). The common feature of all the models we have mentioned is that they rely on a globular combinatoric.

 \vspace{1cm}
The main result of this text is to demonstrate that \textit{complicial sets}, defined and extensively studied by Verity, are a model for $(\infty,n)$-categories. Unlike the other models, complicial sets rely on the combinatorics of the iterated lax cone, i.e. the combinatorics of simplicies.

Let's now try to explain the intuition underlying the definition of complicial sets. To do this, we must first introduce stratified simplicial sets.
A \textit{stratified simplicial set} is a pair $(K, tK)$ where $K$ is a simplicial set and $tK$ is a subset of the simplices of $K$ containing the degenerate simplices. A simplex in $tK$ is called \textit{marked}. We denote by $\stratSset$ the category of stratified simplicial sets.

\vspace{1cm}
Given a stratified simplicial set $(K, tK)$, we would like to view it as a "sort of $\omega$-category". The $0$-simplices correspond to objects, the $1$-simplices
\[\begin{tikzcd}
    a & b
    \arrow["u"{description}, no head, from=1-1, to=1-2]
\end{tikzcd}\]
to $1$-cells between $a$ and $b$, the $2$-simplices 
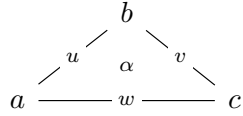
\[\begin{tikzcd}
    & b \\
    a && c
    \arrow["u"{description}, no head, from=2-1, to=1-2]
    \arrow["v"{description}, no head, from=1-2, to=2-3]
    \arrow[""{name=0, anchor=center, inner sep=0}, "w"{description}, no head, from=2-1, to=2-3]
    \arrow["\alpha"{description}, draw=none, from=1-2, to=0]
\end{tikzcd}\]
to $2$-cells with source $w$ and target the composite of $u$ and $v$, and more generally the $n$-simplices to $n$-cells whose source is a composition of odd faces and target a composition of even faces. Marked $n$-simplices are those whose corresponding $n$-cell is weakly invertible.

For this interpretation to be viable, some conditions on $(K, tK)$ must be imposed so that these "cells" "compose", and so that the marked simplices (i.e., "weakly invertible cells") satisfy a two out of three axiom. These conditions have been formalized by Verity in \cite{Verity_weak_complicial_sets_I}. A stratified simplicial set satisfying them is called a \textit{complicial set}.

In practice, these conditions are expressed using lifting properties. As these liftings are non-unique, compositions are not unique either. However, it can be shown that they are unique up to homotopy. Thus, a complicial set resembles a "weak $\omega$-category". Similarly, an \textit{$n$-complicial set} (i.e., a complicial set where all simplices of dimension strictly greater than $n$ are marked) is a kind of "weak $n$-category".

It was therefore conjectured (\cite{Street_algebra_of_orianted_simplexes}, \cite{Verity_a_complicial_compendium}, \cite{Barwick_on_the_unicity_of_the_theory_of_higher_categories}) that that for any $n\in \Nb\cup\{\omega\}$, the Verity model structure for $n$-complicial sets (whose fibrants-cofibrants are exactly $n$-complicial sets) was a model for $(\infty,n)$-categories.

However, for historical reasons, in the original definition of complicial sets, the marked simplices do not satisfy the 2 out 6 property, even though this property is satisfied by the equivalences in $(\infty,n)$-categories.\footnote{Given a category $C$, a class $W$ of morphisms in $C$ satisfies the 2 out of 6 property if, whenever $f$, $g$, and $g$ are composable morphisms such that $fg$ and $gh$ belong to $W$, then $f$, $g$, $h$, and $fgh$ all belong to $W$.} To remedy this issue, the notion of a \emph{saturated complicial set} was introduced, and a model structure on which they are the fibrant objects was introduced by Ozornova and Rovelli in \cite{Ozornova_model_structure_for_infini_n_categories}.
 Consequently, the preceding conjecture admits two versions: the Verity model structure on $n$-complicial sets models non-complete $(\infty,n)$-categories, and the the Ozornova--Rovelli model structure on saturated $n$-complicial sets models $(\infty,n)$-categories.

\vspace{1cm}

The case $n=1$ is was proven Verity in \cite{Verity_weak_complicial_sets_I}, and the case $n=2$ by Gagna, Harpaz, and Lanari in \cite{Gagna_on_the_equivallence_of_all_model_for_infini2_cat}. The goal of this text is to provide a positive answer to this conjecture in the general case, i.e. for any $n\in \Nb\cup\{\omega\}$.

\addcontentsline{toc}{subsection}{Summary of results} 
\subsection*{Summary of results}
\paragraph{Chapter 1.}
The first section is devoted to the definition of $\omega$-categories and of the category $\Theta$ of Joyal. We also show that the category $\Theta$ presents the category of $\omega$-categories, and we also exhibit an other presentation of this category (corollary \ref{cor:xi and ocat}).

The second section begins with a review of Steiner theory, which is an extremely useful tool for providing concise and computational descriptions of $\omega$-categories. Following Ara and Maltsiniotis, we employ this theory to define the Gray tensor product, denoted by $\otimes$, in $\omega$-categories. We then introduce the Gray operations, starting with the Gray cylinder $\uvar\otimes[1]$ which is the Gray tensor product with the directed interval $[1]:=0\to 1$. Then, we have the \textit{Gray cone}, the \textit{Gray $\circ$-cone} and the \textit{Gray op-cone}, denoted by $\uvar\star 1$, $1\costar \uvar$ and $1\star \uvar$, that send an $\omega$-category $C$ onto the following pushouts:
\[\begin{tikzcd}
	{C\otimes\{1\}} & {C\otimes[1]} & {C\otimes\{0\}} & {C\otimes[1]} & {\{0\}\otimes C} & {[1]\otimes C} \\
	1 & {C\star 1} & 1 & {1\costar C} & 1 & {1\star C}
	\arrow[from=1-4, to=2-4]
	\arrow[from=1-3, to=2-3]
	\arrow[from=2-3, to=2-4]
	\arrow[from=1-3, to=1-4]
	\arrow[from=1-2, to=2-2]
	\arrow[from=2-1, to=2-2]
	\arrow[from=1-1, to=2-1]
	\arrow[from=1-1, to=1-2]
	\arrow["\lrcorner"{anchor=center, pos=0.125, rotate=180}, draw=none, from=2-2, to=1-1]
	\arrow["\lrcorner"{anchor=center, pos=0.125, rotate=180}, draw=none, from=2-4, to=1-3]
	\arrow[from=1-5, to=2-5]
	\arrow[from=2-5, to=2-6]
	\arrow[from=1-5, to=1-6]
	\arrow[from=1-6, to=2-6]
	\arrow["\lrcorner"{anchor=center, pos=0.125, rotate=180}, draw=none, from=2-6, to=1-5]
\end{tikzcd}\]

We also present a formula that illustrates the interaction between the suspension and the Gray cylinder. As this formula plays a crucial role in this text, we provide its intuition at this stage.

 If $A$ is any $\omega$-category, the \textit{suspension} of $A$, denoted by $[A,1]$, is the $\omega$-category having two objects - denoted by $0$ and $1$- and such that 
$$\Hom_{[A,1]}(0,1) := A,~~~\Hom_{[A,1]}(1,0) := \emptyset,~~~\Hom_{[A,1]}(0,0)=\Hom_{[A,1]}(1,1):=\{id\}.$$
We also define $[1]\vee[A,1]$ as the gluing of $[1]$ and $[A,1]$ along the $0$-target of $[1]$ and the $0$-source of $[A,1]$. We define similarly $[A,1]\vee[1]$.
These two objects come along with \textit{whiskerings}:
$$\triangledown:[A,1]\to [1]\vee [A,1] ~~~~\mbox{and}~~~~ \triangledown:[A,1] \to [A,1]\vee [1]$$ 
that preserve the extremal points.

The $\omega$-category $[1]\otimes [1]$ is induced by the diagram:
\[\begin{tikzcd}
	00 & 01 \\
	10 & 11
	\arrow[from=1-1, to=2-1]
	\arrow[from=2-1, to=2-2]
	\arrow[from=1-1, to=1-2]
	\arrow[from=1-2, to=2-2]
	\arrow[shorten <=4pt, shorten >=4pt, Rightarrow, from=1-2, to=2-1]
\end{tikzcd}\]
and is then equal to the colimit of the following diagram: 
$$[1]\vee [1]\xleftarrow{\triangledown} [1]\hookrightarrow [[1],1]\hookleftarrow[1]\xrightarrow{\triangledown } [1]\vee [1].$$
The $\omega$-category $ [[1],1]\otimes [1]$ is induced by the diagram:
\[\begin{tikzcd}
	00 & 01 & 00 & 01 \\
	10 & 11 & 10 & 11
	\arrow[from=1-1, to=1-2]
	\arrow[""{name=0, anchor=center, inner sep=0}, from=1-1, to=2-1]
	\arrow[from=2-1, to=2-2]
	\arrow[""{name=1, anchor=center, inner sep=0}, from=1-2, to=2-2]
	\arrow[shorten <=4pt, shorten >=4pt, Rightarrow, from=1-2, to=2-1]
	\arrow[""{name=2, anchor=center, inner sep=0}, from=1-3, to=2-3]
	\arrow[from=1-3, to=1-4]
	\arrow[""{name=3, anchor=center, inner sep=0}, from=1-4, to=2-4]
	\arrow[shorten <=4pt, shorten >=4pt, Rightarrow, from=1-4, to=2-3]
	\arrow[""{name=4, anchor=center, inner sep=0}, curve={height=30pt}, from=1-1, to=2-1]
	\arrow[from=2-3, to=2-4]
	\arrow[""{name=5, anchor=center, inner sep=0}, curve={height=-30pt}, from=1-4, to=2-4]
	\arrow["{ }"', shorten <=6pt, shorten >=6pt, Rightarrow, from=0, to=4]
	\arrow["{ }"', shorten <=6pt, shorten >=6pt, Rightarrow, from=5, to=3]
	\arrow[shift left=0.7, shorten <=6pt, shorten >=8pt, no head, from=1, to=2]
	\arrow[shift right=0.7, shorten <=6pt, shorten >=8pt, no head, from=1, to=2]
	\arrow[shorten <=6pt, shorten >=6pt, from=1, to=2]
\end{tikzcd}\]
and is then equal to the colimit of the following diagram: 
 $$[1]\vee[[1],1]\xleftarrow{\triangledown} [[1]\otimes\{0\},1]\hookrightarrow[[1]\otimes[1],1]\hookleftarrow [[1]\otimes\{1\},1]\xrightarrow{\triangledown}[[1],1]\vee[1]$$
We prove a formula that combines these two examples:

\begin{itheorem}[\ref{theo:appendice formula for otimes}]
In the category of $\omega$-categories, there exists an isomorphism, natural in $A$, between $[A,1]\otimes[1]$ and the colimit of the following diagram
\[\begin{tikzcd}
	{[1]\vee[A,1]} & {[A\otimes\{0\},1]} & { [A\otimes[1],1]} & {[A\otimes\{1\},1]} & {[A,1]\vee[1]}
	\arrow["\triangledown"', from=1-2, to=1-1]
	\arrow[from=1-4, to=1-3]
	\arrow["\triangledown", from=1-4, to=1-5]
	\arrow[from=1-2, to=1-3]
\end{tikzcd}\]
\end{itheorem} 

We also provide similar formulas for the {Gray cone}, the {Gray $\circ$-cone} and the Gray op-cone.
\begin{itheorem}[\ref{theo:appendice formula for star}]
There is a natural identification between $1\costar [A,1]$ and the colimit of the following diagram
\[\begin{tikzcd}
	{[1]\vee[A,1]} & {[A,1]} & { [A\star 1,1]}
	\arrow["\triangledown"', from=1-2, to=1-1]
	\arrow[from=1-2, to=1-3]
\end{tikzcd}\]
There is a natural identification between $[A,1]\star 1$ and the colimit of the following diagram
\[\begin{tikzcd}
	{ [1\costar A,1]} & {[A,1]} & {[A,1]\vee[1]}
	\arrow[from=1-2, to=1-1]
	\arrow["\triangledown", from=1-2, to=1-3]
\end{tikzcd}\]
There is a natural identification between $1\star [A,1]$ and the colimit of the following diagram.
\[\begin{tikzcd}
	{ [1\star A,1]} & {[A,1]} & {[1]\vee[A,1]}
	\arrow[from=1-2, to=1-1]
	\arrow["\triangledown", from=1-2, to=1-3]
\end{tikzcd}\]
\end{itheorem}

\paragraph{Chapter 2.}
This chapter is dedicated to the study of \textit{Verity complicial sets}, defined and extensively studied by Verity (\cite{Verity_weak_complicial_sets_I})

One of the benefits of complicial sets is that they admit a simple definition of the Gray tensor product. Being strongly linked to $\omega$-categories by the Street nerve, they are also a privileged framework for stating and proving strictification results, as done in \cite{Ozornova_Fundamental_pushouts_of_n_complical_set}, \cite{Gagna_Nerves_and_cones_of_free_loop_free_omega_categories}, \cite{Ozornova_a_quillen_adjunction_between_globular_and_complicial} and \cite{Maehara_oriental_as_free_weak_omega_categories}. 
However, they do not interact \textit{a priori} well with the globular language. The goal of this chapter is to show that, with some computation, it is possible to have a globular point of view on theses objects.

The first section is a recollection of usual results and definitions about complicial sets. 
In the second section, we aim to prove an analogue of the formula given in \ref{theo:appendice formula for otimes} to the complicial setting.
We also have a suspension in this category, which is denoted by $X\mapsto \Sigma X$. Objects $[1]\fwedge \Sigma X$ and $\Sigma X\fwedge [1]$ are defined in \ref{cons:wedge}, but for now, we can suppose that they are fibrant replacements of respectively $[1]\coprod_{[0]}\Sigma X$ and $\Sigma X\coprod_{[0]}[1]$.
They come along with morphisms that are analogue to whiskerings, and that we also note by $\triangledown$: 
$$\triangledown:\Sigma X\to [1]\fwedge\Sigma X ~~~~\mbox{and}~~~~ \triangledown:\Sigma X\to\Sigma X\fwedge [1].$$ 
We then show the following theorem:
\begin{itheorem}[\ref{theo:interval_first_formula}]
There exists a zigzag of acyclic cofibrations, natural in $X$, between $(\Sigma X)\otimes [1]$ and the colimit of the following diagram:
 $$\Sigma X\fwedge [1]\xleftarrow{\triangledown} \Sigma (X\otimes\{0\}) \hookrightarrow \Sigma (X\otimes[1])\hookleftarrow \Sigma (X\otimes\{1\})\xrightarrow{\triangledown} [1]\fwedge \Sigma X.$$
\end{itheorem}
We also provide similar formulas for the \textit{Gray cone} and Gray \textit{$\circ$-cone}:
\begin{itheorem}[\ref{theo:cyl_formula}]
There exists a zigzag of acyclic cofibrations, natural in $X$, between $\Sigma X \star[0]$ and the colimit of the following diagram: 
$$ \Sigma X\fwedge [1]\leftarrow \Sigma X\to \Sigma([0]\costar X).$$
There exists a zigzag of acyclic cofibrations, natural in $X$, between  $[0]\costar \Sigma X$ and the colimit of the following diagram: 
$$\Sigma(X\star[0]) \leftarrow \Sigma X\to [1]\fwedge\Sigma X.$$
\end{itheorem}

The third section uses this formula and the strictification result of Gagna, Ozornova and	 Rovelli (\cite{Gagna_Nerves_and_cones_of_free_loop_free_omega_categories}) to demonstrate a criterion for detecting autoequivalences of complicial sets by their behavior on globes.
Indeed, in section \ref{section:Globular equivalences}, by iterating the suspension, we construct a globular object: 
\[\begin{tikzcd}
	{\Db_0} & {\Db_1} & {\Db_2} & {...}
	\arrow["{i_0^+}", shift left=2, from=1-1, to=1-2]
	\arrow["{i_1^+}", shift left=2, from=1-2, to=1-3]
	\arrow["{i_3^+}", shift left=2, from=1-3, to=1-4]
	\arrow["{i_0^-}"', shift right=2, from=1-1, to=1-2]
	\arrow["{i_1^-}"', shift right=2, from=1-2, to=1-3]
	\arrow["{i_3^-}"', shift right=2, from=1-3, to=1-4]
\end{tikzcd}\]
We the show:
\begin{itheorem}[\ref{theo:criterion_to_be_linked_to_identity}]
Let $i$ be a left Quillen endofunctor for the model category for complicial sets. Suppose that there exists a zigzag of weakly invertible natural transformations:
$$i(\Db_{\uvar}) \leftrightsquigarrow \Db_{\uvar}.$$
Then, there exists a zigzag of weakly invertible natural transformations between $i$ and $id$.
\end{itheorem} 
Proposition 15.10 of \cite{Barwick_on_the_unicity_of_the_theory_of_higher_categories} provides a similar result for models of $(\infty,n)$-categories.

\paragraph{Chapter 3.}

Results of Gagna, Harpaz et Lanari (\cite{Gagna_on_the_equivallence_of_all_model_for_infini2_cat}) states that saturated $2$-complicial sets are a model of $(\infty,2)$-categories 
The purpose of this chapter is to generalize this result to any $n\in \Nb\cup\{\omega\}$ but also to the non-saturated case.

To this end, we first address the more general problem of finding sufficient conditions on a model category $A$ to build a \textit{Gray cylinder} $C\mapsto I\otimes C$ and a \textit{Gray cone} $C\mapsto e\star C$ on stratified Segal precategories enriched in $A$. These two operations should be linked by the following homotopy cocartesian square
\[\begin{tikzcd}
	{\{0\}\otimes C} & {I\otimes C} \\
	e & {e\star C}
	\arrow[from=1-2, to=2-2]
	\arrow[from=1-1, to=2-1]
	\arrow[from=2-1, to=2-2]
	\arrow[from=1-1, to=1-2]
\end{tikzcd}\]
where $e$ is the terminal object. The conditions that $A$ has to fulfill are encapsulated in the notion of a \textit{Gray module} (paragraph \ref{defi:Gray module}). Thanks to the Gray cylinder and cone, we can show the following theorem:

\begin{itheorem}[\ref{theo:complicialGray module}]
If $A$ is a complicial Gray module, there is a Quillen adjunction between the model structure for complicial sets on stratified simplicial sets and stratified Segal precategories enriched in $A$, where the left adjoint sends $[n]$ to $e \star e \star \ldots \star e \star \emptyset$.
\end{itheorem}

By iterating the previous construction, we will be able to construct a Quillen adjunction between a model of (resp. non-complete) $(\infty,n)$-categories and stratified simplicial sets. We will then demonstrate the conjecture:

\begin{itheorem}[\ref{theo:theorem model 1}]
Let $n\in \mathbb{N}\cup \{\omega\}$. The model structure for $n$-complicial sets is a model of non-complete $(\infty,n)$-categories. The model structure for saturated $n$-complicial sets is a model of $(\infty,n)$-categories.
\end{itheorem}

\adjustmtc
\chapter{$\omega$-Categories and presheaves on $\Theta$}
\label{chapter:The category of zocategories}
\minitoc
\vspace{1cm}

%
%
%
%
%
%
%
%

\section{Basic constructions}
\label{chapter:Basica construciton preliminaire}
\subsection{$\omega$-Categories}
\label{section:zocategories}

\begin{definition}
 A \notion{globular set} is a presheaf on the \textit{category of globes} $\Gb$, which is the category induces by the diagram
\[\begin{tikzcd}
	{\Db_0} & {\Db_1} & {\Db_2} & {...}
	\arrow["{i_0^+}", shift left=2, from=1-1, to=1-2]
	\arrow["{i_1^+}", shift left=2, from=1-2, to=1-3]
	\arrow["{i_3^+}", shift left=2, from=1-3, to=1-4]
	\arrow["{i_0^-}"', shift right=2, from=1-1, to=1-2]
	\arrow["{i_1^-}"', shift right=2, from=1-2, to=1-3]
	\arrow["{i_3^-}"', shift right=2, from=1-3, to=1-4]
\end{tikzcd}\]
with the relations $i_n^{+} i_{n-1}^\epsilon = i_n^{-} i_{n-1}^\epsilon $ for any $n>0$ and $\epsilon \in \{+,-\}$. 
For any $n>k$ and $\epsilon \in \{+,-\}$, we also denote by $i^{\epsilon}_k$ the composite $\Db_{k} \xrightarrow{i^{\epsilon}_k} \Db_{k+1}\xrightarrow{f} \Db_n$ where $f$ is any map. These and the identity arrows are the only maps in the category $\Gb$.

If $X$ is a globular set, we denotes by $X_n$ the set $X(\Db_n)$. Its elements are called \wcsnotion{$n$-cells}{cell@$n$-cell}{for $\omega$-categories}. The $0$-cells are sometimes called \textit{objects}. The maps $X_n \to X_k$ induced by $i^\epsilon_k : \Db_k \to \Db_n$ is denoted by $\pi^\epsilon_k$.
\end{definition}

\begin{definition}
\label{defi:def of omega cat}
An \wcnotion{$\omega$-category}{category@$\omega$-category} is a globular set $X$ together with
\begin{enumerate}
\item operations of \textit{compositions}
\[ X_n\times_{X_k} X_n\to X_n ~~~(0\leq k<n) \]
which associate to two $n$-cells $(x,y)$ verifying $\pi_k^-(x) = \pi_k^+(y)$, a $n$-cells $x\circ_ky$,
\item as well as \textit{units}
\[X_n\to X_{n+1}\]
which associate to an $n$-cell $x$, a $(n+1)$-cell $\Ib_x$, 
\end{enumerate}
and satisfying the following axioms:
\begin{enumerate}

\item $\forall x \in X_n, \pi^\epsilon_n(\Ib_x) = x $.

\item $\pi^+_k (x \circ_n y) = \pi_k^{+}(x)$ and $\pi^-_k(x \circ_n y) = \pi_k^-(y)$ whenever the composition is defined and $k \leqslant n$.

\item $\pi^\epsilon_k (x \circ_n y) = \pi_k^{\epsilon}(x) \circ_n \pi^\epsilon_k(y)$ whenever the composition is defined and $k > n$.

\item $ x \circ_n \Ib_{\pi^-_n x} = x$ and $ \Ib_{\pi^+_n x} \circ_n x = x$.

\item $(x \circ_n y) \circ_n z = x \circ_n (y \circ_n z) $ as soon as one of these is defined.

\item If $k <n$
\[ (x \circ_n y) \circ_k ( z \circ_n w) = (x \circ_k z) \circ_n (y \circ_k w) \]
when the left-hand side is defined.

\end{enumerate}
A $n$-cell $a$ is \textit{non trivial}\index[notion]{non trivial $n$-cell} if is not in the image of the application $\Ib:X_{n-1}\to X_n$.

A \notion{morphism of $\omega$-categories} is a map of globular sets commuting with compositions and units. The category of $\omega$-categories is denoted by \textit{$\omegacat$}.
\end{definition}

\begin{definition}
\index[notion]{globe@$n$-globe!\textit{for $\omega$-categories}}
By abuse of notation, we also denote by \wcsnotation{$\Db_n$}{(da@$\Db_n$}{for $\omega$-categories} the $\omega$-category that admits for any $k<n$ only two $k$-non-trivial cells, denoted by $e_k^-$ and $e_k^+$, and a single $n$-non-trivial cell, denoted by $e_n$ verifying :
\[
\begin{array}{rcl}
\pi_l^{-}(e_k^\epsilon)= e_l^{-}&\pi_l^{+}(e_k^\epsilon)= e_l^{+}& \mbox{ for $l\leq k<n$}\\
\pi_l^{-}(e_n)= e_l^{-}&\pi_l^{+}(e_n)= e_l^{+}& \mbox{ for $l\leq n$}\\
\end{array}
\]
\end{definition}
Remark furthermore that the $\omega$-category $\Db_n$ represents $n$-cells, in the sense that $\Hom(\Db_n,C)\cong C_n$. We will not make the difference between $n$-cells and the corresponding morphism $\Db_n\to C$. 

\begin{definition}
The $\omega$-category $\partial\Db_n$ is obtained from $\Db_n$ by removing the $n$-cell $e_n$. We thus have a morphism
\[i_n: \partial\Db_n\to \Db_n.\]
Note that $\partial \Db_0 = \emptyset$.

\end{definition}
\begin{definition}
We say that an $\omega$-category $X$ is a \notion{polygraph} if it can be constructed from the empty $\omega$-category by freely adding arrows with specified source and target. That is if $X$ can be obtained as a transfinite composition $\emptyset = X_0 \to X_1 \to \dots \to X_i \to \colim X_i = X$ where for each $i$, the map $X_i \to X_{i+1}$ is a pushout of $\coprod_S \partial \Db_n \to \coprod_S \Db_{n+1}$.

 An arrow of a polygraph is said to be a \emph{generator} if it is one of the arrows that has been freely added at some stage.
 \end{definition}

Each cell in a polygraph can be written as a composite of generators or iterated unit of generators (not necessarily in a unique way). For a $n$-cell $f$, the set of generators of dimension $n$ that appear in such an expression (and even the number of times they appear) is the same for all such expressions. As a consequence, a composition of non trivial cells is always non trivial.

\begin{definition}
 \label{defi:dualities strict case}
 For any subset $S$ of $\Nb^*$, we define the functor $(\uvar)^S:\omegacat\to \omegacat$  sending a $\omega$-category $C$ to the category $C^S$ such that for any $n$, there is an isomorphism $C_n\to C_{n}^S$ that sends every $n$-cell $f$ to a cell $\overline{f}$ fulfilling
$$\pi_{n-1}^-(\overline{f})=\overline{\pi^+_{n-1}(f)}~~~~\pi_{n-1}^+(\overline{f})=\overline{\pi^-_{n-1}(f)}$$
if $i\in S$ and 
$$\pi_{n-1}^-(\overline{f})=\overline{\pi^-_{n-1}(f)}~~~~\pi_{n-1}^+(\overline{f})=\overline{\pi^+_{n-1}(f)}$$
if $i\notin S$.
These functors are called \notion{dualities} as they are inverse of themselves. Even if there are plenty of them, we will be interested in only a few of them. In particular, we have the \notionsym{odd duality}{((b60@$(\uvar)^{op}$} $(\uvar)^{op}$, corresponding to the set of odd integers, the \notionsym{even duality}{((b50@$(\uvar)^{co}$} $(\uvar)^{co}$, corresponding to the set of non negative even integers and the \notionsym{full duality}{((b80@$(\uvar)^{\circ}$} $(\uvar)^{\circ}$, corresponding to the set of all non negative integers. Eventually, we have equivalences
$$((\uvar)^{co})^{op}\sim (\uvar)^{\circ} \sim ((\uvar)^{op})^{co}.$$
\end{definition}

\begin{definition}
\label{defi:suspension zocat}
 Let $\Psh{\Gb}_{\bullet,\bullet}$ be the category of globular set with two distinguished points, i.e. of triples $(X,a,b)$ where $a$ and $b$ are elements of $X_0$.
Let $[\uvar,1]:\Gb\to \Psh{\Gb}_{\bullet,\bullet}$ be the functor sending $\Db_n$ on $(\Db_{n+1},\{0\},\{1\})$ and $i_n^{\epsilon}$ on $i_{n+1}^{\epsilon}$. This induces by left Kan extension a functor $[\uvar,1]:\Psh{\Gb}\to \Psh{\Gb}_{\bullet,\bullet}$ that we call the \textit{suspension}. We leave it to the reader to check that whenever $C$ has a structure of $\omega$-category, $[C,1]$ inherits one from it. This functor then induces a functor 
$$[\uvar,1]:\omegacat\to \omegacat$$
that we calls again the \snotionsym{suspension}{((d60@$[\uvar,1]$}{for $\omega$-categories}. Eventually, we denote by $i_0^-:\{0\}\to [C,1]$ (resp. $i_0^+:\{1\}\to [C,1]$) the morphism corresponding to the left point (resp. to the right point). For an integer $n$, we define by induction the functor $\Sigma^n:\Psh{\Gb}\to \Psh{\Gb}$\index[notation]{(sigma@$\Sigma^n$} with the formula:
$$\Sigma^0:= id ~~~~~\Sigma^{n+1}:=\Sigma^n[\uvar,1].$$
\end{definition}

\begin{definition}
Let $n$ be an integer. An \wcnotion{$n$-category}{category2@$n$-category} is an $\omega$-category whose cell of dimension strictly higher than $n$ are units. The category of $n$-categories is denoted by \wcnotation{$\ncat{n}$}{(ncat@$\ncat{n}$} and is the full subcategory of $\ocat$ whose objects are $n$-categories.
\end{definition}

\begin{definition}
 Let $n$ be a non null integer.
A $n$-cells $f:s\to t$ is an \notion{equivalence} if there exists $n$-cells $g:t\to s$ and $g':t\to s$ such that 
$$f\circ_{n-1} g=\Ib_t~~~~~~g\circ_{n-1} f=\Ib_s$$
\end{definition}
\begin{definition}
\label{defi:of complete cat}
A \notion{complete $\omega$-category} is an $\omega$-category whose only equivalences are the identities.
These objects are called \textit{Gaunt $\omega$-categories} in \cite{Barwick_on_the_unicity_of_the_theory_of_higher_categories} and \textit{rigid $\omega$-categories} in \cite{Rezk_a_cartesian_of_weak_n_categories}. Remark that $\omega$-categories are stable under suspensions and dualities.

Given $n\in \Nb\cup\{\omega\}$,
we denote by \wcnotation{$\ncat{n}^\comp$}{(ncatcomp@$\ncat{n}^\comp$}  the full subcategory of $\ncat{n}$ whose objects are the complete $n$-categories
\end{definition}

\begin{construction}
 Remark that the category $\ncat{n}$ is the localization of $\ocat$ along morphisms $\Db_{k}\to \Db_{n}$ for $k\geq n$. We then have for any $n$ an adjunction 
\[\begin{tikzcd}
	{\iota_n:\ncat{n}} & {\ocat:\tau_n}
	\arrow[""{name=0, anchor=center, inner sep=0}, shift left=2, from=1-2, to=1-1]
	\arrow[""{name=1, anchor=center, inner sep=0}, shift left=2, from=1-1, to=1-2]
	\arrow["\dashv"{anchor=center, rotate=-90}, draw=none, from=1, to=0]
\end{tikzcd}\]
The right adjoint is called the \wcsnotionsym{$n$-truncation}{(tau@$\tau_n$}{truncation@$n$-truncation}{for $\omega$-categories}.
\end{construction}

\begin{construction}
\label{cons:inteligent truncatoin}
For any $n$, we define the colimit preserving functor $\tau^i_n:\ocat\to \ncat{n}$, called the \snotionsym{intelligent $n$-truncation}{(taui@$\tau^i_n$}{for $\omega$-categories}, sending $\Db_k$ on $\Db_{\min(n,k)}$. The functor $\tau^i_n$ fits in an adjunction
\[\begin{tikzcd}
	{\tau^i_n:\ocat} & {\ncat{n}:\iota_n}
	\arrow[""{name=0, anchor=center, inner sep=0}, shift left=2, from=1-1, to=1-2]
	\arrow[""{name=1, anchor=center, inner sep=0}, shift left=2, from=1-2, to=1-1]
	\arrow["\dashv"{anchor=center, rotate=-90}, draw=none, from=0, to=1]
\end{tikzcd}\]
\end{construction}
\begin{notation}
We will identify objects of $\ncat{n}$ with their image in $\ocat$ and we will then also note by $\tau_n$ and $\tau^i_n$ the composites $\iota_n\tau^i_n$ and $\iota_n\tau^i_n$.
\end{notation}

\begin{remark}
The family of truncation functor induces a sequence 
$$...\to \ncat{n+1}\xrightarrow{\tau_{n}} \ncat{n}\to...\to \ncat{1}\xrightarrow{\tau_{0}}\ncat{0}.$$
The canonical morphism
$$\ocat\to \lim_{n:\Nb}\ncat{n},$$
that sends an $\omega$-category $C$ to the sequence $(\tau_n C,\tau_n\tau_{n+1}C\cong \tau_n C)$, has an inverse given by the functor
$$\colim_{\Nb}:\lim_{n:\Nb}\ncat{n}\to \ocat$$
that sends a sequence $(C_n, \tau_{n}C_{n+1}\cong C_n) $ to the colimit of the induced sequence:
$$\iota_0C_0\to \iota_1C_1\to...\to \iota_nC_n\to...$$	
 We then have an equivalence 
$$\ocat\cong \lim_{n:\Nb}\ncat{n}.$$
\end{remark}
\subsection{Elegant Reedy Category}

In the next sections, we will identify two dense subcategories of $\ocat$. One of their important features is to be elegant Reedy categories, a notion that we now introduce.

\begin{definition}
 \label{defi:reedy}
A \notion{Reedy category} is a small category $A$ equipped with two subcategories $A_+$, $A_-$ and a \textit{degree} function $d:ob(A)\to \Nb$ such that: 
\begin{enumerate}
\item for every non-identity morphism $f:a\to b$, if $f$ belongs to $A_-$, then $d(a)>d(b)$, and if $f$ belongs to $A_+$, then $d(a)<d(b)$.
\item every morphism of $A$ uniquely factors as a morphism of $A_-$ followed by a morphism of $A_+$.
\end{enumerate}

A Reedy category $A$ is \wcnotion{elegant}{elegant Reedy category} if for any presheaf $X$ on $A$, for any $a\in A$ and any $c\in X(a)$, there exists a unique morphism $f:a\to a'\in A_{-}$ and a unique non-degenerate object $c'\in X(a')$ such that $c=X(f)(c')$.
\end{definition}

\begin{prop}
\label{prop:elelangat stable by slice}
Let $X$ be a presheaf on an elegant Reedy category $A$. The category $A_{/X}$ is an elegant Reedy category.
\end{prop}

\begin{proof}
We have a canonical projection $\pi:A_{/X}\to A$. A morphism is positive (resp. negative) if its image by $\pi$ is. The degree of an element $c$ of $A_{/X}$ is the degree of $\pi(c)$. We leave it to the reader to check that this endows $A_{/X}$ with a structure of Reedy category.

The fact that $A_{/X}$ is elegant is a direct consequence of the isomorphism $\Psh{A_{/X}}\cong \Psh{A}_{/X}$.
\end{proof}

\begin{definition}
\label{defi: of Delta[..]}
Given a small category $A$, we denote $\Delta[A]$ the category fitting in the pushout:
\[\begin{tikzcd}
	{\{[0]\}\times A} & {\Delta\times A} \\
	{[0]} & {\Delta[A]}
	\arrow[from=1-1, to=1-2]
	\arrow[from=1-1, to=2-1]
	\arrow[from=1-2, to=2-2]
	\arrow[from=2-1, to=2-2]
	\arrow["\lrcorner"{anchor=center, pos=0.125, rotate=180}, draw=none, from=2-2, to=1-1]
\end{tikzcd}\]
Given an element $a\in A$, and an integer $n$, we will denote $[a,n]$\ssym{((g10@$[a,n]$}{for $\Delta[A]$} the element of $\Delta[A]$ that is the image of $(n,a)$ by the canonical functor $\Delta\times A\to \Delta[A]$. \sym{(deltaa@$\Delta[A]$}
\end{definition}

\begin{prop}
\label{prop:delta[B] is reedy}
Let $A$ be an elegant Reedy category. Then the category $\Delta[A]$ also has the structure of an elegant Reedy category.
\end{prop}

\begin{proof}
Remark first that $\Hom_{\Delta[A]}([a,n],[b,m])$ fits in the following cocartesian square:
\[\begin{tikzcd}
	{\coprod_{k\leq m}\Hom_{A}(a,b)\times \Hom_{\Delta}([n],\{k\})} & {\Hom_{A}(a,b)\times \Hom_{\Delta}([n],[m])} \\
	{\coprod_{k\leq m} \Hom_{\Delta}([n],\{k\})} & {\Hom_{\Delta[A]}([a,n],[b,m])}
	\arrow[from=1-2, to=2-2]
	\arrow[from=1-1, to=2-1]
	\arrow[from=2-1, to=2-2]
	\arrow[from=1-1, to=1-2]
\end{tikzcd}\]
We then define the degree functor $ob(\Delta[A])\to \Nb$ by the formula $d([b,n])=d(b)d(n)$. The subcategory $(\Delta[A])_{+}$ is the image of $\Delta_+\times A_+$, and the subcategory $(\Delta[A])_{-}$ is the image of $\Delta_-\times A_-$. 

We recall that we suppose that the Reedy category $A$ is elegant. Let $X$ be a presheaf on $\Delta[A]$, $[a,n]$ an element of $\Delta[A]$, $[f,g]:[a,n]\to [a',n']$ and $[h,i]:[a,n]\to [a',n']$ two negative morphisms, an element $x$ of $X([a,n])$, two non-degenerate elements $y\in X([a',n'])$ and $z\in X([a'',n''])$ such that $[f,g]^*y=x$, $[h,i]^*z=x$.

We suppose first that $n\neq 0$. We denote by $\pi:A\times \Delta\to \Delta[A]$ the canonical projection and 
$$\pi^*:\Psh{\Delta[A]}\to \Psh{\Delta\times A}$$ the functor obtained by precocomposing. Remark that for any $a,n$, $(\pi^*X)(a,n)=X([a,n])$. Furthermore, we have again equalities $(f,g)^*y=x$, $(h,i)^*z=x$. As $\Delta\times A$ is Reedy elegant, this implies that $f=h$, $g=i$ and $y=z$. 

If $n=0$, then $[f,g]$ and $[h,i]$ are the identity, and we directly have $y=z$. The Reedy category $\Delta[A]$ is then elegant.
\end{proof}

\subsection{The category $\Theta$}
\label{subsection:the category theta}

\begin{definition}
 Let $n$ be a non-negative integer and $\textbf{a}:=\{a_0,a_1,...,a_{n-1}\}$ a sequence of $\omega$-categories. We denote \wcnotation{$[\textbf{a},n]$}{((g00@$[\textbf{a},n]$} the colimit of the following diagram
\[\begin{tikzcd}
	& 1 && 1 && 1 \\
	{[a_0,1]} && {[a_1,1]} && {...} && {[a_{n-1},1]}
	\arrow["{i_0^+}"', from=1-2, to=2-1]
	\arrow["{i_0^-}", from=1-2, to=2-3]
	\arrow["{i_0^+}"', from=1-4, to=2-3]
	\arrow["{i_0^-}", from=1-4, to=2-5]
	\arrow["{i_0^+}"', from=1-6, to=2-5]
	\arrow["{i_0^-}", from=1-6, to=2-7]
\end{tikzcd}\]
where $[\uvar,1]$ is the suspension functor defined in \ref{defi:suspension zocat}.
\end{definition}

\begin{definition}
 \label{defi:les sommes glob}
We define \wcnotation{$\Theta$}{(theta@$\Theta$} as the smallest full subcategory of $\ocat$ that includes the terminal $\omega$-category $[0]$, and such that for any non-negative integer $n$, and any finite sequence $\textbf{a}:=\{a_0,a_1,...,a_{n-1}\}$ of objects of $\Theta$, it includes the $\omega$-category $[\textbf{a},n]$. Objects of $\Theta$ are called \notion{globular sums}.
\end{definition}

\begin{remark}
\label{remark:morphism of theta}
A morphism $g:[\textbf{a},n]\to [\textbf{b},m]$ is exactly the data of a morphism $f:[n]\to [m]$, and for any integer $i$, a morphism
$$a_i\to \prod_{f(i)\leq k< f(i+1)}b_k.$$
\end{remark}

\begin{example}
\label{exemple:of globular sum}
For any $n$, $\Db_n$ is a globular sum. The $\omega$-category induced by the $\omega$-graph 
\[\begin{tikzcd}
	\bullet & \bullet & \bullet
	\arrow[from=1-1, to=1-2]
	\arrow[""{name=0, anchor=center, inner sep=0}, curve={height=-24pt}, from=1-2, to=1-3]
	\arrow[""{name=1, anchor=center, inner sep=0}, curve={height=24pt}, from=1-2, to=1-3]
	\arrow[""{name=2, anchor=center, inner sep=0}, from=1-2, to=1-3]
	\arrow[shorten <=3pt, shorten >=3pt, Rightarrow, from=0, to=2]
	\arrow[shorten <=3pt, shorten >=3pt, Rightarrow, from=2, to=1]
\end{tikzcd}\]
is a globular sum.
\end{example}

\begin{definition}
 For a globular sum $a$ and an integer $n$, we define $[a,n]:=[\{a,a,\ldots,a\},n]$.\ssym{((g10@$[a,n]$}{for $\omega$-categories} For a sequence of integers $\{n_0,\ldots,n_k\}$ and a sequence of globular sums $\{a_0,\ldots,a_k\}$, we define \wcsnotation{$[a_0,n_0]\vee[a_1,n_1]\vee\ldots\vee [a_k,n_k]$}{((g20@$[a_0,n_0]\vee[a_1,n_1]\vee\ldots\vee [a_k,n_k]$}{for $\Theta$} as the globular sum $[\{a_0,\ldots,a_1,\ldots,a_k,\ldots\},n_0+n_1+\ldots+n_k]$.

We denote by $[0]$ the terminal $\omega$-category, and $[n]$ the globular sum $[[0],n]$. This induces a fully faithful functor $\Delta\to \Theta$ sending $[n]$ onto $[n]$.
\end{definition}

\begin{definition}
 We define by induction the \wcnotion{dimension}{dimension of a globular sum} of a globular sum $a$, denoted by $|a|$. The dimension of $[0]$ is $0$, and the dimension of $[\textbf{a},n]$ is the maximum of the set $\{|a_k|+1\}_{k< n}$. We denote by \wcnotation{$\Theta_n$}{(thetan@$\Theta_n$} the full subcategory of $\Theta$ whose objects are the globular sums of dimension less than or equal to $n$. We set by convention $\Theta_\omega:=\Theta$.
\end{definition}

\begin{prop}[Berger, Bergner-Rezk]
\label{prop:theta is elegan reedy}
For any $n\in \Nb\cup\{\omega\}$, the category $\Theta_n$ is an elegant Reedy category.

A morphism $g:[\textbf{a},n]\to [\textbf{b},m]$ is \wcnotion{degenerate}{degenerate morphism of $\Theta$} (i.e., a morphism of $\Theta_{-}$) if the corresponding morphism $f:[n]\to [m]$ is a degenerate morphism of $\Delta$, and for any $i<n$ and any $f(i)\leq k<f(k+1)$, the corresponding morphism $a_i\to b_k$ is degenerate. Furthermore, a morphism is degenerate if and only if it is an epimorphism in $\Psh{\Theta}$.

A morphism is in $\Theta^+$ if and only if it is a monomorphism in $\Psh{\Theta}$.
\end{prop}

\begin{proof}
The Reedy structure is a consequence of lemma 2.4 of \cite{Berger_a_cellular_nerve}. The fact that for any $n<\omega$, $\Theta_n$ is elegant is stated in \cite[corollary 4.5.]{Bergner_reedy_category_and_the_theta_construction}. As for any $n<\omega$, the inclusion $\Theta_n\to \Theta$ preserves strong pushouts, the characterization of elegant Reedy categories given by \cite[proposition 3.8.]{Bergner_reedy_category_and_the_theta_construction} implies that $\Theta$ is also elegant.
\end{proof}

\begin{definition}
\label{defi:algebraic and globular}
We recall that a morphism $g:[\textbf{a},n]\to [\textbf{b},m]$ is exactly the data of a morphism $f:[n]\to [m]$, and for any integer $i$, a morphism
$$a_i\to \prod_{f(i)\leq k< f(i+1)}b_k.$$
The morphism $g$ is \textit{globular} \index[notion]{globular morphism} if for any $k<n$, $f(k+1)=f(k)+1$ and the morphism $a_k\to b_k$ is globular. The morphism $g$ is \wcnotion{algebraic}{algebraic morphism of $\Theta$} if it cannot be written as a composite $ig'$ where $i$ is a globular morphism.
\end{definition}

\begin{example}
\label{example:algebraic morphism}
The morphism 
\[\begin{tikzcd}
	\bullet & \bullet & \bullet && \bullet & \bullet & \bullet
	\arrow[from=1-5, to=1-6]
	\arrow[""{name=0, anchor=center, inner sep=0}, curve={height=-24pt}, from=1-6, to=1-7]
	\arrow[""{name=1, anchor=center, inner sep=0}, curve={height=24pt}, from=1-6, to=1-7]
	\arrow[""{name=2, anchor=center, inner sep=0}, from=1-6, to=1-7]
	\arrow[""{name=3, anchor=center, inner sep=0}, from=1-2, to=1-3]
	\arrow[""{name=4, anchor=center, inner sep=0}, curve={height=-24pt}, from=1-2, to=1-3]
	\arrow[from=1-1, to=1-2]
	\arrow[shorten <=14pt, shorten >=14pt, maps to, from=1-3, to=1-5]
	\arrow[shorten <=3pt, shorten >=3pt, Rightarrow, from=0, to=2]
	\arrow[shorten <=3pt, shorten >=3pt, Rightarrow, from=2, to=1]
	\arrow[shorten <=3pt, shorten >=3pt, Rightarrow, from=4, to=3]
\end{tikzcd}\]
is globular. This is not the case for the morphism
\[\begin{tikzcd}
	\bullet & \bullet & \bullet && \bullet & \bullet & \bullet
	\arrow[from=1-5, to=1-6]
	\arrow[""{name=0, anchor=center, inner sep=0}, curve={height=-24pt}, from=1-6, to=1-7]
	\arrow[""{name=1, anchor=center, inner sep=0}, curve={height=24pt}, from=1-6, to=1-7]
	\arrow[""{name=2, anchor=center, inner sep=0}, from=1-6, to=1-7]
	\arrow[""{name=3, anchor=center, inner sep=0}, curve={height=24pt}, from=1-2, to=1-3]
	\arrow[""{name=4, anchor=center, inner sep=0}, curve={height=-24pt}, from=1-2, to=1-3]
	\arrow[from=1-1, to=1-2]
	\arrow[shorten <=14pt, shorten >=14pt, maps to, from=1-3, to=1-5]
	\arrow[shorten <=3pt, shorten >=3pt, Rightarrow, from=0, to=2]
	\arrow[shorten <=3pt, shorten >=3pt, Rightarrow, from=2, to=1]
	\arrow[shorten <=6pt, shorten >=6pt, Rightarrow, from=4, to=3]
\end{tikzcd}\]
that sends the $2$-cell of the left globular sum to the $1$-composite of the two $2$-cells of the right globular sum.
\end{example}

\begin{prop}[{\cite[Proposition 3.3.10]{Ara_thesis}}]
\label{prop:algebraic ortho to globular}
Every morphism in $\Theta$ can be factored uniquely into an algebraic morphism followed by a globular morphism.
\end{prop}

\begin{remark}

Globular morphisms belong to $\Theta_+$ (and so morphisms of $\Theta_-$ are algebraic), but the converse is false. For example, the second morphism of example \ref{example:algebraic morphism} is not globular but belongs to $\Theta_+$. We then have two different factorizations on $\Theta$: the one coming from the Reedy elegant structure, and the one given in proposition \ref{prop:algebraic ortho to globular}.
\end{remark}

\begin{definition}
\label{defi:definition of W}
The suspension functor $[\uvar,1]:\Theta\to \Theta$ induces by left Kan extension a functor 
$$[\uvar,1]:\Psh{\Theta}\to \Psh{\Theta}.$$
We define by induction on $a$ a $\Theta$-presheaf \wcnotation{$\Sp_a$}{(sp@$\Sp_{a}$} and a morphism $\Sp_a\to a$. 
If $a$ is $[0]$, we set $\Sp_{[0]}:=[0]$. For $n>0$, we define 
$\Sp_{[\textbf{a},n]}$ as the set-valued presheaf on $\Theta$ obtained as the colimit of the diagram
\[\begin{tikzcd}
	& 1 && 1 && 1 \\
	{[ \Sp_{a_0},1]} && {[\Sp_{a_1},1]} && \cdots && {[\Sp_{a_{n-1}},1]}
	\arrow["{i_0^-}", from=1-2, to=2-3]
	\arrow["{i_0^+}"', from=1-4, to=2-3]
	\arrow["{i_0^+}"', from=1-2, to=2-1]
	\arrow["{i_0^-}", from=1-6, to=2-7]
	\arrow["{i_0^-}", from=1-4, to=2-5]
	\arrow["{i_0^+}"', from=1-6, to=2-5]
\end{tikzcd}\]
We define \wcnotation{$E^{eq}$}{(eeq@$E^{eq}$} as the set-valued presheaf on $\Delta$ obtained as the colimit of the diagram
\[\begin{tikzcd}
	& {[1]} && {[1]} && {[1]} \\
	{[0]} && {[2]} && {[2]} && {[0]}
	\arrow[from=1-2, to=2-1]
	\arrow["{d^1}", from=1-2, to=2-3]
	\arrow["{d^1}"', from=1-6, to=2-5]
	\arrow["{d^0}"', from=1-4, to=2-3]
	\arrow["{d^2}", from=1-4, to=2-5]
	\arrow[from=1-6, to=2-7]
\end{tikzcd}\]
For any integer $n$, the functor $\Sigma^n:\Theta\to \Theta$, which is the $n$-iteration of $[\uvar,1]$, induces by left Kan extension a functor 
$$\Sigma^n:\Psh{\Theta}\to \Psh{\Theta}.$$
 We define two sets of morphisms of $\Psh{\Theta}$:\sym{(w@$\W$}\sym{(wsat@$\W^\sat$}
$$\W := \{\Sp_a\to a,~a\in\Theta\}~~~~\W^\sat:= \{\Sigma^n E^{eq}\to \Db_n,~n\geq 0\}$$
For any $n$, we also define $$\W_n:= \W\cap \Theta_n.~~~~ \W_n^\sat:= \W^{\sat}\cap \Theta_n.$$
\end{definition}

\begin{definition}
Let $C$ be a category and $S$ a set of monomorphisms. A morphism $f: x \to y$ is \wcnotion{$S$-local}{Slocal@$S$-local} if it has the unique right lifting property against morphisms of $S$. An object $x$ is \textit{$S$-local} if $x \to 1$ is {$S$-local}, or equivalently, if for any $i: a \to b \in S$, the induced functor $\Hom(i, x): \Hom(b, x) \to \Hom(a, x)$ is an isomorphism.
\end{definition}

\begin{construction}
Let $C$ be a presentable category and $S$ a set of monomorphisms with small codomains. We define \textit{$C_{S}$} as the full subcategory of $C$ composed of $S$-local objects. The theorem 4.1 of \cite{BOUSFIELD1977207} implies that $\iota:C_S\to C$ is part of an adjunction
\[\begin{tikzcd}
	{\Fb_S:C} & {C_S:\iota}
	\arrow[""{name=0, anchor=center, inner sep=0}, shift left=2, from=1-1, to=1-2]
	\arrow[""{name=1, anchor=center, inner sep=0}, shift left=2, from=1-2, to=1-1]
	\arrow["\dashv"{anchor=center, rotate=-90}, draw=none, from=0, to=1]
\end{tikzcd}\]
where $\Fb_S:C\to C_S$ is the localization of $C$ by the smallest class of morphisms containing $S$ and stable under composition and colimits.
\end{construction}

\begin{theorem}[Berger]

\label{theo:theta and ocat}
Let $n\in \Nb\cup\{\omega\}$. 
The functor $\Psh{\Theta_n}\to \ncat{n}$ defined as the left Kan extension of the canonical inclusion $\Theta_n\to \ncat{n}$ induces isomorphisms 
$$\Psh{\Theta_n}_{\W_n}\cong \ncat{n}~~~~~\Psh{\Theta_n}_{\W_n\cup \W_n^{\sat}}\cong \ncat{n}^\comp$$
\end{theorem}

\begin{proof}
The first equivalence is \cite[corollary 12.3]{Barwick_on_the_unicity_of_the_theory_of_higher_categories}; the second follows directly.
\end{proof}

\subsection{The category $\Xi$}
\label{section:Xi}

We now introduce a new category, denoted $\Xi$, along with subcategories denoted $\Xi_n$ consisting of elements of dimension less than or equal to $n$. We recall that given a category $A$, we denote $\Delta[A]$ the category fitting in the pushout:

\[\begin{tikzcd}
	{\{[0]\}\times A} & {\Delta\times A} \\
	{[0]} & {\Delta[A]}
	\arrow[from=1-1, to=1-2]
	\arrow[from=1-1, to=2-1]
	\arrow[from=1-2, to=2-2]
	\arrow[from=2-1, to=2-2]
	\arrow["\lrcorner"{anchor=center, pos=0.125, rotate=180}, draw=none, from=2-2, to=1-1]
\end{tikzcd}\]

\begin{definition}
We define by induction on $n\in \Nb\cup\{\omega\}$ the category $\Xi_n$ by setting $\Xi_0:=\{[0]\}$, $\Xi_{n+1}:=\Delta[\Xi_n]$, and $\Xi:=\Xi_{\omega}$ as the limit of the sequence $\Xi_0\to \Xi_1\to  \Xi_2...  $. 

Objects of $\Xi_n$ are then of the shape $[a,m]$ where $m$ is an integer and $a$ is an element of $\Xi_{n-1}$. We will denote simply $[m]$ the element $[[0],m]$, where $[0]$ is the terminal element. \sym{(xi@$\Xi$}
\end{definition}

\begin{prop}
For any $n\in \Nb\cup\{\omega\}$, the category $\Xi_n$ is Reedy elegant.
\end{prop}

\begin{proof}
This follows  by induction using proposition  \ref{prop:delta[B] is reedy}.
\end{proof}

\begin{definition}
\label{defi:globes for xi}
For any integer $k\leq n$, we define the element $\Db_k$, called the \wcsnotionsym{$n$-globe}{(da@$\Db_n$}{globe@$n$-globe}{for $\Xi$} of $\Xi_n$ by setting $\Db_0:=[0]$ and $\Db_k:=[\Db_{k-1},1]$ for $k>0$. This defines a functor $\Gb_{\leq n}\to \Xi_n$.
\end{definition}

\begin{definition}
\label{definition: of Mn}\sym{(m@$\M$}\sym{(ma@$\M^\sat$}
We define the set of maps $\M_n$ of $\Psh{\Xi_n}$ by setting $\M_0:=\{[0]\to [0]\}$, 
$$\M_n:= \{[a,\Sp_n]\to [a,n],~a\in \Xi_{n-1}\}\cup \{[f,1],~ f\in \M_{n-1}\}$$
and $\M:=\M_\omega:= \cup_n\M_n$. 

We also define the set $\M^\sat_n$ by setting $\M^\sat_1:=\{E^{eq}\to [1]\}$, $$\M^\sat_n:= \{E^{eq} \to [0]\}\cup \{[f,1],~ f\in \M^\sat_{n-1}\}$$
and $\M^\sat:=\M^\sat_\omega:= \cup_n\M_n^\sat.$ 
\end{definition}

\begin{construction}
\label{cons: of the functor $j$}
Let $n\in\Nb\cup \{\omega\}$. We have a canonical functor 
$j:\Xi_n\to \Theta_n$ sending $[a,n]$ onto the globular sum $[\{a,a,\ldots,a\},n]$.
This induces an adjunction
\[\begin{tikzcd}
	{j_!:\Psh{\Xi_n}} & {\Psh{\Theta_n}:j^*}
	\arrow[""{name=0, anchor=center, inner sep=0}, shift left=2, from=1-1, to=1-2]
	\arrow[""{name=1, anchor=center, inner sep=0}, shift left=2, from=1-2, to=1-1]
	\arrow["\dashv"{anchor=center, rotate=-90}, draw=none, from=0, to=1]
\end{tikzcd}\]
\end{construction}
\begin{remark}
We can  easely check that $j_!$ and $j^*$ preserve the globes.
\end{remark}
We will now study the link between presheaves on $\Theta$ and on $\Xi$. To this extent, we will use the notion of \textit{precocomplete class}, defined in \ref{defi:precocomplet}. Given a set of morphisms $S$, we denote $\overline{S}$ the smallest precocomplete class contaiining $S$.

\begin{theorem}
\label{theo:unit and counit are in Sigma}
Let $n\in \Nb\cup\{\omega\}$. We have inclusions
$$j_!(\M_n)\subset \mbox{$\W_n$},~~~~~~~~~j^*(\mbox{$\W_n$})\subset \overline{\M_n},$$
and equalities
$$j_!(\M^\sat_n)=\W^\sat_n, ~~~~~~~~~ j^*(\W^\sat_n)=\M^\sat_n.$$

Moreover, for any $a\in \Xi$ and $b\in \Delta[\Theta]$, the morphism $j_!j^*a\to a$ and $b\to j^*j_! b$ are respectively in $\overline{\W}$ and $\overline{\Sigma}$.
\end{theorem}

We directly give a corollary of this theorem, which states that $\Xi_n$ is also a dense subcategory of $\ncat{n}$.

\begin{cor}
\label{cor:xi and ocat}
Let $n\in \Nb\cup\{\omega\}$. 
The functor $\Psh{\Xi_n}\to \ncat{n}$ defined as the left Kan extension of the canonical inclusion $\Xi_n\xrightarrow{j} \Theta_n \to \ncat{n}$ induces isomorphisms 
$$\Psh{\Xi_n}_{\M_n}\cong \ncat{n}~~~~~\Psh{\Xi_n}_{\M_n\cup \M_n^{\sat}}\cong \ncat{n}^\comp$$
\end{cor}

\begin{proof}
The inclusion $j_!(\M_n)\subset \Sigma_n$ induces a left adjoint $\Psh{\Xi_n}_{\M_n}\to \Psh{\Theta_n}_{\W_n}$. Now note that the class of morphisms sent to an identity by the functors $\Psh{\Xi_n}\to \Psh{\Xi_n}_{\M_n}$ and $\Psh{\Theta_n}\to \Psh{\Theta_n}_{\W_n}$ are saturated. The result then directly follows from the theorems \ref{theo:theta and ocat} and \ref{theo:unit and counit are in Sigma}.
\end{proof}

To show theorem \ref{theo:unit and counit are in Sigma}, we will proceed by induction on $n$ and consider the category $\Delta[\Theta_n]$, which fits between $\Xi_n$ and $\Theta_n$.

\begin{definition}
\label{defi:defi of delta theta}
Let $n\in\Nb\cup \{\omega\}$.
We define a set of morphisms of $\Psh{\Delta[\Theta_n]}$:\sym{(m@$\M$} 
$$
\begin{array}{c}
\Sigma_{n+1} := \{[a,\Sp_n]\to [a,n],~a:\Theta\}\cup\{[f,1],~f\in \W_n\}\\
\end{array}$$
\end{definition}

\begin{definition}
We denote by 
$$[\uvar,\uvar]: \Psh{\Theta}\times \Psh{\Delta}\to \Psh{\Delta[\Theta]}$$
the left Kan extension of the functor $\Theta\times \Delta\to \Psh{\Delta[\Theta]}$ sending $(a,n)$ to $[a,n]$.
For an integer $n$, we denote
$$[\uvar,n]:\Psh{\Theta}^n\to \Psh{\Theta}$$ 
the left Kan extension of the functor 
$\Theta^n\to\Psh{\Theta}$ sending $\textbf{a}:=\{a_1,...,a_n\}$ to $[\textbf{a},n]$. Eventually, we define 
$$[\uvar,d^0\cup d^n]:\Psh{\Theta}^n\to \Psh{\Theta}$$ 
the left Kan extension of the functor 
$\Theta^n\to\Psh{\Theta}$ sending $\textbf{a}:=\{a_1,...,a_n\}$ to the colimit of the span.
$$[\{a_0,...,a_{n-2}\},{n-1}]\leftarrow [\{a_1,...,a_{n-2}\},{n-2}]\to [\{a_1,...,a_{n-1}\},{n-1}]$$
\end{definition}

\begin{lemma}
\label{lemma:the functor [] preserves classes}
The image of $\overline{\W}\times \overline{\W_1}$ by the functor $[\uvar,\uvar]:\Psh{\Theta}\times \Psh{\Delta}\to \Psh{\Delta[\Theta]}$ is included in $\overline{\W}$.
\end{lemma}

\begin{proof}
As $[\uvar,\uvar]$ preserves colimits and monomorphisms, it is enough to show that for any pair $f,g\in \W\times \W_1$, $[f,g]$ is in $\W$, which is obvious.
\end{proof}

\begin{lemma}
\label{lemma:i etoile of W is in M 0}
For any globular sum $v$, and any integer $n$,
the morphism $[v,d^0\cup d^n]\cup[\partial v,n]\to [v,n]$ appearing in the diagram
\[\begin{tikzcd}
	{[\partial v,d^0\cup d^n]} & {[v,d^0\cup d^n]} \\
	{[\partial v,n]} & {[\partial v,n]\cup[v,d^0\cup d^n]} \\
	&& {[ v,n]}
	\arrow[from=1-1, to=2-1]
	\arrow[from=1-2, to=2-2]
	\arrow[from=2-2, to=3-3]
	\arrow[curve={height=18pt}, from=2-1, to=3-3]
	\arrow[curve={height=-18pt}, from=1-2, to=3-3]
	\arrow[from=1-1, to=1-2]
	\arrow[from=2-1, to=2-2]
\end{tikzcd}\]
is in $\overline{\Sigma}$.
\end{lemma}

\begin{proof}
Let $a$ be a globular sum. Remark that the morphism $[a,\Sp_n]\to [a,d^0\cup d^n]$ is in $\overline{\Sigma}$. By two out of three, this implies that $[a,d^0\cup d^n]\to [a,n]$ is in $\overline{\Sigma}$. Let $X$ be a presheaf on $\Theta$. As $X$ is a colimit of globular sums indexed by the Reedy cofibrant diagram $\Theta_{/X}\to \Psh{\Theta}$ (definition \ref{defi:reedycof}), and as $[\uvar,d^0\cup d^n]\to [\uvar,n]$ preserves cofibrations, this implies that $[X,d^0\cup d^n]\to [X,n]$ is in $\overline{\Sigma}$. In particular, $[\partial v,d^0\cup d^n]\to [\partial v,n]$ is in $\overline{\Sigma}$, and so is $[v,d^0\cup d^n]\to [\partial v,n]\cup[v,d^0\cup d^n]$ by stability by coproduct. A final use of the stability by two out of three then concludes the proof.
\end{proof}

\begin{definition} 
Let $[b,m]$ be an element of $\Delta[\Theta]$. We denote $\Hom^*(i([b,m]),[\textbf{a},n])$ the subset of $\Hom(i([b,m]),[\textbf{a},n])$ that consists of morphisms that preserve extremal objects. The explicit expression of morphisms in $\Theta$ given in remark \ref{remark:morphism of theta} implies the bijection:
\begin{equation}
\label{eq:hom in theta}
\Hom_{\Theta}^*(i([b,m]),[\textbf{a},n])\cong\Hom_{\Delta}([n],[m])^*\times \prod_{i<n}\Hom_{\Theta}(b,a_i)
\end{equation}
where $\Hom_{\Delta}^*([n],[m])$ is the subset of $\Hom_{\Delta}([n],[m])$ consisting of morphisms that preserve extremal objects.

Let $\textbf{a}:=\{a_0,a_1,...,a_{n-1}\}$ be a finite sequence of globular sums. We define $\Theta^{\hookrightarrow}_{/\textbf{a}}$ as the category whose objects are collections of maps $\{b\to a_i\}_{i<n}$ such that there exists no degenerate morphism $b\to b'$ factorizing all $b\to a_i$. Morphisms are monomorphisms $b\to b'$ making all induced triangles commute.

The bijection \eqref{eq:hom in theta} induces a bijection between the objects of $\Theta^{\hookrightarrow}_{/\textbf{a}}$ and the morphisms $[b,n]\to i^*[\textbf{a},n]$ that are the identity on objects and that cannot be factored through any degenerate morphism $[b,n]\to [\tilde{b},n]$.
\end{definition}

\begin{lemma}
\label{lemma:i etoile of W is in M 1}
For any morphism $p:[b,m]\to i^*[\textbf{a},n]$ in $\Psh{\Delta[\Theta]}$ that preserves extremal objects, there exists a unique pair $(\{b'\to a_i\}_{i<n},[f,i]:[b,m]\to [b',n])$ where $\{b'\to a_i\}_{i<n}$ is an element of $\Theta^{\hookrightarrow}_{/\textbf{a}}$, $f$ is a degenerate morphism, and such that the induced triangle
\[\begin{tikzcd}
	{[b,m]} & {[b',n]} \\
	& {i^*[\textbf{a},n]}
	\arrow["{[f,i]}", from=1-1, to=1-2]
	\arrow["{p'}", from=1-2, to=2-2]
	\arrow["p'", from=1-1, to=2-2]
\end{tikzcd}\]
commutes.
\end{lemma}

\begin{proof}
By adjunction and thanks to the bijection \eqref{eq:hom in theta}, $p$ corresponds to a pair $(j:[m]\to [n], \{b\to a_i\}_{i<n})$, and $i$ has to be equal to $j$.

Using this bijection once again, and the fact that degeneracies are epimorphisms, we have to show that there exists a unique degenerate morphism $g:b\to b'$ that factors the morphisms $b\to a_i$ for all $i<n$, and such that the induced family of morphisms $\{b'\to a_i\}_{i<n}$ is an element of $\Theta^{\hookrightarrow}_{/\textbf{a}}$.

As any infinite sequence of degenerate morphisms is constant at some point, the existence is immediate.

Suppose we are given two morphisms $b\to b'$, $b\to b''$ fulfilling the previous condition. Proposition 3.8 of \cite{Bergner_reedy_category_and_the_theta_construction} implies that there exists a globular sum $\tilde{b}$ and two degenerate morphisms $b'\to \tilde{b}$ and $b''\to \tilde{b}$ such that the induced square
\[\begin{tikzcd}
	b & {b'} \\
	{b''} & {\tilde{b}}
	\arrow[from=1-1, to=2-1]
	\arrow[from=2-1, to=2-2]
	\arrow[from=1-1, to=1-2]
	\arrow[from=1-2, to=2-2]
\end{tikzcd}\]
is cartesian. The universal property of the pushout implies that $b\to \tilde{b}$ also fulfills the previous condition. By the definition of $b'$ and $b''$, this implies that they are equal to $\tilde{b}$, and this shows the uniqueness.
\end{proof}

\begin{lemma}
\label{lemma:i etoile of W is in M 0.5}
Let $\{b \to a_i\}_{i < n}$ be an element of $\Theta^{\hookrightarrow}_{/\textbf{a}}$ and $i: b' \to b$ a monomorphism of $\Theta$. The induced family $\{b' \to b \to a_i\}_{i < n}$ is an object of $\Theta^{\hookrightarrow}_{/\textbf{a}}$.
\end{lemma}

\begin{proof}
The lemma \ref{lemma:i etoile of W is in M 1} implies that there exists a unique degenerate morphism $j:b'\to \tilde{b}$ that factors all the morphisms $b'\to b\to a_i$ for $i<n$, and such that the induced family of morphisms $\{\tilde{b}\to a_i\}_{i<n}$ is an element of $\Theta^{\hookrightarrow}_{/\textbf{a}}$. We proceed by contradiction, and we then suppose that $j$ is different from the identity.

We then have, for any $i<n$, a commutative square
\[\begin{tikzcd}
	{b'} & b \\
	{\tilde{b}} & {a_i}
	\arrow["i", from=1-1, to=1-2]
	\arrow[from=1-2, to=2-2]
	\arrow["j"', from=1-1, to=2-1]
	\arrow[from=2-1, to=2-2]
\end{tikzcd}\]
As the morphism $j$ is degenerate and different from the identity, there exists an integer $k$ and a non-trivial $k$-cell $d$ of $b'$ that is sent to an identity by $j$. Now, let $d'$ be a $k$-generator of the polygraph $b$ that appears in the decomposition of $i(d)$. The commutativity of the previous square and the fact that the $\omega$-categories $a_i$ are polygraphs imply that for any $i$, the $k$-cell $a'$ is sent to an identity by the morphism $b\to a_i$. As for any $i< n$ and any $l\geq k$, there is no non-trivial $l$-cell in $a_i$ whose $(k-1)$-source and $(k-1)$-target are the same, this implies that every $l$-cell of $b$ that is $(k-1)$-parallel with $d'$ is sent to the identity by the morphism $b\to a_i$.

We denote $\bar{b}$ the globular sum obtained by crushing all $l$-cells of $b$ that are $(k-1)$-parallel with $d'$. The induced degenerate morphism $b\to \bar{b}$ factors all the morphisms $b\to a_i$, which is in contradiction with the fact that $\{b\to a_i\}_{i<n}$ is an element of $\Theta^{\hookrightarrow}_{/\textbf{a}}$.
\end{proof}

\begin{definition}
We say that an element $\{v\to a_i\}_{i<n}$ in the category $\Theta^{\hookrightarrow}_{/\textbf{a}}$ is \textit{of height $0$} if $v\to a_0$ factors through $\partial a_0$ or $v\to a_{n-1}$ factors through $\partial a_{n-1}$. The \textit{height of an element $w$} is the maximal integer $m$ such that there exists a sequence 
$v_0\to v_1\to...\to v_m=w$ in $\Theta^{\hookrightarrow}_{/\textbf{a}}$ with $v_i\neq v_{i+1}$ for any $i<m$ and such that $v_0$ is of height $0$ and $v_1$ is not. As $\Theta$ is a Reedy category, all elements have finite height.
\end{definition}

\begin{lemma}
\label{lemma:i etoile of W is in M 1.5}
For any morphism $p:[b,m]\to i^*[\textbf{a},n]$ that preserves extremal objects, there exists a unique integer $k$, a unique element $\{b'\to a_i\}_{i<n}$ of height $k$, and a unique morphism $[f,i]:[b,m]\to [b',n]$ that doesn't factor through $[\partial b',n]$, and such that the induced triangle
\[\begin{tikzcd}
	{[b,m]} & {[b',n]} \\
	& {i^*[\textbf{a},n]}
	\arrow["{p'}", from=1-2, to=2-2]
	\arrow["{[f,i]}", from=1-1, to=1-2]
	\arrow[from=1-1, to=2-2]
\end{tikzcd}\]
commutes.

If $\{\tilde{b}\to a_i\}_{i<n}$ is any other object of non-negative height, and $[\tilde{f},j]:[b,m]\to [\tilde{b},n]$ is a morphism that makes the induced triangle
\[\begin{tikzcd}
	{[b,m]} & {[\tilde{b},n]} \\
	& {i^*[\textbf{a},n]}
	\arrow["{\tilde{p}}", from=1-2, to=2-2]
	\arrow["{[\tilde{f},j]}", from=1-1, to=1-2]
	\arrow[from=1-1, to=2-2]
\end{tikzcd}\]
commutative, then $\{\tilde{b}\to a_i\}_{i<n}$ is of height strictly greater than $k$ and $[\tilde{f},j]$ factors through $[\partial\tilde{b},n]$.
\end{lemma}

\begin{proof}
The lemma \ref{lemma:i etoile of W is in M 1} implies the first assertion. For the second one, suppose we are given an object $\{\tilde{b}\to a_i\}_{i<n}$ of non-negative height and a morphism  $[\tilde{f},j]:[b,m]\to [\tilde{b},n]$ fulfilling the desired condition. The bijection \eqref{eq:hom in theta} directly implies that $j$ is equal to $i$, and the first assertion implies that $\tilde{f}$ is non-degenerate.

We can then factor $\tilde{f}:b\to \tilde{b}$ into a degenerate morphism $b\to \bar{b}$ followed by a monomorphism $ \bar{b}\to \tilde{b}$ which is not the identity. The lemma \ref{lemma:i etoile of W is in M 0.5} then implies that $\{\bar{b}\to \tilde{b}\to a_i\}_{i<n}$ is an element of  $\Theta^{\hookrightarrow}_{/\textbf{a}}$. The first assertion then implies that the two morphisms $[b,m]\to [b',n]$ and $[b,m]\to [\bar{b},n]$ are equal. As the monomorphism $[\bar{b},n]\to [\tilde{b},n]$ is not the identity, this concludes the proof.
\end{proof}

\begin{lemma}
\label{lemma:i etoile of W is in M 2}
The morphism $i^*[\partial^0\textbf{a},n]\cup i^*[\partial^{n-1}\textbf{a},n]\to i^*[\textbf{a},n]$ is in $\overline{\Sigma}$, where $\partial^j\textbf{a}$ corresponds to the sequence $\{a_1,\ldots,\partial a_j,\ldots,a_n\}$.
\end{lemma}

\begin{proof}
For $k\in\Nb\cup\{\infty\}$, we define $x_k$ as the smallest subobject of $i^*[\textbf{a},n]$ such that for any element of height less than or equal to $k$ of $\Theta^{\hookrightarrow}_{/\textbf{a}}$, the corresponding morphism $[b,n]\to i^*[\textbf{a},n]$ factors through $x_k$. In particular, we have $x_0= i^*[\partial^0\textbf{a},n]\cup i^*[\partial^{n-1}\textbf{a},n]$, and lemma \ref{lemma:i etoile of W is in M 1} implies that $x_{\infty} =i^*[\textbf{a},n]$. We denote $(\Theta^{\hookrightarrow}_{/\textbf{a}})_{k}$ the set of elements of $\Theta^{\hookrightarrow}_{/\textbf{a}}$ of height $k$.

Every morphism $[b,m]\to i^*[\textbf{a},n]$ that does not preserve extremal points then factors through $x_0$. The lemma \ref{lemma:i etoile of W is in M 1.5} implies that for any integer $k$, the canonical square 
\begin{equation}
\label{eq:lemma:i etoile of W is in M 2}
\begin{tikzcd}
	{\coprod_{(\Theta^{\hookrightarrow}_{/\textbf{a}})_{k+1}}[b,d^0\cup d^n]\cup[\partial b,n]} & {x_k} \\
	{\coprod_{(\Theta^{\hookrightarrow}_{/\textbf{a}})_{k+1}}[b,n]} & {x_{k+1}}
	\arrow[from=1-1, to=2-1]
	\arrow[from=2-1, to=2-2]
	\arrow[from=1-1, to=1-2]
	\arrow[from=1-2, to=2-2]
\end{tikzcd}
\end{equation}
is cocartesian. The lemma \ref{lemma:i etoile of W is in M 0} and the stability under pushout of $\overline{\Sigma}$ imply that $x_k\to x_{k+1}$ is in $\overline{\Sigma}$. As $i^*[\textbf{a},n]$ is the transfinite composition of the sequence $x_0\to x_1\to...$, this implies that $x_0\to i^*[\textbf{a},n]$ is in $\overline{\Sigma}$, which concludes the proof.
\end{proof}

\begin{lemma}
The morphism $i^*\Sp_a \to i^*a$ is in $\overline{\Sigma}$ for any globular sum $a$.
\end{lemma}

\begin{proof}
Let $[\textbf{a},n]:= a$. As $\overline{\Sigma}$ is closed under pushouts and composition, lemma \ref{lemma:i etoile of W is in M 2} implies that the morphism 
$$i^*[\{a_0,...,a_{n-2}\},n-1]\cup i^*[\{a_1,...,a_{n-1}\},n-1]\to i^*[\textbf{a},n]$$
is in $\widehat{\M}$. 
An easy induction on $n$ shows that this is also the case for the morphism 
$$[a_0,1]\cup... \cup [a_{n-1},1]= i^*[a_0,1]\cup... \cup i^*[a_{n-1},1]\to i^*[\textbf{a},n].$$
Now note that $i^*\Sp_{[\textbf{a},n]}$ is equivalent to 
$$[\Sp_{a_0},1]\cup... \cup [\Sp_{a_{n-1}},1].$$
As the morphisms $[\Sp_i,1]\to [a_i,1]$ are by definition in $\M$, this concludes the proof.
\end{proof}

\begin{lemma}
\label{lemma:i etoile of W is in M}
Let $n\in \Nb\cup \{\omega\}$.
There is an inclusion $i^*(\W_n)\subset \overline{\Sigma_n}$.
\end{lemma}

\begin{proof}
The case where $n=\omega$ is precisely the content of the last lemma. The case $n<\omega$ directly follows.
\end{proof}

We are now ready to prove theorem \ref{theo:unit and counit are in Sigma}.

\begin{construction}
We denote $k:\Xi_n\to \Delta[\Theta_n]$ the functor sending $[a,n]$ onto $[j(a),n]$. This induces an adjunction
\[\begin{tikzcd}
	{k_!:\Psh{\Xi_{n+1}}} & {\Psh{\Delta[\Theta_{n}]}:k^*}
	\arrow[""{name=0, anchor=center, inner sep=0}, shift left=2, from=1-1, to=1-2]
	\arrow[""{name=1, anchor=center, inner sep=0}, shift left=2, from=1-2, to=1-1]
	\arrow["\dashv"{anchor=center, rotate=-90}, draw=none, from=0, to=1]
\end{tikzcd}\]
\end{construction}

\begin{proof}[Proof of theorem \ref{theo:unit and counit are in Sigma}]
First, note that the two equalities $j_!(\M^\sat_n)=\W^\sat_n$ and $j^*(\W^\sat_n)=\M^\sat_n$ and the inclusion $j_!(\M_n)\subset \Sigma_n$ are obvious. We show by induction on $n$ that we have an inclusion $j^*(\W_n)\subset \overline{\M_n}$. The case $n=1$ is straightforward. Suppose then the result is proven at the stage $n$. As we have $k^*[a,1]\sim [j^*a,1]$, the induction implies that $k^*(\Sigma_{n+1})\subset \overline{\M_{n+1}}$. Combined with proposition \ref{lemma:i etoile of W is in M}, this then implies that $j^*(\W_{n+1})\subset \overline{\M_{n+1}}$.

It remains to show the last part of the assertion. We can restrict to the case where $n=\omega$. Let $a$ be an element of $\Theta$. If $a$ is of the shape $\Db_n$, then $j_!j^*a = a$. Suppose now that $a$ is any globular sum. We then have a commutative diagram
\[\begin{tikzcd}
	{j_!j^*\Sp_a} & {\Sp_a} \\
	{j_!j^*a} & a
	\arrow[from=1-1, to=2-1]
	\arrow[Rightarrow, no head, from=1-1, to=1-2]
	\arrow[from=1-2, to=2-2]
	\arrow[from=2-1, to=2-2]
\end{tikzcd}\]
where the upper horizontal morphism is an identity. The two inclusions $j_!(\Sigma)\subset \W$ and $j^*(\W)\subset \overline{\Sigma}$ imply that the vertical morphisms of the previous diagram are in $\overline{\W}$. By two out of three, this implies that $i_!i^*a\to a$ belongs to $\overline{\W}$ for any globular sum. We proceed analogously to show that for any $b\in \Xi$, $b\to j^*j_! b$ is in $\overline{\Sigma}$.
\end{proof}

\section{Gray Operations}
\subsection{Recollection on Steiner theory}
\label{section:Steiner thery} 

We present here the Steiner theory developed in \cite{Steiner_omega_categories_and_chain_complexes}.

\begin{definition}
An augmented directed complex $(K,K^*,e)$ is given by a complex of abelian groups $K$, with an augmentation $e$: $$\Zb \xleftarrow{e} K_0 \xleftarrow{\partial_0} K_1 \xleftarrow{\partial_1} K_2 \xleftarrow{\partial_2} K_3 \xleftarrow{\partial_3}. .. $$
and a graded set $K^* = (K^*_n)_{n\in\Nb}$ such that for any $n$, $K_n^*$ is a submonoid of $K_n$. A morphism of directed complexes between $(K,K^*,e)$ and $(L,L^*,e')$ is given by a morphism of augmented complexes of abelian groups $f : (K,e)\to (L,e')$ such that $f(K^*_n)\subset L^*_n$ for any $n$. We note by \wcnotation{$\CDA$}{(adc@$\CDA$} the category of augmented directed complexes. 
\end{definition}

Steiner then constructs an adjunction
\[\begin{tikzcd}
	{\lambda:\omegacat} & {\CDA:\nu}
	\arrow[""{name=0, anchor=center, inner sep=0}, shift left=2, from=1-1, to=1-2]
	\arrow[""{name=1, anchor=center, inner sep=0}, shift left=2, from=1-2, to=1-1]
	\arrow["\dashv"{anchor=center, rotate=-90}, draw=none, from=0, to=1]
\end{tikzcd}\]
The functor $\lambda$ is the simplest to define: \sym{(lambda@$\lambda:\omegacat\to \CDA$}

\begin{construction}
Let $C$ be a $\omega$-category.
We denote by $(\lambda C)_n$ the abelian group generated by the set $\{[x]_n: x\in C_n\}$ and the relations
$$[x*_m y]_n \sim [x]_n + [y]_n \mbox{ for $m<n$ }.$$
We define the morphism $\partial_n: (\lambda C)_{n+1}\to (\lambda C)_n$ on generators by the formula:
$$\partial_n([x]_{n+1}) := [d_n^+ x]_{n} - [d_n^- x]_{n}.$$
We can easily check that the morphism $\partial$ is a differential. We define an augmentation $e:(\lambda C)_{0}\to \Zb$ by setting $e([x]_0) = 1$ on generators. 
We denote by $(\lambda C)_n^*$ the additive submonoid generated by the elements $[x]_n$. We then set:
$$\lambda C := (\{(\lambda C)_n \}_{n\in \Nb},\{(\lambda C)^*_n \}_{n\in \Nb},e ).$$ This assignation lifts to a functor:
$$\begin{array}{ccccc}
\lambda &:& \omegacat&\to&\CDA\\
&&C&\mapsto&\lambda C.
\end{array}$$
\end{construction}
\begin{example}~
\begin{enumerate}
\item
For any integer $n$, $\lambda\Db_n$ is the augmented directed complex whose underlying chain complex is given by:
$$
\Zb\xleftarrow{e}
\Zb[e_0^-,e_0^+] \xleftarrow{\partial_0}
... \xleftarrow{\partial_{n-2}}
\Zb[e_{n-1}^-,e_{n-1}^+] \xleftarrow{\partial_{n-1}}
\Zb[e_{n}] \xleftarrow{\partial_{n}}
0\leftarrow ...$$
where for any $0<k<n$ and $\alpha\in\{-,+\}$
$$e(e_0^\alpha)=1~~~\partial_{k-1}(e_k^\alpha)= e_{k-1}^+-e_{k-1}^-~~~\partial_{n-1}(e_n)= e_{n-1}^+-e_{n-1}^-.$$
\item
The augmented directed complex $\lambda[n]$ has for underlying chain complex:
$$
\Zb\xleftarrow{e}
\Zb[v_0,v_1,...,v_n] \xleftarrow{\partial_0}
\Zb[v_{0,1},v_{1,2}...,v_{n-1,n}] \xleftarrow{\partial_{1}}
0\leftarrow ...$$
where for any $k<n$ and $\alpha\in\{-,+\}$
$$e(v_k)=e(v_n)=1~~~ \partial_{1}(v_{k,k+1})=v_{k+1}-v_k.$$
\end{enumerate}
\end{example}

\begin{definition}
 We now define the functor $\nu:\CDA\to \omegacat$. Throughout, we fix an augmented directed complex $(K,K^*,e)$.
A \textit{Steiner array} (or simply a \notion{array}) of dimension $n$ is the data of a finite double sequence: \sym{(nu@$\nu:\CDA\to \omegacat$}
$$\left(\begin{matrix}
x^-_0 &x^-_1&x^-_2&x^-_3 &...&x_n^-\\
x^+_0 &x^+_1&x^+_2&x^+_3 &...&x_n^+
\end{matrix}\right)$$
such that
\begin{enumerate}
\item $x^-_n=x^+_n$;
\item For any $i\leq n$ and $\alpha\in\{-,+\}$, $x_i^\alpha$ is an element of $K^*_i$;
\item For any $0<i\leq n$, $\partial_{i-1}(x_i^\alpha)= x_{i-1}^+ - x_{i-1}^-$;
\end{enumerate}
An array is said to be \wcnotion{coherent}{coherent array} if $e(x^+_0) = e(x^-_0) = 1$.
\end{definition}

\begin{definition}
We define the globular set $\nu K$, whose $n$-cells are the coherent arrays of dimension $n$. The source and target maps are defined for $k<n$ by the formula: 

$$d^\alpha_k\begin{pmatrix}
x^-_0 &x^-_1&x^-_2&...&x^-_n\\
x^+_0 &x^+_1&x^+_2&...& x^+_n
\end{pmatrix} = \begin{pmatrix}
x^-_0 &x^-_1&x^-_2&...& x^-_{k-1}&x^\alpha_k\\
x^+_0 &x^+_1&x^+_2&...& x^+_{k-1}&x^\alpha_k\end{pmatrix}$$

There is an obvious group structure on the arrays:
$$\begin{pmatrix}
x^-_0 &x^-_1&...& x^-_n\\
x^+_0 &x^+_1&...& x^+_n
\end{pmatrix}
+
\begin{pmatrix}
y^-_0 &y^-_1&...& y^-_n\\
y^+_0 &y^+_1&...& y^+_n
\end{pmatrix}
=
\begin{pmatrix}
x^-_0+y^-_0 &x^-_1+ y^-_1&...&x^-_n+ y^-_n \\
x^+_0+y^+_0 &x^+_1+ y^+_1&...&x^+_n +y^+_n 
\end{pmatrix}
$$
\label{defi:definition of composition and units of nu k}

\begin{itemize}
\item[$-$]
For two coherent arrays $x$ and $y$ such that $d^-_k(x) =d^+_k(y) = z$, we define their $k$-composition by the following formula: 
$$x*_k y := x- z + y .$$ More explicitly:
$$\begin{pmatrix}
x^-_0 &...& x^-_n\\
x^+_0 &...& x^+_n
\end{pmatrix}
*_k
\begin{pmatrix}
y^-_0 &...& y^-_n\\
y^+_0 &...& y^+_n
\end{pmatrix}
 := 
\begin{pmatrix}
y^-_0&...&y_k^-& y_{k+1}^- + x_{k+1}^- & ...& y_{n}^- + x_{n}^-\\
x^+_0 &...&x_k^+& y_{k+1}^+ + x_{k+1}^+ & ...& y_{n}^+ + x_{n}^+ 
\end{pmatrix}
$$
\item[$-$]
For an integer $m>n$, we define the $m$-sized array $1^m_x$ as follows:
$$1^m_x :=
\begin{pmatrix}
x^-_0 &...& x^-_n& 0 &...&0\\
x^+_0 &...& x^+_n& 0 &...&0	
\end{pmatrix}$$
\end{itemize}
The globular set $\nu K$, equipped with these compositions and units is an $\omega$-category.
\end{definition}

\begin{construction}
We define the functor $\nu: \CDA \to \omegacat$ which associates to an augmented directed complex $K$, the $\omega$-category $\nu K$, and to a morphism of augmented directed complexes $f: K \to L$, the morphism of $\omega$-categories.
$$
\begin{array}{rccc}
\nu f : &\nu K &\to& \nu L\\
& \left(\begin{matrix}
x^-_0 &...&x_n^-\\
x^+_0&...&x_n^+
\end{matrix}\right) 
&\mapsto&
\left(\begin{matrix}
f_0(x^-_0) &...&f_n(x_n^-)\\
f_0(x^+_0)&...&f_n(x_n^+)
\end{matrix}\right) 
\end{array}
$$
\end{construction}

\begin{theorem}[Steiner]
\label{theo:ajdonction de steiner avec unite et counite explicite}
The functors $\lambda$ and $\nu$ form an adjoint pair 
\[\begin{tikzcd}
	{\lambda:\omegacat} & {\CDA:\nu}
	\arrow[""{name=0, anchor=center, inner sep=0}, shift left=2, from=1-1, to=1-2]
	\arrow[""{name=1, anchor=center, inner sep=0}, shift left=2, from=1-2, to=1-1]
	\arrow["\dashv"{anchor=center, rotate=-90}, draw=none, from=0, to=1]
\end{tikzcd}\]
For a $\omega$-category $C$, the unit of the adjunction is given by:
$$\begin{array}{rrcl}
~~~~~\eta :& C &\to & \nu \lambda C \\
& x\in C_n &\mapsto & 
\begin{pmatrix}
[d^-_0(x)]_0&...&[d^-_{n-1}(x)]_{n-1}&[x]_n\\
[d^+_0(x)]_0&...& [d^+_{n-1}(x)]_{n-1}&[x]_n
\end{pmatrix}
\end{array}
$$
For an augmented directed complex $K$, the counit is given by:
$$\begin{array}{rrcl}
\pi :& \lambda \nu K &\to & K~~~~~~~~~~~~~~~~ \\
& [x ]_n \in (\lambda \nu K)_n&\mapsto & x_n^+ = x_n^-
\end{array}
$$
\end{theorem}
\begin{proof}
This is \cite[theorem 2.11]{Steiner_omega_categories_and_chain_complexes}.
\end{proof}

\begin{definition}
A \snotion{basis}{for augmented directed complexes} for an augmented directed complex $(K,K^*,e)$ is a graded set $B = (B_n)_{n\in\Nb}$ such that for every $n$, $B_n$ is both a basis for the monoid $K_n^*$ and for the group $K_n$.
\end{definition}

\begin{remark}
The elements of $B_n$ can be characterized as the minimal elements of $K_n^*\backslash{0}$ for the following order relation:
	$$x\leq y \mbox{ iff } y-x \in K_n^*$$
This shows that if a basis exists, it is unique.
\end{remark}
\vspace{1cm}

Any element of $K_n$ can then be written uniquely as a sum $\sum_{b\in B_n} \lambda_b b$. This leads us to define new operations:
\begin{definition}
\label{defi:support}
For an element $x := \sum_{b\in B_n} \lambda_b b$ of $K_n$, we define the \textit{positive part} and the \textit{negative part}:
$$
\begin{array}{rcl}
(x)_+ &:=& \sum_{b\in B_n, \lambda_b> 0} ~\lambda_bb\\
(x)_- &:=& \sum_{b\in B_n, \lambda_b< 0} -\lambda_bb
\end{array}
$$
We then have $x = (x)_+ - (x)_-$. An element $x$ is \textit{positive} (resp. \textit{negative}) when $x =(x)_+$ (resp. when $x =-(x)_-$).
Let $y = \sum_{b\in B_n} \mu_b b$, we set : 
$$
\begin{array}{rcl}
x\wedge y &:=& \sum_{b\in B_n} \mbox{ min}(\lambda_b, \mu_b)~ b \\
\end{array}
$$
Eventually, we set \sym{(partialna@$\partial_n^+(\uvar)$}\sym{(partialnb@$\partial_n^-(\uvar)$}
$$
\begin{array}{rcl}
\partial_n^+(\uvar) &:=& (\partial_n(\uvar))_+ : K_{n+1}\to K^*_n\\
\partial_n^-(\uvar) &:= &(\partial_n(\uvar))_- : K_{n+1}\to K^*_n
\end{array}
$$

When an element $b$ of the basis is in the support of $x$, i.e $\lambda_b\neq 0$, we say that \textit{$b$ belongs to $x$}, which is denoted by $b\in x$.
\end{definition}

\begin{example}
For any integer $n$, $\lambda\Db_n$ admits a basis, given by the graded set $B_{\lambda\Db_n}$ fulfilling:
$$(B_{\lambda\Db_n})_k:= \left\{ 
\begin{array}{ll}
\{e_k^-,e_k^+\}&\mbox{ if $k<n$}\\
\{e_n\}&\mbox{ if $k=n$}\\
\emptyset&\mbox{ if $k>n$}\\
\end{array}\right.$$ 
The augmented directed complex $\lambda[n]$ also admits a basis, given by the graded set $B_{\lambda\Db_n}$ fulfilling:
$$(B_{\lambda\Db_n})_k:= \left\{ 
\begin{array}{ll}
\{v_0,v_1,...,v_n\}&\mbox{ if $k=0$}\\
\{v_{0,1},v_{1,2}...,v_{n-1,n}\}&\mbox{ if $k=1$}\\
\emptyset&\mbox{ if k>1}\\
\end{array} \right.$$ 
\end{example}

\begin{definition}
Let $a\in K^*_n$. We set by a decreasing induction on $k\leq n$ : 
 $$ \begin{array}{rclc}
 \langle a\rangle_k^\alpha &:= & a & \mbox{if $k = n$}\\
 &:= & \partial_k^\alpha\langle a\rangle^\alpha_{k+1} & \mbox{if not}
\end{array} 
$$
The array associated to $a$ is then: 
$$\langle a\rangle := \begin{pmatrix}
\langle a\rangle^-_0 &...&\langle a\rangle^-_{n-1}&a\\
\langle a\rangle^+_0 &...&\langle a\rangle^+_{n-1}&a
\end{pmatrix}$$
The basis is said to be \wcnotion{unitary}{unitary basis} when for any $b\in B$, the array $\langle b\rangle$ is coherent.
\end{definition}

\begin{definition}
 We define the relation $\odot$ on $B$ as being the smallest transitive and reflexive relation such that for any pair of elements of the basis $a,b$, 
$$a\odot b \mbox{ if } \mbox{($|a|>0$ and $b\in\langle a\rangle_{|a|-1}^-$)}~~\mbox{or}~~\mbox{($|b|>0$ and $a\in \langle b\rangle_{|b|-1}^+$)}$$
A basis is said to be \wcsnotion{loop free}{loop free basis}{for augmented directed complexes} the relation $\odot$ is a (partial) order on $B$.
\end{definition}

\begin{remark}
In \cite{Ara_Maltsiniotis_joint_et_tranche}, this notion is called \textit{strongly loop free}.
\end{remark}

\begin{example}
For any integer $n$, $\lambda\Db_n$ and $\lambda[n]$ admit a loop free and unitary basis.
\end{example}

\begin{definition}
 We now define the subcategory \wcnotation{$\CDAB$}{(adcb@$\CDAB$} of $\CDA$ composed of augmented directed complexes which admit a unitary and loop free basis. 
 \end{definition}

We will now describe the analog of the notion of basis for $\omega$-categories. 

\begin{definition}
A $\omega$-category $C$ is \wcnotion{generated by composition}{generated by composition} by a set $E\subset C$ when any cell can be written as a composition of elements of $E$ and iterated units of elements of $E$. This set is a \snotion{basis}{for $\omega$-categories} if $\{[e]_{d(e)}\}_{e\in E}$ is a basis of the augmented directed complex $\lambda C$. 
\end{definition}

\begin{prop}
An $\omega$-category $C$ that admits a basis is an $\omega$-category.
\end{prop}
\begin{proof}
Let $C$ be an $\omega$-category that admits a basis $E$. Suppose that there exists a non trivial $n$-cell $\alpha$ that admits an inverse $\beta$. We then have $[\alpha]_n+ [\beta]_n=[\alpha \circ_{n-1} \beta]_n =0$. As $\lambda C$ is free, we have $[\alpha]_n=0$. This implies the equality $[e]_n=0$ for any element $e\in E$ of dimension $n$ that appears in a decomposition of $\alpha$. This is obviously in contradiction with the fact that $\{[e]_{d(e)}\}_{e\in E}$ is a basis of the augmented directed complex $\lambda C$. 
\end{proof}

\begin{definition}
\label{defi:loop free and atomic}
A basis $E$ of an $\omega$-category is : 
\begin{enumerate}
\item \wcsnotion{Loop free}{loop free basis}{for $\omega$-categories} when $\{[e]_{d(e)}\}_{e\in E}$ is.
\item \wcnotion{Atomic}{atomic basis} when $[d_n^+ e]_n \wedge [d_n^- e]_n = 0$ for any $e\in E$ and any natural number $n$ strictly smaller than the dimension of $e$. 
\end{enumerate}
\end{definition}

\begin{prop}
 If a loop free basis $E$ is atomic then $\{[e]\}_{e\in E}$ is unitary.
 \end{prop}
\begin{proof}
 This is \cite[proposition 4.6]{Steiner_omega_categories_and_chain_complexes}.
 \end{proof}

\begin{example}
For any integer $n$, $\Db_n$ and $[n]$ admit a loop free and atomic basis.
More generally, \cite[proposition 4.13]{Ara_Maltsiniotis_joint_et_tranche} states that 
any globular sum admits a loop free and atomic basis. 
\end{example}

\begin{definition}
Proposition $1.23$ of \cite{Ara_a_categorical_characterization_of_strong_Steiner_omega_categories} states that if an $\omega$-category admits a loop-free and atomic basis, it is unique.
We then define the category \wcnotation{$\ocatB$}{(ocat@$\ocatB$} as the full subcategory of $\omegacat$ composed of $\omega$-categories admitting an atomic and loop-free basis.
\end{definition}

 \begin{theorem}[Steiner]
 \label{theorem:steiner}
 Once restricted to $\ocat_B$ and $\CDAB$, the adjunction 
\[\begin{tikzcd}
	{\lambda:\omegacat} & {\CDA:\nu}
	\arrow[""{name=0, anchor=center, inner sep=0}, shift left=2, from=1-1, to=1-2]
	\arrow[""{name=1, anchor=center, inner sep=0}, shift left=2, from=1-2, to=1-1]
	\arrow["\dashv"{anchor=center, rotate=-90}, draw=none, from=0, to=1]
\end{tikzcd}\]
becomes an adjoint equivalence, i.e. :
$$ \lambda_{|\ocatB } \circ \nu_{|\CDAB} \cong id_{|\CDAB}~~~~~~~ id_{|\ocatB }\cong \nu_{|\CDAB} \circ \lambda_{|\ocatB }$$
\end{theorem}
\begin{proof}
See \cite[theorem 5.11]{Steiner_omega_categories_and_chain_complexes}.
\end{proof}

\begin{remark}
If $K$ is an augmented directed complex admitting a unitary and loop-free basis $B$, then the $\omega$-category $\nu K$ admits an atomic and loop-free basis given by the set $\langle B\rangle := \{\langle b\rangle,b\in B\}$. Conversely if an $\omega$-category $C$ admits an atomic and loop-free basis $E$, then the augmented directed complex $\lambda C$ admits a unitary and loop-free basis given by the family of sets $[E_n] := \{[e]_{d(e)}, e\in E_n\}$. 
The isomorphisms
$$\lambda \nu K\cong K \mbox{~~~ and ~~~} C\cong \nu\lambda C$$
induce isomorphisms:
$$[\langle B\rangle ]\cong B \mbox{~~~ and ~~~} E \cong \langle [E]\rangle.$$
\end{remark}

\begin{definition}
Let $f:M\to N$ be a morphism between two augmented directed complexes admitting unitary and loop-free bases $B_M$ and $B_N$. The morphism $f$ is \wcnotion{quasi-rigid}{quasi-rigid morphism} if for any $n$, and any $b\in (B_M)_n$,
$$f_n(b)\neq 0 ~\Rightarrow ~ f_n(b)\in B_N\mbox{ and }\nu(f)\langle b\rangle = \langle f_n(b)\rangle.$$
\end{definition}

\begin{theorem}
\label{theo:Kan condition}
Suppose given a commutative square in $\CDAB$
\[\begin{tikzcd}
	K & {M_1} \\
	{M_0} & M
	\arrow["{k^0}", from=1-1, to=1-2]
	\arrow["{l^1}", from=1-2, to=2-2]
	\arrow["{k^0}"', from=1-1, to=2-1]
	\arrow["{l^0}"', from=2-1, to=2-2]
\end{tikzcd}\]
and such that all morphisms are quasi-rigid. Let $B_K,~B_{M_0},~B_{M_1},~B_{M}$ be the bases of $K,~M_0,~M_1,~ M$.

Then, this square is cocartesian if and only if for any $n$, the induced diagram of sets
\[\begin{tikzcd}
	{(B_{K})_n\cup\{0\}} & {(B_{M_1})_n\cup\{0\}} \\
	{(B_{M_0})_n\cup\{0\}} & {(B_{M})_n\cup\{0\}}
	\arrow["{k^0_n}", from=1-1, to=1-2]
	\arrow["{l^1_n}", from=1-2, to=2-2]
	\arrow["{k^0_n}"', from=1-1, to=2-1]
	\arrow["{l^0_n}"', from=2-1, to=2-2]
\end{tikzcd}\]
is cocartesian. Furthermore, the induced square in $\ocat$
\[\begin{tikzcd}
	{\nu K} & {\nu M_1} \\
	{\nu M_0} & {\nu M}
	\arrow["{\nu k^0}", from=1-1, to=1-2]
	\arrow["{\nu l^1}", from=1-2, to=2-2]
	\arrow["{\nu k^0}"', from=1-1, to=2-1]
	\arrow["{\nu l^0}"', from=2-1, to=2-2]
\end{tikzcd}\]
is cocartesian.
\end{theorem}
\begin{proof}
This is a combination of theorems 3.1.2 and 3.2.7 of \cite{Loubaton_condition_de_kan}.
\end{proof}

\subsection{$2$-Polygraphs and presheaves on $\Theta_2$}

The objective of this section is to prove the following theorem
\begin{theorem}
\label{theo: case of 1 and 2 category}
Let $k\leq 1$ be an integer, and let $C$ and $D$ be two $2$-categories admitting loop-free and atomic bases (definition \ref{defi:loop free and atomic}). Suppose there is a cocartesian square in $\ocat$ of shape:
\[\begin{tikzcd}
	{\partial [[k],1]} & C \\
	{[[k],1]} & D
	\arrow["x"', from=2-1, to=2-2]
	\arrow["{\partial x}", from=1-1, to=1-2]
	\arrow["j", from=1-2, to=2-2]
	\arrow["\lrcorner"{anchor=center, pos=0.125, rotate=180}, draw=none, from=2-2, to=1-1]
	\arrow[from=1-1, to=2-1]
\end{tikzcd}\]
Then, viewed as a morphism of $\Psh{\Theta_2}$, the morphism $j:C\cup x\to D$ is in $\overline{\W_2}$ which is the smallest precocomplete class of morphism (definition \ref{defi:precocomplet}) containing $\W_2$ ( definition \ref{defi:definition of W}).
\end{theorem}
Informally, this theorem shows that the square appearing in the previous statement is homotopically cocartesian. This result is therefore a special case of the similar but much more general theorem proved by Campion in \cite{campion2023infty}.

\vspace{1cm}

We  fix  a $2$-category $D$  admitting a loop free and atomic basis until the end of this section.

\begin{definition}
Let $v$ be a 2-cell of $ D$. The \textit{2-support of $v$}, denoted $B_2^v$, is the support of $[v]_2$ (definition \ref{defi:support}). 
The \textit{1-support of $ v$}, denoted $B_1^v$, is the union of the support of $[\pi_1^+v]_1$ with   $(\partial^-_1B_2^v)\cup B_2^v$.

For $i=1,2$, we define the relation $ <^v_i$ as the smallest transitive relation on $ B_i^v$ such that $ c<_i d$ whenever 
\[ \langle c \rangle^-_i  \wedge  \langle d \rangle^+_i  \neq 0. \]
\end{definition}

\begin{remark}
Remark that the two inclusions $(B_0^v,<^v_0)\to (B,\odot)$ and $(B_1^v,<^v_1)\to (B,\odot)$ are strictly increasing. As a consequence, $<^v_0$ and $<^v_1$ are (partial) orders. 
\end{remark}

\begin{remark}
The theorem \ref{theorem:steiner} implies that $B_1^v$ is also equal to the union of the support of $[\pi_1^-v]_1$ with $(\partial^+_1B_2^v)\cup B_2^v$.
\end{remark}

\begin{lemma}
\label{lemma:other characterization of <0}
Let $v$ be a 2-cell of $D$, and $b,b'$ be two elements of $B_1^v$. The assertion $b<^v_0b'$ holds if and only if there exists a well-defined $0$-composite 
$$b*_0....*_0 b'.$$
\end{lemma}
\begin{proof}
Straightforward.
\end{proof}

\begin{definition}
Given a finite set $E$ endowed with a strict order $<$, \textit{an ordering of $E$} is a bijective sequence $(x_i)_{i\leq n}$ of elements of $E$ such that for every $i<j$, $\neg (x_j<x_i)$. 
\end{definition}

\begin{theorem}
\label{theo:decomposition de condition de Kan}
Let $v$ be a $2$-cell of $D$, and $(w_i)_{i\leq n}$ an ordering of $B_2^v$.
There exists a decomposition of $v$ as 
$$v:=v_0*_1...*_1 v_{n}$$
such that for every $i<n$, $v_i$ is a $0$-composition of an element of $w_i$ with several $1$-generators of $D$. 

Moreover, for any decomposition of $v$ as 
$$v:=v_0'*_1...*_1 v'_{n}$$
such that $v_i'$ is a $0$-composition of a unique element $w_i'$ of $B_2^v$ with several $1$-generators of $D$, then the sequence $\{w_i\}_{i\leq n}$ is an ordering of $B_2^v$.
\end{theorem}
\begin{proof}
The first assertion is a consequence of \cite[theorem 2.47]{Loubaton_condition_de_kan}. 

To show the second assertion, suppose given such a decomposition. We will proceed by contradiction and then suppose that there exist $i<j$ such that $w'_j<w'_i$. We can suppose without loss of generality that $i=0$ and $j=n$. 

By a direct induction on $n$ using \cite[lemma 2.43]{Loubaton_condition_de_kan}, we have
$$\partial_1^+([v'_0]_2) \leq \partial_1^+([v_0'*_1...*_1 v'_{n}]_2)=\partial_1^+([v]_2)$$
$$\partial_1^-([v'_n]_2) \leq \partial_1^-([v_0'*_1...*_1 v'_{n}]_2)=\partial_1^-([v]_2)$$
Moreover, the inequality $w'_n<w'_0$ implies
$$\partial_1^+([v'_0]_2)\wedge \partial_1^-([v'_n]_2)\neq 0$$ 
and then
$$\partial_1^+([v]_2)\wedge \partial_1^-([v]_2)\neq 0$$
which is absurd as $\partial_1^+([v]_2)$ and $\partial_1^-([v]_2)$ are respectively defined as the positive part and the negative part of $\partial([v]_2)$.
\end{proof}

\begin{lemma}
\label{lem:functorialite de <}
Let $ D$ be a $ 2$-category and $ f:C\to D$ be a morphism. 
Let $ v$ be a $2$-cell of $C$ and $ b,b'$ two elements in the $1$-support of $ v$.
\begin{enumerate}
\item $ b<^v_0b'$ implies that for all $ c\in B_1^{f(b)}$ and $ c'\in B_1^{f(b')}$, $ c<^{f(v)}_0c'$.
\item $ b<^v_1b'$ implies that for all $ c\in B_2^{f(b)}$ and $ c'\in B_2^{f(b')}$, $ \neg (c'<^{f(v)}_1 c)$.
\end{enumerate}
\end{lemma}
\begin{proof}
According to lemma \ref{lemma:other characterization of <0}, we have a well defined $0$-composite 
$$b*_0...*_0b'$$
and so a well defined $0$-composite
$$f(b)*_0...*_0f(b')$$
Applying the decomposition given in theorem \ref{theo:decomposition de condition de Kan} to $f(b)$ and $f(b')$, we get a well-defined composite 
$$w*_0...*_0w'.$$
where $ w$ (resp. $w'$) is a $0$-composite of $c$ (resp. $c'$) with 1-generators. This then implies $c<^{f(v)}_0c'$. 

We now deal with the second case. Let $c\in B_2^{f(b)}$ and $c'\in B_2^{f(b')}$.
According to theorem  \ref{theo:decomposition de condition de Kan} there exists a decomposition of $v$ of shape 
\[ v:=v_0*v_1*_1....*_1v_n \]
where for all $i\leq n$, $v_0$ is a $0$-composite of a unique $2$-generator with $1$-generators. Moreover, the unique $i$ (resp. the unique $j$) such that $b$ belongs to $v_i$ (resp. such that $b'$ belongs to $v_j$) verifies $i<j$.

Applying the morphism $f$ and decomposing each $f(v_i)$ the same way, we get a decomposition
\[ f(v):=u_0*u_1*_1....*_1u_m \]
where for all $i\leq m$, $u_0$ is a $0$-composite of a $2$-generator with $1$-generators, and such that the unique $i$ (resp. the unique $j$) such that $c$ belongs to $u_i$ (resp. such that $c'$ belongs to $w_j$) verifies $i<j$. The second assertion of theorem \ref{theo:decomposition de condition de Kan} then implies that $ \neg (c'<^{f(v)}_1 c)$.
\end{proof}

\begin{lemma}
\label{lemma:2 incompatiblite1}
Let $v$ be a 2-cell, and $b,b'$ two different elements of the 2-support of $v$. Then $\neg (b<^v_1 b')\wedge \neg (b'<^v_1 b)$ implies that $(b<^v_0 b')\vee (b'<^v_0 b)$ holds.
\end{lemma}
\begin{proof}
We suppose that $\neg (b<^v_1 x)\wedge \neg (x<^v_1 b)$. We can then find an ordering with respect to $<^v_i$ of $B_2^v$ such that $b$ and $b'$ are one after the other. According to theorem \ref{theo:decomposition de condition de Kan}, we have a decomposition of $v$ of shape
 $...*_1 v_i*_1 v_{i+1}*_1....$ such that $v_i$ can be written as a $0$-composite of $b$ and $1$-generators and $v_{i+1}$ can be written in a $0$-composite of $b'$ and $1$-generators. We then have 
 \[ v_{i}:= ...*_0b*_0...~~~~ v_{i+1}:= ...*_0b'*_0... \]
 and then an equality between the following $1$-cells
 \[ ...*_0\pi^-_1b*_0...=\pi^-_1v_i=\pi^+_1v_{i+1}= ...*_0\pi^+_1b'*_0... \]
 As $\pi^-_1b\wedge \pi^+_1b'=0$, this implies that $\pi^-_1v_i=\pi^+_1v_{i+1}$ can be written as 
 \[ ... *_0 \pi^-_1b *_0 ... *_0 \pi^+_1b' *_0... \mbox{~~~or as~~~}... *_0 \pi^+_1b'*_0 ... *_0 \pi^-_1 b *_0...  \]
The cell $v_i*_1v_{i+1}$ can then be written as 
 \[ ... *_0 b *_0 ... *_0 b' *_0... \mbox{~~~or as~~~}... *_0 b'*_0 ... *_0 b *_0...  \]
This implies that $(b<_0 x)\vee (x<_0 b)$ holds.
\end{proof}

\begin{lemma}
\label{lemma:2 incompatiblite2}
Let $v$ be a $2$-cell, and $b,b'$ two elements of the $2$-support of $v$. Then $b<^v_0 b'$ implies that for all $\alpha\in \{-,+\}$, for all $c$ in $\langle b\rangle^\alpha_1$, $c<^v_0 b'$ holds.
\end{lemma}
\begin{proof}
By lemma \ref{lemma:other characterization of <0},
there exists a sequence $(b_i)_{i\leq n}$ such that $b_0=b$, $b_n=b'$ and for all $i<n$, $b_i$ and $b_{i+1}$ are $0$-composable. The sequence 
$$b*_0b_1*_0... *_0b_{n-1}*_0b'$$ is well defined, and then so is the sequence
$$\pi_1^\alpha b*_0b_1*_0... *_0b_{n-1}*_0b'.$$
As $\pi_1^\alpha b$ is a $0$-composite of $c$ with other elements of $B_1^v$, this concludes the proof.
\end{proof}

%

\begin{lemma}
\label{lemma:decompasiition}
 Let $r,u$ be two $2$-cells of $D$ such that $B_1^u\subset B_1^r$. Let $x$ in $B_2^u$.
Then there exists a unique decomposition of $u$ of shape
$$u= v*_{1}w*_{1}t$$ such that
\begin{enumerate}
\item for any element $b$ in $B_2^v$, $b<^r_1x$;
\item for any element $b$ in $B_2^t$, $x  <^r_1 b$;
\item for any element $b$ in $B_2^w$, $\neg (b<^r_1x) \vee  \neg (x<^r_1b)$
\end{enumerate}
If for any element of $b$ in $B_2^u$ different from $x$, $\neg (b<^r_1x) \vee  \neg (x<^r_1b)$, then  there exists a unique decomposition of $u$ of shape
$$u= v*_{0}w*_{0}t$$ such that
\begin{enumerate}
\item for any element $b$ in $B_1^v$, $b<^r_0x$;
\item for any element $b$ in $B_1^t$, $x  <^r_0 b$;
\item $w$ is either $x$ or a cell of lower dimension.
\end{enumerate}
\end{lemma}
\begin{proof}
We will construct these two decompositions at the same time. To this extend, we will use the Steiner theory recalled in section \ref{section:Steiner thery}.

Let $i$ be either $1$ or $0$. If $i=0$, we then suppose furthermore that for any element of $b$ in $B_2^u$ different from $x$, $\neg (b<^r_1x) \vee  \neg (x<^r_1b)$.
We denote by
$$\left(
\begin{array}{rcl}
u_0^-& u_1^-& u_2^-\\
u_0^+& u_1^+& u_2^+
\end{array}\right)$$
the array corresponding to the cell $u$.  
For any $i<j\leq 2$ and $\alpha\in\{-,+1\}$, we denote
$$
v_j^{\alpha}:=\sum \{b\in [u]_j^{\alpha}, ~b<_i x\} ~~~~~~
t_j^{\alpha}:=\sum \{b\in [u]_j^{\alpha},~ b>_i x\}  $$
$$
w_j^{\alpha}:=\sum \{b\in [u]_j^{\alpha}, ~\neg (b<_j x) \wedge \neg (b<_jx)\} 
$$
and 
$$
\begin{array}{lll}
v_i^{+}:= u_i^+& w_i^{+}:= v_i^{-} & t_i^{+}:= w_i^{-} \\
v_i^{-}:=  u_i^+- \partial (v_{i+1}^{-}) &
w_i^{-}:= v_i^{-} - \partial (w_{i+1}^{-}) &
t_i^{-}:= u_i^- 
\end{array}
$$
and for any $j<i$ and $\alpha\in\{-,+1\}$
$$
\begin{array}{lll}
v_j^{\alpha}:= u_j^{\alpha} & w_j^{\alpha}:= u_j^{\alpha} & t_j^{\alpha}:= u_j^{\alpha}\\
\end{array}
$$

By construction, 
we then have for any $i\leq j\leq 2$
$$u_j^{\alpha}=v_j^{\alpha}+w_j^{\alpha}+t_j^{\alpha}.$$
and 
$$\partial(v_{i+1}^-)= v_i^+-v_i^-~~~~
\partial(w_{i+1}^-)= w_i^+-w_i^- ~~~~
\partial(t_{i+1}^-)= t_i^+-t_i^-$$
and 
$$\partial(u_{i}^\alpha)= \partial(v_{i}^\alpha)=\partial(w_{i}^\alpha)=\partial(t_{i}^\alpha)$$
It then remains to show that  for any $i+1<j\leq 2$
\begin{equation}
\label{eq:first eq to show}
\partial v_j^\alpha = v_{j-1}^+-v_{j-1}^-~~~~~
\partial w_j^\alpha = w_{j-1}^+-w_{j-1}^-~~~~~
\partial t_j^\alpha = t_{j-1}^+-t_{j-1}^-
\end{equation}
and
\begin{equation}
\label{eq:seond eq to show}
v_i^-\geq 0~~~~ w_i^-\geq 0
\end{equation}
Indeed, if the assertions \eqref{eq:first eq to show} and \eqref{eq:seond eq to show} are fulfilled, this implies that the sequences $\{v_j^\beta\}$, $\{w_j^\beta\}$ and $\{t_j^\beta\}$  are arrays and then correspond respectively to the unique cells $v,w$ and $t$  fulfilling the desired condition. 

We first deal with the assertion \eqref{eq:first eq to show}.
Suppose first that there exists an integer $j$ such that $i+1<j\leq 2$. This implies that $i=0$. The lemma \ref{lemma:2 incompatiblite1} then implies that $w^\alpha_2=\lambda x$ with $\lambda\in \{0,1\}$.
By assumption, we have 
$$\partial (u^\beta_2) = u^+_1-u^-_1$$
and then 
$$\partial (v^\beta_2)+ \partial (w^\beta_2)  + \partial (t^\beta_2) = v^+_1-v^-_1+ w^+_1-w^-_1+ t^+_1-t^-_1$$
The lemma \ref{lemma:2 incompatiblite2} implies that  any element of the base belonging to $\partial (v^\beta_2)$ (resp. to $\partial (t^\beta_2)$) is $0$-inferior to $x$ (resp. $0$-superior to $x$). Moreover,  for any $b\in \partial (w^\beta_2)=\lambda \partial x$, we have $ \neg (b<^r_1x) \vee  \neg (x<^r_1b)$.

 The previous equality then implies
$$\partial (v^\beta_2) = v^+_1-v^-_1 ~~~~~~
\partial (w^\beta_2) =  w^+_1-w^-_1~~~~~~
 \partial (t^\beta_2) =  t^+_1-t^-_1$$

We now deal with the assertion \eqref{eq:first eq to show}.
We claim that we have
$$\partial^+ v^\alpha_{i+1}\wedge \partial^- w^\alpha_{i+1}=0~~~~~~~~\partial^+ w^\alpha_{i+1}\wedge \partial^- t^\alpha_{i+1}=0~~~~~~~~
\partial^+ v^\alpha_{i+1}\wedge \partial^- t^\alpha_{i+1}=0
$$

Indeed, suppose that $\partial^+ v^\alpha_{i+1}\wedge \partial^- w^\alpha_{i+1}\neq 0$. This implies that there exists an element of the base $b\in w^\alpha_{i+1}$ and $c\in v^\alpha_{i+1}$ such that $b<_i c$. As we have by definition $c<_i x$, this directly implies that $b<_i x$ which is absurd. We show similarly the two other equalities.
This implies that 
$$
\begin{array}{rcl}
u_i^{+} &\geq& \partial(u_{i+1}^-)\\
 &=& \partial^+(v_{i+1}^{-}+w_{i+1}^{-}+t_{i+1}^{-})\\
&=&\partial^+(v_{i+1}^{-}) + (\partial^+(w_{i+1}^{-}) - \partial^-(v_{i+1}^{-}))_+ +  (\partial^+(t_{i+1}^{-}) - \partial^-(w_{i+1}^{-}) -  \partial^-(v_{i+1}^{-}))_+
\end{array}$$
As a consequence, we have 
$$
\begin{array}{rcl}
v^-_{i} &=& u_i^+- \partial (v_{i+1}^{-})\\
&=& u_i^+- \partial^+ (v_{i+1}^{-})+ \partial^- (v_{i+1}^{-})\\
&\geq&  (\partial^+(w_{i+1}^{-}) - \partial^-(v_{i+1}^{-}))_+ +  (\partial^+(t_{i+1}^{-}) - \partial^-(w_{i+1}^{-}) -  \partial^-(v_{i+1}^{-}))_+ + \partial^- (v_{i+1}^{-})\\
&\geq&  (\partial^+(w_{i+1}^{-}) - \partial^-(v_{i+1}^{-}))_+ + \partial^- (v_{i+1}^{-})\\
&\geq & \partial^+(w_{i+1}^{-})\\
\end{array}
$$
and
$$w^-_{i} =  v_i^{-} - \partial (w_{i+1}^{-})= v_i^{-} - \partial^+ (w_{i+1}^{-})+ \partial^- (w_{i+1}^{-})\geq 0$$
The two assertions \eqref{eq:first eq to show} and \eqref{eq:seond eq to show} are then fulfilled, which concludes the proof.
\end{proof}

\begin{lemma}
\label{lemma:unicity of fact 1}
Let $C$ be a  $2$-category with a atomic and loop free basis. Let $x$ be a element of the base of $C$, and $y$ an element of the base of $D$. Let $f:C\to D$ be a morphism such that $\lambda f x=y$. Let $u$ be an $2$-cell of $C$. We denote by $u=:u_0*_0u_1*_0u_2$ and $f(u)=:v_0*_1v_1*_1v_2$ the decomposition given by the lemma \ref{lemma:decompasiition}. Then
$$f(u_0)=v_0~~~~ f(u_1)=v_1~~~~ f(u_2)=v_2$$
\end{lemma}
\begin{proof}
This is a direct consequence of lemma \ref{lemma:unicity of fact 1}.
\end{proof}

\begin{lemma}
\label{lemma:unicity of fact 2}
Let $C$ be a  $2$-category with a atomic and loop free basis. Let $x$ be a element of the base of $C$, and $y$ an element of the base of $D$. Let $f:C\to D$ be a morphism such that $y$ belongs to $\lambda f x$. Let $u$ be an $2$-cell of $C$. We denote by $u=:u_0*_1u_1*_1u_2$  and $f(u)=:v_0*_1v_1*_1v_2$ the decompositions given by lemma \ref{lemma:decompasiition}. For any $i\leq 2$,  we denote by $f(u_i)=:u_{i0}*_1u_{i1}*_1u_{i2}$ the decomposition given by lemma \textit{op cit}.Then
$$v_0=u_{00}~~~~v_1 = u_{01}*_1u_{02}*_1u_{10}*_1u_{11}*_1u_{12}*_1 u_{20}*_1u_{21}~~~~ v_2=u_{22}$$
\end{lemma}
\begin{proof}
This is a direct consequence of lemma \ref{lemma:unicity of fact 1}.
\end{proof}

\begin{notation}
Let $a$ be a globular sum of dimension lower or equal to $2$. We denote by $\triangledown$ the unique algebraic morphism $\Db_2\xrightarrow{} a$. The  $2$-cell $\triangledown$ is called the \textit{composite cell} of $a$.
\end{notation}

\begin{remark}
\label{rem:composite cell an algebraic morphism}
If $i:a\to a'$ is an algebraic morphism, and $f:a'\to C$ any morphism, the composite cell of $f:a'\to C$ is the same as the composite cell of $fi:a\to C$.
\end{remark}

\begin{definition}
Let $b$ be an element of the base of $D$.
A $2$-cell $v$ of $D$ is \textit{$0$-comparable} with $b$ if $b\in B_2^v$ and if for any $b'\in B_2^v$, the assertion 
$\neg (b<_1^v b') \wedge \neg (b'<_1^v b)$
holds.
\end{definition} 
\begin{lemma}
\label{lemma:simplification gamma,0}
Let $a$ be a globular sum of dimension lower or equal to $2$. Let $x$ be a $2$-cell of $D$.
Let $f:a\to D$ be a morphism such that $f(\triangledown)$ is $0$-comparable with $x$.
Then there exists a commutative triangle 
\[\begin{tikzcd}
	& {a'\vee [[1],1]\vee a''} \\
	a & D
	\arrow["f"', from=2-1, to=2-2]
	\arrow["i", from=2-1, to=1-2]
	\arrow["{f'\vee x\vee f''}", from=1-2, to=2-2]
\end{tikzcd}\]
Moreover, this factorization is functorial in $C$.
\end{lemma} 
\begin{proof}
Let $d$ be the (necessarily unique) element of the basis of $a$ such that $x\in [f(d)]_2$. Let $k\leq1$ and $j:[[k],1]\to \Sp_a$ be an element of the basis, i.e., a globular morphism.

If $j=d$, we consider the diagram
\[\begin{tikzcd}
	& {[[1],1]\vee[[1],1]\vee[[1],1]} \\
	{[[1],1]} & D
	\arrow["fj"', from=2-1, to=2-2]
	\arrow[from=2-1, to=1-2]
	\arrow["{f'\vee x\vee f''}", from=1-2, to=2-2]
\end{tikzcd}\]
and if $j$ is different of $d$ by share the same $0$-source and $0$-target, we consider the diagram
\[\begin{tikzcd}
	& {[[k],1]\vee[1]\vee[[k],1]} \\
	{[[k],1]} & D
	\arrow["fj"', from=2-1, to=2-2]
	\arrow[from=2-1, to=1-2]
	\arrow[from=1-2, to=2-2]
\end{tikzcd}\]
where these two decompositions are induced by lemma \ref{lemma:decompasiition}. If the $0$-source and $0$-target of $j$ are different of the one of $d$, we consider the diagram
\[\begin{tikzcd}
	& {[[k],1]} \\
	{[[k],1]} & D
	\arrow["fj"', from=2-1, to=2-2]
	\arrow[from=2-1, to=1-2]
	\arrow["fj", from=1-2, to=2-2]
\end{tikzcd}\] 
Taking the colimit over all such $j:[[k],1]\to a$, this induces a factorization
\[\begin{tikzcd}
	& {a'\vee [[1],1]\vee a''} \\
	a & D
	\arrow["f"', from=2-1, to=2-2]
	\arrow["i", from=2-1, to=1-2]
	\arrow["{f'\vee x\vee f''}", from=1-2, to=2-2]
\end{tikzcd}\]
fulfilling the desired property. Eventually, the functoriality of this factorization is a consequence of the unicity of the decomposition given in lemma \ref{lemma:decompasiition} and of lemma \ref{lem:functorialite de <}.
\end{proof}
\vspace{1cm}

Until the end of this section, we fix an other $2$-category  $C$ admitting a loop-free and atomic basis, and fitting in a cocartesian square of $\ocat$ of shape:
\[\begin{tikzcd}
	{\partial[[1],1]} & C \\
	{[[1],1]} & D
	\arrow["f", from=1-2, to=2-2]
	\arrow["x"', from=2-1, to=2-2]
	\arrow[from=1-1, to=2-1]
	\arrow["{\partial x}", from=1-1, to=1-2]
	\arrow["\lrcorner"{anchor=center, pos=0.125, rotate=180}, draw=none, from=2-2, to=1-1]
\end{tikzcd}\]

\begin{construction}
We define $\Gamma_0$ as the full subcategory of $(\Theta_2)_{/D}$ whose objects are morphisms $f:a\to D$ such that either $f$ factors through $C$, or the following conditions are fulfilled:
\begin{enumerate}
\item $f(\triangledown)$ is $0$-comparable with $x$.
\item $\Sp_a\to a\to D$ factors through the $\Theta$-set $C\cup x$.
\end{enumerate}
We define $\Gamma_1$ as the full subcategory of $(\Theta_2)_{/D}$ whose objects are morphisms $v:a\to D$ such that $\Sp_a\to a\to D$ factors through the $\Theta$-set $C\cup \colim_{\Gamma_0}a$.
\end{construction}

\begin{lemma}
\label{lem:injectif 1}
The canonical morphism of $\Theta$-sets $\iota:\colim_{\Gamma_0}a\to D$ is injective. Its image corresponds to morphisms $f:a\to D$ such that either $f$ factors through $C$, or the $2$-cell $f(\triangledown)$ is $0$-comparable with $x$.
\end{lemma}
\begin{proof}
First, remark that the morphism $C\to \colim_{\Gamma_0}a$ is injective. To complete the characterization of the image of $\iota$, let $f:a\to D$ be a morphism such that $f(\triangledown)$ is $0$-comparable with $x$.

Consider now the factorization $a\xrightarrow{i} a'\xrightarrow{g} D$ of $f$ given by lemma \ref{lemma:simplification gamma,0}. Every element of $\Sp a'$ is sent to either an element of $C$ or to $x$. This implies that $g$ belongs to $\Gamma_0$, which concludes the characterization of the image of $\iota$.

Now, for the injectivity, suppose that there exists another element $h:b\to D$ of $\Gamma_0$ and a decomposition $a\xrightarrow{j} b\xrightarrow{h} D$ of $f:a\to D$. Up to further factorization, we can suppose that $j$ is algebraic and, according to lemma \ref{lemma:simplification gamma,0}, that $j(\triangledown)$ is $0$-comparable with the (necessarily unique) element of the basis $c$ of $b$ such that $g(c)=x$. 

Using once again the factorization lemma \ref{lemma:simplification gamma,0} on the morphism $j$ and the object $c$, and using the functoriality of this factorization, we get a commutative diagram
\[\begin{tikzcd}
	a & b \\
	{a'} & D
	\arrow["i"', from=1-1, to=2-1]
	\arrow["g"', from=2-1, to=2-2]
	\arrow["h", from=1-2, to=2-2]
	\arrow["j", from=1-1, to=1-2]
	\arrow[from=2-1, to=1-2]
\end{tikzcd}\]
 completing the proof of injectivity.
\end{proof}


\begin{lemma}
\label{lem:injectif 2}
The canonical morphism of $\Theta$-sets $\iota:\colim_{\Gamma_1}a\to D$ is an equivalence.
\end{lemma}
\begin{proof}
First, remark that the morphism $C\to \colim_{\Gamma_1}a$ is injective. To complete the surjectivity of $\iota$, let $f:a\to D$ be a morphism such that $x$ belongs to $[f(\triangledown)]_2$. We denote by $c$ as the (necessary) unique element of the base of $a$ such that $x\in [f(c)]_2$.

Let $k\leq 1$ and $j:[[k],1]\to \Sp_a$ be an element of the basis. If $j$ is $c$, we consider the following diagram
$$[[1],1]\to [[3],1]\to D$$
induced by the decomposition of lemma \ref{lemma:decompasiition}. Moreover, lemma \ref{lem:injectif 1} implies that $l$ belongs to $\Gamma_1$.
If $j$ is different from $c$, we consider the diagram
$$[[k],1]\to [[k],1]\to D$$
Moreover, $fj$ factors through $C$ and then belongs to $\Gamma_1$. Taking the colimit over all such $j$, this induces a diagram
$$a\xrightarrow{i} a' \xrightarrow{g} D$$
whose composite is $f$ and such that $g$ is in $\Gamma_1$. This concludes the proof of the surjectivity of $\iota$.

To prove the injectivity, suppose now that there exists another element $h:b\to D$ and a decomposition $a\xrightarrow{j} b\xrightarrow{h} D$ of $f:a\to D$ with $h$ in $\Gamma_1$. If $j$ is $c$, we consider the diagram
\[\begin{tikzcd}
	{[[1],1]} & {[[3],1]} & {a'} \\
	{[[3],1]} & {[[9],1]} \\
	b && D
	\arrow[from=2-1, to=2-2]
	\arrow[from=1-1, to=2-1]
	\arrow["g", from=1-3, to=3-3]
	\arrow["h"', from=3-1, to=3-3]
	\arrow["t", from=2-2, to=3-3]
	\arrow[from=2-1, to=3-1]
	\arrow[from=1-2, to=1-3]
	\arrow["{[\sigma,1]}", from=1-2, to=2-2]
	\arrow[from=1-1, to=1-2]
\end{tikzcd}\]
where the left vertical morphisms are induced by the decomposition of  lemma \ref{lemma:decompasiition}, the morphism $t$ obtained in applying for each $2$-cell the decomposition of lemma \textit{op cit}, and the morphism $\sigma$ send $0$ on $0$, $1$ on $1$, $2$ on $8$ and $3$ on $9$. The commutativity of this diagram is a consequence of  lemma \ref{lemma:unicity of fact 2}. 

If $j$ is different from $c$, we consider the diagram
\[\begin{tikzcd}
	{[[k],1]} & {[[k],1]} & {a'} \\
	{[[k],1]} & {[[k],1]} \\
	b && D
	\arrow[Rightarrow, no head, from=2-1, to=2-2]
	\arrow[Rightarrow, no head, from=1-1, to=2-1]
	\arrow[from=1-3, to=3-3]
	\arrow["h"', from=3-1, to=3-3]
	\arrow[from=2-2, to=3-3]
	\arrow[from=2-1, to=3-1]
	\arrow[from=1-2, to=1-3]
	\arrow[Rightarrow, no head, from=1-1, to=1-2]
	\arrow[Rightarrow, no head, from=1-2, to=2-2]
\end{tikzcd}\]
Taking the colimit over all such $j$, this induces a diagram
\[\begin{tikzcd}
	a & {a'} & {a'} \\
	{a'} & {a''} \\
	b && D
	\arrow[from=2-1, to=2-2]
	\arrow["i"', from=1-1, to=2-1]
	\arrow["g", from=1-3, to=3-3]
	\arrow["h"', from=3-1, to=3-3]
	\arrow[from=2-2, to=3-3]
	\arrow[from=2-1, to=3-1]
	\arrow[Rightarrow, no head, from=1-2, to=1-3]
	\arrow["i", from=1-1, to=1-2]
	\arrow[from=1-2, to=2-2]
\end{tikzcd}\]
where $a''\to D$ is in $\Gamma_1$,
which concludes the proof of injectivity.	
\end{proof}

\begin{lemma}
\label{lemma: lambdA gamma}
Let $f: a\to D$ be a morphism of $\Gamma_0$. We denote by $\Lambda^{\Gamma_0} a$  the subobject of $a$ composed of all $i\in {\Theta_2}_{/a}$ such that $fi$ factors through the  $\Theta_2$-set $C\cup x$. Then the morphism $\Lambda^{\Gamma_0} a\to a$ is in $\overline{\W_2}$.
\end{lemma}
\begin{proof}
If $f$ factors through $C$, then $\Lambda^{\Gamma_0}a$ is equal to $a$. Suppose then that there exists a (necessarily unique) element of the base $b$ such that $f(b)=x$. 

There exists a unique decomposition of $a$ as
$$a\cong a' \vee [[k]\vee[1]\vee[k'],1]\vee a''$$
where the cell $[[1],1]\to a$ is $b$ and where
$$[[k],1]\to a\to D~~~~\mbox{and}~~~~[[k'],1]\to a\to D$$
factors through $C$.

We then have
$$\Lambda^{\Gamma_0} a\cong a'\vee[[k]\coprod_{[0]} [1]\coprod_{[0]}[k'],1]\vee a''$$
As the functor $a'\vee[\uvar,1]\vee a:\Psh{\Delta}\to \Psh{\Theta}$ sends $\overline{\W_1}$ to $\overline{\W_2}$, and as $$[k]\coprod_{[0]} b\coprod_{[0]}[k']\to [k+1+k']$$ is in $\overline{\W_1}$, this concludes the proof.
\end{proof}

\begin{lemma}
\label{lemma: lambdA gamma2}
Let $f: a\to D$ be a morphism of $\Gamma_1$. We denote by $\Lambda^{\Gamma_1} a$  the subobject of $a$ composed of all $i\in \Theta_{/a}$ such that $fi$ factors through $\colim_{\Gamma_0}a$. Then the morphism $\Lambda^{\Gamma_1} a\to a$ is in $\overline{\W_2}$.
\end{lemma}
\begin{proof}
If $f$ factors through $C$, then $\Lambda^{\Gamma_1}a$ is equal to $a$. Suppose then that there exists a (necessarily unique) element of the base $b$ such that $x$ belongs to $[f(b)]_2$.

There exists a unique decomposition of $a$ as
$$a\cong a' \vee [[n]\vee[k]\vee[1]\vee[k']\vee[n'],1]\vee a'' $$
where the cell $[[1],1]\to a$ is $b$, and where $k$ and $k'$ are the maximal integers such that the image by the composite cell of 
$$[[k]\vee[1]\vee[k'],1]\to a$$
is $0$-comparable with $x$, and such that
$$[[k],1]\coprod [[k'],1]\to a\to D$$
factors through $C$.

We then have
$$\Lambda^{\Gamma_0} a\cong a'\vee[[n+k]\coprod_{[k]} [k+1+k']\coprod_{[k']}[k'+n'],1]\vee a''$$
As the functor $a'\vee[\uvar,1]\vee a:\Psh{\Delta}\to \Psh{\Theta}$ sends $\overline{\W_1}$ to $\overline{\W_2}$, and as 
$$[n+k]\coprod_{[k]} [k+1+k']\coprod_{[k']}[k'+n']\to [n+k +1+k'+n']$$ is in $\overline{\W_1}$, this concludes the proof.
\end{proof}

\begin{prop}
\label{prop: case of 2 category}
Let $C$ and $D$ be two $2$-categories admitting loop-free and atomic bases, fitting in a  cocartesian square of shape:
\[\begin{tikzcd}
	{\partial[[1],1]} & C \\
	{[[1],1]} & D
	\arrow["f", from=1-2, to=2-2]
	\arrow["x"', from=2-1, to=2-2]
	\arrow["{\partial x}", from=1-1, to=1-2]
	\arrow["\lrcorner"{anchor=center, pos=0.125, rotate=180}, draw=none, from=2-2, to=1-1]
	\arrow[from=1-1, to=2-1]
\end{tikzcd}\]
Then, viewed as a morphism of $\Psh{\Theta_2}$, the morphism $j:C\cup x\to D$ is in $\overline{\W_2}$.
\end{prop}
\begin{proof}
The category $\Gamma_0$ inherits from $\Theta_{/D}$ a structure of Reedy elegant category. The two functors 
$$\begin{array}{ccccccc}
\Gamma_0&\to &\Psh{\Delta} &~~~~~&\Gamma_0&\to &\Psh{\Delta}\\
a\to D&\mapsto &\Lambda^{\Gamma_0}a &~~~~~& a\to D &\mapsto & a
\end{array}$$
are Reedy cofibrant (definition \ref{defi:reedycof}). The morphism 
$$C\cup x\cong \colim_{\Gamma_0}\Lambda^{\Gamma_0}a\to \colim_{\Gamma_0}\Lambda^{\Gamma_0}a$$
is then in $\overline{\W_2}$. We proceed similarly to demonstrate that the morphism
$$ \colim_{\Gamma_0}\Lambda^{\Gamma_0}a\cong \colim_{\Gamma_1}\Lambda^{\Gamma_1} a\to \Lambda^{\Gamma_1} a\cong D$$
is in $\overline{\W_2}$. By stability by composition of $\overline{\W_2}$, this concludes the proof.
\end{proof}

\begin{prop}
\label{prop: case of 1 category}
Let $C$ and $D$ be two $1$-categories admitting loop-free and atomic bases,  fitting in a  cocartesian square of shape:
\[\begin{tikzcd}
	{\partial[1]} & C \\
	{[1]} & D
	\arrow["f", from=1-2, to=2-2]
	\arrow["x"', from=2-1, to=2-2]
	\arrow["{\partial x}", from=1-1, to=1-2]
	\arrow["\lrcorner"{anchor=center, pos=0.125, rotate=180}, draw=none, from=2-2, to=1-1]
	\arrow[from=1-1, to=2-1]
\end{tikzcd}\]
Then, viewed as a morphism of $\Psh{\Delta}$, the morphism $j:C\cup x\to D$ is in $\overline{\W_1}$.
\end{prop}
\begin{proof}
We denote by $\Upsilon$  the full subcategory of $\Delta_{/D}$ whose objects are morphisms $f:[n]\to D$ such that $\Sp_{[n]}\to [n]\to D$ factors through the $\Theta$-set $C\cup x$. 

Given $f:[n]\to D$ in $\Upsilon$, we denote by $\Lambda^{\Upsilon}[n]$ the subobject of $[n]$ composed of all $i\in \Delta_{/[n]}$ such that $fi$ factors through $C\cup x$. We can proceed as in lemma \ref{lemma: lambdA gamma} to show that the canonical morphism $\Lambda^{\Upsilon}[n]\to [n]$ is in $\overline{\W_1}$.

Now, remark that the category $\Upsilon$ inherits from $\Delta_{/D}$ a structure of Reedy elegant category. The two functors 
$$\begin{array}{ccccccc}
\Upsilon&\to &\Psh{\Delta} &~~~~~&\Upsilon&\to &\Psh{\Delta}\\
{[n]}\to D&\mapsto &\Lambda^{\Upsilon}[n] &~~~~~& [n]\to D &\mapsto & [n]
\end{array}$$
are Reedy cofibrant (definition \ref{defi:reedycof}). As the colimit of the first one is $C\cup x$ and the colimit of the second one is $D$, this concludes the proof.
\end{proof}

\begin{proof}[Proof of theorem \ref{theo: case of 1 and 2 category}]
If $n=0$, this is straightforward, and if $n=2$, it follows from proposition \ref{prop: case of 2 category}.

It then remains to prove the case $n=1$. 
Let $S$ be the set of generators of $C$ of dimension $2$. A repeated application of proposition \ref{prop: case of 2 category} and the stability by pushout  and transfinite composition of $\overline{\W_2}$ implies that the two vertical morphisms of the following square are in $\overline{\W_2}$:
\[\begin{tikzcd}
	{\tau_1C\cup x\cup_{y\in S} y} & {\tau_1D \cup_{y\in S} y} \\
	{C\cup x} & D
	\arrow[from=1-2, to=2-2]
	\arrow[from=1-1, to=2-1]
	\arrow[from=2-1, to=2-2]
	\arrow[from=1-1, to=1-2]
\end{tikzcd}\]
Moreover, the proposition \ref{prop: case of 1 category} implies that the canonical morphism 
$$\tau_1 C\cup x\to \tau_1 D$$ is in $\overline{\W_2}$, and so is the top horizontal morphism of the previous square. By two out of three of $\overline{\W_2}$, this concludes the proof.
\end{proof}

\subsection{Gray operations on augmented directed complexes}
We follow Steiner (\cite{Steiner_omega_categories_and_chain_complexes}) and Ara-Maltsiniotis (\cite{Ara_Maltsiniotis_joint_et_tranche}) for the definitions and first properties of Gray operations on augmented directed complexes.

\begin{definition}
Let $(K,K^*,e)$ and $(L,L^*,f)$ be two augmented directed complexes. We define the \snotion{Gray tensor product}{for augmented directed complexes} of $(K,K^*,e)$ and $(L,L^*,f)$ as the augmented directed complex
$$(K,K^*,e)\otimes (L,L^*,f):= (K\otimes L,(K\otimes L)^*,e\otimes f)$$
where 
\begin{enumerate}
\item[$-$] $K\otimes L$ is the chain complex whose value on $n$ is:
$$(K\otimes L)_n:= \oplus_{k+l=n}K_k\otimes L_l$$
and the differential is the unique graded group morphism fulfilling: 
$$\partial (x\otimes y):= \partial x\otimes y + (-1)^{|x|}x\otimes \partial y$$
where we set the convention $\partial x:=0$ if $|x|=0$.
\item[$-$] $(K\otimes L)^*$ is given on all integer $n$ by :
$$(K\otimes L)^*_n:= \oplus_{k+l=n}K_k^*\otimes L_l^*.$$
\item[$-$] $e\otimes f:K_0\otimes L_0\to \Zb$ is the unique morphism fulfilling 
$$(e\otimes f)(x\otimes y)= e(x)f(y).$$
\end{enumerate}
\end{definition}
The Gray tensor product induces a monoidal structure on $\CDA$. Its unit is given by $\lambda \Db_0$. Furthermore, Steiner shows that if $K$ and $L$ admit loop free and unitary bases, so does $K\otimes L$. The basis of $K\otimes L$ is given by the set of elements of shape $b\otimes b'$ where $b$ and $b'$ are respectively elements of the bases of $K$ and $L$.
 The monoidal structure then restricts to a monoidal structure on $\CDAB$.

\begin{notation}
 To simplify notation, the augmented directed complex $\lambda[1]$ will simply be denoted by $[1]$. 
\end{notation}
\begin{definition}
The induced functor 
$$\uvar\otimes [1]:\CDA\to \CDA$$
is called the \snotionsym{Gray cylinder}{((d30@$\uvar\otimes[1]$}{for augmented directed complexes}. 
For $(K,K^*,e)$ an augmented directed complex, we then have
$$(K,K^*,e)\otimes [1]:=(K\otimes [1] ,(K\otimes [1])^*,e)$$
where
\begin{enumerate}
\item[$-$] $K\otimes [1]$ is the chain complex whose value on $n$ is:
$$(K\otimes [1])_n:=\left\{
\begin{array}{ll}
\{x\otimes \{\epsilon\},x\in K_0,\epsilon=0,1\}&\mbox{if $n=0$}\\
\{x\otimes \{\epsilon\},x\in K_n,\epsilon=0,1\}\oplus \{x\otimes[1],x\in K_{n-1}\} &\mbox{if $n>0$}
\end{array}\right.$$
and the differential is the unique graded group morphism fulfilling: 
$$\partial (x\otimes [1]):= \partial x\otimes [1] + (-1)^{|x|}(x\otimes \{1\}-x\otimes \{0\} )~~~~~\partial (x\otimes\{\epsilon\}) = (\partial x)\otimes\{\epsilon\}$$
for $\epsilon\in\{0,1\}$, and
where we set the convention $\partial x:=0$ if $|x|=0$.
\item[$-$] $(K\otimes [1])^*$ is given on all integer $n$ by :
$$(K\otimes [1])^*_n:=\left\{
\begin{array}{ll}
\{x\otimes \{\epsilon\},x\in K^*_0,\epsilon=0,1\}&\mbox{if $n=0$}\\
\{x\otimes\{ \epsilon\},x\in K^*_n,\epsilon=0,1\}\oplus \{x\otimes[1],x\in K^*_{n-1}\} &\mbox{if $n>0$}
\end{array}\right.$$
\item[$-$] $e:(K\otimes [1])_0\to \Zb$ is the unique morphism fulfilling 
$$e(x\otimes \{0\})=e(x\otimes \{1\})= e(x).$$
\end{enumerate}
\end{definition}

\begin{prop}
\label{prop:non trivial automorphisme 0}
Let $A$ be an augmented directed complex admitting no non-trivial automorphisms. Then the augmented directed complexe $A\otimes [1]$ has no non-trivial automorphisms.
\end{prop}
\begin{proof}
Let $\phi:A \otimes[1]\to A\otimes [1]$ be an automorphism. The morphism $\phi$ then induces a bijection on the elements of the basis of $A\otimes [1]$.

Let $(E,F)$ be a partition of the set $(B_{A\otimes[1]})_0$ such that
\begin{enumerate}
\item there exists no element of $(B_{A\otimes[1]})_1$ whose source is in is $F$ and target in $E$.
\item for any $x,y\in E$ and $v\in (B_{A\otimes[1]})_1$ such that $\partial v=y-x$, there exist an element $w\in (B_{A\otimes[1]})_1$  such that $\partial^- w=y$ and an element $\alpha\in (B_{A\otimes[1]})_2$ with $\partial^+\alpha=w+v$.
\item for any $x,y\in F$ and $v\in (B_{A\otimes[1]})_1$ such that $\partial v=y-x$, there exist an element $w\in (B_{A\otimes[1]})_1$  such that $\partial^+ w=x$ and an element $\alpha\in (B_{A\otimes[1]})_2$ with $\partial^-\alpha=w+v$.
\end{enumerate}
Suppose now that there exists an object $a$ of $(B_A)_0$ such that $a\otimes \{1\}$ in $E$. As we have $\partial a\otimes[1]=a\otimes\{1\}-a\otimes\{0\}$, $a\otimes\{1\}$ is in $E$. There exist then an element $\alpha \in  (B_{A\otimes[1]})_2$ with $\partial^+ \alpha = a\otimes [1]+w$ with $\partial^+ a\otimes [1]= \partial^- w$. However, by construction of $A\otimes[1]$, there exist no such element $\alpha$. This implies that any element of $E$ is of shape $a\otimes\{0\}$ and we can show similarly that every element of $F$ is of shape $a\otimes\{1\}$. 

Conversely, we claim that the partition $((B_{A\otimes\{0\}})_0, (B_{A\otimes\{1\}})_0)$ fulfills these conditions.  The first one is obvious. For the second, there exist $a\in (B_A)_0$ and $u\in (B_A)_0$ such that $y=a\otimes\{0\}$ and $v:= u\otimes\{0\}$ and we then choose $w:=a\otimes[1]$ and $\alpha:= u\otimes [1]$. We proceed similarly for the last condition. 

The partition $((B_{A\otimes\{0\}})_0, (B_{A\otimes\{1\}})_0)$ is then the unique one fulfilling the previous three condition. As $\phi$ preserves such partition, this implies that $\phi(B_{A\otimes\{0\}})= B_{A\otimes\{0\}}$ and $\phi(B_{A\otimes\{1\}})= B_{A\otimes\{1\}}$.

Now, remark that for any element $e\in (A\otimes[1])^*_{n+1}$, there exists $x\in A^*_n$ such that $x\otimes[1] \leq e$ if and only if there exists $y\in A^*_{n-1}$ such that $y\otimes[1] \leq \partial^+e$. By a direct induction, this implies that there exists $x\in A^*_n$ such that $x\otimes[1]\leq  e$ if and only if $\partial^-_0e$ is in $A_0^*\otimes\{0\}$ and $\partial^+_0e$ is in $A_0^*\otimes\{1\}$.

Combined with the previous observation, this implies that for any element $x$ of the basis of $A_{n}$, $\phi(x\otimes\{\epsilon\})$ is of shape $x'\otimes\{\epsilon\}$ with $\epsilon\in\{0,1\}$.
The automorphism $\phi$ then induces by restriction  automorphisms $\phi_{|A\otimes\{0\}}:A\otimes\{0\}\to A\otimes\{0\}$ and $\phi_{|A\otimes\{1\}}:A\otimes\{1\}\to A\otimes\{1\}$, and the hypothesis implies that they are the identity.

We now show by induction on $n$ that $\phi_n:(A\otimes[1])_n\to (A\otimes[1])_n$ is the identity. Suppose the result true at the stage $n$. For any element $x$ of the basis of $A_{n}$, we then have 
$$\partial \phi(x\otimes[1]) = \phi(\partial (x\otimes[1])) = \partial (x\otimes[1]).$$
By the definition of the derivative of $A\otimes[1]$, and as $\phi$ preserves the basis, this forces the equality $\phi(x\otimes[1])=x\otimes[1]$. As we already know that for any element $x$ of the basis of $A_{n+1}$ we have $\phi(x\otimes\{\epsilon\})=x\otimes\{\epsilon\}t$ for any $\epsilon\in\{0,1\}$, this concludes the induction.

We then have $\phi=id$ and $A\otimes[1]$ has no non trivial automorphisms.
\end{proof}

\begin{definition}
We define the \snotionsym{Gray cone}{((d40@$\uvar\star 1$}{for augmented directed complexes}
$$\begin{array}{ccc}
\CDA &\to&\CDA\\
K&\mapsto &K\star 1 
\end{array}
$$
where $K\star 1$ is defined as the following pushout: 
\begin{equation}
\label{eq:defin of cstar costar CDA}
\begin{tikzcd}
	{K\otimes\{1\}} & {K\otimes [1]} \\
	1 & {K\star 1}
	\arrow[from=1-1, to=2-1]
	\arrow[from=1-1, to=1-2]
	\arrow[from=2-1, to=2-2]
	\arrow[from=1-2, to=2-2]
	\arrow["\lrcorner"{anchor=center, pos=0.125, rotate=180}, draw=none, from=2-2, to=1-1]
\end{tikzcd}
\end{equation}

According to \cite[corollary 6.21]{Ara_Maltsiniotis_joint_et_tranche}, if $K$ admits a loop free and unitary basis, this is also the case for $K\star 1$. The {Gray cone}  then induces a functor:
$$\begin{array}{ccc}
\CDAB&\to&\CDAB\\
K&\mapsto &K\star 1 \\
\end{array}
$$
\end{definition}

\begin{remark}
 Unfolding the definition, we have
$$(K,K',e)\star 1:=(K\star 1, (K\star 1)^*,e)$$
where
\begin{enumerate}
\item[$-$] $K\star 1$  is the chain complex whose value on $n$ is:
$$(K\star 1)_n:=\left\{
\begin{array}{ll}
\Zb[\emptyset\star 1]\oplus \{x\star \emptyset,x\in K_0\}&\mbox{if $n=0$}\\
\{\emptyset\star x,x\in K_n\}\oplus \{x\star 1,x\in K_{n-1}\} &\mbox{if $n>0$}
\end{array}\right.$$
and the differentials are the unique graded group morphisms fulfilling: 
$$\begin{array}{rr}
\partial (x\star 1)= \partial x\star 1 + (-1)^{|x|} x\star \emptyset&\partial( x \star \emptyset )=\partial x\star \emptyset \\
\end{array}$$
where we set the convention $\partial x:=0$ if $|x|=0$.
\item[$-$] The graded monoids $(K\star 1)^*$ is given on any integer $n$ by :
$$(K\star 1)^*:=\left\{
\begin{array}{ll}
\Nb[\emptyset\star 1]\oplus \{x\star \emptyset,x\in K^*_0\}&\mbox{if $n=0$}\\
\{\emptyset\star x,x\in K^*_n\}\oplus \{x\star 1,x\in K^*_{n-1}\} &\mbox{if $n>0$}
\end{array}\right.$$

\item[$-$] The augmentation $e:(K\star 1)_0\to \Zb$ is the unique ones fulfilling 
$$
\begin{array}{cc}
e( \emptyset \star 1) =1 & e(x\star \emptyset)=e(x)\\
\end{array}$$
\end{enumerate}
The basis of $K\star 1$ is given by the reunion of $\emptyset\star 1$ and of the set of elements of shape $b\star 1$  where $b$ is an element of the basis of $K$. 
\end{remark}

\begin{prop}
\label{prop:non trivial automorphisme 1}
Let $A$ be an augmented directed complex admitting no non-trivial automorphisms. Then the augmented directed complexe $A\star 1$  has no non-trivial automorphisms.
\end{prop}
\begin{proof}
Let $\phi:A\star 1\to A\star 1$ be an automorphism. The morphism $\phi$ then induces a bijection on the elements of the basis of $A\star 1$.

 As the element $\emptyset\star 1\in (A\star 1)_0$ is the only element of the basis such that for all $v\in (A\star 1)_1$  $\partial_0^-(v)\neq \emptyset\star 1$, it is preserved by $\phi$. As a consequence, for any element $x$ of the basis of $A_0$, $\phi(x\star \emptyset)$ is of shape $x'\star \emptyset$. The morphism $\phi$ then preserves $(A\star \emptyset)_0$.

Now, remark that for any element $e\in (A\star 1)^*_{n+1}$, there exists $x\in A^*_n$ such that $x\star 1\leq e$ if and only if there exists $y\in A^*_{n-1}$ such that $y\star 1\leq \partial^+e$. By a direct induction, this implies that there exists $x\in (A\star 1)^*_n$ such that $x\star 1\leq e$ if and only if $\partial^+_0e\in \Zb[\emptyset\star 1]$.

Combined with the previous observation, this implies that for any element $x$ of the basis of $A_{n}$, $\phi(x\star \emptyset)$ is of shape $x'\star \emptyset$.
The automorphism $\phi$ then induces by restriction an automorphism $\phi_{|A\star\emptyset}:A\to A$, and the hypothesis implies that it is the identity.

We now show by induction on $n$ that $\phi_n:(A\star 1)_n\to (A\star 1)_n$ is the identity. Suppose the result true at the stage $n$. For any element $x$ of the basis of $A_{n}$, we then have 
$$\partial \phi(x\star 1) = \phi(\partial (x\star 1)) = \partial (x\star 1).$$
By the definition of the derivative of $A\star 1$, and as $\phi$ preserves the basis, this forces the equality $\phi(x\star 1)=x\star 1$. As we already know that for any element $x$ of the basis of $A_{n+1}$ we have $\phi(x\star \emptyset)=x\star \emptyset$, this concludes the induction.

We then have $\phi=id$ and $A\star 1$ has no non trivial automorphisms.

\end{proof}

\begin{definition}
We define the \snotionsym{suspension}{((d60@$[\uvar,1]$}{for augmented directed complexes} as the functor 
$$[\uvar,1]:\CDA\to \CDA$$
where $[K,1]$ is defined as the following pushout:
\begin{equation}
\label{eq:def of suspension cda}
\begin{tikzcd}
	{K\otimes \{0,1\}} & {K\otimes [1]} \\
	{1\coprod 1} & {[K,1]}
	\arrow[from=1-1, to=2-1]
	\arrow[from=2-1, to=2-2]
	\arrow[from=1-1, to=1-2]
	\arrow[from=1-2, to=2-2]
	\arrow["\lrcorner"{anchor=center, pos=0.125, rotate=180}, draw=none, from=2-2, to=1-1]
\end{tikzcd}
\end{equation}
We leave to the reader to check that $[K,1]$ admits a loop free and unitary basis when this is the case for $K$. This functor then induces a functor:
$$[\uvar,1]:\CDAB\to \CDAB$$
\end{definition}

\begin{remark}
 Unfolding the definition, we have
$$[(K,K',e),1]:=([K,1] ,([K,1])^*,e)$$
where
\begin{enumerate}
\item[$-$] $[K,1]$ is the chain complex whose value on $n$ is:
$$[K,1]:=\left\{
\begin{array}{ll}
 \Zb[\{0\},\{1\}]&\mbox{if $n=0$}\\
 \{[x,1],x\in K_{n-1}\} &\mbox{if $n>0$}
\end{array}\right.$$
and the differential is the unique graded group morphism fulfilling: 
$$\partial([x,1]):= \left\{
 \begin{array}{lll} 
 \{1\}-\{0\}&\mbox{if $|x|=0$}\\ 
 ~[\partial x,1]&\mbox{if $|x|>0$}
 \end{array}\right.
$$
\item[$-$] $([K,1])^*$ is given on all integer $n$ by:
$$([K,1])^*_n:=\left\{
\begin{array}{ll}
\Nb[0,1]&\mbox{if $n=0$}\\
 \{[x,1],x\in K^*_{n-1}\} &\mbox{if $n>0$}
\end{array}\right.$$
\item[$-$] $e:([K,1])_0\to \Zb$ is the unique morphism	 fulfilling 
$$e( 0)=e( 1)= e(x).$$
\end{enumerate}
The basis of $[K,1]$ is given by the reunion of $\{0\}$, $\{1\}$ and of the set of elements of shape $[b,1]$  where $b$ is an element of the basis of $K$. 
\end{remark}

\begin{prop}
\label{prop:non trivial automorphisme 2}
Let $A$ be a non null augmented directed complex admitting no non-trivial automorphisms. Then the augmented directed complex $[A,1]$ has no non-trivial automorphisms.
\end{prop}
\begin{proof}
Let $\phi:[A,1]\to [A,1]$ be an automorphism. As the element $\{1\}\in ([A,1])_0$ is the only element of the basis such that for all $v\in [A,1]_1$  $\partial_0^-(v)\neq \{1\}$, it is preserved by $\phi$. As a consequence, $\phi$ also preserves $\{0\}$. The induced morphism $\phi_0:[A,1]_0\to [A,1]_0$ is then the identity. 

Now, remark that $(\phi_{n+1})_{n\in \Nb}:A\to A$ is an automorphism and is then the identity. This implies that for all $n>0$, $\phi_n:[A,1]_n\to [A,1]_n$ is then identity, which concludes the proof.
\end{proof}

\begin{definition}
We define the \textit{wedges} as the functors
$$[\uvar,1]\vee[1]:\CDA\to \CDA~~~~~~ [1]\vee[\uvar,1]:\CDA\to \CDA$$
where $[K,1]\vee [1]$ and $[1]\vee[K,1]$ are defined as the following pushouts:
\[\begin{tikzcd}
	{\lambda [0]} & {[1]} && {\lambda [0]} & {[K,1]} \\
	{[K,1]} & { [K,1]\vee[1]} && {[1]} & {[1]\vee[K,1]}
	\arrow["{\{1\}}"', from=1-1, to=2-1]
	\arrow[from=2-1, to=2-2]
	\arrow["{\{0\}}", from=1-1, to=1-2]
	\arrow[from=1-2, to=2-2]
	\arrow["\lrcorner"{anchor=center, pos=0.125, rotate=180}, draw=none, from=2-2, to=1-1]
	\arrow["{\{0\}}", from=1-4, to=1-5]
	\arrow["{\{1\}}"', from=1-4, to=2-4]
	\arrow[from=1-5, to=2-5]
	\arrow[from=2-4, to=2-5]
	\arrow["\lrcorner"{anchor=center, pos=0.125, rotate=180}, draw=none, from=2-5, to=1-4]
\end{tikzcd}\]
Once again, we can easily check that $[K,1]\vee[1]$ and $[1]\vee[K,1]$ have a loop free and unitary basis when this is the case for $K$. These functors then induce functors
$$[\uvar,1]\vee[1]:\CDAB\to \CDAB~~~~~~ [1]\vee[\uvar,1]:\CDAB\to \CDAB$$
\end{definition}

\vspace{1cm}

 Unfolding the definition, we have
$$[(K,K',e),1]\vee [1]:=([K,1]\vee [1] ,([K,1]\vee [1])^*,e)$$ $$
[1]\vee(K,K',e),1]:=([1]\vee[K,1] ,([1]\vee[K,1])^*,e)$$
where
\begin{enumerate}
\item[$-$] $[K,1]\vee [1]$ and $[1]\vee[K,1]$ are the chain complexes whose value on $n$ are:
$$[K,1]\vee[1]:=\left\{
\begin{array}{ll}
\Zb[\{0\},\{1\},\{2\}]&\mbox{if $n=0$}\\
 \{[x,1],x\in K_{0}\}\oplus \Zb[e_1] &\mbox{if $n=1$}\\
 \{[x,1],x\in K_{n-1}\} &\mbox{if $n>1$}
\end{array}\right.$$
$$[1]\vee[K,1]:=\left\{
\begin{array}{ll}
\Zb[\{0\},\{1\},\{2\}]&\mbox{if $n=0$}\\
\Zb[e_1] \oplus \{[x,1],x\in K_{0}\} &\mbox{if $n=1$}\\
 \{[x,1],x\in K_{n-1}\} &\mbox{if $n>1$}
\end{array}\right.$$
and the differentials are the unique graded group morphism fulfilling: 
$$\partial_{[K,1]\vee[1]} (e_1):= \{2\}-\{1\}
~~~
\partial_{[K,1]\vee[1]} ([x,1]):=
\left\{
\begin{array}{ll}
 \{1\}-\{0\}&\mbox{if $|x|=0$}\\
 ~[\partial x,1]&\mbox{if $|x|>0$}\\
\end{array}\right.
$$
$$
\partial_{[1]\vee[K,1]} (e_1):= \{1\}-\{0\}
~~~
\partial_{[1]\vee[K,1]} ([x,1]):=
\left\{
\begin{array}{ll}
 \{2\}-\{1\}&\mbox{if $|x|=0$}\\
~ [\partial x,1]&\mbox{if $|x|>0$}\\
\end{array}\right.
$$
\item[$-$] $([K,1]\vee [1])^*$ and $([1]\vee[K,1])^*$ are given on all integer $n$ by:
$$([K,1]\vee[1])^*:=\left\{
\begin{array}{ll}
\{\{0\},\{1\},\{2\}\}&\mbox{if $n=0$}\\
 \{[x,1],x\in K_0^*\}\oplus \Nb[e_1] &\mbox{if $n=1$}\\
 \{[x,1],x\in K_{n-1}\} &\mbox{if $n>1$}
\end{array}\right.$$
$$([1]\vee[K,1])^*:=\left\{
\begin{array}{ll}
\{\{0\},\{1\},\{2\}\}&\mbox{if $n=0$}\\
\Nb[e_1]\oplus\cup \{[x,1],x\in K^*_{0}\} &\mbox{if $n=1$}\\
 \{[x,1],x\in K^*_{n-1}\} &\mbox{if $n>1$}
\end{array}\right.$$
\item[$-$] The augmentations $e$ are the unique morphism fulfilling 
$$e( \{0\})=e(\{ 1\})= e(\{2\})=1.$$
\end{enumerate}

\begin{prop}
\label{prop:non trivial automorphisme 3}
Let $A$ be a non null augmented directed complex admitting no non-trivial automorphisms. Then the augmented directed complexes $[A,1]\vee[1]$ and $[1]\vee[A,1]$ have no non-trivial automorphisms.
\end{prop}
\begin{proof}
The proof is similar to the one of proposition \ref{prop:non trivial automorphisme 2} and we leave it to the reader.
\end{proof}

\begin{definition}
There are two canonical morphisms 
$$\triangledown:\Sigma K\to \Sigma K \vee [1]
~~~~~~~ \triangledown:\Sigma K\to [1]\vee \Sigma K $$
that are the unique ones fulfilling
$$\triangledown(\{0\}):= \{0\}~~~\triangledown(\{1\}):= \{2\}~~~
\triangledown([x,1]):=\left\{ 
\begin{array}{ll}
~[x,1]+e_1&\mbox{if $|x|=0$}\\
~[x,1]&\mbox{if $|x|>0$}\\
\end{array}\right.$$
When we write $ \Sigma K\to \Sigma K \vee [1]$ and $\Sigma K\to [1]\vee \Sigma K$ and nothing more is specified, it will always mean that we considered the morphisms $\triangledown$.
\end{definition}

\begin{prop}
 \label{prop:appendice formula for otimes cda}
 Let $K$ be an augmented directed complex. 
 There is a natural transformation between the colimit of the following diagram
$$
\begin{tikzcd}
	{[1]\vee [K,1]} & {[K\otimes\{0\},1]} & {[K\otimes [1],1]} & {[K\otimes\{1\},1]} & {[K,1]\vee [1]}
	\arrow[from=1-2, to=1-1]
	\arrow[from=1-2, to=1-3]
	\arrow[from=1-4, to=1-3]
	\arrow[from=1-4, to=1-5]
\end{tikzcd}$$
and $[K,1]\otimes [1]$.
\end{prop}
\begin{proof}
The cone is induced by morphisms
$$
\begin{array}{rl}
&[1]\vee [K,1]\to [K,1]\otimes [1]\\
(\mbox{resp}.&[ K,1]\vee[1]\to [ K,1] \otimes [1])
\end{array}
$$ sending an element $x$ in the basis of $[1]$ to $\{0\}\otimes x$ (resp. $\{1\}\otimes x$), an element $y$ in the basis of $[ K,1]$ to $y\otimes\{1\}$ (resp. $y\otimes\{0\}$), 
and by the morphism 
$$f:[K\otimes [1],1]\to [K,1]\otimes [1]$$
defined by the formula 
$$f([x\otimes y,1]):= [ x,1]\otimes y$$ 
for $x$ in the basis of $K$ and $y$ in the basis of $[1]$.
We leave it to the reader to check the compatibilities of this three morphisms.
\end{proof}

\subsection{Gray operations on $\omega$-categories}
\label{section:definition of Gray operations}
We follow Ara-Maltsiniotis \cite{Ara_Maltsiniotis_joint_et_tranche} for the definitions and first properties of Gray operations on $\omega$-categories. 

\begin{theorem}[Steiner, Ara-Maltsiniotis]
\label{theo:otimes in zocat}
There is a unique colimit preserving monoidal structure on $\ocat$,
up to a unique monoidal isomorphism, making the functor
$\nu_{|\CDAB}:\CDAB\to \ocat$
a monoidal functor, when $\CDAB$ is endowed with the monoidal structure given by the Gray tensor product.
\end{theorem}
\begin{proof}
This is \cite[theorem A.15]{Ara_Maltsiniotis_joint_et_tranche}.
\end{proof}

\begin{definition}
The monoidal product on $\ocat$ induced by the previous theorem is called the \snotionsym{Gray tensor product}{((d00@$\otimes$}{for $\omega$-categories} and is denoted by $\otimes$. It's unit is $ \Db_0$. If $C$ and $D$ are $\omega$-categories with an atomic and loop free basis, we have by construction
$$C\otimes D := \nu(\lambda C\otimes \lambda D).$$
\end{definition}

\begin{prop}
\label{prop:otimes and duality}
There are equivalences
$$(C\otimes D)^{op}\cong D^{op}\otimes C^{op}~~~~~~ (C\otimes D)^\circ\cong C^{\circ}\otimes D^{\circ}~~~~~~(C\otimes D)^{co}\cong D^{co}\otimes C^{co}$$
natural in $C,D:\ocat$.
\end{prop}
\begin{proof}
This is \cite[proposition A.20]{Ara_Maltsiniotis_joint_et_tranche}.
\end{proof}

\begin{definition}
The functors
$$\uvar\otimes[1]:\ocat\to \ocat~~~~[1]\otimes\uvar:\ocat\to \ocat$$
are respectively called the \snotionsym{Gray cylinder}{((d30@$\uvar\otimes[1]$}{for $\omega$-categories} and the \snotionsym{Gray $\circ$-cylinder}{((d31@$[1]\otimes \uvar$}{for $\omega$-categories}.
\end{definition}

\begin{prop}
\label{prop:suspensino and gray}
Let $C$ be an $\omega$-category.
The following canonical square 
\[\begin{tikzcd}
	{C\otimes\{0,1\}} & {C\otimes[1]} \\
	{1\coprod 1} & {[C,1]}
	\arrow[from=1-1, to=2-1]
	\arrow[from=1-1, to=1-2]
	\arrow[from=2-1, to=2-2]
	\arrow[from=1-2, to=2-2]
	\arrow["\lrcorner"{anchor=center, pos=0.125, rotate=180}, draw=none, from=2-2, to=1-1]
\end{tikzcd}\]
is cocartesian
\end{prop}
\begin{proof}
As all these functors commute with colimits, it is sufficient to demonstrate this assertion when $C$ is a globular sum, and \textit{a fortiori} when $C$ admits a loop free and atomic basis. In this case, remark that all the morphisms appearing in canonical cartesian square
\[\begin{tikzcd}
	{\lambda C\otimes\{0,1\}} & {\lambda C\otimes[1]} \\
	{1\coprod 1} & {[\lambda C,1]}
	\arrow[from=1-1, to=2-1]
	\arrow[from=1-1, to=1-2]
	\arrow[from=2-1, to=2-2]
	\arrow[from=1-2, to=2-2]
	\arrow["\lrcorner"{anchor=center, pos=0.125, rotate=180}, draw=none, from=2-2, to=1-1]
\end{tikzcd}\]
 are quasi-rigid. 
The results then follow from an application of theorem \ref{theo:Kan condition}.
\end{proof}

\begin{remark}
\label{defi:explicit Dbn otiomes [1]}
Applying the duality $(\uvar)^{op}$ to the computation achieved in appendix B.1 of \cite{Ara_Maltsiniotis_joint_et_tranche}, we can give an explicit expression of $\Db_n\otimes [ 1]$. As a polygraph, the generating arrows of $\Db_n\otimes [1]$ are:
$$ e^\epsilon_k\otimes\{0\}~~~~~e^\epsilon_k\otimes\{1\}~~~~~e^\epsilon_k\otimes[1]$$
 \[ a^-_0 \otimes e^\epsilon_k \qquad a^+_0 \otimes e^\epsilon_k \qquad a \otimes e^\epsilon_k \]
 where $\epsilon$ is either $+$ or $-$, $k \leqslant n$ and $e^+_n = e^-_n$. Their source and target are given as follows:
 \[ \pi^-( e^\epsilon_k \otimes\{0\}) = e^-_{k-1} \otimes\{0\} \qquad\qquad\qquad \pi^+(e^\epsilon_k \otimes\{0\}) = e^+_{k-1}\otimes\{0\} \]
 \[ \pi^-(e^\epsilon_k \otimes\{1\} ) = e^-_{k-1} \otimes\{1\}\qquad\qquad\qquad \pi^+(e^\epsilon_k\otimes\{1\} ) = e^+_{k-1}\otimes\{1\} \]
$$\pi^{-}(e^\epsilon_{2k}\otimes[1]) =...\circ_2(e^+_0\otimes[1])\circ_0(e^\epsilon_{2k}\otimes\{0\})\circ_1 (e^-_1\otimes[1])\circ_3... \circ_{2k-1}(e_{2k-1}^-\otimes[1])$$
$$\pi^{+}(e^\epsilon_{2k}\otimes[1]) = (e_{2k-1}^+\otimes[1])\circ_{2k-1}...\circ_3(e^+_1\otimes[1])\circ_1(e^\epsilon_{2k}\otimes\{1\})\circ_0 (e^-_0\otimes[1])\circ_2...$$
$$\pi^{-}(e^\epsilon_{2k+1}\otimes[1]) = ...\circ_3(e^+_1\otimes[1])\circ_1(e^\epsilon_{2k+1}\otimes\{1\})\circ_0 (e^-_0\otimes[1])\circ_2...\circ_{2k}(e_{2k}^-\otimes[1])$$
$$\pi^{+}(e^\epsilon_{2k+1}\otimes[1]) = (e_{2k}^+\otimes[1])\circ_{2k}...\circ_2(e^+_0\otimes[1])\circ_0(e^\epsilon_{2k+1}\otimes\{0\})\circ_1 (e^-_1\otimes[1])\circ_3...$$
 We did not put parenthesis in the expression above, to keep them shorter, the default convention is to do the composition $\circ_i$ in order of increasing values of $i$.
\end{remark}
 
\begin{example}
The $\omega$-category $\Db_1\otimes[1]$ is the polygraph: 
\[\begin{tikzcd}
	00 & 01 \\
	10 & 11
	\arrow[from=1-1, to=2-1]
	\arrow[from=2-1, to=2-2]
	\arrow[from=1-1, to=1-2]
	\arrow[from=1-2, to=2-2]
	\arrow[shorten <=4pt, shorten >=4pt, Rightarrow, from=1-2, to=2-1]
\end{tikzcd}\]
The $\omega$-category $\Db_2\otimes[1]$ is the polygraph: 
\[\begin{tikzcd}
	00 & 01 & 00 & 01 \\
	10 & 11 & 10 & 11
	\arrow[from=1-1, to=1-2]
	\arrow[""{name=0, anchor=center, inner sep=0}, from=1-1, to=2-1]
	\arrow[from=2-1, to=2-2]
	\arrow[""{name=1, anchor=center, inner sep=0}, from=1-2, to=2-2]
	\arrow[shorten <=4pt, shorten >=4pt, Rightarrow, from=1-2, to=2-1]
	\arrow[""{name=2, anchor=center, inner sep=0}, from=1-3, to=2-3]
	\arrow[from=1-3, to=1-4]
	\arrow[""{name=3, anchor=center, inner sep=0}, from=1-4, to=2-4]
	\arrow[shorten <=4pt, shorten >=4pt, Rightarrow, from=1-4, to=2-3]
	\arrow[""{name=4, anchor=center, inner sep=0}, curve={height=30pt}, from=1-1, to=2-1]
	\arrow[from=2-3, to=2-4]
	\arrow[""{name=5, anchor=center, inner sep=0}, curve={height=-30pt}, from=1-4, to=2-4]
	\arrow["{ }"', shorten <=6pt, shorten >=6pt, Rightarrow, from=0, to=4]
	\arrow["{ }"', shorten <=6pt, shorten >=6pt, Rightarrow, from=5, to=3]
	\arrow[shift left=0.7, shorten <=6pt, shorten >=8pt, no head, from=1, to=2]
	\arrow[shift right=0.7, shorten <=6pt, shorten >=8pt, no head, from=1, to=2]
	\arrow[shorten <=6pt, shorten >=6pt, from=1, to=2]
\end{tikzcd}\]
\end{example}

\begin{construction}
\label{cons:Gray cone for omega cat}
 We define the \snotionsym{Gray cone}{((d40@$\uvar\star 1$}{for $\omega$-categories}, the \snotion{Gray $\circ$-cone}{for $\omega$-categories}\index[notation]{((d50@$1\overset{co}{\star}\_$!\textit{for $\omega$-categories}} and the \snotion{Gray op-cone}{for $\omega$-categories}\index[notation]{((d55@$1\star\_$!\textit{for $\omega$-categories}}:
$$\begin{array}{ccccccccccc}
\ocat &\to&\ocat_{\cdot}&&\ocat &\to&\ocat_{\cdot}&&\ocat &\to&\ocat_{\cdot}\\
C&\mapsto &C\star 1 & &C &\mapsto &1\costar C&& C&\mapsto  &1\star C
\end{array}
$$
where $C\star 1$, $1\costar C$ and $1\star C$ are defined as the following pushouts: 
\[\begin{tikzcd}
	{C\otimes\{1\}} & {C\otimes [1]} & {C\otimes\{0\}} & {C\otimes [1]} & {\{0\}\otimes C} & {[1]\otimes C} \\
	1 & {C\star 1} & 1 & {1\costar C} & 1 & {1\star C}
	\arrow[from=1-1, to=2-1]
	\arrow[from=1-1, to=1-2]
	\arrow[from=2-1, to=2-2]
	\arrow[from=1-2, to=2-2]
	\arrow["\lrcorner"{anchor=center, pos=0.125, rotate=180}, draw=none, from=2-2, to=1-1]
	\arrow[from=1-3, to=2-3]
	\arrow[from=2-3, to=2-4]
	\arrow[from=1-3, to=1-4]
	\arrow[from=1-4, to=2-4]
	\arrow["\lrcorner"{anchor=center, pos=0.125, rotate=180}, draw=none, from=2-4, to=1-3]
	\arrow[from=2-5, to=2-6]
	\arrow[from=1-5, to=2-5]
	\arrow[from=1-5, to=1-6]
	\arrow[from=1-6, to=2-6]
\end{tikzcd}\]
\end{construction}

\begin{remark}
We could also define the \textit{Gray co-cone} $C\costar 1$, but we have omitted it as it will not appear in this text.
\end{remark}

\begin{prop}
\label{prop:star and duality}
There are equivalences
$$(C\star 1)^{\circ}\cong 1\costar C^{\circ} ~~~~~(1\star C)^{op}\cong C^{op}\star 1~~~~~ (1\costar C)^{co}\cong 1\star C^{co}$$
natural in $C:\ocat$.
\end{prop}
\begin{proof}
This directly follows from the definition of these operations and from proposition \ref{prop:otimes and duality}. 
\end{proof}

\begin{example}
The $\omega$-categories $\Db_1\star 1$ and $1\costar \Db_1$ correspond respectively to the polygraphs: 
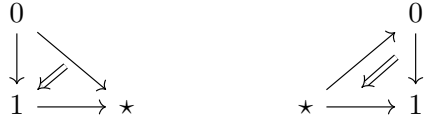
\[\begin{tikzcd}
	0 &&&& 0 \\
	1 & \star && \star & 1
	\arrow[from=1-1, to=2-1]
	\arrow[from=2-1, to=2-2]
	\arrow[""{name=0, anchor=center, inner sep=0}, from=1-1, to=2-2]
	\arrow[""{name=1, anchor=center, inner sep=0}, from=1-5, to=2-5]
	\arrow[from=2-4, to=1-5]
	\arrow[""{name=2, anchor=center, inner sep=0}, from=2-4, to=2-5]
	\arrow[shorten <=2pt, Rightarrow, from=0, to=2-1]
	\arrow[shift right=2, shorten <=4pt, shorten >=4pt, Rightarrow, from=1, to=2]
\end{tikzcd}\]
The $\omega$-categories $\Db_2\star 1$ and $1\costar \Db_2$ correspond respectively to the polygraphs: 
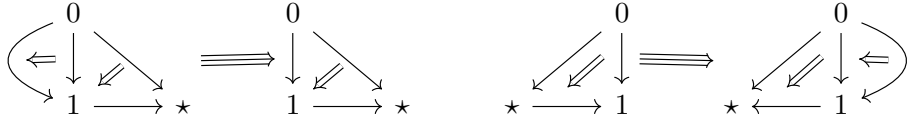
\[\begin{tikzcd}
	0 & {~} & 0 &&& 0 & {~} & 0 \\
	1 & \star & 1 & \star & \star & 1 & \star & 1
	\arrow[""{name=0, anchor=center, inner sep=0}, from=1-1, to=2-1]
	\arrow[from=2-1, to=2-2]
	\arrow[""{name=1, anchor=center, inner sep=0}, from=1-3, to=2-3]
	\arrow[""{name=2, anchor=center, inner sep=0}, curve={height=30pt}, from=1-1, to=2-1]
	\arrow[from=2-3, to=2-4]
	\arrow[""{name=3, anchor=center, inner sep=0}, from=1-1, to=2-2]
	\arrow[""{name=4, anchor=center, inner sep=0}, draw=none, from=1-2, to=2-2]
	\arrow[""{name=5, anchor=center, inner sep=0}, from=1-3, to=2-4]
	\arrow[from=1-6, to=2-5]
	\arrow[""{name=6, anchor=center, inner sep=0}, from=1-6, to=2-6]
	\arrow[""{name=7, anchor=center, inner sep=0}, from=2-5, to=2-6]
	\arrow[from=1-8, to=2-7]
	\arrow[""{name=8, anchor=center, inner sep=0}, from=1-8, to=2-8]
	\arrow[""{name=9, anchor=center, inner sep=0}, from=2-8, to=2-7]
	\arrow[""{name=10, anchor=center, inner sep=0}, curve={height=-30pt}, from=1-8, to=2-8]
	\arrow[""{name=11, anchor=center, inner sep=0}, draw=none, from=1-7, to=2-7]
	\arrow["{ }"', shorten <=6pt, shorten >=6pt, Rightarrow, from=0, to=2]
	\arrow[shorten <=2pt, shorten >=2pt, Rightarrow, from=3, to=2-1]
	\arrow[shift left=0.7, shorten <=6pt, shorten >=8pt, no head, from=4, to=1]
	\arrow[shift right=0.7, shorten <=6pt, shorten >=8pt, no head, from=4, to=1]
	\arrow[shorten <=6pt, shorten >=6pt, from=4, to=1]
	\arrow[shorten <=2pt, Rightarrow, from=5, to=2-3]
	\arrow[shorten <=6pt, shorten >=6pt, Rightarrow, from=10, to=8]
	\arrow[shift right=2, shorten <=4pt, shorten >=4pt, Rightarrow, from=8, to=9]
	\arrow[shift right=2, shorten <=4pt, shorten >=4pt, Rightarrow, from=6, to=7]
	\arrow[shift right=0.7, shorten <=6pt, shorten >=8pt, no head, from=6, to=11]
	\arrow[shorten <=6pt, shorten >=6pt, from=6, to=11]
	\arrow[shift left=0.7, shorten <=6pt, shorten >=8pt, no head, from=6, to=11]
\end{tikzcd}\]
\end{example}

\begin{prop}
Let $C$ be an $\omega$-category with an unitary and loop free basis. The canonical comparaison
$$ (\lambda C)\star 1\to \lambda (C\star 1) $$
is an equivalence.

Let $K$ be an augmented directed complex with a loop free and unitary basis. The canonical comparaisons 
$$(\nu K)\star 1\to \nu(K\star 1)$$
is an equivalence.
\end{prop}
\begin{proof}
The first assertion directly follows from the fact $\lambda$ commutes with colimits. For the second one,
we can easily check that all the morphisms appearing in the squares \eqref{eq:defin of cstar costar CDA} are quasi-rigid.
The results then follow from an application of theorem \ref{theo:Kan condition}.
\end{proof}

The following theorems express the link between the Gray operations and the suspension. They will play a fundamental role in the rest of this work.
\begin{theorem}
 \label{theo:appendice formula for otimes} 
 Let $C$ be an $\omega$-category.
There is a natural identification between $[ C,1]\otimes [1]$ and the colimit of the following diagram
$$
\begin{tikzcd}
	{[1]\vee [ C,1]} & {[C\otimes\{0\},1]} & {[C\otimes [1],1]} & {[C\otimes\{1\},1]} & {[C,1]\vee[1]}
	\arrow[from=1-2, to=1-1]
	\arrow[from=1-2, to=1-3]
	\arrow[from=1-4, to=1-3]
	\arrow[from=1-4, to=1-5]
\end{tikzcd}$$
\end{theorem}
\begin{proof}
As all these functors preserve colimits, it is sufficient to construct the comparison when $C$ is a globular sum, and to show that it is an equivalence when $C$ is a globe. 
As globular sums have atomic and loop free bases, the comparison is induced by proposition \ref{prop:appendice formula for otimes cda}. Using the explicit description of the $\omega$-category $\Db_n\otimes[1]$ given in definition \ref{defi:explicit Dbn otiomes [1]}, it is straightforward to see that it induces an equivalence on globes.
\end{proof}

\begin{theorem}
 \label{theo:appendice formula for star} 
There is a natural identification between $1\costar [C,1]$ and the colimit of the following diagram
\[\begin{tikzcd}
	{[1]\vee [C,1]} & {[C,1]} & {[C\star 1,1]}
	\arrow[from=1-2, to=1-3]
	\arrow[from=1-2, to=1-1]
\end{tikzcd}\]
There is a natural identification between $[C,1]\star 1$ and the colimit of the following diagram
\[\begin{tikzcd}
	{[1\costar C,1]} & {[C,1]} & {[C,1]\vee[1]}
	\arrow[from=1-2, to=1-3]
	\arrow[from=1-2, to=1-1]
\end{tikzcd}\]
There is a natural identification between $1\star[C,1]$ and the colimit of the following diagram   
\[\begin{tikzcd}
	{[1\star C,1]} & {[C,1]} & {[1]\vee[C,1]}
	\arrow[from=1-2, to=1-3]
	\arrow[from=1-2, to=1-1]
\end{tikzcd}\]
\end{theorem}
\begin{proof}
This directly follows from the definition of these operations, from theorem \ref{theo:appendice formula for otimes}  and from proposition \ref{prop:star and duality}.
\end{proof}

\vspace{1cm}

We are now willing to show the following  theorem: 
\begin{theorem}
\label{theo:appendince unicity of operation}
Let $F$ be an endofunctor of $\ocat$ such that the induced functor $\ocat\to \ocat_{F(\emptyset)/}$ is colimit preserving and $\psi$ an invertible natural transformation between $F(\Db_n)$ and $G(\Db_n)$ where $G$ is either the Gray cylinder, the Gray cone, the Gray $\circ$-cone, the Gray op-cone or an iterated suspension.

Then, the natural transformation $\psi$ can be uniquely extended to an natural transformation between $F$ and $G$.  Moreover, this natural transformation is unique.
\end{theorem}
The previous theorem implies that the equations given in theorem \ref{theo:appendice formula for otimes} and \ref{theo:appendice formula for star} characterize respectively the Gray cylinder, the Gray cone, the Gray $\circ$-cone and the Gray op-cone.

\begin{lemma}
\label{lemma:sub categgory of Theta}
A sub category $\Theta'$ of $\Theta$, stable by colimit is equal to $\Theta$ iff
\begin{enumerate}
\item for any integer $n$ and $\alpha\in\{-,+1\}$, $i_n^{\alpha}:\Db_n\to \Db_{n+1}$ belongs to $\Theta'$.
\item For any integer $n$, the unit $\Ib_n:\Db_{n+1}\to \Db_n$ belongs to $\Theta'$.
\item For any pair of integers $k<n$, the composition $\triangledown_{k,n}:\Db_n\to \Db_n\coprod_{k}\Db_n$ belongs to $\Theta'$.
\end{enumerate}
\end{lemma}
\begin{proof}
Suppose that $\Theta'$ fulfills these conditions.
As globular morphisms are compositions of pushouts along morphisms of shape $i_n^{\alpha}$, they belong to $\Theta'$.
 As algebraic morphisms are compositions of colimits of morphism of shape $\triangledown_{k,n}$ or $\Ib_n$, they belong to $\Theta'$.
The result then follows from proposition \ref{prop:algebraic ortho to globular} that states that every morphism factors as an algebraic morphism followed by a globular morphism.
\end{proof}

\begin{lemma}
\label{lemma:unit forced}
Let $n$ be an integer, and $G$ be either the Gray cylinder, the Gray cone, the Gray $\circ$-cone, the Gray op-cone or an iterated suspension, and suppose
given a square 
\[\begin{tikzcd}
	{G(\Db_n)} \\
	& {G(\Db_{n+1})} & {G(\Db_n)} \\
	{G(\Db_n)}
	\arrow["f", from=2-2, to=2-3]
	\arrow["{G(i_n^-)}", from=1-1, to=2-2]
	\arrow["{G(i_n^+)}"', from=3-1, to=2-2]
	\arrow["id", curve={height=-18pt}, from=1-1, to=2-3]
	\arrow["id"', curve={height=18pt}, from=3-1, to=2-3]
\end{tikzcd}\]
Then, the morphism $f$ is $G(\Ib_n)$.
\end{lemma}
\begin{proof}
As the proof for any possibilities of $G$ are similar, we will show only the case $G:=\uvar\otimes [1]$.
As for any integer $n$, $\Db_n\otimes[1]$ admits a loop free and atomic basis, we can then show the desired assertion after applying the functor $\lambda$.
Remark first that the assumption implies that $\partial f((e_{n+1}\otimes \{\alpha\})=0$, and so $f((e_{n+1}\otimes \{\alpha\}) =0$. We also have $f(e_{n+1}\otimes[1])=0$ as $\lambda (\Db_n\otimes[1])_{n+2} =0$. This implies that $f$ is equal to $\lambda(G(\Ib_n))$.
\end{proof} 
\begin{lemma}
\label{lemma:comp forced}
Let $k<n$ be two integers, and $G$ be either the Gray cylinder, the Gray cone, the Gray $\circ$-cone or an iterated suspension, and suppose
given a square 
\[\begin{tikzcd}
	{G(\Db_{n-1})} && { G(\Db_{n-1}\coprod_k\Db_{n-1})} \\
	& {G(\Db_n)} && { G(\Db_{n}\coprod_k\Db_n)} \\
	{G(\Db_{n-1})} && { G(\Db_{n-1}\coprod_k\Db_{n-1})}
	\arrow["f", from=2-2, to=2-4]
	\arrow["{G(i_n^-)}"{description}, from=1-1, to=2-2]
	\arrow["{G(i_n^+)}"{description}, from=3-1, to=2-2]
	\arrow["{G(i_n^+)\coprod_k G(i_n^+)}"{description}, from=3-3, to=2-4]
	\arrow["{G(i_n^-)\coprod_k G(i_n^-)}"{description}, from=1-3, to=2-4]
	\arrow["{\triangledown_{n-1,k}}"', from=3-1, to=3-3]
	\arrow["{\triangledown_{n-1,k}}", from=1-1, to=1-3]
\end{tikzcd}\]
where we set $\triangledown_{n,n}:=id$.
Then, the morphism $f$ is $G(\triangledown_{n,k})$.
\end{lemma}
\begin{proof}
As the proof for any possibilities of $G$ are similar, we will show only the case $G:=\uvar\otimes [1]$.
As for any integer $n$, $\Db_n\otimes[1]$ admits a loop free and atomic basis, we can then show the desired assertion after applying the functor $\lambda$. Suppose first that $k<n-1$.
By assumption, we have 
$$
\begin{array}{rcl}
\partial f(e_n\otimes \{\alpha\})&=& \partial (e_n^0\otimes \{\alpha\} +e_n^1\otimes \{\alpha\})\\
\partial f(e_n\otimes [1])&=& \partial (e_n^0\otimes [1]) + \partial (e_n^1\otimes [1]) \\
\end{array}
$$
This forces the equalities
$$
\begin{array}{rcl}
 f(e_n\otimes \{\alpha\})&=& e_n^0\otimes \{\alpha\} +e_n^1\otimes \{\alpha\}\\
 f(e_n\otimes [1])&=& e_n^0\otimes [1] + e_n^1\otimes [1] \\
\end{array}
$$
and $f$ is then equal to $\triangledown_{n,k}\otimes[1]$. The case $k=n-1$ is similar.
\end{proof}

\begin{proof}[Proof of theorem \ref{theo:appendince unicity of operation}]
As every globular sum is a colimit of globes, we can extend $\psi$ to a (\textit{a priori} non natural) transformation, 
$\psi:F_{|\Theta}\to G_{|\Theta}$.
Let $\Theta'$ be the maximal sub category of $\Theta$ such that $\psi_{|\Theta'}$ is natural.  The category $\Theta'$ is  closed by colimit. 
The assumption implies that	 $\Theta'$ fulfills the first condition of lemma \ref{lemma:sub categgory of Theta}. The lemma \ref{lemma:unit forced} implies that it fulfills the second condition, and an easy induction on $(n-k)$ using lemma \ref{lemma:comp forced} implies that it fulfills the last condition. Applying the lemma \ref{lemma:sub categgory of Theta}, $\psi:F_{|\Theta}\to G_{|\Theta}$ is then poitwise an isomorphism, and can be extended by colimits to a invertible natural transformation between $F$ and $G$. The unicity of this extension is a consequence of lemma \ref{lemma:non trivial automorphisme 4}.
\end{proof}

\vspace{1cm}

We conclude this section by giving some technical results that we will use later.

\begin{lemma}
\label{lemma:non trivial automorphisme 4}
The set of $\omega$-categories admitting no non-trivial automorphisms is  stable 
\begin{enumerate}
\item by isomorphisms,
\item  by $[\uvar,1]\vee[1]$ and $[1]\vee[\uvar,1]$,
\item by the Gray cylinder, the Gray cone, the Gray $\circ$-cone, the Gray op-cone and the  iterated suspensions,
\end{enumerate}
and contains globular sums.
\end{lemma}
\begin{proof}
Let $S$ be the smallest set of $\omega$-categories stable by isomorphism,  $[\uvar,1]\vee[1]$, $[1]\vee[\uvar,1]$, the Gray cylinder, the Gray cone and by  iterated suspensions. As the set of $\omega$-categories admitting no non-trivial automorphisms is stable by dualities and by proposition \ref{prop:star and duality}, we have to show that it includes $S$.

Remarks now that $S$ is contained in the set of $\omega$-categories admitting an atomic and loop free basis also fulfills. Using theorem \ref{theorem:steiner}, it is then sufficient to show that any augmented directed complex in $\lambda(S)$ has no non-trivial automorphisms. This directly follows from propositions \ref{prop:non trivial automorphisme 0}, \ref{prop:non trivial automorphisme 1}, \ref{prop:non trivial automorphisme 2} and \ref{prop:non trivial automorphisme 3}.

It remains to show that $S$ contains globular sums. We proceed by induction, and we suppose that $S$ contains any globular sum of dimension $k$. Let $[\textbf{a},n]$ be a globular sum of dimension $k+1$, and let $\phi:[\textbf{a},n]\to [\textbf{a},n]$ be an isomorphism. In particular $\phi$ induces an automorphism on $[n]$, and we then have $\phi{i}=i$ for any $i\leq n$. The automorphism $\phi$ then induces for all $i<n$ an automorphism $\phi_i:[a_i,1]\cong [a_i,1]$. However, the stability by suspension of $S$ and the induction hypothesis implies that for any $i<n$, $[a_i,1]$ has no non trivial automorphisms and $\phi_i$ is then the identity. This implies that $\phi$ is also the identity which concludes the proof.
\end{proof}

\begin{prop}
\label{prop:the globes a non non trivial automorphisms}
Let $n$ be an integer $n$. The $\omega$-categories $\Db_n$ and $\underbrace{1\star 1\star ... \star 1}_{n}$ have no non-trivial automorphisms.
\end{prop}
\begin{proof}
This is a direct consequence of lemma \ref{lemma:non trivial automorphisme 4} as these two $\omega$-categories belong to $S$.
\end{proof}

\subsection{Gray tensor product of simplicial sets}

\begin{notation}
We denote  by
\[\begin{tikzcd}
	{\Psh{\Theta}} & \ocat
	\arrow["\Fb", shift left=2, from=1-1, to=1-2]
	\arrow["\iota", shift left=2, from=1-2, to=1-1]
\end{tikzcd}\]
the adjunction between presheaves on $\Theta$ and $\omega$-categories.
\end{notation}

\begin{construction}
We define the functor  $\uvar\otimes \uvar :\Psh{\Theta}\times \Psh{\Theta}\to \Psh{\Theta}$, called once again the \textit{Gray tensor product}, as the left Kan extension of the functor 
$$\Theta\times \Theta \xrightarrow{\uvar\otimes \uvar} \ocat\xrightarrow{\iota} \Psh{\Theta}$$
where $\otimes:\Theta\times \Theta\to \ocat$ is the Gray tensor product defined in theorem \ref{theo:otimes in zocat}.

By construction, the functor $\Fb$ preserves the Gray tensor product, and the functor $\iota$ preserves the Gray tensor product of globular sums.
\end{construction}

The aim of this section is to prove the following result:

\begin{theorem} 
\label{theo:otimes presserves W}
The functor
$$\uvar\otimes\uvar:\Psh{\Delta}\times \Psh{\Delta}\to \Psh{\Theta_2}$$
sends $\W_1\times \W_1$ onto $\overline{\W_2}$, and $\W_1^\sat\times \W_1^\sat$ onto $\overline{\W_2^\sat}$, where $\W_i$ and $\W_i^\sat$ are defined in \ref{defi:definition of W}, and $\overline{(\uvar)}$ in \ref{defi:precocomplet}.
\end{theorem}

In the second part of this section, we will show a similar result for the op-joint.

\begin{prop}
\label{prop:explicit Gray}
The $\Theta$-set $[1]\otimes[1]$ is the colimit, computed in $\Psh{\Theta}$, of the diagram
\[\begin{tikzcd}
	{[2]} & {[1]} & {[[1],1]} & {[1]} & {[2]}
	\arrow["\triangledown"', from=1-2, to=1-1]
	\arrow["{[d^1,1]}", from=1-2, to=1-3]
	\arrow["{[d^0,1]}"', from=1-4, to=1-3]
	\arrow["\triangledown", from=1-4, to=1-5]
\end{tikzcd}\]
\end{prop}
\begin{proof}
We denote by $P$ the colimit of this diagram. Remark that $\Fb P$ is the $\omega$-category generated by the diagram
\[\begin{tikzcd}
	00 & 01 \\
	10 & 11
	\arrow[from=1-2, to=2-2]
	\arrow[from=2-1, to=2-2]
	\arrow[from=1-1, to=2-1]
	\arrow[from=1-1, to=1-2]
	\arrow[shorten <=4pt, shorten >=4pt, Rightarrow, from=1-2, to=2-1]
\end{tikzcd}\]
and we then have $\Fb P\cong [1]\otimes[1]$.  To conclude the proof, we have to show that $P$ is a $\omega$-category, i.e. that it has the unique right lifting property against $\W$. 

Let $f:[\textbf{a},n]\to P$ (resp. $f:\Sp_{[a,n]}\to p$) be a morphism. If there exists an integer $i<n$ such that  $f(i)=00$ and $f(i+1)=11$, then $f$ uniquely factors through $[[1],1]\to P$. If there exists an integer $i$ such that $f(i)=10$, then $f$ uniquely factors through the left inclusion $[2]\to P$. If there exists an integer $i$ such that $f(i)=01$, then $f$ uniquely factors through the right inclusion $[2]\to P$.
If none of these conditions is satisfied, then $f$ factors through $00$ or $11$.

As $[2]$ and $[[1],1]$ are $\omega$-categories, they have the unique right lifting property against $\W$, and so has $P$.
\end{proof}

\begin{lemma}
\label{lemma:when Gray and pullback work}
Let $C,D, E$ be three $\omega$-categories with loop-free and atomic bases. Let $f:C\to D$ be a morphism such that $f$ sends every generator of $C$ to a cell which is not a unit. The following square is then cartesian:
\[\begin{tikzcd}
	{C\otimes E} & {D\otimes E} \\
	C & D
	\arrow["f"', from=2-1, to=2-2]
	\arrow[from=1-2, to=2-2]
	\arrow[from=1-1, to=2-1]
	\arrow["{f\otimes E}", from=1-1, to=1-2]
\end{tikzcd}\]
\end{lemma}
\begin{proof}
We can show this result at the level of the corresponding augmented directed complex, where it is an easy computation. 
\end{proof}
\begin{lemma}
\label{lemma:product is a nice colimit}
Let $(a)_{i\leq n}$ be a sequence of elements of $\Theta$. There exists a diagram $F:I\to \Theta^+$ such that the presheaf $a_0\times...\times a_n$ is the colimit of $F$. 
\end{lemma}
\begin{proof}
Let $I$ be the full subcategory of $\Theta_{/a_0\times...\times a_n}$ whose objects are $n$-tuples of morphisms $(j_i:b\to a_i)_{i\leq n}$ such that there exists no morphism $b\to b'$ in $\Theta^{-}$ that factors all the $j_i$. Morphisms are the ones such that $b\to b'$ in $\Theta^{+}$.

Let $(j_i:b\to a_i)_{i\leq n}$ be any element of $\Theta_{/a_0\times...\times a_n}$. 
We will show that there exists a unique degenerate morphism $g:b\to b'$ that factors the morphisms $b\to a_i$ for all $i<n$, and such that the induced family of morphisms $\{b'\to a_i\}_{i<n}$ is an element of $I$. It will implies that $g$ is an initial object of  the category $I_{(j_i)_{i\leq n}/}$, and then that $\alpha: I\to  \Theta_{/a_0\times...\times a_n}$ is final, which will concludes the proof.

As any infinite sequence of degenerate morphisms is constant at some point, the existence is immediate.
Suppose given two morphisms $b\to b'$, $b\to b''$ fulfilling the previous condition. By proposition \ref{prop:theta is elegan reedy}, the category $\Theta$ is Reedy elegant  and 
the proposition 3.8 of \cite{Bergner_reedy_category_and_the_theta_construction} then implies that there exists a globular sum $\tilde{b}$ and two degenerate morphisms $b'\to \tilde{b}$ and $b''\to \tilde{b}$ such that the induced square
\[\begin{tikzcd}
	b & {b'} \\
	{b''} & {\tilde{b}}
	\arrow[from=1-2, to=2-2]
	\arrow[from=2-1, to=2-2]
	\arrow[from=1-1, to=2-1]
	\arrow[from=1-1, to=1-2]
\end{tikzcd}\]
is cocartesian. The universal property of pushout implies that $b\to \tilde{b}$ also fulfills the previous condition. By definition of $b'$ and $b''$, this implies that they are equal to $\tilde{b}$, and this shows the uniqueness.
\end{proof}

\begin{lemma}
\label{lemma:nice colimit in presheaves and Gray}
Let $C$ be a $\omega$-category such that there exists a diagram $F:I\to \Theta^+$ with $\iota(C)$ being  the colimit of $F$. Let $a$ be an element of $\Theta$. The canonical morphism $\iota(C)\otimes a \to \iota (C\otimes a)$ is an isomorphism.
\end{lemma}
\begin{proof}
The lemma \ref{lemma:when Gray and pullback work} implies that the natural transformation $F(i)\otimes b\to F(i)$ is cartesian. As a consequence, for any $i$, the square
\[\begin{tikzcd}
	{F(i)\otimes a} & {(\colim_I F)\otimes a\cong \iota(C)\otimes a} \\
	{F(i)} & {\colim_I F\cong \iota(C)}
	\arrow[from=1-2, to=2-2]
	\arrow[from=1-1, to=1-2]
	\arrow[from=2-1, to=2-2]
	\arrow[from=1-1, to=2-1]
	\arrow["\lrcorner"{anchor=center, pos=0.125}, draw=none, from=1-1, to=2-2]
\end{tikzcd}\]
is cartesian. 

Now, to show the desired result, we have to demonstrate that the $\Theta$-set $\iota(C)\otimes a$ already has a structure of $\io$-category, i.e. that it is $\W$-local. It is sufficient to show that for all $f:X\to Y$ in $\W$, any square
\[\begin{tikzcd}
	X & { \iota(C)\otimes a} \\
	Y & { \iota(C)}
	\arrow[from=1-2, to=2-2]
	\arrow[from=1-1, to=2-1]
	\arrow[from=1-1, to=1-2]
	\arrow[from=2-1, to=2-2]
\end{tikzcd}\]
admits a unique lift. Indeed, as $\iota(C)$ is an $\omega$-category, it is $\W$-local, and this will imply that $\iota(C)\otimes a$ also is. Suppose then given such a square. As every codomain of morphism of $\W$ is representable, there exists a (not necessarily unique) element $i$ of $I$, such that the bottom morphism factors as $Y\to F(i)\to \iota(C)$. The previous square then factors as
\[\begin{tikzcd}
	X & {F(i)\otimes a} & { \iota(C)\otimes a} \\
	Y & {F(i)} & { \iota(C)}
	\arrow[from=1-3, to=2-3]
	\arrow[from=1-2, to=1-3]
	\arrow[from=2-2, to=2-3]
	\arrow[from=1-2, to=2-2]
	\arrow["\lrcorner"{anchor=center, pos=0.125}, draw=none, from=1-2, to=2-3]
	\arrow[from=2-1, to=2-2]
	\arrow[from=1-1, to=1-2]
	\arrow[from=1-1, to=2-1]
\end{tikzcd}\]
where the right square is a pullback. The middle vertical morphism is $\W$-local because it's domain and codomain are, and this concludes the proof. 
\end{proof}
\begin{lemma}
\label{lemma:times et otimes}
Given $(a)_{i\leq n}$ and $b$ elements of $\Theta$, we have 
$$\iota((a_0\times ...\times a_n\otimes b)\cong (a_0\times ...\times a_n)\otimes b$$
\end{lemma}
\begin{proof}
This is a direct consequence of lemmas \ref{lemma:product is a nice colimit} and \ref{lemma:nice colimit in presheaves and Gray}
\end{proof}

\begin{lemma}
\label{lemma:times et otimes2}
Let $A,B, C$ be three presheaves on $\Theta$. We have a canonical morphism 
$$A \otimes (B\otimes C)\to (A\times B)\otimes C$$
\end{lemma}
\begin{proof}
It is sufficient to demonstrate the result when $A$, $B$ and $D$ are representable. In this case the lemma \ref{lemma:times et otimes} implies that $(A\times B)\otimes C$ is in the image of $\iota$. By adjunction, the desired comparaison morphism is induced by
$$\iota (A \otimes (B\otimes C))\cong  \iota (A) \otimes (\iota(B)\otimes \iota(C))) \cong (\iota(A)\otimes \iota(B))\otimes \iota(X) \to  (\iota(A)\times \iota(B))\otimes \iota (C)$$
\end{proof}

\begin{lemma}
\label{lemma:technical steiner}
Let $A$, $B$, $C$, $D$, and $E$ be presheaves on $\Theta$, and $k, m, n$ be integers. There exists a natural morphism 
$$(\uvar)_A:\Hom([B,m],C\otimes[D,n])\to \Hom([A\otimes B,m],C\otimes[A\otimes D,n])$$
such that for any pair of morphisms $f:[B,m]\to C\otimes[n]$ and $g:[F,k]\to E\otimes[m]$, 
$$\Fb(((E\otimes f)\circ g_B)_A)=\Fb((E\otimes f_A)\circ (g_B)_A)$$
\end{lemma}
\begin{proof}
It is sufficient to describe this morphism when $A,B,C,D$, and $E$ are representable. This allows us the use of Steiner theory to construct this application.
Let $f:[B,m]\to C\otimes[D,n]$ be a morphism. We set  $f_A:[A\otimes B,m]\to C\otimes[A\otimes D,n]$ as the unique morphism of $\omega$-categories such that for every $a\in B_A$, $b\in B_B$, and $m\in B_m$
$$\lambda f_A( [a\otimes b,m]):= \sum_{i\leq n} c_i\otimes [a\otimes d_i,n_i]$$
where $(c_i,d_i,n_i)$ is the unique sequence of elements of $B_C\times B_D\times B_{[n]}$ such that $\lambda f([b,m])= \sum_{i\leq n} c_i\otimes [d_i,n_i]$.
The equality $\lambda f_A\partial=\partial\lambda f_A$ and the equality $\Fb(((E\otimes f)\circ g_B)_A)=\Fb((E\otimes f_A)\circ (g_B)_A)$ is a straightforward calculation using Steiner theory.
\end{proof}

\begin{lemma}
\label{lemma:technical steiner2}
Let $A$, $B$, $C$, $D$, $E$, and $F$ be presheaves on $\Delta$, and $k, m, n, l$ be integers. There exists a natural morphism 
$$\alpha: \Hom([A,k],B\otimes[m])\times \Hom([C,m],D\otimes[n])\to \Hom([C\times A,k],(B\times D)\otimes[n])$$
and such that for any $f:[A,k]\to B\otimes[m]$, $g:[C,m]\to D\otimes[n]$, and $h:[E,n]\to F\otimes[l]$, 
\begin{equation}
\label{eq:alpha}
\alpha(\alpha(f,g),h) = \alpha(f,\alpha(g,h))
\end{equation}
\end{lemma}
\begin{proof}
Let $f:[A,k]\to B\otimes[m]$ and $g:[C,m]\to D\otimes[n]$ be two morphisms. Using the application of lemma \ref{lemma:technical steiner} and the canonical morphism of \ref{lemma:times et otimes2}, we get a sequence of arrows 
\[\begin{tikzcd}
	{[C\otimes A,k]} & { B\otimes[C,m]} & {B\otimes (D\otimes[n])} & {(B\times D)\otimes [n]}
	\arrow["{f_C}", from=1-1, to=1-2]
	\arrow["{B\otimes g}", from=1-2, to=1-3]
	\arrow[from=1-3, to=1-4]
\end{tikzcd}\]
whose composite is denoted $\alpha'(f,g)$.
Remark now that $(B\times D)\otimes [n]$ is a $\Theta_2$-set. Moreover, we have an isomorphism 
$$\tau^i_2([C\otimes A,k])\cong [\tau^i_1(C\otimes A),k]\cong [C\times A,k]$$
We then set 
$$\alpha(f,g):=\tau^i_2(\alpha'(f,g)):=[C\times A,k]\to (B\times D)\otimes [n].$$

Now, suppose given two arrows $f:[A,k]\to B\otimes[m]$, $g:[C,m]\to D\otimes[n]$, and $h:[E,n]\to F\otimes[l]$. Unfolding the definition, we have that $\alpha(\alpha(f,g),h)$ and $\alpha(f,\alpha(g,h))$ are respectively the image by $\tau^i_2$ of the morphism 	
\[\begin{tikzcd}
	{[E\otimes (C\otimes A),k]} && {B\otimes (D\otimes[E,n])} & {B\otimes (D\otimes( F\otimes [k]))} & {(B\times D\times F)\otimes[l]}
	\arrow["{B\otimes (D \otimes h)}", from=1-3, to=1-4]
	\arrow[from=1-4, to=1-5]
	\arrow["{(B\otimes g_E)\circ(f_C)_E}", from=1-1, to=1-3]
\end{tikzcd}\]
and
\[\begin{tikzcd}
	{[E\otimes (C\otimes A),k]} && {B\otimes (D\otimes[E,n])} & {B\otimes (D\otimes( F\otimes [k]))} & {(B\times D\times F)\otimes[l]}
	\arrow[from=1-4, to=1-5]
	\arrow["{(B\otimes g\circ (f_C)_E}", from=1-1, to=1-3]
	\arrow["{B\otimes (D \otimes h)}", from=1-3, to=1-4]
\end{tikzcd}\]
Remark moreover that if $A$, $B$, $C$,  $D$, and $E$ are representable, lemma \ref{lemma:times et otimes} implies that $\alpha(\alpha(f,g),h)$ and $\alpha(f,\alpha(g,h))$ are morphisms of $\omega$-categories, and they are then equal to the image by $\tau^i_2$ and $\Fb$ of the two previous morphism. The equality given in lemma \ref{lemma:technical steiner} then implies 
$$\alpha(\alpha(f,g),h) =\alpha(f,\alpha(g,h))$$
\end{proof}

\begin{lemma}
\label{lemma:otimes presserves W1}
Let $n$ and $m$ be two integers. The canonical morphism 
$$\Sp_{[n]}\otimes\Sp_{[m]}\to [n]\otimes[m]$$ 
is in $\overline{\W_2}$.
\end{lemma}
\begin{proof}
Let $\Delta^{glob}$ be the subcategory of $\Delta$ whose morphisms are the globular ones.
We consider the functor $g:\Delta^{glob}\times \Delta^{glob}\to \Psh{\Theta_2}$ by the formula 
$$g([n],[m]):=\tau_0([n]\otimes [m])\cup_{x\in S_{n,m}}x$$
where $S_{n,m}$ is the set of $1$-generators of $\tau_1([n]\otimes[m])$.
We have a canonical transformation $g(n,m)\to \tau_1([n]\otimes[m])$ which is pointwise in $\widehat{\W_2}$ by repeated application of theorem \ref{theo: case of 1 and 2 category}. For any pair of integers $n,m$, the morphism 
$$g([n],[m])\cong \colim_{\Sp_{[n]}\times \Sp_{[m]}}g\to \tau_1(\Sp_{[n]}\otimes\Sp_{[m]})$$ 
then also belongs to $\overline{\W_2}$. 
By two out of three, so is the morphism
$$\tau_1(\Sp_{[n]}\otimes \Sp_{[m]})\to \tau_1([n]\otimes[m])$$
Remark now that we have a cocartesian square
\[\begin{tikzcd}
	{\tau_1(\Sp_{[n]}\otimes\Sp_{[m]})} & {\Sp_{[n]}\otimes\Sp_{[m]}} \\
	{\tau_1([n]\otimes[m])} & {\tau_1([n]\otimes[m])\cup_{x\in T_{n,m}}x}
	\arrow[from=1-1, to=2-1]
	\arrow[from=1-1, to=1-2]
	\arrow[from=2-1, to=2-2]
	\arrow[from=1-2, to=2-2]
\end{tikzcd}\]
where $T_{n,m}$ is the set of $2$-generators of  $[n]\otimes[m]$. The theorem \ref{theo: case of 1 and 2 category} implies that 
$$\tau_1([n]\otimes[m])\cup_{x\in T_{n,m}}x\to [n]\otimes[m]$$
is in $\widehat{\W_2}$, and by stability by composition and pushout, so is 
$$\Sp_{[n]}\otimes\Sp_{[m]}\to [n]\otimes[m].$$
\end{proof}

\begin{prop}
\label{prop:otimes and suspension in presheaves}
Let $K$ be a simplicial set. The canonical morphism
$$1\coprod_{K\otimes\{0\}}K\otimes[1]\coprod_{K\otimes\{1\}}1\to [K,1]$$
is in $\overline{\W_2}$.
\end{prop}
\begin{proof}
As $K$ is a colimit of representables indexed by the Reedy cofibrant diagram $\Delta_{/K}\to \Sset$ (definition \ref{defi:reedycof}), and as 
$1\coprod_{\uvar\otimes\{0\}}\uvar\otimes[1]\coprod_{\uvar\otimes\{1\}}1$ and  $[\uvar,1]$ preserve cofibrations,  it is sufficient to demonstrate the result when $K:=[n]$ for $n$ an integer. As  $[\uvar,1]$ and, by lemma \ref{lemma:otimes presserves W1}, $\uvar\otimes[1]$ send $\Sp_{[n]}\to [n]$ to $\overline{\W_2}$, it is sufficient to demonstrate the result when $[n]=[1]$. By proposition \ref{prop:explicit Gray}, the morphism $$1\coprod_{[1]\otimes\{0\}}[1]\otimes[1]\coprod_{[1]\otimes\{1\}}1\to [[1],1]$$ fits in the cocartesian square
\[\begin{tikzcd}
	{[0]\coprod_{[1]}[2]~~~\coprod~~~[0]\coprod_{[1]}[2]} & {1\coprod_{[1]\otimes\{0\}}[1]\otimes[1]\coprod_{[1]\otimes\{1\}}1} \\
	{[1]~~~\coprod~~~[1]} & {[[1],1]}
	\arrow[from=2-1, to=2-2]
	\arrow[from=1-2, to=2-2]
	\arrow[""{name=0, anchor=center, inner sep=0}, from=1-1, to=1-2]
	\arrow[from=1-1, to=2-1]
	\arrow["\lrcorner"{anchor=center, pos=0.125, rotate=180}, draw=none, from=2-2, to=0]
\end{tikzcd}\]
As the canonical morphisms $[0]\coprod_{[1]}[2]\to [1]$ and $[2]\coprod_{[1]}\to [1]$ are in $\overline{\W_2}$, this concludes the proof.
\end{proof}

\begin{lemma}
\label{lemma:otimes presserves W2}
Let $n$ be an integer. The two morphisms
$$E^{eq}\otimes[n]\to [n] \quad\text{and}\quad [n]\otimes E^{eq}\to [n]$$
are in $\overline{\W_2^\sat}$.
\end{lemma}
\begin{proof}
As $\otimes$ sends spine inclusions to $\overline{\W_2}\subset \overline{\W_2^\sat}$, we can reduce to the case where $n=1$. By stability by pushouts along monomorphisms, and using lemma \ref{prop:otimes and suspension in presheaves}, the composite
$$E^{eq}\otimes[1]\to 1\coprod_{E^{eq}\otimes\{0\}}E^{eq}\otimes[1]\coprod_{E^{eq}\otimes\{1\}}1\to [E^{eq},1]$$
is in $\overline{\W_2^\sat}$. As $[E^{eq},1]\to [1]$ is in $\W_2^\sat$, this concludes the first assertion. We show the second one similarly.
\end{proof}

\begin{proof}[Proof of theorem \ref{theo:otimes presserves W}]
This is the content of lemmas \ref{lemma:otimes presserves W1} and \ref{lemma:otimes presserves W2}.
\end{proof}

\vspace{1cm}

 We will also need the same analysis for the op-cone.

\begin{construction}
We define $1\star \uvar :\Psh{\Theta}\to \Psh{\Theta}$ as the left Kan extension of the functor 
$$\Theta \xrightarrow{1\star \uvar} \ocat\xrightarrow{\iota} \Psh{\Theta}.$$
\end{construction}

\begin{theorem}
\label{theo:otimes presserves Wversions star}
The functor
$$1\star \uvar:\Psh{\Delta}\to \Psh{\Theta_2}$$
sends $\W_1$ onto $\overline{\W_2}$ and $\W_1^\sat$ onto $\overline{\W_2^\sat}$.
\end{theorem}
\begin{proof}
The proof is similar to the proof of theorem \ref{theo:otimes presserves W}.
\end{proof}

\begin{prop}
\label{prop:exppliocit Graycone}
The $\Theta$-set $1\star[1]$ is the colimit, computed in $\Psh{\Theta}$, of the diagram
\[
\begin{tikzcd}
    {[[1],1]} & {[1]} & {[2]}
    \arrow["{d^1}", from=1-2, to=1-3]
    \arrow["{[d^0,1]}"', from=1-2, to=1-1]
\end{tikzcd}
\]
\end{prop}
\begin{proof}
We denote by $P$ the colimit of this diagram. Remark that $\Fb P$ is the $\omega$-category
\[
\begin{tikzcd}
    1\star\emptyset \\
    {\emptyset\star\{0\}} & {\emptyset\star\{1\}}
    \arrow[from=2-1, to=2-2]
    \arrow[""{name=0, anchor=center, inner sep=0}, from=1-1, to=2-2]
    \arrow[from=1-1, to=2-1]
    \arrow[shorten <=2pt, shorten >=2pt, Rightarrow, from=0, to=2-1]
\end{tikzcd}
\]
and we then have $\Fb P\cong 1\star [1]$. To conclude the proof, we have to show that $P$ is a $\omega$-category, i.e. that it has the unique right lifting property against $\W$. 

Let $f:[\textbf{a},n]\to P$ (resp. $f:\Sp_{[\textbf{a},n]}\to P$) be a morphism. If there exists an integer $i<n$ such that  $f(i)=1\star\emptyset$ and $f(i+1)=\emptyset\star \{1\}$, then $f$ uniquely factors through $[[1],1]\to P$. If there exists an integer $i$ such that $f(i)=\emptyset\star\{0\}$, then $f$ uniquely factors through $[2]\to P$. 
If none of these conditions is satisfied, then $f$ factors through $1\star\emptyset$ or $\emptyset\star\{1\}$.

As $[2]$ and $[[1],1]$ are $\omega$-categories, they have the unique right lifting property against $\W$, and so have $P$.
\end{proof}

\begin{lemma}
\label{lemma:exppliocit Grayconetech}
The $\Theta$-set $\iota(1\star[1]\coprod_{\emptyset\star \{1\}}[1])$ is the colimit, computed in $\Psh{\Theta}$, of the diagram
\[
\begin{tikzcd}
    {[[1],1]\vee[1]} & {[2]} & {[3]}
    \arrow["{d^1}", from=1-2, to=1-3]
    \arrow["{[d^0,2]}"', from=1-2, to=1-1]
\end{tikzcd}
\]
\end{lemma}
\begin{proof}
We denote by $P$ the colimit of this diagram. Remark that $\Fb P$ is the $\omega$-category
\[
\begin{tikzcd}
    1\star\emptyset \\
    {\emptyset\star\{0\}} & {\emptyset\star\{1\}} & {\emptyset\star\{2\}}
    \arrow[from=2-1, to=2-2]
    \arrow[""{name=0, anchor=center, inner sep=0}, from=1-1, to=2-2]
    \arrow[from=1-1, to=2-1]
    \arrow[from=2-2, to=2-3]
    \arrow[shorten <=2pt, shorten >=2pt, Rightarrow, from=0, to=2-1]
\end{tikzcd}
\]
and we then have $\Fb P\cong 1\star[1]\coprod_{\emptyset\star \{1\}}[1]$.  To conclude the proof, we have to show that $P$  is a $\omega$-category, i.e. that it has the unique right lifting property against $\W$. The proof of this assertion is an easy adaptation of the one of lemma \ref{prop:exppliocit Graycone}.
\end{proof}

\begin{prop}
\label{prop:exppliocit Graycone2}
The $\Theta$-set $1\star[2]$ is the colimit, computed in $\Psh{\Theta}$, of the diagram
\[
\begin{tikzcd}
    {[[2],1]} & {[[1],1]} & {[[1],1]\vee[1]} & {[2]} & {[3]}
    \arrow["{d^1}", from=1-4, to=1-5]
    \arrow["{[d^0,2]}"', from=1-4, to=1-3]
    \arrow["{[[1],d^1]}", from=1-2, to=1-3]
    \arrow["{[d^0,1]}"', from=1-2, to=1-1]
\end{tikzcd}
\]
\end{prop}
\begin{proof}
We recall that lemma \ref{lemma:exppliocit Grayconetech} states that the colimit of the subdiagram
\[
\begin{tikzcd}
    {[[1],1]\vee[1]} & {[2]} & {[3]}
    \arrow["{d^1}", from=1-2, to=1-3]
    \arrow["{[d^0,2]}"', from=1-2, to=1-1]
\end{tikzcd}
\] 
is equivalent to $\iota(1\star[1]\coprod_{\emptyset\star \{1\}}[1])$.

We denote by $P$ the colimit of this diagram given in the statement of the proposition. Remark that $\Fb P$ is the $\omega$-category
\[
\begin{tikzcd}
    1\star\emptyset & {} \\
    {\emptyset\star\{0\}} & {\emptyset\star\{1\}} & {\emptyset\star\{2\}}
    \arrow[from=2-1, to=2-2]
    \arrow[""{name=0, anchor=center, inner sep=0}, from=1-1, to=2-2]
    \arrow[from=1-1, to=2-1]
    \arrow[from=2-2, to=2-3]
    \arrow[curve={height=-12pt}, from=1-1, to=2-3]
    \arrow[""{name=1, anchor=center, inner sep=0}, draw=none, from=1-2, to=2-3]
    \arrow[shorten <=2pt, shorten >=2pt, Rightarrow, from=0, to=2-1]
    \arrow[shorten <=2pt, Rightarrow, from=1, to=2-2]
\end{tikzcd}
\]
and we then have $\Fb P\cong 1\star [2]$.  To conclude the proof, we have to show that $P$  is a $\omega$-category, i.e. that it  has the unique right lifting property against $\W$. 

Let $f:[\textbf{a},n]\to P$ (resp. $f:\Sp_{[\textbf{a},n]}\to P$) be a morphism. If there exists an integer $i<n$ such that  $f(i)=1\star\emptyset$ and $f(i+1)=\emptyset\star \{2\}$, then $f$ uniquely factors through $[[2],1]\to P$. If there exists an integer $i$ such that $f(i)=\emptyset\star\{0\}$, then $f$ uniquely factors through $\iota(1\star[1]\coprod_{\emptyset\star \{1\}}[1])$. 
If none of these conditions is satisfied, then $f$ factors through $1\star\emptyset$ or $\emptyset\star\{2\}$.

As $[2]$ and $\iota(1\star[1]\coprod_{\emptyset\star \{1\}}[1])$ are $\omega$-categories, they have the unique right lifting property against $\W$, and so have $P$.
\end{proof}

\begin{lemma}
\label{lemma:techiniqcal steiner2version star}
Let $A$, $B$ be presheaves on $\Delta$, and $k,m,n,l$ be integers. There exists a natural morphism 
$$\beta: \Hom([A,k],1\star [m])\times \Hom([B,m],1\star [n])\to \Hom([B\times A,k],1\star [n])$$
such that for any $f:[A,k]\to 1\star [m]$, $g:[B,m]\to 1\star [n]$ and $h:[B,n]\to 1\star[l]$,
\begin{equation}
\label{eq:beta} 
\beta(\beta(f,g),h) = \beta(f,\beta(g,h))
\end{equation}
\end{lemma}
\begin{proof}
Similar to the proof of lemma \ref{lemma:technical steiner2}.
\end{proof}

\begin{prop}
\label{prop:comparaison in theta set between otimes and star}
Let $K$ be a simplicial set. The canonical morphism 
$$1\coprod_{\{0\}\otimes K}[1]\otimes K\to 1\star K$$
is in $\overline{\W_2}$.
\end{prop}
\begin{proof}
As $K$ is a colimit of representables indexed by the Reedy cofibrant diagram $\Delta_{/K}\to \Sset$ (definition \ref{defi:reedycof}), and as 
$1\coprod_{\{0\}\otimes \uvar}[1]\otimes \uvar$ and  $1\star \uvar$ preserve cofibrations,  it is sufficient to demonstrate the result when $K:=[n]$ for $n$ an integer. By theorems \ref{theo:otimes presserves W}  and \ref{theo:otimes presserves Wversions star}, the functors $1\coprod_{\{0\}\otimes \uvar}[1]\otimes \uvar$ and $1\star \uvar$  send $\Sp_{[n]}\to [n]$ to $\overline{\W_2}$. It  is then sufficient to demonstrate the result when $[n]=[1]$. By propositions \ref{prop:explicit Gray} and \ref{prop:exppliocit Graycone}, the morphism $$1\coprod_{\{0\}\otimes [1]}[1]\otimes[1]\to 1\star[1]$$ fits in the cocartesian square
\[\begin{tikzcd}
	{[0]\coprod_{[1]}[2]} & {1\coprod_{\{0\}\otimes [1]}[1]\otimes[1]} \\
	{[1]} & {1\star[1]}
	\arrow[from=2-1, to=2-2]
	\arrow[from=1-2, to=2-2]
	\arrow[""{name=0, anchor=center, inner sep=0}, from=1-1, to=1-2]
	\arrow[from=1-1, to=2-1]
	\arrow["\lrcorner"{anchor=center, pos=0.125, rotate=180}, draw=none, from=2-2, to=0]
\end{tikzcd}\]
As the canonical morphism $[0]\coprod_{[1]}[2]\to [1]$ is in $\overline{\W_2}$, this concludes the proof.
\end{proof}



\chapter{Study of complicial sets}
\label{chapter:Studies of the complicial model}

\minitoc
\vspace{1cm}

\section{Preliminaries}

\subsection{Homotopy colimits}
\label{chapter:Generalities on model categories V2}
For this section, we fix a model category $C$ whose cofibrations are monomorphisms.
We give first some results on homotopy colimits. These results will be used freely throughout these text.

\begin{prop}
\label{prop:hom colimit 1}
Suppose given a square
\[\begin{tikzcd}
	a & b \\
	c & d
	\arrow[from=1-1, to=2-1]
	\arrow[from=1-2, to=2-2]
	\arrow[from=2-1, to=2-2]
	\arrow[from=1-1, to=1-2]
\end{tikzcd}\]
such that the two horizontal morphisms are weak equivalences. Then this square is homotopy cocartesian. 
\end{prop}
\begin{proof}
This is \cite[proposition 2.3.26]{Cisinski_Higher_categories_and_homotopical_algebra}.
\end{proof}
\begin{prop}
\label{prop:hom colimit 2}
Suppose given a cocartesian square
\[\begin{tikzcd}
	a & b \\
	c & d
	\arrow[from=1-1, to=2-1]
	\arrow[from=1-2, to=2-2]
	\arrow[from=2-1, to=2-2]
	\arrow[from=1-1, to=1-2]
	\arrow["\lrcorner"{anchor=center, pos=0.125, rotate=180}, draw=none, from=2-2, to=1-1]
\end{tikzcd}\]
where the left vertical morphism is a cofibration. Then this square is homotopy cocartesian. 
\end{prop}
\begin{proof}
This is \cite[corollary 2.3.28]{Cisinski_Higher_categories_and_homotopical_algebra}.
\end{proof}

\begin{prop}
\label{prop:hom colimit 2.5}
Weak equivalences are stable by pushout along cofibrations.
\end{prop}
\begin{proof}
This is \cite[proposition 13.1.2]{Hirschhorn_Model_categories_and_their_localizations}.
\end{proof}

\begin{prop}
\label{prop:hom colimit 3}
Let $F:\alpha\to C$ be a diagram indexed by an ordinal. The transfinite composition $\colim_\alpha F$ is the homotopy colimit of the diagram $F$.
\end{prop}
\begin{proof}
This is \cite[proposition 2.3.13]{Cisinski_Higher_categories_and_homotopical_algebra}.
\end{proof}

\begin{prop}
\label{prop:hom colimit 4}
Suppose given a diagram 
\[\begin{tikzcd}
	& {b_0} && {...} && {b_{n-1}} \\
	{a_0} && {a_1} && {a_{n-1}} && {a_{n}}
	\arrow[from=1-2, to=2-1]
	\arrow[hook, from=1-2, to=2-3]
	\arrow[from=1-4, to=2-3]
	\arrow[hook, from=1-4, to=2-5]
	\arrow[from=1-6, to=2-5]
	\arrow[hook, from=1-6, to=2-7]
\end{tikzcd}\]
where all morphisms labelled by $\hookrightarrow$ are cofibrations.
The colimit of this diagram is also the homotopy colimit of this diagram.
\end{prop}
\begin{proof}
Let $I_n$ be the category indexing the previous diagram. We denote by $i_0$, $j_0$,..., $i_{n-1}$, $j_{n-1}$, $i_{n}$ it's objects.
The projective model structure on $\Fun(I_n,C)$ is given by functor $G$ such that for any $k<n$, $F(j_k)\to F(i_k)$, $F(j_k)\to F(i_{k+1})$ are monomorphisms, and such that for any $0<k<n$, $F(j_k)\coprod F(j_{k+1}) \to F(i_k)$ is a monomorphism. Remark that such presheaves verify the condition given in the statement of the proposition.

 We will show on induction on $n$ that a natural transformation $\psi$ between two diagrams $F,G:I_n\to C$ that fulfills the desired condition induces a weak equivalence between their colimits. As we can always chose $F$ to be the cofibrant replacement of $G$ in the projective model structure on $\Fun(I_n,C)$, it will imply the desired result. 

The case $n=1$ is proposition \ref{prop:hom colimit 2}. Suppose now the result is true at the stage $(n-1)$ and let $\psi$ be a weakly invertible natural transformation between two diagram $F,G:I_n\to C$ that fulfills the desired condition. We denote by $\iota:I_{n-1}\to I_n$ the canonical inclusion that sends $i_k$(resp. $j_k$) on $i_k$(resp. $j_k$) for $k<n$ (resp. $k<n-1$).
We then have a diagram 
\[\begin{tikzcd}
	{\colim_{I_{n-1}}F\circ\iota} & {F(j_{n-1})} & {F(i_n)} \\
	{\colim_{I_{n-1}}G\circ\iota} & {G(j_{n-1})} & {G(i_n)}
	\arrow[hook, from=1-2, to=1-3]
	\arrow[from=1-2, to=1-1]
	\arrow["\sim"', from=1-1, to=2-1]
	\arrow["\sim"', from=1-3, to=2-3]
	\arrow["\sim"', from=1-2, to=2-2]
	\arrow[from=2-2, to=2-1]
	\arrow[hook, from=2-2, to=2-3]
\end{tikzcd}\]
where all arrows labeled by $\sim$ are weak equivalences. Remark furthermore that the limit of the two lines are respectively $\colim_{I_{n}}F$ and $\colim_{I_{n}}G$. A last application of proposition \ref{prop:hom colimit 2} concludes the proof.
\end{proof}

\begin{definition}
\label{defi:nice model structure}
A model structure is \wcnotion{nice}{nice model structure} if it is simplicial, combinatorial, and its cofibrations are monomorphisms. 
\end{definition}

 \begin{definition}
\label{defi:reedycof}
Let $C$ be a cocomplete category. A functor $F:A\to C$ is \textit{Reedy cofibrant}\index[notion]{Reedy cofibrant functor} if $A$ has a structure of Reedy elegant category (definition \ref{defi:reedy}) and for every object $a$, the  induced morphism $\colim_{\partial a}F\to F(a)$ is a monomorphism.
\end{definition}

As all the presheaves categories that we will encounter through this text are presheaves on elegant Reedy categories, we will use freely the following theorem:
\begin{theorem}[Hirschhorn]
\label{theo:hom colimi}
We suppose that $C$ is a nice model category.
Let $A$ be a elegant Reedy category, and $F:A\to C$ a Reedy cofibrant diagram. The object $\colim_A F$ is the homotopy colimit of $F$. In particular, if $C$ is $\Psh{B}$, with $B$ an elegant Reedy category, every object $X$ is the homotopy colimit of the diagram $B_{/X}\to B\to \Psh{B}$.
\end{theorem}
\begin{proof}
Using the characterization of elegant Reedy category given by proposition 3.8 of \cite{Bergner_reedy_category_and_the_theta_construction}, and \cite[proposition 15.10.2]{Hirschhorn_Model_categories_and_their_localizations},
it's easy to see that they have fibrant constant in the sens of \cite[definition 15.10.1]{Hirschhorn_Model_categories_and_their_localizations}. We can then apply the theorem 19.9.1	 of \cite{Hirschhorn_Model_categories_and_their_localizations}.
\end{proof}

The previous results justifies the following definition.
\begin{definition}
\label{defi:precocomplet}
A class of morphisms $T$ of a cocomplete category $C$ is \wcnotionsym{precocomplete}{(ss@$\overline{S}$}{precocomplete class of morphisms} if
\begin{enumerate}
\item monomorphisms in $T$ are closed under transfinite compositions and pushouts.
\item $T$ is closed by pushout along monomorphism
\item It is closed under two out of three.
\item For any Reedy cofibrant diagram $F:A\to \Arr(C)$ that is pointwise in $T$, the morphism $\colim_AF$ is in $T$.
\end{enumerate} 
For a set of morphisms $T$, we denote $\overline{T}$ as the smallest precocomplete class of morphisms containing $T$.
\end{definition}

\subsection{Building model structures}

In this section, we provide a theorem to build model structures starting from a few data. All the results here are specifications of much more general results taken from the Cisinski theory presented, for example, in \cite{Cisinski_Higher_categories_and_homotopical_algebra}.

\begin{prop}
\label{prop:weak equivalence are precocomplete}
Weak equivalences of a nice model category form a precocomplete class in the sense of definition \ref{defi:precocomplet}.
\end{prop}

\begin{proof}
The conditions $(1)$ and $(3)$ of definition \ref{defi:precocomplet} are obviously fulfilled by the class of weak equivalences. The condition $(2)$ follows from proposition \ref{prop:hom colimit 2.5}, and the last one from theorem \ref{theo:hom colimi}.
\end{proof}

\begin{theorem}
\label{theo:free_model_structure}
Let $B$ be an elegant reedy category. Suppose we are given a left adjoint $\alpha:\Psh{\Delta}\to \Psh{B}$ preserving monomorphisms and finite products, and let $J$ be a set of monomorphisms in $\Psh{B}$. Then there exists a nice model structure denoted $\Psh{B}^{\alpha,J}$ whose weak equivalences are the smallest precocomplete class of maps containing $J$ and $b\times \alpha([n])\to b$ for $b\in B$ and integer $n$.
\end{theorem}

\begin{proof}
We denote $J_{\mathrm{Kan}}$ as a set of generating acyclic cofibrations of the Kan-Quillen model structure on $\Psh{\Delta}$. We fix a cellular model $\cell(B)$ of $\Psh{B}$ and a cellular model $\cell(\Delta)$ of $\Psh{\Delta}$. The functor $X\mapsto X\times \alpha([1])$ is an exact cylinder in the sense of \cite[definition 2.4.8]{Cisinski_Higher_categories_and_homotopical_algebra}. We denote $\Lambda$ as the smallest set of morphisms in $\Psh{B}$ that is closed under pushout, transfinite composition, and retract, and that includes 
$$f\hat{\times}\alpha(g):b\times \alpha(K)\cup a\times \alpha(L)\to b\times \alpha(L)$$
whenever $(f,g)\in J\times \cell(\Delta)$ or $(f,g)\in \cell(B)\times J_{\mathrm{Kan}}$. Using the fact that $\alpha$ preserves finite products and by \cite[proposition 3.1.2]{Cisinski_Higher_categories_and_homotopical_algebra}, we have that $\Lambda$ is a class of $\alpha([1])$-anodyne extensions in the sense of \cite[definition 2.4.11]{Cisinski_Higher_categories_and_homotopical_algebra}. By \cite[theorem 2.4.19]{Cisinski_Higher_categories_and_homotopical_algebra}, this induces a model structure on $\Psh{B}$. 
By \cite[Scholie 1.3.45]{cisinski_prefaisceaux_comme_modele}, we have that the class of weak equivalences $W$ of this model structure is the smallest one stable under two out of three, and containing $\Lambda$.

We will denote $W'$ the smallest precocomplete class of maps containing $J$ and $b\times \alpha([n])\to b$ for $b\in B$ and integer $n$.
We can easily check that for any anodyne extension $K\to L$ of the Kan-Quillen model structure, the morphism $b\times K\to b\times L$ is a weak equivalence of the model structure of $\Psh{B}$. In particular, it implies that $b\times \alpha(\{0\})\to b\times \alpha([n])$ is a weak equivalence. By two out of three, this implies that $b\times \alpha([n])\to b$ is also a weak equivalence. As $J$ is also contained in $\Lambda$, proposition \ref{prop:weak equivalence are precocomplete} implies that we have $W'\subset W$.

Note now that the class of morphisms of simplicial sets $K\to L$ such that $b\times \alpha(K)\to b\times\alpha(L)$ is in $W'$ verifies the three conditions of \cite[corollaire 2.1.20]{cisinski_prefaisceaux_comme_modele}, and then includes all weak equivalences of the Kan-Quillen model structure. 
Let $a\to b$ be a morphism in $W'$. By two out of three, the class of simplicial sets $K$ such that $a\times \alpha(K)\to b\times \alpha(K)$ is in $W'$ contains every representable simplicial set and is closed under colimits indexed by Reedy cofibrant diagrams by definition. It then includes all simplicial sets by proposition \ref{prop:elelangat stable by slice}.

Suppose now we are given $f:a\to b$ in $\Psh{B}$, $g:K\to L$ in $\Psh{\Delta}$ and consider the induced diagram
\[\begin{tikzcd}
	{a\times\alpha(K)} & {a\times\alpha(L)} & \\
	{b\times\alpha(K)} & {b\times\alpha(K)\cup a\times\alpha(L)} \\
	&& {b\times\alpha(L)}
	\arrow["{a\times \alpha(g)}", from=1-1, to=1-2]
	\arrow["{f\times \alpha(K)}"', from=1-1, to=2-1]
	\arrow[from=1-2, to=2-2]
	\arrow["{f\times \alpha(L)}", curve={height=-18pt}, from=1-2, to=3-3]
	\arrow[from=2-1, to=2-2]
	\arrow["{b\times \alpha(g)}"', curve={height=18pt}, from=2-1, to=3-3]
	\arrow["{f\hat{\times} g}", from=2-2, to=3-3]
\end{tikzcd}\]
If $(f,g)\in J\times \cell(\Delta)$, then $f\times \alpha(K)$ and $f\times \alpha(L)$ are in $W'$. By stability under pushout and two out of three, $f\hat{\times} g$ also belongs to it. If $(f,g)\in \cell(A)\times J_{\mathrm{Kan}}$, then $a\times \alpha(g)$ and $b\times \alpha(g)$ are in $W'$. By stability under pushout and two out of three, so is $f\hat{\times} g$. We then have an inclusion $\Lambda\subset W'$ and then $W\subset W'$ which concludes the proof.
\end{proof}

\begin{remark}
\label{remark:free model structure and loc}
Remark that $\Psh{B}^{\alpha,J}$ is always the left Bousfield localization of  $\Psh{B}^{\alpha,\emptyset}$ along $J$.
\end{remark}

\begin{remark}
\label{rem:mapping from a free model structure}
Let $F:C\to D$ be a left adjoint between two nice model structures that preserves monomorphisms. The proposition \ref{prop:weak equivalence are precocomplete} implies that the set $F^{-1}(\W_{D})$ is a precocomplete class. In particular, if the weak equivalences of $C$ are of shape $\overline{S}$ for $S$ a set of morphisms (as is the case for all model structures constructed from theorems \ref{theo:free_model_structure} and \ref{theo:free_model_structure_on_marking}), then $F$ is a left Quillen adjoint if and only if $F(S)$ consists of weak equivalences.
\end{remark}

\subsection{Zig-zag of acyclic cofibrations}

\begin{definition}
 Let $i:A\to B$ and $i':A'\to B'$ be two cofibrations. A \notion{zigzag of acyclic cofibrations} between $i$ and $i'$, denoted $i\leftrightsquigarrow i'$, is a zigzag in the category of arrows such that all the horizontal maps are acyclic cofibrations, and all the vertical maps are cofibrations.
\end{definition}

\begin{lemma}
Let $i$ and $j$ be two cofibrations, and $f:X\to Y$ a fibration between fibrant objects. Suppose that we have a morphism in the category of arrows $i\to j$ which is pointwise an acyclic cofibration. Then, if $j$ has the left lifting property against $f$, so does $i$.
\end{lemma}

\begin{proof}
We consider a diagram of the following shape:
\[\begin{tikzcd}
	A & {A'} & X \\
	B & {B'} & Y.
	\arrow["i", from=1-1, to=2-1]
	\arrow["\sim", from=1-1, to=1-2]
	\arrow["\sim"', from=2-1, to=2-2]
	\arrow["j"', from=1-2, to=2-2]
	\arrow[from=1-3, to=2-3]
	\arrow[curve={height=-18pt}, from=1-1, to=1-3]
	\arrow[curve={height=18pt}, from=2-1, to=2-3]
	\arrow["{l_0}"', dotted, from=2-2, to=2-3]
	\arrow["{l_1}"{description}, dotted, from=1-2, to=1-3]
	\arrow["{l_2}"{description}, dotted, from=2-2, to=1-3]
\end{tikzcd}\]
We construct, one after the other, the liftings $l_0$, $l_1$, and $l_2$.
\end{proof}

\begin{lemma}
Let $i$ and $j$ be two cofibrations, and $f:X\to Y$ a fibration between fibrant objects. Suppose that we have a morphism in the category of arrows $i\to j$ which is pointwise an acyclic cofibration. Then, if $i$ has the right lifting property against $f$, so does $j$.
\end{lemma}

\begin{proof}
We consider a diagram of the following shape:
\[\begin{tikzcd}
	A & {A'} &&& X \\
	B & {B\coprod_A A'} \\
	&& {B'} && Y.
	\arrow[from=1-1, to=2-1]
	\arrow["\sim"', from=1-1, to=1-2]
	\arrow["\sim", from=2-1, to=2-2]
	\arrow["\sim"{description}, from=2-2, to=3-3]
	\arrow["\sim"', curve={height=6pt}, from=2-1, to=3-3]
	\arrow[curve={height=-6pt}, from=1-2, to=3-3]
	\arrow[from=1-2, to=2-2]
	\arrow[from=1-2, to=1-5]
	\arrow[from=3-3, to=3-5]
	\arrow[from=1-5, to=3-5]
	\arrow["{l_0}"{description}, dotted, from=2-2, to=1-5]
	\arrow["{l_1}"{description}, dotted, from=3-3, to=1-5]
\end{tikzcd}\]
We construct, one after the other, the liftings $l_0$, $l_1$.
\end{proof}

\begin{prop}
\label{prop:lifting_property_zigzag_of_acyclic_cofibration}
Let $f$ be a fibration between fibrant objects and $i$ and $j$ two cofibrations such that there exists a zigzag of acyclic cofibrations $i\leftrightsquigarrow j$. Then $f$ has the right lifting property against $i$ if and only if it has the right lifting property against $j$.
\end{prop}

\begin{proof}
This is a direct consequence of the last two lemmas.
\end{proof}

\subsection{Marked and stratified presheaves}
\label{section:Marked and stratified presheaves}
\begin{definition}
\label{defi:of pshMB}
Let $B$ be an elegant Reedy category and $M$ a subset of the set of objects of $B$. A \wcnotion{$M$-stratified presheaf on $B$}{stratified presheaf on $B$}, or just a \textit{stratified prehsheaf on $B$} when the subset $M$ will be non-ambiguous, is a pair $(X,tX)$ where $X$ is a presheaf on $B$ and $tX:=\coprod_{a\in M} tX_a$ is the disjoint union of sets, such that for any $a\in M$, $tX_a$ is a subset of $X_a$ including degeneracies, i.e the image of morphisms $X_p:X_b\to X_a$ for $p:b\to a$ in $B_-$.

A \notion{stratified morphism} $f:(X,tX)\to (Y,tY)$ is the data of a morphism on the underlying presheaf such that $f(tX_n)\subset tY_n$.
The category of stratified presheaves is denoted by \wcnotation{$\tPshM{B}$}{(tpsh@$\tPshM{B}$}. 
\end{definition}

\begin{definition}
\label{defi:entire}
A morphism between two stratified presheaves is \wcnotion{entire}{entire morphism} if it is the identity on the underlying presheaves.
\end{definition}

\begin{construction}
\label{construction: definition of sharp}
We have adjunctions
\[\begin{tikzcd}
	{(\uvar)^\natural:\tPshM{B}} && {\Psh{B}}
	\arrow[""{name=0, anchor=center, inner sep=0}, from=1-1, to=1-3]
	\arrow[""{name=1, anchor=center, inner sep=0}, "{(\uvar)^\flat}", shift left=3, from=1-3, to=1-1]
	\arrow[""{name=2, anchor=center, inner sep=0}, "{(\uvar)^\sharp}"', shift right=3, from=1-3, to=1-1]
	\arrow["\dashv"{anchor=center, rotate=90}, draw=none, from=0, to=2]
	\arrow["\dashv"{anchor=center, rotate=90}, draw=none, from=1, to=0]
\end{tikzcd}\]
where $(\uvar)^\flat$ and \wcsnotation{$(\uvar)^\sharp$}{((b61@$(\uvar)^{\sharp}$}{for stratified presheaves} are fully faithful inclusions that send a presheaf $X$ onto $(X,S)$ where $S$ is respectively the smaller and maximal stratification on $X$, and where the right adjoint is the obvious forgetful functor. We will identify presheaves on $B$ with their image by the functor $(\uvar)^\flat$, and will denote $(\uvar)^{\sharp}:\tPshM{B}\to \tPshM{B}$ the functor $((\uvar)^{\natural})^{\sharp}$.
\end{construction}

\begin{construction}
\label{construction: of tB}
 If $b$ is an object of $M$, we denote by $b_t$ the stratifed presheaf $(b,S)$, where $S$ is the smaller stratification that includes $id:b\to b$.

We then define $t_MB$ as the full subcategory of $\tPshM{B}$ spanned by the objects of shape $a$ or $b_t$ with $a\in B$ and $b\in M$. We then have equalities:
$$\begin{array}{rcl}
\Hom_{t_MB}(a,b)&:=& \Hom_B(a,b),\\
\Hom_{t_MB}(a,b_t)&:=& \Hom_B(a,b),\\
\Hom_{t_MB}(a_t,b)&:=& \Hom_B(a,b)\cap B_- \diagdown \{id_a\}, \\
\Hom_{t_MB}(a_t,b_t)&:=&\Hom_B(a,b)\cap B_-.\\
\end{array}$$
The canonical functor $B\to t_MB$ is then fully faithful and we will identify object of $B$ with their image through this functor.

The category of $M$-stratified presheaves is then equivalent to the fully faithful subcategory of presheaves $X$ on $t_MB$ such that for any $b\in M$, $X(b_t)\to X(b)$ is a monomorphism. 
In particular, we have an adjunction 
\begin{equation}
\label{eq:entre presheaveds on tB et stratified presheages}
\begin{tikzcd}
	{\pi:\Psh{t_MB}} & {\tPshM{B}:\iota}
	\arrow[""{name=0, anchor=center, inner sep=0}, shift left=2, from=1-1, to=1-2]
	\arrow[""{name=1, anchor=center, inner sep=0}, shift left=2, from=1-2, to=1-1]
	\arrow["\dashv"{anchor=center, rotate=-90}, draw=none, from=0, to=1]
\end{tikzcd}
\end{equation}
\end{construction}

\begin{prop}
\label{prop:reedy structure on tB}
The category $t_MB$ admits a structure of elegant Reedy category, that makes the inclusion $B\to t_MB$ a morphism of Reedy category. There is no non trivial negative morphism whose codomain is of shape $b_t$ for $b\in M$. There is no non trivial positive morphism whose domain is of shape $b_t$ for $b\in M$. 
\end{prop}
\begin{proof}
We define the degree degree function $ob(t_MB)\to \Nb$ by the assignment 
$$d'(b):= 2 d(b)~~~~~d'(b_t):= 2d(b)+1$$
The category $(t_MB)_+$ is the smallest that includes $B_+$ and morphisms of shape $a\to a_t$. The category $(t_MB)_-$ is the smallest that includes $B_-$ and morphisms of shape $b_t\to a$.

To prove the axioms of Reedy category, we can replicate the strategy used in proposition C.2 of \cite{Ozornova_model_structure_for_infini_n_categories} with obvious modification to this more general framework.

We still have to show that $tB$ is elegant. Let $X$ be a presheaf on $t_MB$, $a$ an element of $t_MB$, $f:a\to a'$ and $g:a\to a'$ two negative morphisms, an element $x$ of $X(a)$, two non degenerate elements $y\in X(a')$ and $z\in X(a'')$ such that $f^*y=x$, $g^*z=x$. 

Suppose first that $a$ is in $B$. In this case, $f$ and $g$ are also in $B$, and as this Reedy category is elegant by assumption, this implies $f=g$ and $y=z$. Suppose now that $a$ is of shape $b_t$ for $b\in B$. We denote by $\alpha$ the canonical morphism $\alpha:b\to b_t$. By definition of negative morphism, the codomain of $f$ and $g$ are in $B$. The morphisms $\alpha f $ and $\alpha g$ then are in $B$. Moreover, these two morphisms are negative, and we have $(\alpha f)^*y=\alpha^* x$, $(\alpha g)^*z=\alpha^* x$. As $B$ is elegant, $\alpha f=\alpha g$ and $y=z$. Eventually, remark that the first equality implies that $f$ is equal to $g$. 
\end{proof}
\begin{remark}
A cellular model for $t_MB$ is given by $C\cup\{b\to b_t,b\in M\}$ where $C$ is a cellular model for $B$.
\end{remark}

\begin{prop}
\label{prop:transfert from presheaves on tB to stratified presheaves}
Suppose given a combinatorial model structure on $\Psh{t_MB}$ whose cofibrations are monomorphisms. Then there exists a combinatorial model structure on $\tPshM{B}$ making the adjunction \ref{eq:entre presheaveds on tB et stratified presheages} a Quillen equivalence.

A morphism of $\tPshM{B}$ is a cofibration if and only if it is a monomorphism. A morphism is a fibration (resp. a weak equivalence) if and only if its image by $\iota$ is.

\end{prop}
\begin{proof}
We are willing to apply \cite[theorem 11.3.2]{Hirschhorn_Model_categories_and_their_localizations}. As two adjoints of \eqref{eq:entre presheaveds on tB et stratified presheages} preserve smallness, the first condition is obviously fulfilled. Using the fact that $\iota$ is fully faithful, the second condition of theorem \textit{op cit} is equivalent to asking that for any acyclic cofibration $i$ of $\Psh{t_MB}$, the morphism $\iota\pi i$ is a weak equivalence. 

However, remark that the unit $X\to \iota \pi X$ is a trivial fibration. Indeed, a cellular model is given $C\cup\{b\to b_t,b\in M\}$, where  $C$ is a cellular model for $B$, and the unit obviously has the right lifting property against it. The result then directly follows from the stability of weak equivalences by two out of three.

This provides the model structure. As the unit is pointwise a trivial fibration and the counit is the identity, the adjunction
\eqref{eq:entre presheaveds on tB et stratified presheages} induces a Quillen equivalence. 
\end{proof}

\begin{theorem}

\label{theo:free_model_structure_on_marking}
Let $B$ be an elegant Reedy category and $M$ a marking on it. Suppose we are given a left adjoint $\alpha:\Psh{\Delta}\to \tPshM{B}$ such that $\iota\alpha:\Psh{\Delta}\to \Psh{t_MB}$ preserves monomorphisms and finite products, and $J$ a set of monomorphisms in $\tPshM{B}$. Then there exists a nice model structure on $\tPshM{B}$ whose weak equivalences are the smallest precocomplete class of maps containing $J$ and $b\times \alpha([n])\to b$ for $b\in t_MB$ and integer $n$.
\end{theorem}

\begin{proof}
By proposition \ref{prop:transfert from presheaves on tB to stratified presheaves}, we can transfer to $\tPshM{B}$ the model structure on $\Psh{t_MB}$ obtained from the theorem \ref{theo:free_model_structure} applied to $\iota\alpha$ and $\iota(J)$.
\end{proof}

\begin{remark}
\label{rem: free model structure on marking}
 Remark that in the previous theorem, the adjunction between $\Psh{t_MB}$ and $\tPshM{B}$ given in \ref{construction: of tB} is a Quillen equivalence.
\end{remark}

\vspace{1cm}

We now fix a Reedy category $B$, a subset $M$ of objects of $B$, and we suppose given a nice model structure on $\tPshM{B}$ (as defined in definition \ref{defi:nice model structure}).
\begin{definition}
A \wcnotion{$M$-marked presheaf on $B$}{marked presheaf on $B$} is a stratified presheaf having the unique right lifting property against all entire acyclic cofibrations. In particular, any fibrant objects is marked. 

We denote by \wcnotation{$\mPshM{B}$}{(mpsh@$\mPshM{\uvar}$} the full subcategory of marked presheaves on $B$. We then have an adjunction: \sym{((b91@$(\uvar)_{\mk}$}
\begin{equation}
\label{eq:adj beetwen stratified and marked}
\begin{tikzcd}
	{(\uvar)_{\mk}:\tPshM{B}} & {\mPshM{B}:\iota}
	\arrow[""{name=0, anchor=center, inner sep=0}, shift left=2, from=1-2, to=1-1]
	\arrow[""{name=1, anchor=center, inner sep=0}, shift left=2, from=1-1, to=1-2]
	\arrow["\dashv"{anchor=center, rotate=-90}, draw=none, from=1, to=0]
\end{tikzcd}
\end{equation}
where the left adjoint $(\uvar)_{\mk}$ sends a stratified presheaf $(X,tX)$ to the marked presheaf $(X,\overline{tX})$, where $\overline{tX}$ is the smaller stratification that includes $tX$ and makes $(X,\overline{tX})$ a marked presheaf, and where the right adjoint is a fully faithful inclusion.
Remark furthermore that at the level of presheaves, these two adjoints are the identity. 
\end{definition}

\begin{prop}
\label{prop:X to Xmk is acycli cof}
Let $X$ be a $M$-stratified presheaf on $B$.
The canonical morphism $X\to \iota (X_{\mk})$ is an entire acyclic cofibration.
\end{prop}
\begin{proof}
Let $\kappa$ be a regular cardinal such that $X$ is $\kappa$-small. Remark first the domain of a entire monomorphism is $\kappa$-small if and only if its codomain is.

Let $I$ be the set of entire acyclic cofibrations with $\kappa$-small codomains and domains. This set generates via the small object argument a weak factorization system, and we denote by $X\to X'\to 1$ the factorization of $X\to 1$. We are willing to show that $X'$ is $M$-marked. As $X\to X'$ is an entire acyclic cofibration by construction, this will directly imply that $X'$ is equal to $\iota (X_{\mk})$ and so demonstrate the desired result.

Suppose then given a diagram
\[\begin{tikzcd}
	K & {X'} \\
	L & 1
	\arrow[from=2-1, to=2-2]
	\arrow[from=1-1, to=1-2]
	\arrow[from=1-2, to=2-2]
	\arrow["i"', from=1-1, to=2-1]
\end{tikzcd}\]
with $i$ an entire acyclic cofibration. We have to show that it admits a lift.
Remark that this square factors as:
\[\begin{tikzcd}
	K & {X'} & {X'} \\
	L & {X'\coprod_KL} & 1
	\arrow[from=1-1, to=1-2]
	\arrow["i"', from=1-1, to=2-1]
	\arrow[from=1-3, to=2-3]
	\arrow[Rightarrow, no head, from=1-2, to=1-3]
	\arrow[from=2-1, to=2-2]
	\arrow[from=2-2, to=2-3]
	\arrow["{i'}", from=1-2, to=2-2]
	\arrow["\lrcorner"{anchor=center, pos=0.125, rotate=180}, draw=none, from=2-2, to=1-1]
\end{tikzcd}\]
The morphism $i'$ is an entire acyclic cofibration with $\kappa$-small codomain and domain and then belongs to $i$. The right square of the previous diagram then admits a lift. This induces a lift in the in the original square, and this concludes the proof.
\end{proof}

\begin{prop} 
\label{prop:model structure on marked presheaves}
Suppose given a nice model structure on $\tPshM{B}$.
This induces a nice model structure on $\mPshM{B}$, making the adjunction \eqref{eq:adj beetwen stratified and marked} a Quillen equivalence. A morphism between two marked presheaves is a cofibration (resp. a fibration) (resp. a weak equivalence) if it is a cofibration (resp. a fibration) (resp. a weak equivalence) when seen as a morphism of $\tPshM{B}$. 
\end{prop} 
\begin{proof} Let $f:X\to Y$ be a fibration between stratified presheaves. If $Y$ is marked, so is $X$. The two weak factorization systems on $\mPshM{B}$ are then induced by the one of $\tPshM{B}$. We leave it to the reader to check that this model structure is nice. 

The unit is pointwise a weak equivalence according to proposition \ref{prop:X to Xmk is acycli cof} and the counit is the identity. 
The adjunction \eqref{eq:adj beetwen stratified and marked} is then a Quillen equivalence.
\end{proof}

\section{The complicial model}

\subsection{Model structure on stratified and marked simplicial sets}

\begin{definition}
Let $M$ be the set of simplices of non-negative dimension. We will simply denote $\tPsh{\Delta}$ and $t\Delta$ the categories $\tPshM{\Delta}$ and $t_M\Delta$ from constructions \ref{defi:of pshMB} and \ref{construction: of tB}. 

Unfolding the definition, the elements of $\tPsh{\Delta}$, called \notion{stratified simplicial set}, correspond to pairs $(X,tX)$ where $X$ is a simplicial set and $tX := \cup_{n>0}tX_n$ a graded set such that for any $n\geq 1$, $tX_n$ is a subset of $X_n$ that includes all degenerate simplices. A simplex in $tX$ is called a \wcnotion{thin}{thin simplex}.  

Morphisms $f:(X,tX)\to (Y,tY)$ of $\tPsh{\Delta}$, called \textit{stratified morphisms}, correspond to the data of a morphism on the underlying simplicial set such that $f(tX_n)\subset tY_n$.
\end{definition}

\begin{remark}
Given a functor $i:I\mapsto (F(i),tF(i))$ with values in stratified simplicial sets, its colimit is given by $(\colim F(i),M)$ where $M$ is the smaller stratification that includes the image of $tF(i)\to \colim F(i)$ for any $i:I$.
\end{remark}

\begin{definition}[Verity]
We can extend the join to stratified simplicial sets as follows: 
If $(X,tX)$ and $(Y,tY)$ are two stratified simplicial sets, we define $tX\star tY$ as the set of simplices of $X\star Y$ of shape $x\star y$ where either $x$ or $y$ is thin. We then define 
$$(X,tX)\star (Y,tY) := (X\star Y, tX\star tY).$$
\end{definition}

\begin{definition}
A stratified monomorphism $f:X\to Y$ is 
\begin{enumerate}
\item \textit{entire} if it is an identity on underlying simplicial sets.
\item \wcnotion{regular}{regular morphism} if for every $n\geq 1$ the following diagram is a pullback:
\[\begin{tikzcd}
	{tX_n} & {X_n} \\
	{tY_n} & {Y_n}.
	\arrow[from=2-1, to=2-2]
	\arrow[from=1-2, to=2-2]
	\arrow[from=1-1, to=1-2]
	\arrow[from=1-1, to=2-1]
	\arrow["\lrcorner"{anchor=center, pos=0.125}, draw=none, from=1-1, to=2-2]
\end{tikzcd}\]
\end{enumerate}
\end{definition}

\begin{definition}[Verity]
We define several stratified structures on $[n]$. 
\begin{enumerate}
\item \wcnotation{$[n]_t$}{((g31@$[n]_t$}. The top $n$-simplex is thin. All degeneracies are thin.
\item \wcnotation{$[n]^k$}{((g32@$[n]^k$}. All simplices that include $\{k-1,k,k+1\}\cap[n]$ are thin. All degeneracies are thin.
\item \wcnotation{$([n]^k)'$}{((g33@$([n]^k)'$}. All simplices that include $\{k-1,k,k+1\}\cap[n]$, together with the $(k-1)$-face and the $(k+1)$-face are thin. All degeneracies are thin.
\item \wcnotation{$([n]^k)''$}{((g34@$([n]^k)''$}. All simplices that include $\{k-1,k,k+1\}\cap[n]$, together with the $(k-1)$-face, the $k$-face, and the $(k+1)$-face are thin. All degeneracies are thin.
\item \wcnotation{$[3]^{eq}$}{((g35@$[3]^{eq}$}. All simplices of dimension strictly higher than $2$, together with $[0,2]$ and $[1,3]$ are thin. All degeneracies are thin.
\item \wcnotation{$[n]^\sharp$}{((g36@$[n]^\sharp$}. All simplices are thin.
\end{enumerate}
\end{definition}

\begin{definition}	

\label{defi:anodyne extension complicial}
 We define several sets of morphisms:
\begin{enumerate}
\item The \notion{complicial horn inclusions} are the regular extensions:
$$\Lambda^k[n]\to [n]^k,~n\geq 1,~ n\geq k\geq 0.$$
\item The \notion{complicial thinness extensions}:
$$([n]^k)'\to ([n]^k)'',~n\geq 2,~ n\geq k\geq 0.$$
\item The \notion{saturation extensions}:
$$[n]\star[3]^{eq}\star[m]\to [n]\star[3]^{\sharp}\star[m],~ n,m\geq -1.$$
\end{enumerate}
The union of complicial horn inclusions and complicial thinness extensions will be called \notion{complicial elementary anodyne extension}, and the further union with saturation extensions will be called \notion{saturated complicial elementary anodyne extension}.
\end{definition}

\begin{definition}[Verity]
\label{defi:complicial set}
Let $n\in \Nb\cup\{\omega\}$.	
A \wcnotion{$n$-complicial set}{complicial set} is a stratified set having the right lifting property against complicial elementary extensions and against all morphisms $[k]\to [k]_t$ for $k>n$. It is \wcnotion{saturated}{saturated complicial set} if it also has the right lifting property against saturation extensions. 

The $\omega$-complicial sets will simply be called \textit{complicial sets}.
\end{definition}

\begin{theorem}[Verity]
\label{theo:model structure on complicial set}
Let $n\in \Nb\cup\{\omega\}$.	
There exists a nice and cartesian model structure on stratified simplicial sets, denoted by \wcnotation{$\stratSset^n$}{(spshdeltasa@$\stratSset^n$} and called the \wcnotion{$n$-complicial model structure}{complicial model structure}, whose fibrant objects are $n$-complicial sets. Its weak equivalences are the smallest precocomplete class of maps containing complicial horn inclusions and complicial thinness extensions.
\end{theorem}

\begin{theorem}[Ozornova-Rovelli]
\label{theo:model structure on complicial set2}
Let $n\in \Nb\cup\{\omega\}$.	
The model structure $\stratSset^n$ admits a left Bousfield localization along saturation extensions, denoted \wcnotation{$\stratSset^n_\sat$}{(spshdeltasat@$\stratSset^n_\sat$} and called the \wcnotion{saturated $n$-complicial model structure}{saturated complicial model structure}, whose fibrants are saturated $n$-complicial sets.
\end{theorem}

\begin{proof}[Proof of theorems \ref{theo:model structure on complicial set} and \ref{theo:model structure on complicial set2}]
The first assertion is \cite[Theorem 98]{Verity_weak_complicial_sets_I} applied to the set $J_c$ given in \cite[Example 91]{Verity_weak_complicial_sets_I}, and the second one is \cite[Theorem 1.25]{Ozornova_model_structure_for_infini_n_categories}.
\end{proof}

\begin{notation}
When speaking about \textit{acyclic cofibrations} or \textit{weak equivalences} of stratified simplicial sets, we will always refer implicitly to the complicial model structure. We will speak of \textit{saturated weak equivalences} and \textit{saturated acyclic cofibrations} when we are working in the saturated complicial model structure.
\end{notation}

\begin{remark}
\label{rem:on forgetting some marking}
Let $M$ be the set of all simplices $[k]$ with $0<k\leq n$, and let's denote $\mathrm{t}_{\leq n}\Psh{\Delta}$ the category $\tPshM{\Delta}$ from constructions \ref{defi:of pshMB}. Remark that we have a canonical functor $\tPsh{\Delta}\to\mathrm{t}_{\leq n}\Psh{\Delta}$ that just forgets thin $k$-cells for $k>n$.

We can easily check that the model structure $t\Psh{\Delta}^n$ and $\tPsh{\Delta}^n_\sat$ induce two model structures denoted respectively $\mathrm{t}_{\leq n}\Psh{\Delta}^n$ and $\mathrm{t}_{\leq n}\Psh{\Delta}_\sat^n$ such that the comparison functor becomes a Quillen equivalence.
\end{remark}

In the case $n=1$, we can obtain the following easier characterization of the $1$-complicial model structure.

\begin{prop}
\label{prop:characterization of morphism depuis strratsset.}
Weak equivalences in $\stratSset^1$ are the smallest precocomplete class of morphisms that contains:
\begin{enumerate}
\item for any $n$, the morphism $\Sp_n\to [n]$,
\item the morphism $[1]_t\to [0]$,
\item for any $n>1$, the morphism $[n]_t\to [n]$.
\end{enumerate}
The model category $\stratSset^1_\sat$ is the Bousfield localization of this model structure along the morphism $E^{eq}\to [0]$.
\end{prop}

\begin{proof}
Let $W$ be the class of weak equivalences of $\stratSset^1$ and let $W'$ be the precocomplete class described in the statement. We can easely check that the functor $(\uvar)^{\flat}:\Sset\to \stratSset$ sends inner horn inclusions to weak equivalences. By proposition 3.7.4 of \cite{Cisinski_Higher_categories_and_homotopical_algebra}, it then sends spine inclusions to weak equivalences. Finally, as $[n]\to [n]_t$ for $n\geq 2$, and $[1]_t \to [0]$ are weak equivalences in $\stratSset^1$, and as $W$ is a saturated class by \ref{prop:weak equivalence are precocomplete}, we have $W'\subset W$.

Let's now show the other direction. We will first show that $W'$ is closed under the functor $\uvar\star K$ for any simplicial set $K$.

To this extent, remark that for all integers $n$, the morphism $\Sp_{[n+1]}\to \Sp_{[n]}\star [0]$ is a sequence of pushouts along $\Sp_{[2]}\to [2]$ and is therefore in $W'$. By two out of three, so is the morphism $ \Sp_{[n]}\star [0]\to [n]\star[0]\cong [n+1]$. Now remark that we have a cocartesian square
\[\begin{tikzcd}
	{[1]_t} & {[1]_t\coprod_{[0]}[1]} \\
	{\{0\}} & {[1]}
	\arrow[from=1-1, to=1-2]
	\arrow[from=1-1, to=2-1]
	\arrow[from=1-2, to=2-2]
	\arrow[from=2-1, to=2-2]
\end{tikzcd}\]
where the horizontal morphisms are monomorphisms. This implies that the right vertical morphism is in $W'$. Furthermore, this morphism factors as the composite of $[1]_t\coprod_{[0]}  [1]\to \Delta^2[2]$, which is in $W'$, and of $\Delta^2[2]\to [1]$. By  two out of three, the morphism $[1]_t\star[0]\cong \Delta^2[2]\to [0]\star[0]\cong [1]$ is then in $W'$. Eventually is clear that $[n]\star[0]\to [n]_t\star [0]$ for $n>1$ is a pushout of morphisms $[k]\to [k]_t$ for $k>1$, and so is also in $W'$. As a consequence, $W'$ is closed under the functor $\uvar\star[0]$ as this functor preserves monomorphisms. By a direction induction, $W'$ is then closed under the functor $\uvar\star[n]$ for any integer $n$. As any simplicial set $K$ is the colimit of the Reedy cofibrant diagram $\Delta_{/K}\to \Delta\to \Sset$, $W'$ is closed under $\uvar\star K$.

Now let $f:X\to Y$ be a monomorphism in $W'$. By stability under pushout of $W'$, the morphism 
$$X\star[n]\to X\star[n]\coprod_{X\star\partial [n]}Y\star[n]$$
is in $W'$.
By two out of three, so is the morphism 
$$X\star[n]\coprod_{X\star\partial [n]}Y\star \partial [n]\to Y\star [n].$$
Monomorphisms of $W'$ are then closed under the Leibniz product $\uvar~ \hat{\star}~ (\partial[n]\to [n])$. We can show similarly that they are closed under the Leibniz product $ (\partial[n]\to [n]) ~\hat{\star}~\uvar$.

Remark now that $[2]\to [2]^1$ is in $W'$, and by two out of three, this implies that $ \Lambda^{1}[2]\to [2]^1$ is in $W'$. By two out of three, we also have that $\Lambda^0[1]\to [1]^0$ and $\Lambda^1[1]\to [1]^1$ are in $W'$.

All put together, and as for any pair of integers $0<i<n$, $\Lambda^{i}[n]\to [n]$ is the Leibniz product
$$ (\partial[i-1]\to [i-1])~ \hat{\star}~(\mbox{$\Sp_2$}\to [2])~  \hat{\star}~ (\partial[n-i-1]\to [n-i-1])$$
and $\Lambda^0[n]\to[n]^0$ (resp. $\Lambda^n[n]\to[n]^n$) is the Leibniz product
$$(\Lambda^0[1]\to[1]^0)~\hat{\star}~ (\partial [n-2]\to [n-2])~~~~\mbox{(resp.}~~(\partial [n-2]\to [n-2])~\hat{\star}~(\Lambda^0[1]\to[1]^0) \mbox{)}$$
we deduce that $W'$ contains all complicial horn inclusions.

Remark now that the complicial thinness extensions for $n>3$ are pushouts along morphisms $[k]\to [k]_t$ for $k>1$, and so belong to $W'$. In the case $n=2$, remark that we have a pushout:
\[\begin{tikzcd}
	{([2]^k)'} & {[1]_t} \\
	{([2]^k)''} & {[1]_t}
	\arrow[from=1-1, to=1-2]
	\arrow[from=1-1, to=2-1]
	\arrow[from=1-2, to=2-2]
	\arrow[from=2-1, to=2-2]
\end{tikzcd}\]
where the top horizontal morphism is $s^0$ if $k\leq 1$ or $s^1$ if not. As seen earlier, the morphism $[2]_k\to [1]_t$ is in $W'$, and then so is the top horizontal morphism in the previous diagram. This is also the case for the lower horizontal one. By two out of three, this implies that $([2]^k)'\to ([2]^k)''$ is in $W'$. All complicial elementary anodyne extensions are then in $W'$, which implies that $W\subset W'$ and so these two sets are equal. 

The last assertion follows directly from the characterization of weak equivalences of $\tPsh{\Delta}^1$ and theorem \ref{theo:model structure on complicial set2}.
\end{proof}

\vspace{1cm}

\begin{definition} A \notion{marked simplicial set} is a stratified simplicial set that has the right lifting property against entire acyclic cofibrations of the model structure $\tPsh{\Delta}^\omega$. In particular, all complicial sets are marked. The category of marked simplicial sets is denoted by \wcnotation{$\mSset$}{(mpsh@$\mSset$}. There is an adjunction:
\begin{equation}
\label{adj:between marked and stratified}
\begin{tikzcd}
	{(\uvar)_{\mk}:\stratSset} & {\mSset:\iota}
	\arrow[""{name=0, anchor=center, inner sep=0}, shift left=2, from=1-1, to=1-2]
	\arrow[""{name=1, anchor=center, inner sep=0}, shift left=2, from=1-2, to=1-1]
	\arrow["\dashv"{anchor=center, rotate=-90}, draw=none, from=0, to=1]
\end{tikzcd}
\end{equation}
The left adjoint $(\uvar)_{\mk}$ sends a stratified simplicial set $(X,tX)$ to the marked simplicial set $(X,\overline{tX})$, where $\overline{tX}$ is the smaller stratification that includes $tX$ and makes $(X,\overline{tX})$ a marked simplicial set. Moreover, proposition \ref{prop:X to Xmk is acycli cof} implies that the canonical morphism $X\to \iota (X)_{\mk}$ is an entire acyclic cofibration.
\end{definition}

\begin{remark}
Given a functor $i:I\mapsto (F(i),tF(i))$ with value in marked simplicial sets, its colimit is given by $(\colim F(i),\overline{M})$ where $M$ is the smaller stratification that includes the image of $tF(i)\to \colim F(i)$ for any $i:I$.
\end{remark}

\begin{prop}
\label{prop:model structure on marked simplicial set}
The category of marked simplicial sets admits two model structures, denoted $\mSset$ and $\mSset_\sat$, that make the adjunction \ref{adj:between marked and stratified} a Quillen equivalence with respect to the complicial and saturated complicial model structures.
\end{prop}

\begin{proof}
This is a direct consequence of proposition \ref{prop:model structure on marked presheaves} and theorems \ref{theo:model structure on complicial set} and \ref{theo:model structure on complicial set2}
\end{proof}

\begin{notation}
When speaking about \textit{acyclic cofibrations} or \textit{weak equivalences} of marked simplicial sets, we will always refer implicitly to the complicial model structure. We will speak of \textit{saturated weak equivalences} and \textit{saturated acyclic cofibrations} when we are working in the saturated complicial model structure.
\end{notation}

\begin{construction}
\label{cons:intelligent truncation for simplicial set}
Let $n$ be an integer, and $(X,tX)$ a stratified simplicial set. We define $\tau^i_n(tX)$ as the union of $tX$ and all simplices of dimension strictly greater than $n$. This induces a functor, called the \snotionsym{intelligent $n$-truncation}{(taui@$\tau^i_n$}{for stratified simplicial sets}:
$$\begin{array}{rcll}
\tau^i_n :& \tPsh{\Delta}&\mapsto &\tPsh{\Delta}\\
 &(X,tX)&\mapsto &(X, {\tau^i_n(tX)}).
\end{array}$$
This functor preserves colimits and cofibrations. Moreover, we can easily check that it sends (resp. saturated) complicial elementary extensions to (resp. saturated) weak equivalences, and this functor is then left Quillen for the complicial and saturated complicial model structure. Its associated right adjoint is called the \wcnotion{$n$-truncation}{truncation@$n$-truncation} and is denoted by 
$$\tau_n:\mSset\to \mSset.$$

We will also denote $\tau^i_n,\tau_n:\mSset\to \mSset$, and call the \snotionsym{intelligent $n$-truncation}{(taui@$\tau^i_n$}{for marked simplicial sets}
and the \textit{$n$-truncation}, the two induced functors $(\tau_n^i(\iota(\uvar))_{\mk}$ and $(\tau_n(\iota(\uvar))_{\mk}$. \sym{(tau@$\diamond$}
\end{construction}

\subsection{Gray operations on marked simplicial sets}

From now one, we will always assume  that $\mSset$ is endowed with the complcial model structure from proposition \ref{prop:model structure on marked simplicial set}. We will speak of \textit{saturated weak equivalence} and \textit{saturated acyclic cofibrations} when we are working in the saturated complicial model structure.

\begin{construction}[Verity] For any $n,p,q\geq 0$ such that $n=p+q$, we define:
\begin{itemize}
\item the \notion{degeneration partition operator}:
$$
\begin{array}{rclllrrclll}
\invamalg^1_{p,q}:&[n]&\to&[p]&&~~~~~~&\invamalg^2_{p,q}:&[n]&\to&[q]&\\
&k&\mapsto &k &\mbox{if $k\leq p$} &&&k&\mapsto &0& \mbox{if $k\leq p$}\\
&k&\mapsto &p 	&\mbox{if $k>p$} &&&k&\mapsto &k-p& \mbox{if $k> p$}.
\end{array}
$$
\item the \notion{face partition operator}:
$$
\begin{array}{rcllrrcll}
\amalg^1_{p,q}:&[p]&\to&[n]&~~~~~~&\amalg^2_{p,q}:&[q]&\to&[n]\\
&k&\mapsto &k &&&k&\mapsto &k+p.
\end{array}
$$
\end{itemize}
\end{construction}

\begin{definition}[Verity]
Let $(X,tX)$ and $(Y,tY)$ be two stratified simplicial sets. 
We define the \snotionsym{Gray tensor product}{((d00@$\otimes$}{for stratified simplicial sets} of $(X,tX)$ and $(Y,tY)$ as the stratified simplicial set 
$$(X,tX)\otimes (Y,tY):=(X\times Y,tX\otimes tY)$$ where $tX\otimes tY$ is the set of pairs $(x,y)$ such that for any partitions $(p,q)$ of $n$ either $\amalg^1_{p,q}x$ or $\amalg^2_{p,q}y$ is thin. 
\end{definition}
\begin{remark}
Let $X,Y$ be two stratified simplicial sets such that all simplices of $X$ are thin. The morphism 
$X\otimes Y\to X\times Y$ is then an isomorphism.
\end{remark}

\begin{prop}
\label{prop:otimes et op simplocoal sets}
There is a canonical isomorphism 
$$(X\otimes Y)^{op}\cong Y^{op}\otimes X^{op}$$
natural in $X$ and $Y$.
\end{prop}
\begin{proof}
At the level of simplicial sets, this two objects are obviously isomorphic in a unique way. It is sufficient to check that the unique isomorphism preserves the marking, which is left to the reader.
\end{proof}

\begin{remark}
In \cite{Verity_weak_complicial_sets_I}, it is shown that the Gray tensor is associative. The problem of this operation comes from the fact that it doesn't commute with colimits. Verity then defines an other binary operation, which is cocontinuous, the \textit{Gray pretensor} (\cite[definition 135]{Verity_weak_complicial_sets_I}) $(X,tX)\boxtimes(Y,tY):=(X\times Y, tX\boxtimes tY)$, together with a natural transformation: 
$$\uvar\boxtimes\uvar\to \uvar\otimes\uvar$$
that is pointwise an entire acyclic cofibration (\cite[lemma 149]{Verity_complicial_set_characterising_the_simplicial_nerve}). Moreover, in \cite{Verity_weak_complicial_sets_I} and  \cite{Ozornova_Gray_tensor_product_and_saturated_n_complicia}, it is shown that this pretensor is a Quillen bifunctor for the model structures  $\stratSset$ and  $\stratSset_\sat$. 
\end{remark}

\begin{definition}
Let $X$ and $Y$ be two marked simplicial sets. We define the \snotionsym{Gray tensor product}{((d00@$\otimes$}{for marked simplicial sets} of $X$ and $Y$ as the marked simplicial set 
$$X\otimes Y:= (\iota(X)\otimes \iota(Y))_{\mk}$$
 where $((\uvar)_{\mk},\iota)$ is the adjunction \ref{adj:between marked and stratified}.
As $\uvar\boxtimes\uvar\to \uvar\otimes\uvar$ is pointwise a entire acyclic cofibration, we have an equality: 
$$X\otimes Y:= (\iota(X)\boxtimes \iota(Y))_{\mk}.$$
\end{definition}

\begin{prop}
\label{prop:R_commutes_with_gray_tensor}
We have equalities
$$(\uvar\boxtimes \uvar)_{\mk}=(\uvar\otimes \uvar)_{\mk}= (\uvar)_{\mk}\otimes (\uvar)_{\mk}.$$
\end{prop}
\begin{proof}
The first equality is a consequence of the fact that $\uvar\boxtimes\uvar\to \uvar\otimes\uvar$ is pointwise a entire acyclic cofibration.

 For the second one, we have to show that $(X\otimes Y)_{\mk}=(\iota(X_{\mk})\otimes \iota(Y_{\mk}))_{\mk}$.
The unit of the adjunction $(\iota,(\uvar)_{\mk})$ induces a morphism $h:(X\otimes Y)_{\mk}\to (\iota(X_{\mk})\otimes \iota(Y_{\mk}))_{\mk}$. This morphism is an entire acyclic cofibration according to proposition \ref{prop:X to Xmk is acycli cof}, 
and the corollary 2.2 of \cite{Ozornova_Gray_tensor_product_and_saturated_n_complicia} and the fact that $(\uvar)_{\mk}$ is a left Quillen functor.

 We then have lifts in the following diagram:
\[\begin{tikzcd}
	{(X\otimes Y)_{\mk}} & {(X\otimes Y)_{\mk}} \\
	{(\iota(X_{\mk})\otimes \iota(Y_{\mk}))_{\mk}}
	\arrow["id", from=1-1, to=1-2]
	\arrow["h"', from=1-1, to=2-1]
	\arrow["k"', from=2-1, to=1-2]
\end{tikzcd}\]
As both $k$ and $h$ are the identity on the underlying simplicial sets, this implies that the stratifications of $(X\otimes Y)_{\mk}$ and $(X\otimes Y)_{\mk}$ coincide, and this two objects are then equal. 
\end{proof}

We can then deduce the following proposition:
\begin{prop}
\label{prop:gray_product_is_a_left_Quillen_bifunctor}
The Gray tensor product is associative, and is a left Quillen bifunctor on $\mSset$ and $\mSset_\sat$.
\end{prop}
\begin{proof}
The first assertion is a consequence of proposition \ref{prop:R_commutes_with_gray_tensor} and the fact that the binary operation $\otimes$ on $\stratSset$ is associative. The second one is a consequence of proposition \ref{prop:R_commutes_with_gray_tensor} and \cite[Theorem 2.1]{Ozornova_Gray_tensor_product_and_saturated_n_complicia}.
\end{proof}

\begin{construction}
\label{cons:suspension functor of simplciial set}
Let $X$ be a marked simplicial set. We define the \snotion{suspension}{for marked simplicial sets} of $X$, noted by \wcnotation{$\Sigma X$}{(sigma@$\Sigma\uvar$}, as the following pushout:
\[\begin{tikzcd}
	{X\otimes\partial [1]} & {X\otimes [1]} \\
	{\partial[1]} & {\Sigma X}
	\arrow[from=1-2, to=2-2]
	\arrow["\lrcorner"{anchor=center, pos=0.125, rotate=180}, draw=none, from=2-2, to=1-1]
	\arrow[from=1-1, to=2-1]
	\arrow[from=2-1, to=2-2]
	\arrow[from=1-1, to=1-2]
\end{tikzcd}\]
This assignation defines a cocontinuous functor $\Sigma:\mSset\to \mSset_{\partial[1]/}.$ For every acyclic cofibration $K\to L$, we have cartesian squares
\[\begin{tikzcd}
	{L\otimes\partial[1]} & {K\otimes[1]\cup L\otimes\partial[1]} & {L\otimes[1]} \\
	{\partial[1]} & {\Sigma K} & {\Sigma L}
	\arrow[from=1-1, to=2-1]
	\arrow[""{name=0, anchor=center, inner sep=0}, from=1-1, to=1-2]
	\arrow[from=2-1, to=2-2]
	\arrow[from=1-2, to=2-2]
	\arrow[from=1-3, to=2-3]
	\arrow[from=2-2, to=2-3]
	\arrow[""{name=1, anchor=center, inner sep=0}, from=1-2, to=1-3]
	\arrow["\lrcorner"{anchor=center, pos=0.125, rotate=180}, draw=none, from=2-2, to=0]
	\arrow["\lrcorner"{anchor=center, pos=0.125, rotate=180}, draw=none, from=2-3, to=1]
\end{tikzcd}\]
The suspension then preserves ((saturated) acyclic cofibration and is then a left Quillen functor for the complicial and saturated complicial model structure.

This functor admits a right adjoint, that sends a pair $(a,b,C)$ to \wcnotation{$C(a,b)$}{(cab@$C(a,b)$} where $a,b$ are two $0$-simplices of $C$. If $p:C\to D$ is a morphism between complicial sets, and $a,b$ two $0$-simplices of $C$, we denote by 
$$p(a,b):C(a,b)\to D(pa,pb)$$
the induced morphism.

\end{construction}

\begin{construction}
We introduce an other operation, the \notion{diamond product}, that makes the link between the Gray tensor product and the join. 
Let $X$ and $Y$ be two marked simplicial sets. We define \sym{((d21@$\diamond$}$X\diamond Y$ as the colimit of the diagram:
\[\begin{tikzcd}
	X & {X\otimes \{0\}\otimes Y} & {X\otimes[1]\otimes Y} & {X\otimes \{1\}\otimes Y} & Y
	\arrow[from=1-4, to=1-3]
	\arrow[from=1-4, to=1-5]
	\arrow[from=1-2, to=1-1]
	\arrow[from=1-2, to=1-3]
\end{tikzcd}\]
The functors 
$$\uvar\diamond X:\mSset\to \mSset_{/X} ~~~~\mbox{and}~~~~ X\diamond \uvar:\mSset\to \mSset_{/X}$$
are colimit preserving. Furthermore, for (saturated) weak equivalence $K\to L$, the morphism $K\diamond X\to L\diamond X$ is the horizontal colimit of the diagram: 
\[\begin{tikzcd}
	{K\amalg X} & {K\otimes \partial[1]\otimes X} & {K\otimes [1]\otimes X} \\
	{L\amalg X} & {L\otimes \partial[1]\otimes X} & {L\otimes [1] \otimes X}
	\arrow[from=1-2, to=1-1]
	\arrow[from=1-2, to=1-3]
	\arrow[from=2-2, to=2-3]
	\arrow[from=2-2, to=2-1]
	\arrow[from=1-2, to=2-2]
	\arrow[from=1-1, to=2-1]
	\arrow[from=1-3, to=2-3]
\end{tikzcd}\]
However, these two horizontal colimits are homotopy colimits, and all the horizontal maps of the previous diagram are (saturated) weak equivalences. This morphism is then (saturated) weak equivalence. This shows that 
 $\uvar\diamond X$ is a left Quillen functor for the complicial and saturated complicial model structure. We show analogously that $X\diamond \uvar$  has the same property.
\end{construction}

\begin{prop}
\label{prop:diamond et op simplocoal sets}
There is a canonical isomorphism 
$$(X\diamond Y)^{op}\cong Y^{op}\diamond X^{op}$$
natural in $X$ and $Y$.
\end{prop}
\begin{proof}
This directly follows from proposition \ref{prop:otimes et op simplocoal sets}.
\end{proof}

\begin{lemma}
There exists a unique natural transformation $\gamma_{X,Y}:X\diamond Y\to X\star Y$ that fits in the following diagram: 
\[\begin{tikzcd}
	{X\coprod Y} & {X\star Y} \\
	{X\diamond Y} & {[1]}
	\arrow[from=1-1, to=2-1]
	\arrow[from=1-1, to=1-2]
	\arrow[from=1-2, to=2-2]
	\arrow[from=2-1, to=2-2]
	\arrow["{\gamma_{X,Y}}", from=2-1, to=1-2]
\end{tikzcd}\]
\end{lemma}
\begin{proof}
We begin by defining this morphism on simplicial sets, and for this we can suppose that both $X$ and $Y$ are representables, ie $X:=[n]$, $Y:=[m]$.
On object, this morphism is induced by the assignation:
$$p(k,0,l) := k~~~p(k,1,l) := l.$$ 

We need to verify that this morphism preserves thin cells. Suppose now that $(x,v,y)$ is a thin $n$-simplex of $X\diamond Y$. There are several cases to consider. \textbf{Case $v_n=0$.} The simplex $x$ is then thin, and is sent to $x\star \emptyset$ which is also thin. \textbf{Case $v_0=1$.} Similar. \textbf{Case $v_0=0$ and $v_n=1$.} Let $p$ be the smaller integer such that $v_p=1$. Either $\amalg_{p-1,n-p+1}^1(x)$ or $\amalg_{p,n-p}^2(y)$ is thin. This implies that $\phi_{X,Y}(x,v,y)= \amalg_{p-1,n-p+1}^1(x)\star \amalg_{p,n-p}^2(y)$ is thin. 
\end{proof}

\begin{prop}
\label{prop:equivalence between diamond and join product}
For any marked simplicial sets $X,Y$, the morphism $\gamma_{X,Y}$ is a weak equivalence. 
\end{prop}
\begin{proof}
The functor 
$$t\Delta_{/X}\times t\Delta_{/Y}\to \mSset\times\mSset\xrightarrow{\gamma} \Arr(\mSset)$$
is Reedy cofibrant (definition \ref{defi:reedycof}). It is then enough to show the result for any couples of representables. 

Let's start by the case $(X,Y)=([n],[m])$. Let $s:X\star Y\to X\diamond Y$ be the morphism defined on objects by the formula: 
$$s(k\star \emptyset) := (k,0,0)~~~s(\emptyset \star l) := (n,1,l)$$
We have
$$\gamma_{X,Y}s = id ~~~~s\gamma_{X,Y} (k,\epsilon,l) =(k + \epsilon (n-k), \epsilon,\epsilon l).$$

Let $\eta:[n]\diamond [m]\to [n]\diamond [m]$ be induced by the application
$$(k,\epsilon,l)\mapsto (k,\epsilon,\epsilon l).$$
We are now going to construct two morphisms
$$\epsilon_0: ([n]\diamond[m])\times [1]_t\to [n]\diamond[m]~~~~\mbox{ and }~~~~\epsilon_1: ([n]\diamond[m])\times [1]_t\to [n]\diamond[m]$$
such that $$
\begin{array}{rrl}
\epsilon_0(\uvar,0)=\eta&&\epsilon_0(\uvar,1)=s\gamma_{X,Y}\\
\epsilon_1(\uvar,0)=\eta&~~~~&\epsilon_1(\uvar,1)=id\\
\end{array}$$
The first one is induced on the level of simplicial sets by
$$(k,\epsilon,l,\alpha)\mapsto (k + \alpha\epsilon (n-k),\epsilon,\epsilon l ),$$
and the second one by
$$(k,\epsilon,l,\alpha)\mapsto (k,\epsilon,(\epsilon\vee\alpha)l),$$
where $\epsilon\vee\alpha := \epsilon+\alpha - \epsilon\alpha.$
These two morphisms extend to marked simplicial sets. 

We proceed in a similar way with cases $(X,Y) = ([n]_t,[m]), ([n],[m]_t)$ or $([n]_t,[m]_t)$. 
\end{proof}

\begin{remark}
As we already now that functors $\uvar\diamond X$ and $X\diamond \uvar$ preserve (saturated) weak equivalences, the previous proposition implies that for any marked simplicial sets $X$, functors $\uvar\star X$ and $X\star \uvar$ preserves (saturated) weak equivalences and are then
 left Quillen functors for the complicial and saturated complicial model structure.
\end{remark}

\begin{construction}
\label{cons:sigma star}
Let $X$ be a marked simplicial set. We now describe an variation on the suspension. We define \wcnotation{$\Sigma^\star X$}{(sigmastar@$\Sigma^\star\uvar$}, as the following pushout:
\[\begin{tikzcd}
	X & {X\star [0]} \\
	1 & {\Sigma^\star X}
	\arrow[from=1-2, to=2-2]
	\arrow["\lrcorner"{anchor=center, pos=0.125, rotate=180}, draw=none, from=2-2, to=1-1]
	\arrow[from=1-1, to=2-1]
	\arrow[from=1-1, to=1-2]
	\arrow[from=2-1, to=2-2]
\end{tikzcd}\]
This assignation defines a cocontinuous functor $\Sigma^\star:\mSset\to \mSset_{\partial[1]/}.$ Using proposition \ref{prop:equivalence between diamond and join product}, all the vertical morphisms of the following diagram are weak equivalences:
\[\begin{tikzcd}
	1 & X & {X\diamond 1} \\
	1 & X & X\star1
	\arrow[from=1-2, to=1-3]
	\arrow[from=1-2, to=1-1]
	\arrow[from=2-2, to=2-1]
	\arrow[from=1-3, to=2-3]
	\arrow[from=2-2, to=2-3]
	\arrow[from=1-2, to=2-2]
	\arrow[from=1-1, to=2-1]
\end{tikzcd}\]
Remark furthermore that the colimits of these lines are also homotopy colimits. Taking the horizontal colimit, this induces a weak equivalence
\begin{equation}
\label{eq:sigma et sigam star}
\Sigma X\to \Sigma^{\star}X
\end{equation}
natural in $X$, where $\Sigma$ is the functor construted in \ref{cons:suspension functor of simplciial set}.
\end{construction}

\begin{construction}
We define the \notion{co-join} of $X$ and $Y$, denoted by \index[notation]{((d22@$\overset{co}{\star}$}$X\costar Y$, as the colimit of the following diagram:
\[\begin{tikzcd}
	Y & {Y\otimes \{1\}\otimes X} & {Y\otimes [1]\otimes X} & {Y\otimes \{0\}\otimes X} & X
	\arrow[from=1-2, to=1-1]
	\arrow[from=1-2, to=1-3]
	\arrow[from=1-4, to=1-5]
	\arrow[from=1-4, to=1-3]
\end{tikzcd}\]
The functors 
$$\uvar\costar X:\mSset\to \mSset_{/X} ~~\mbox{and}~~ X\costar \uvar:\mSset\to \mSset_{/X}$$
are colimit preserving. Furthermore, for (saturated) weak equivalence $K\to L$, the morphism $K\costar X\to L\costar X$ is the horizontal colimit of the diagram:
\[\begin{tikzcd}
	{K\amalg X} & {X\otimes \partial[1]\otimes K} & {X\otimes [1]\otimes K} \\
	{L\amalg X} & {X\otimes \partial[1]\otimes L} & {X\otimes [1] \otimes K}
	\arrow[from=1-2, to=1-1]
	\arrow[from=1-2, to=1-3]
	\arrow[from=2-2, to=2-3]
	\arrow[from=2-2, to=2-1]
	\arrow[from=1-2, to=2-2]
	\arrow[from=1-3, to=2-3]
	\arrow[from=1-1, to=2-1]
\end{tikzcd}\]
However, all the horizontal maps of the previous diagram are (saturated) weak equivalences. This morphism is then also a (saturated) weak equvivalence
This shows that $\uvar\costar X$ is a left Quillen functor for the complicial and saturated complicial model structure. We show analogously that $X\costar \uvar$ enjows the same property.
\end{construction}
\begin{construction}
\label{cons:wedge} Let $X$ be a simplicial set. We define the \notion{wedge} of $\Sigma X$ and $[1]$, noted by \sym{(sigmavee@$\Sigma X~\rotatebox[origin=c]{270}{$\gtrdot$}~[1]$}\sym{(sigmave@$[1]~\rotatebox[origin=c]{270}{$\gtrdot$}\Sigma X$}$\Sigma X\fwedge [1]$, as the colimit of the following diagram:
\[\begin{tikzcd}
	{X\otimes[0,1]} & {X\otimes[2]_t} & {X\otimes[1,2]} \\
	{\Sigma X} & {X\fwedge[1]} & {[1,2]}
	\arrow[from=1-1, to=2-1]
	\arrow[from=1-3, to=2-3]
	\arrow[from=1-1, to=1-2]
	\arrow[from=1-3, to=1-2]
	\arrow[from=1-2, to=2-2]
	\arrow[from=2-3, to=2-2]
	\arrow[from=2-1, to=2-2]
\end{tikzcd}\]
This assignation defines a cocontinuous functor $\uvar\fwedge [1]:\mSset\to \mSset_{[0]\amalg [1]/}.$ For every  weak equivalence $K\to L$, the morphism $K\fwedge [1]\to L\fwedge [1]$ is the horizontal colimit of the diagram:
\[\begin{tikzcd}
	{[0]\coprod[1]} & {K\otimes([0]\coprod[1,2])} & {K\otimes[2]_t} \\
	{K\otimes[2]_t} & {L\otimes[2]_t} & {L\otimes[2]_t}
	\arrow[from=1-2, to=1-1]
	\arrow[from=1-2, to=1-3]
	\arrow[from=1-1, to=2-1]
	\arrow[from=2-2, to=2-1]
	\arrow[from=2-2, to=2-3]
	\arrow[from=1-3, to=2-3]
	\arrow[from=1-2, to=2-2]
\end{tikzcd}\]
However, these two horizontal colimits are homotopy colimits, and all the horizontal maps of the previous diagram are (saturated) weak equivalences. This morphism is then also a (saturated) weak equivalence.
This shows that this functor is a left Quillen functor for the complicial and saturated complicial model structure. We denote by $$\triangledown:\Sigma X\to \Sigma X\fwedge [1]$$ the morphism induced by the inclusion $X\otimes [0,2]\subset X\otimes [2]_t$ and 
$$\Sigma X\hookrightarrow \Sigma X\fwedge [1]$$
the morphism induced by the inclusion $X\otimes [1,2]\subset X\otimes [2]_t$.
We define similarly the left Quillen functor $$[1]\fwedge\uvar:\mSset\to \mSset_{[1]\amalg [0]/}$$ and the morphisms
$$\triangledown:\Sigma X\to [1]\fwedge\Sigma X~~~\mbox{and}~~~\Sigma X\hookrightarrow [1]\fwedge\Sigma X .$$
\end{construction}
\begin{prop}
Morphisms 
$$ \Sigma X\coprod_{[0]}[1]\to \Sigma X\fwedge [1]~~~~\mbox{and}~~~~ [1]\coprod_{[0]}\Sigma X\to [1]\fwedge \Sigma X$$
are acyclic cofibrations. 
\end{prop}
\begin{proof}
We have cartesian squares:
\[\begin{tikzcd}
	{X\otimes([0]\coprod[1,2])} & {X\otimes \Lambda^{1}[2]} & {X\otimes[2]_t} \\
	{[0]\coprod[1]} & {\Sigma X\coprod_{[0]} [1]} & {\Sigma X\fwedge [1].}
	\arrow[from=1-1, to=2-1]
	\arrow[from=2-1, to=2-2]
	\arrow[from=1-1, to=1-2]
	\arrow[from=1-2, to=1-3]
	\arrow[from=1-3, to=2-3]
	\arrow[from=1-2, to=2-2]
	\arrow["\lrcorner"{anchor=center, pos=0.125, rotate=180}, draw=none, from=2-3, to=1-2]
	\arrow["\lrcorner"{anchor=center, pos=0.125, rotate=180}, draw=none, from=2-2, to=1-1]
	\arrow[from=2-2, to=2-3]
\end{tikzcd}\]
The upper right horizontal morphism is an acyclic cofibration, and so is the downer right horizontal one. We proceed similarly for the other morphism.
\end{proof}

\begin{definition}
The Gray tensor product induces a functor
$$\uvar\otimes[1]:\mSset\to \mSset$$
called the \snotionsym{Gray cylinder}{((d30@$\uvar\otimes[1]$}{for marked simplicial sets}, that is left Quillen for the  the complicial and saturated complicial model structure.
The join and the co-join also incuce two left Quillen functors for  the complicial and saturated complicial model structure:
$$\uvar\star [0]:\mSset\to \mSset_{[0]/}~~~~~[0]\costar \uvar:\mSset\to \mSset_{[0]/}$$
called the \snotionsym{Gray cone}{((d40@$\uvar\star 1$}{for marked simplicial sets} and the \snotion{Gray $\circ$-cone}{for marked simplicial sets}\index[notation]{((d50@$1\overset{co}{\star}\_$!\textit{for marked simplicial sets}}. We denote by 
$$
\begin{array}{rclcrcl}
 \mSset_{\cdot} &\to &\mSset & & \mSset_{\cdot}&\to & \mSset\\
(X,x)&\mapsto & X_{/x} &~~~~~ & (X,x)&\mapsto & X_{x/}\\
\end{array}
$$
respectively called the \wcnotionsym{slice of $X$ over $x$}{(cc@$C_{c/}$}{slice over} and the \wcnotionsym{slice of $X$ under $x$}{(cc@$C_{/c}$}{slice under}, the right adjoints of the Gray cone and the Gray $\circ$-cone.

Remark furthermore that we have canonical natural transformation $X_{x/}\to X$ and $X_{/x}\to X$, induced by the natural transformation $X\to X\star [0]$ and $X\to [0]\costar X$.
\end{definition}

\subsection{Street nerve}

We recall that $\omega$-categories are defined in section \ref{section:zocategories}. The Gray operations on $\omega$-categories - 
$\uvar\otimes[1]$, $\uvar\star 1$, $1\costar \uvar$ -
are defined in section \ref{section:definition of Gray operations}.

\begin{construction}
\label{cons:Street nerve}
In \cite{Street_algebra_of_orianted_simplexes}, Street defines a cosimplicial object in $\ocat$, that associates to $n$, the $n^{th}$ \notion{oriental} $\O_n$. 
The original construction of this object is complicated, but Ara and Maltsiniotis have shown that it can be easily defined using Gray operations. Indeed, in \cite[Corollaire 7.10]{Ara_Maltsiniotis_joint_et_tranche}, these authors construct an isomorphism
$$\O_n\cong \overbrace{1\star...\star 1}^{n+1}$$
natural in $n$.

We can extend the functor $\O_{\uvar}:\Delta\to \ocat$ to $t\Delta$ by defining
$$(\O_n)_t:=\tau^i_{n-1}(\O_n),$$
where $\tau^i_{n-1}$ denote the intelligent truncation defined in construction \ref{cons:intelligent truncation for simplicial set}.

By extention by colimit, this induces a functor 
$$\R:\stratSset\to \ocat.$$
As explained in example 11 of \cite{Verity_weak_complicial_set_part2_nerve_of_complicial_Gray_categories}, $\R$ preserves the Gray tensor product, and so also the suspension, the wedge, the Gray cone and the Gray $\circ$-cone. 
 Moreover, \cite[Theorem 249]{Verity_complicial_set} states that this functor sends complicial horn inclusions and complicial thinness extensions to isomorphisms. This functor then sends every weak equivalences to isomorphisms, and then lifts to a colimit preserving functor $\R:\mSset\to \ocat$ and induces an adjoint pair: \sym{(r@$\R:\mSset\to \ocat$}\sym{(n@$\N:\ocat\to \mSset$}
\[\begin{tikzcd}
	{\R:\mSset} & {\ocat:\N}
	\arrow[""{name=0, anchor=center, inner sep=0}, shift left=2, from=1-1, to=1-2]
	\arrow[""{name=1, anchor=center, inner sep=0}, shift left=2, from=1-2, to=1-1]
	\arrow["\dashv"{anchor=center, rotate=-90}, draw=none, from=0, to=1]
\end{tikzcd}\]
\end{construction}

We now recall two fundamental results of strictification:
\begin{theorem}[Gagna--Ozornova--Rovelli, Maheara]
\label{theo:strict representable}
Let $n$ be an integer. The canonical morphism
$$[n]\to \N(\R([n]))$$
is an acyclic cofibration.
\end{theorem}
\begin{proof}
This is \cite[corollary 5.4]{Gagna_Nerves_and_cones_of_free_loop_free_omega_categories}.
\end{proof}
\begin{theorem}[Ozornova, Rovelli]
\label{theo:strict susension}
Let $C$ be an $\omega$-category.
The canonical morphism
$$\Sigma \N C \to \N([C,1])$$
is an acyclic cofibration.
\end{theorem}
\begin{proof}
The morphism \eqref{eq:sigma et sigam star} provides a weak equivalence
$\Sigma \N C\to \Sigma^{\star} \N C$.
As $R$ preserves the Gray tensor product and the Gray cone, it sends this morphism to an isomorphism. We then have a commutative triangle
\[\begin{tikzcd}
	& {\Sigma^{\star} \N C} \\
	{\Sigma \N C} && {\N([C,1])}
	\arrow[from=2-1, to=2-3]
	\arrow["\sim", from=2-1, to=1-2]
	\arrow[from=1-2, to=2-3]
\end{tikzcd}\]
The theorem 3.22 of \cite{Ozornova_a_quillen_adjunction_between_globular_and_complicial} stipulates that $\Sigma^{\star} \N C\to \N([C,1])$ is a weak equivalence, which concludes the proof.
\end{proof}

\begin{definition}
We define the \notion{Street endofunctor} \wcnotation{$i_{str}$}{(istr@$i_{str}$} to be the colimit preserving functor defined on representables by: 
$$i_{str}([n]) := \N(\R([n]))~~~\mbox{ and }~~~i_{str}([n]_t) :=\tau^i_{n-1} (i_{str}([n]))$$
\end{definition}

\begin{prop}
\label{prop:i_str_is_Quillen}
 The functor $i_{srt}$ is left Quillen and 
the natural transformation 
$$id \to i_{srt}$$ 
is weakly invertible.
\end{prop}
\begin{proof}
As noticed earlier, for any integer $n$, the map $[n]\to i_{srt}([n])$ is a weak equivalence.
We recall that the intelligent truncation functor $\tau^i_{n-1}:\mSset\to \mSset$ is a left Quillen functor, and so preserves weak equivalences between cofibrant objects. The morphism $[n]_t\to i_{str}([n]_t)$ is then a weak equivalence.
The set of objects $X$ such that the morphism $X\to i_{srt}X$ is a weak equivalence is closed by homotopy colimits and includes all representables. As $i_{srt}$ preserves monomorphisms, it then consists of all marked simplicial sets. Now let $K\to L$ be an acyclic cofibration. We have a commutative square:
\[\begin{tikzcd}
	K & {i_{str}(K)} \\
	L & {i_{str}(L)}
	\arrow["\sim"', from=1-1, to=2-1]
	\arrow["\sim", from=1-1, to=1-2]
	\arrow["\sim"', from=2-1, to=2-2]
	\arrow[from=1-2, to=2-2]
\end{tikzcd}\]
By two out of three, $i_{str}(K)\to i_{str}(L)$ is then an acyclic cofibration. The functor $i_{srt}$ is then left Quillen. 
\end{proof}

\section{Suspension and Gray operations}
\label{section:Suspension and Gray operation}
\subsection{Formula for the Gray cylinder}
The aim of this subsection is to demonstrate the following theorem, which is the analogue in marked simplicial sets of the theorem \ref{theo:appendice formula for otimes}.
\begin{theorem}
\label{theo:interval_first_formula}
There is a zigzag of acyclic cofibrations, natural in $X$, between the colimit of the diagram
$$[1]\fwedge\Sigma X\xleftarrow{\triangledown} \Sigma (X\otimes\{0\})\hookrightarrow \Sigma(X\otimes[1])\hookleftarrow \Sigma (X\otimes\{1\})\xrightarrow{\triangledown} \Sigma X\fwedge[1]$$
and $(\Sigma X)\otimes [1]$.
\end{theorem}

\begin{construction}

Let $C$ be the following colimit:
\[\begin{tikzcd}
	{[3]\times\{0\}\coprod [3]\times\{1\}} & {[3]\times[1]} \\
	{[1]\coprod[1]} & {C.}
	\arrow["{s^0s^0\coprod s^2s^3}"', from=1-1, to=2-1]
	\arrow[""{name=0, anchor=center, inner sep=0}, from=1-1, to=1-2]
	\arrow[from=1-2, to=2-2]
	\arrow[from=2-1, to=2-2]
	\arrow["\lrcorner"{anchor=center, pos=0.125, rotate=180}, draw=none, from=2-2, to=0]
\end{tikzcd}\]

We define several marked simplicial sets whose underlying simplicial sets are sub objects of C: 
\[\begin{tikzcd}
	{} & 00 & 01 && 00 & 01 \\
	& 10 & 11 && 20 & 21 \\
	& 20 & 21 && 00 & 01 \\
	& 30 & 31 & {} & 30 & 31
	\arrow[from=1-3, to=2-3]
	\arrow["{\large{A_1:=~~~~~~}}"', Rightarrow, no head, from=2-2, to=3-2]
	\arrow[Rightarrow, no head, from=2-3, to=3-3]
	\arrow["{\large{A_0:=~~~~~~}}"', Rightarrow, no head, from=1-2, to=2-2]
	\arrow[from=1-2, to=1-3]
	\arrow[from=3-2, to=3-3]
	\arrow[from=2-2, to=2-3]
	\arrow[""{name=0, anchor=center, inner sep=0}, from=1-2, to=2-3]
	\arrow[""{name=1, anchor=center, inner sep=0}, from=2-2, to=3-3]
	\arrow["{\large{A_3:=~~~~~~}}"', Rightarrow, no head, from=1-5, to=2-5]
	\arrow[from=1-6, to=2-6]
	\arrow[from=2-5, to=2-6]
	\arrow[""{name=2, anchor=center, inner sep=0}, from=1-5, to=2-6]
	\arrow[from=4-2, to=4-3]
	\arrow["{\large{A_4:=~~~~~~}}"', from=3-5, to=4-5]
	\arrow[from=4-5, to=4-6]
	\arrow[from=3-6, to=4-6]
	\arrow[from=3-5, to=3-6]
	\arrow[""{name=3, anchor=center, inner sep=0}, from=3-5, to=4-6]
	\arrow[from=1-5, to=1-6]
	\arrow["{\large{A_2:=~~~~~~}}"', from=3-2, to=4-2]
	\arrow[""{name=4, anchor=center, inner sep=0}, from=3-2, to=4-3]
	\arrow[Rightarrow, no head, from=3-3, to=4-3]
	\arrow["\sim"{description}, Rightarrow, draw=none, from=0, to=1-3]
	\arrow["\sim"{description}, Rightarrow, draw=none, from=1, to=2-3]
	\arrow["\sim"{description}, Rightarrow, draw=none, from=0, to=2-2]
	\arrow["\sim"{description}, Rightarrow, draw=none, from=2, to=1-6]
	\arrow[shorten <=2pt, Rightarrow, from=2, to=2-5]
	\arrow["\sim"{description}, Rightarrow, draw=none, from=3, to=3-6]
	\arrow[shorten <=2pt, Rightarrow, from=3, to=4-5]
	\arrow[shorten <=2pt, Rightarrow, from=1, to=3-2]
	\arrow["\sim"{description}, Rightarrow, draw=none, from=4, to=4-2]
	\arrow["\sim"{description}, Rightarrow, draw=none, from=4, to=3-3]
\end{tikzcd}\]
where arrows labeled by $=$ are degenerate and simplicies labeled by $\sim$ are thin.

Let $B_0$ be the sub object corresponding to the image of $[0,1,2]\times[0,1]$ where the marking includes all cells of dimension $\leq 2$, except $[10,20,21]$ and $[00,20,21]$.

Let $B_1$ be the sub object corresponding to the image of $[0,2,3]\times[0,1]$ where the marking includes all cells of dimension $\leq 2$, except $[00,20,21]$, $[00,30,31]$ and $[00,20,31]$.

Let $B$ be the reunion of $[0,1,2]\times[0,1]$ and $[0,2,3]\times[0,1]$ where the marking is the reunion of $B_0$ and $B_1$.
\end{construction}

\begin{lemma}
Morphisms $A_0\cup A_1\to B_0$ and $A_3\to B_0$ are acyclic cofibrations. 
\end{lemma}
\begin{proof}
The cofibration $A_0\cup A_1\to B_0$ fits in the following pushout square:
\[\begin{tikzcd}
	{\Lambda^{1}[2]\otimes [1]\cup[2]_t\otimes \partial[1]} & {A_1\cup A_2} \\
	{[2]_t\otimes [1]} & {B_0}
	\arrow[""{name=0, anchor=center, inner sep=0}, from=1-1, to=1-2]
	\arrow[from=1-1, to=2-1]
	\arrow[from=1-2, to=2-2]
	\arrow["{[0,1,2]\times[0,1]}"', from=2-1, to=2-2]
	\arrow["\lrcorner"{anchor=center, pos=0.125, rotate=180}, draw=none, from=2-2, to=0]
\end{tikzcd}\]

The cofibration $A_3\to B_0$ is a sequence of inclusions:
$$A_3=:(D_0,M_0)\subset (D_1,M_1)\subset (D_2,M_2)\subset(D_3,M_3)\subset(D_4,M_4)\subset (D_5,M_5)\subset (D_6,M_6):= B_0,$$ where 

\begin{itemize}[leftmargin=* ,parsep=0cm,itemsep=0cm,topsep=0cm]
\item $D_1 = D_0\cup [00,{01},11]$ ;
\item $D_2 = D_1\cup [ {00},10,11]$ ;
\item $D_2 = D_1\cup [ {00},10,21]$ ;
\item $D_4 = D_3\cup [00, {01},11,21]$; 
\item $D_5 = D_4\cup [ {00},10,11,21]$; 
\item $D_6 = D_5\cup [ {00},10,20,21]$; 
\end{itemize} and
\begin{itemize}[leftmargin=* ,parsep=0cm,itemsep=0cm,topsep=0cm]
\item $(D_0,M_0)\to (D_1,M_1)$ is a pushout of $\Lambda^1[2]\to [2]^1$;
\item $(D_1,M_1)\to (D_2,M_2)$ is a pushout of $\Lambda^0[2]\to [2]^0$;
\item $(D_2,M_2)\to (D_3,M_3)$ is a pushout of $\Lambda^0[2]\to [2]^0$;
\item $(D_3,M_3)\to (D_4,M_4)$ is a pushout of $\Lambda^1[3]\to [3]^1$;
\item $(D_4,M_4)\to (D_5,M_5)$ is a pushout of $\Lambda^0[3]\to [3]^0$;
\item $(D_5,M_5)\to (D_6,M_6)$ is a pushout of $\Lambda^0[3]\to [3]^0$.
\end{itemize}
\end{proof}

\begin{lemma}
Morphisms $A_2\cup A_3\to B_1$ and $A_4\to B_1$ are acyclic cofibrations. 
\end{lemma}
\begin{proof}
The cofibration $A_2\cup A_3\to B_1$ fits in the pushout square:
\[\begin{tikzcd}
	{\Lambda^{1}[2]\otimes [1]\cup [2]_t\otimes \partial[1]} & {A_2\cup A_3} \\
	{[2]_t\otimes [1]} & {B_1}
	\arrow[from=1-1, to=1-2]
	\arrow["{[0,2,3]\times[0,1]}"', from=2-1, to=2-2]
	\arrow[from=1-1, to=2-1]
	\arrow[from=1-2, to=2-2]
\end{tikzcd}\]
The cofibration $A_4\to B_1$ is a sequence of inclusions:
$$A_4=:(D_0,M_0)\subset (D_1,M_1)\subset (D_2,M_2)\subset(D_3,M_3)\subset(D_4,M_4)\subset (D_5,M_5)\subset (D_6,M_6):= B_1$$ where 
\begin{itemize}[leftmargin=* ,parsep=0cm,itemsep=0cm,topsep=0cm]
\item $D_1 = D_0\cup [00,21, {31}]$ ;
\item $D_2 = D_1\cup [20, {30},31]$ ;
\item $D_3 = D_2\cup [20,21, {31}]$;
\item $D_4 = D_3\cup [00,01,21, {31}]$;
\item $D_5 = D_4\cup [00,20, {30},31]$ ;
\item $D_6 = D_5\cup [00,20,21, {31}]$ ;
\end{itemize}
and
\begin{itemize}[leftmargin=* ,parsep=0cm,itemsep=0cm,topsep=0cm]
\item $(D_0,M_0)\to (D_1,M_1)$ is a pushout of $\Lambda^2[2]\to [2]^2$;
\item $(D_1,M_1)\to (D_2,M_2)$ is a pushout of $\Lambda^1[2]\to [2]^1$;
\item $(D_2,M_2)\to (D_3,M_3)$ is a pushout of $\Lambda^2[2]\to [2]^2$;
\item $(D_3,M_3)\to (D_4,M_4)$ is a pushout of $\Lambda^3[3]\to [3]^3$;
\item $(D_4,M_4)\to (D_5,M_5)$ is a pushout of $\Lambda^2[3]\to [3]^2$;
\item $(D_5,M_5)\to (D_6,M_6)$ is a pushout of $\Lambda^3[3]\to [3]^3$.
\end{itemize}
\end{proof}

\begin{lemma}
\label{lemma:formula for gray 0}
The maps $A_0\cup A_1\cup A_2\to B$ and $A_4\to B$ are acyclic cofibrations. 
\end{lemma}
\begin{proof}
This is a direct consequence of the last two lemmas.
\end{proof}

\begin{construction}
The marked simplicial set
$\overline{X\otimes B}$ is the pushout:
\[\begin{tikzcd}
	{X\otimes([00,01]\coprod [30,31])} & {X\otimes B} \\
	{[00,01]\coprod [30,31]} & {\overline{X\otimes B}.}
	\arrow[from=1-1, to=2-1]
	\arrow[from=2-1, to=2-2]
	\arrow[from=1-2, to=2-2]
	\arrow[""{name=0, anchor=center, inner sep=0}, from=1-1, to=1-2]
	\arrow["\lrcorner"{anchor=center, pos=0.125, rotate=180}, draw=none, from=2-2, to=0]
\end{tikzcd}\]

Let $\overline{X\otimes A_i}$ and $\overline{X\otimes B_i}$ be the sub-objects of $\overline{X\otimes B}$ corresponding to image of ${X\otimes A_i}$ and $	{X\otimes B_i}$.
\end{construction}

\begin{lemma}
\label{lemma:formula for gray 1}
The inclusion 
$	\overline{X\otimes A_{0}}\cup \overline{X\otimes A_{1}}\cup \overline{X\otimes A_{2}}\to	\overline{X\otimes B}$ and 
$\overline{X\otimes A_{4}}\to \overline{X\otimes B}$ are acyclic cofibrations.
\end{lemma}
\begin{proof}
Remark that we have cocartesian squares
\[\begin{tikzcd}
	{X\otimes([00,01]\coprod [30,31])} & {{X\otimes A_{0}}\cup {X\otimes A_{1}}\cup {X\otimes A_{2}}} & {{X\otimes B}} \\
	{[00,01]\coprod [30,31]} & {\overline{X\otimes A_{0}}\cup \overline{X\otimes A_{1}}\cup \overline{X\otimes A_{2}}} & {\overline{X\otimes B}}
	\arrow[""{name=0, anchor=center, inner sep=0}, from=1-2, to=1-3]
	\arrow[from=1-1, to=2-1]
	\arrow[from=2-1, to=2-2]
	\arrow[""{name=1, anchor=center, inner sep=0}, from=1-1, to=1-2]
	\arrow[from=1-2, to=2-2]
	\arrow[from=2-2, to=2-3]
	\arrow[from=1-3, to=2-3]
	\arrow["\lrcorner"{anchor=center, pos=0.125, rotate=180}, draw=none, from=2-2, to=1]
	\arrow["\lrcorner"{anchor=center, pos=0.125, rotate=180}, draw=none, from=2-3, to=0]
\end{tikzcd}\]
and 
\[\begin{tikzcd}
	{X\otimes([00,01]\coprod [30,31])} & {{X\otimes A_{4}}} & {{X\otimes B}} \\
	{[00,01]\coprod [30,31]} & {\overline{X\otimes A_{4}}} & {\overline{X\otimes B}}
	\arrow[""{name=0, anchor=center, inner sep=0}, from=1-2, to=1-3]
	\arrow[from=1-1, to=2-1]
	\arrow[from=2-1, to=2-2]
	\arrow[""{name=1, anchor=center, inner sep=0}, from=1-1, to=1-2]
	\arrow[from=1-2, to=2-2]
	\arrow[from=2-2, to=2-3]
	\arrow[from=1-3, to=2-3]
	\arrow["\lrcorner"{anchor=center, pos=0.125, rotate=180}, draw=none, from=2-2, to=1]
	\arrow["\lrcorner"{anchor=center, pos=0.125, rotate=180}, draw=none, from=2-3, to=0]
\end{tikzcd}\]
The result then follows from lemma \ref{lemma:formula for gray 0}.
\end{proof}
\begin{lemma}
\label{lemma:formula for gray 2}
The morphisms 
$\overline{X\otimes A_0} \to [1]\fwedge \Sigma X$ and $\overline{X\otimes A_2} \to \Sigma X\fwedge [1],$
induced by the morphism $A_0\to [00,01,11]_t$ and $A_2\to [20,30,31]_t$, are acyclic cofibrations. 
\end{lemma}
\begin{proof}
We have cocartesian squares
\[\begin{tikzcd}
	{X\otimes ([00,01]\coprod \{11\})} & {X\otimes [00,01]\coprod_{X\otimes[01]} X\otimes[01,11]} & {X\otimes A_0} \\
	{[00,01]\coprod \{11\}} & {[1]\coprod_{[0]}\Sigma X} & {\overline{X\otimes A_0}}
	\arrow[from=1-1, to=2-1]
	\arrow[""{name=0, anchor=center, inner sep=0}, "\sim", from=1-2, to=1-3]
	\arrow[""{name=1, anchor=center, inner sep=0}, from=1-1, to=1-2]
	\arrow["\sim", from=2-2, to=2-3]
	\arrow[from=1-2, to=2-2]
	\arrow[from=1-3, to=2-3]
	\arrow[from=2-1, to=2-2]
	\arrow["\lrcorner"{anchor=center, pos=0.125, rotate=180}, draw=none, from=2-2, to=1]
	\arrow["\lrcorner"{anchor=center, pos=0.125, rotate=180}, draw=none, from=2-3, to=0]
\end{tikzcd}\]
That shows that $[1]\coprod_{[0]} \Sigma X \to \overline{X\otimes A_0}$ is an acyclic cofibration. We then have a commutative diagram: 
\[\begin{tikzcd}
	{[1]\coprod_{[0]}\Sigma X} & {\overline{X\otimes A_0}} & {[1]\fwedge\Sigma X}
	\arrow["\sim", from=1-1, to=1-2]
	\arrow[from=1-2, to=1-3]
	\arrow["\sim", curve={height=-24pt}, from=1-1, to=1-3]
\end{tikzcd}\]
and by two out of three, this shows that $\overline{X\otimes A_0} \to [1]\fwedge \Sigma X$ is an acyclic cofibration. 
We proceed similarly for the second morphism. 
\end{proof}

\begin{lemma}
\label{lemma:formula for gray 3}
Marked simplicial sets $\overline{X\otimes A_1}$ and $\overline{X\otimes A_4}$ are respectively equal to $\Sigma (X\otimes [1])$ and $(\Sigma X)\otimes [1]$.
\end{lemma}
\begin{proof}
This is true by the definition of these objects.
\end{proof} 

\begin{proof}[Proof of theorem \ref{theo:interval_first_formula}]
According to lemma \ref{lemma:formula for gray 3} we have a cocartesian square
\[\begin{tikzcd}
	{\overline{X\otimes A_{0}}\coprod\overline{X\otimes A_{2}}} & {\overline{X\otimes A_{0}}\cup \overline{X\otimes A_{1}}\cup \overline{X\otimes A_{2}}} \\
	{[1]\fwedge\Sigma X\coprod \Sigma X\fwedge[1]} & {[1]\fwedge\Sigma X\coprod_{\Sigma (X\otimes\{0\})} \Sigma(X\otimes[1])\coprod_{\Sigma (X\otimes\{1\})} \Sigma X\fwedge[1]}
	\arrow[from=1-1, to=2-1]
	\arrow[from=1-1, to=1-2]
	\arrow[from=1-2, to=2-2]
	\arrow[from=2-1, to=2-2]
\end{tikzcd}\]
The left vertical morphism is a weak equivalence according to lemma \ref{lemma:formula for gray 2}, and the horizontal morphisms are cofibrations. By left properness, the right vertical morphism is a weak equivalence. Combined with lemmas \ref{lemma:formula for gray 1} and \ref{lemma:formula for gray 3}, this provides a zigzag of weak equivalences between 
$
[1]\fwedge\Sigma X\coprod_{\Sigma (X\otimes\{0\})} \Sigma(X\otimes[1])\coprod_{\Sigma (X\otimes\{1\})} \Sigma X\fwedge[1]$
 and $(\Sigma X)\otimes[1].$
\end{proof}

\subsection{Formulas for the Gray cone and the Gray $\circ$-cone}
The aim of this subsection is to demonstrate the following theorem, which is the analogue in stratified simplicial sets of the theorem \ref{theo:appendice formula for star}.
\begin{theorem}
\label{theo:cyl_formula}
There is a zigzag of acyclic cofibrations, natural in $X$, between the colimit of the diagram 
$$ \Sigma X\fwedge [1]\leftarrow \Sigma X\to \Sigma([0]\costar X)$$
and $\Sigma X \star[0]$.

There is a zigzag of acyclic cofibrations, natural in $X$, between the colimit of the diagram 
$$\Sigma(X\star[0]) \leftarrow \Sigma X\to [1]\fwedge\Sigma X$$
and $[0]\costar \Sigma X$.
\end{theorem}

\begin{proof}
We consider the diagram:
\[\begin{tikzcd}
	{[1]} & {[1]\coprod_{[0]}\Sigma X} & {\Sigma X\fwedge[1]\coprod_{\Sigma X} \Sigma( X\otimes[1]) \coprod_{\Sigma X}[1]\fwedge\Sigma X} \\
	{[1]} & {[1]\fwedge \Sigma X} & {\Sigma X\fwedge[1]\coprod_{\Sigma X} \Sigma( X\otimes[1]) \coprod_{\Sigma X}[1]\fwedge\Sigma X}
	\arrow["id", from=1-3, to=2-3]
	\arrow[from=1-2, to=1-1]
	\arrow[from=2-2, to=2-1]
	\arrow[from=1-2, to=1-3]
	\arrow[from=2-2, to=2-3]
	\arrow["\sim"', from=1-2, to=2-2]
	\arrow["id"', from=1-1, to=2-1]
\end{tikzcd}\]
All vertical morphisms are weak equivalences.
We denote by $A$ the colimit of the first line. The theorem \ref{theo:interval_first_formula} implies that there is a zigzag of acyclic cofibrations between $A$ and $X\diamond [0]$. Colimits of the two lines are homotopy colimits, and the comparison morphism is then an acyclic cofibration. 
We then have a zigzag of acyclic cofibrations: 
$$
X\star [0]\leftarrow X\diamond[0] \leftrightsquigarrow A\to \Sigma X\fwedge [1]\coprod_{\Sigma X} \Sigma([0]\costar X)
$$

The second assertion is demonstrated similarly.
\end{proof}

\begin{cor}
\label{cor:star and zigzag}
Let $f:C\to D$ be a fibration between complicial sets, and $K\to L$ a cofibration. It $f$
has the right lifting property against $$\Sigma( [0]\costar K\cup \emptyset \star L )\to \Sigma([0]\costar L),$$ then $f$
 has the right lifting property against $$(\Sigma K)\star [0]\cup (\Sigma L)\star \emptyset \to \Sigma K\star [0].$$
 
If $f$ has the right lifting property against $\Sigma [1]\to \Sigma[1]_t$, then $f$ has the right lifting property against
$$[1]_t\star\emptyset \cup[1]\star [0] \to [1]_t\star [0]$$
\end{cor}
\begin{proof}
Suppose that $f$ fulfills the condition. The class of cofibration having the right lifting property against $f$ is closed by pushouts and, according to \ref{prop:lifting_property_zigzag_of_acyclic_cofibration}, by zigzag of acyclic cofibration. The morphism 
$$\alpha:\Sigma L\fwedge [1]\coprod\limits_{\Sigma L} \Sigma([0]\costar K\coprod\limits_{\emptyset \star K}\emptyset\star L )\to
 \Sigma L\fwedge [1]\coprod\limits_{\Sigma L} \Sigma([0]\costar L)$$ is then in this class. 
Remark that we have a cocartesian square
\[\begin{tikzcd}
	{\Sigma L \cup[1]\coprod\limits_{\Sigma K \cup[1]}\Sigma K\fwedge [1]} & {\Sigma L \cup[1]\coprod\limits_{\Sigma K \cup[1]}\Sigma K\fwedge [1]\coprod\limits_{\Sigma L} \Sigma([0]\costar K)} \\
	{\Sigma L\fwedge [1]} & {\Sigma L\fwedge [1]\coprod\limits_{\Sigma L} \Sigma([0]\costar K\coprod\limits_{\emptyset \star K}\emptyset\star L )}
	\arrow[from=1-1, to=1-2]
	\arrow[from=1-1, to=2-1]
	\arrow[from=1-2, to=2-2]
	\arrow[from=2-1, to=2-2]
\end{tikzcd}\]
where the left vertical morphism, and so also the right vertical morphism, is an acyclic cofibration. This induces a zigzag of acyclic cofibration between $\alpha$ and $\beta$ where $\beta$ is 
$$\Sigma L \cup[1]\coprod\limits_{\Sigma K \cup[1]}\Sigma K\fwedge [1]\coprod\limits_{\Sigma L} \Sigma([0]\costar K)\to 
 \Sigma L\fwedge [1]\coprod\limits_{\Sigma L} \Sigma([0]\costar L)$$
Eventually, the theorem \ref{theo:cyl_formula} induces a zigzag of acyclic cofibration between $\beta$ and 
$(\Sigma K)\star [0]\cup (\Sigma L)\star \emptyset \to \Sigma K\star [0]$ which concludes the proof of the first assertion.

For the second assertion, remark that $[1]_t\star [0]$ is $\tau^i_1([1]_t\star\emptyset \cup [1]\star [0])$. As $\tau^i_1$ is a left Quillen functor, 
the theorem \ref{theo:cyl_formula} induces a zigzag of acyclic cofibration between $[1]_t\star\emptyset \cup[1]\star [0] \to [1]_t\star [0]$ and 
$$[1]_t\fwedge [1]\coprod_{[1]}\Sigma [1]\to [1]_t\fwedge [1]\coprod_{[1]}\Sigma [1]_t.$$
As this cofibration is a pushout of $\Sigma [1]\to \Sigma [1]_t$, this concludes the proof.
\end{proof}

\begin{cor}
\label{cor:costar and zigzag}
Let $f:C\to D$ be a fibration between complicial sets, and $K\to L$ a cofibration. It $f$
has the right lifting property against
$$\Sigma (L\star \emptyset \cup K \star [0])\to \Sigma (L\star [0]),$$
then $f$ has the right lifting property against 
$$ [0]\costar \Sigma K \cup \emptyset \star \Sigma L \to [0]\costar \Sigma L.$$

If $f$ has the right lifting property against $\Sigma [1]\to \Sigma[1]_t$, then $f$ has the right lifting property against
$$[0]\costar [1]\cup \emptyset \star [1]_t\to [0]\costar [1]_t$$
\end{cor}
\begin{proof}
The proof is similar to the one of corollary \ref{cor:star and zigzag}.

\end{proof}

\section{Globular equivalences}
\label{section:Globular equivalences}

\subsection{Marked homotopy categories}

\begin{definition}
\label{defi:definition of simplicial globes}
The \wcsnotionsym{$n$-globe}{(da@$\Db_n$}{globe@$n$-globe}{for marked simplicial sets} is the marked simplicial set $\Db_n:=\Sigma^n [0]$. We then have 
$\Db_0:=[0]$ and $\Db_{n+1}:= \Sigma \Db_n$.
This defines a globular object in $\mSset$:
\[\begin{tikzcd}
	{\Db_0} & {\Db_1} & {\Db_2} & {...}
	\arrow["{i_0^+}", shift left=2, from=1-1, to=1-2]
	\arrow["{i_1^+}", shift left=2, from=1-2, to=1-3]
	\arrow["{i_3^+}", shift left=2, from=1-3, to=1-4]
	\arrow["{i_0^-}"', shift right=2, from=1-1, to=1-2]
	\arrow["{i_1^-}"', shift right=2, from=1-2, to=1-3]
	\arrow["{i_3^-}"', shift right=2, from=1-3, to=1-4]
\end{tikzcd}\]
and we have equalities:
$$i_{n+1}^- i^+_n=i^+_{n+1} i^-_n~~~~i^+_{n+1} i^-_n=i^+_{n+1} i^+_n.$$
 We also set $(\Db_n)_t:= \tau^i_{n-1}(\Db_n)$ for $n>0$ and $\partial\Db_n:=\Sigma^n \emptyset$. 
 We then have a canonical inclusions 
 $$\partial \Db_0\to \Db_0$$ and
 for any $n>0$, we have canonical inclusions 
 $$\partial \Db_n\to \Db_n\to (\Db_n)_t.$$
\end{definition}

\vspace{1cm}

From now on, and until the end of this section, we fix a complicial set $C$. 
\begin{definition}
Let $C$ be a complicial set. A \wcsnotion{$n$-cell}{cell@$n$-cell}{for marked simplicial sets} $a$ of $C$ is a morphism $a:\Db_n\to C$. If $n$ is non null, the \textit{source} of $a$ (resp. the \textit{target} of $a$)
is the $(n-1)$-cell $a\circ i^-_{n-1}$ (resp. $a\circ i^+_{n-1}$). The cell $a$ is marked if the corresponding morphism $\Db_n\to C$ factorizes via $(\Db_n)_t$.
\end{definition}
	
\begin{definition}
Let $n$ be a non null integer, and $a,b$ two $n$-cells.
Cells $a$ and $b$ are \textit{parallel} if they share the same source and the same target. They are \textit{composable} if the source of $a$ is the target of $b$.

Let $a$ and $b$ be two parallel cells. The cell $a$ is \wcnotion{equivalent}{equivalent $n$-cells} to the cell $b$ if there exists a marked $(n+1)$-cell $d:a\to b$, or equivalently, if there exists a homotopy $\Db_n\times [1]_t$ between $a$ and $b$, and constant on $\partial \Db_n\times [1]_t$. This relation is denoted by $\sim$.
\end{definition}

\begin{lemma}
The relation $\sim$ is reflexive, symmetric and transitive.
\end{lemma}
\begin{proof}
This comes from usual properties of fibrant objects.
\end{proof}

\begin{lemma}
Let $a$, $b$ be two equivalent cells. If $a$ is marked, so is $b$.
\end{lemma}
\begin{proof}
As $\{0\}\to [1]_t$ is a weak equivalence, so is $\Db_n\times [1]_t\cup (\Db_n)_t\times \{0\}\to (\Db_n)_t\times [1]_t$. As $C$ is fibrant, this directly implies the result. 
\end{proof}

\begin{construction}
\label{cons:composition_homotopy_category}
Let $a,b$ be two composable $n$-cells . A composition of ${a}$ and ${b}$ is a $n$-cell $a\circ b$ that fits in a diagram:
\[\begin{tikzcd}
	{\Db_n\coprod_{\Db_{n-1}}\Db_n} \\
	{\Sigma^{n-1}([2]_t)} & C \\
	{\Db_{n}}
	\arrow[from=1-1, to=2-1]
	\arrow["{a\coprod b}", from=1-1, to=2-2]
	\arrow["{a\circ b}"', from=3-1, to=2-2]
	\arrow[from=3-1, to=2-1]
	\arrow[from=2-1, to=2-2]
\end{tikzcd}\]
As $C$ is a fibrant object, if $(a\circ b)'$ is any other composition, $(a\circ b)'\sim a\circ b$.  Morover if two of three cells $a$, $b$ or $a\circ b$ is marked, so is the third.
\end{construction}

\begin{lemma}
\label{lemma:associativity_of_composition_in_homotopy_category}
Let $a,b,c$ be three composable cells. There exists compositions such that $(a\circ b)\circ c = a\circ (b\circ c)$.
\end{lemma}
\begin{proof}
Let $M$ be the marking on $[3]$ that includes all simplices of dimension superior or equal to $2$. We define $\Sp_{[3]}$ as the simplicial set $[1]\coprod_{[0]} [1]\coprod_{[0]} [1]$. Remark that the cofibration $\Sp_{[3]}\to ([3],M)$ is acyclic. We then have a lift $f$ in the following diagram
\[\begin{tikzcd}
	{\Sigma^{n-1}\Sp_{[3]}} & C \\
	{\Sigma^{n-1}([3],M)}
	\arrow[from=1-1, to=2-1]
	\arrow["{a\coprod b\coprod c}", from=1-1, to=1-2]
	\arrow["f"', dashed, from=2-1, to=1-2]
\end{tikzcd}\]
The morphism $f$ provides all the desired compositions.
\end{proof}
\begin{definition}
  A \notion{marked category} is a pair $(C,tC)$ where $C$ is a category and $tC$ is a subset of isomorphisms stable by $2$ out of $3$. A functor $f:(C,tC)\to (D,tD)$ is a \notion{marked equivalence} if there exists a functor $g:(D,tD)\to (C,tC)$ and two natural transformations $fg\to id$ and $gf\to id$ that are pointwise marked.
\end{definition}

\begin{definition}

We define the marked category $\pi_0(C)$ whose objects are $0$-cells $x:s\to t$, and edges between $x,y:s\to t$ are equivalence classes of the set of $1$-cells $f:x\to y$ quotiented by the relation $\sim$. Marked edges correspond to equivalence classes of marked $1$-cells. The composition is given by construction \ref{cons:composition_homotopy_category} which is associative according to lemma \ref{lemma:associativity_of_composition_in_homotopy_category}. 

Let $n>0$ be an integer, and $s,t$ two parallel $(n-1)$-cells. We define the marked category $\pi_n(s,t,C)$ whose objects are $n$-cells $x:s\to t$, and edges between $x,y:s\to t$ are equivalence classes of the set of $(n+1)$-cells $f:x\to y$ quotiented by the relation $\sim$. Marked edges correspond to equivalence classes of marked $(n+1)$-cells. The composition is given by construction \ref{cons:composition_homotopy_category} which is associative according to lemma \ref{lemma:associativity_of_composition_in_homotopy_category}.
\end{definition}

For the definition to make sense, we have to show that the previous marking consists indeed of isomorphisms.

\begin{prop}
\label{prop:in_the_homotopy_category_marked_is_iso}
Let $x,y:s\to t$ be two parallel $n$-cells, and $f:x\to y$ a $(n+1)$-cell. If the cell $f$ is marked then $[f]:x\to y$ is an isomorphism in $\pi_n(s,t,C)$.
\end{prop}
\begin{proof}
There are liftings in the following diagrams:
\[\begin{tikzcd}
	{\Sigma^{n}\Lambda^0[2]} & C & {\Sigma^{n}\Lambda^2[2]} & C \\
	{\Sigma^{n}[2]^0} && {\Sigma^{n}[2]^0}
	\arrow["{f\amalg id}", from=1-1, to=1-2]
	\arrow[from=1-1, to=2-1]
	\arrow["h"', dotted, from=2-1, to=1-2]
	\arrow["{id\amalg f}", from=1-3, to=1-4]
	\arrow["k"', dotted, from=2-3, to=1-4]
	\arrow[from=1-3, to=2-3]
\end{tikzcd}\]
Let $g:y\to z$ be the restriction of $h$ to $\Sigma^{n}[1,2]$ and $l:y\to z$ be the restriction of $k$ to $\Sigma^{n}[0,1]$. We then have $[f][g]= id$, and $[h][f]=id$, and $[f]$ is then an isomorphism. 
\end{proof}

\begin{lemma}
\label{lemma:homotopycategory_are_idenpendant_of}
Let $s,t$ and $s',t'$ be two pairs of parallel cells, and $\psi:\partial \Db_n\times [1]_t\to C$ a homotopy between $s\cup t:\partial \Db_n\to C$ and $s'\cup t':\partial \Db_n\to C$. Then 
$$\pi_n(s,t,C)\cong \pi_n(s',t',C)$$
\end{lemma}
\begin{proof}
For each $x:s\to t$, there exists a lifting $h_x$ in the following diagram:
\[\begin{tikzcd}
	{\Db_n\times\{0\}\cup\partial\Db_n\times [1]_t} & C \\
	{\Db_n\times [1]_t}
	\arrow[from=1-1, to=2-1]
	\arrow["x\cup\psi", from=1-1, to=1-2]
	\arrow["h"', dotted, from=2-1, to=1-2]
\end{tikzcd}\]
and we define $F(x)$ as the restriction of $h_x$ to $\Db_n\times \{1\}$. For a $(n+1)$-cell $f:x\to y$, there exists a lifting $h_f$ in the following diagram:
\[\begin{tikzcd}
	{\Db_{n+1}\times \{0\}\cup \partial\Db_{n+1}\times [1]_t} & C \\
	{\Db_{n+1}\times [1]_t}
	\arrow["{f\cup h_x\cup h_y}", from=1-1, to=1-2]
	\arrow[from=1-1, to=2-1]
	\arrow["{h_f}"', from=2-1, to=1-2]
\end{tikzcd}\]
and we define $F(f)$ as the restriction of $h_f$ to $\Db_{n+1}\times \{1\}$. Furthermore, the unicity up to homotopy of lifting implies that $[F(f)]$ is independent of the choice of the lifting, and that $f\sim g$ implies $[F(f)]=[F(g)]$.
If $g:y\to z$ is an other morphism, and $\psi:\Sigma^n[2]_t \to C$ corresponds to the composition of $f$ and $g$, 
there is a lift in the following diagram:
\[\begin{tikzcd}
	{\Sigma^n [2]_t\cup (\Sigma^n\partial[2])\times [1]_t} && C \\
	{\Sigma^n [2]_t\times [1]_t}
	\arrow["{ \phi \cup h_f\cup h_g\cup h_{f\circ g}}", from=1-1, to=1-3]
	\arrow[from=1-1, to=2-1]
	\arrow[dotted, no head, from=2-1, to=1-3]
\end{tikzcd}\]
Restricted to $\Sigma^n [2]_t\times \{1\}$ this shows that $F$ commutes with compositions. We then have defined a functor 
$$F:\pi_n(s,t,C)\to \pi_n(s',t',C).$$

Using exactly the same procedure, where we just invert $0$ and $1$, we define a functor:
$$G:\pi_n(s',t',C)\to \pi_n(s,t,C).$$
Now, we have a lift in the following diagram:
\[\begin{tikzcd}
	{\Db_{n}\times \Lambda^{2}[2]^\sharp\cup\partial\Db_n\times[2]^\sharp} &&& C \\
	{\Db_n\times[2]^\sharp}
	\arrow["{h_x\cup h_{F(x)}\cup\psi(id\times s^0)}", from=1-1, to=1-4]
	\arrow[from=1-1, to=2-1]
	\arrow["{k_x}"', dotted, from=2-1, to=1-4]
\end{tikzcd}\]
The restriction of $k_x$ to $\Db_n\times [0,1]_t$ provides a marked cell $x\to G(F(x))$, which corresponds to an isomorphism in $\pi_n(s,t,C)$ according to proposition \ref{prop:in_the_homotopy_category_marked_is_iso}. If $f:x\to y$ is a $(n+1)$-cell, there is a lifting in the following diagram:
\[\begin{tikzcd}
	{\Db_{n+1}\times \Lambda^{2}[2]^\sharp\cup\partial\Db_{n+1}\times[2]^\sharp} &&& C \\
	{\Db_{n+1}\times[2]^\sharp}
	\arrow["{h_f\cup h_{F(f)}\cup k_x\cup k_y}", from=1-1, to=1-4]
	\arrow["{k_f}"', dotted, from=2-1, to=1-4]
	\arrow[from=1-1, to=2-1]
\end{tikzcd}\]
The restriction of $k_f$ to $\Db_{n+1}\times[0,1]_t$ induces in $\pi_n(s,t,C)$ a commutative diagram:
\[\begin{tikzcd}
	x & GFx \\
	y & GFy.
	\arrow["{[GFf]}", from=1-2, to=2-2]
	\arrow["{[f]}"', from=1-1, to=2-1]
	\arrow[from=2-1, to=2-2]
	\arrow[from=1-1, to=1-2]
\end{tikzcd}\]
We then have a marked invertible natural transformation $\psi: id\to GF$. Similarly we can construct an other marked invertible natural transformation $id\to GF$, which shows the desired equivalence of categories.
\end{proof}
\begin{definition}
Let $a$ be an element of  $\Hom_{ho(\mSset)}(\partial \Db_n,C)$. We define  the marked category
\begin{equation}
\label{eq:def of pi a}
\pi_n(a, C) := \pi_n(s,t,C)
\end{equation}
where $s,t$ is a pair of parallel arrows such that $s\cup t$ represents $a$.
The previous proposition shows that this is well defined.
\end{definition}

\subsection{A criterion to be a weak equivalence}
\begin{definition}
A morphism $p:C\to D$ between complicial sets is a \wcnotion{$\Db$-equivalence}{Dequivalence@$\Db$-equivalence} if 
$$\pi_0(C)\to \pi_0(D)$$
is an equivalence of  marked categories, and for any $n>0$ and pair of parallel arrow $s,t$, the induced functor
$$\pi_n(s,t,C)\to \pi_n(ps,pt,D)$$
is an equivalence of marked categories. 

A \wcnotion{$\Db$-trivial fibration}{Dtrivial fibration@$\Db$-trivial fibration} is a fibration having the right lifting property against $\partial\Db_n\to \Db_n$ and $\Db_n\to (\Db_{n})_t$.

\end{definition}

\begin{lemma}
\label{lemma:fibration_are_isofibration}
Let $\alpha\in\{-,+\}$.
The morphism $i^\alpha_{n+1}:\Db_n\to (\Db_{n+1})_t$ is an acyclic cofibration. 
\end{lemma}
\begin{proof}
We have a pushout diagram
\[\begin{tikzcd}
	{\Db_n\times\{\alpha\}\cup\partial\Db_n\times [1]_t} & {\Db_n\times \{\alpha\}} \\
	{\Db_n\times [1]_t} & {(\Db_n)_t}
	\arrow[""{name=0, anchor=center, inner sep=0}, "{id\cup \partial\times s^0}", from=1-1, to=1-2]
	\arrow[from=1-1, to=2-1]
	\arrow[from=2-1, to=2-2]
	\arrow["{i^\alpha_{n+1}}", from=1-2, to=2-2]
	\arrow["\lrcorner"{anchor=center, pos=0.125, rotate=180}, draw=none, from=2-2, to=0]
\end{tikzcd}\]
The left hand morphism being an acyclic cofibration, this concludes the proof.
\end{proof}

\begin{lemma}
\label{lemma:acyclic_cofibration_are_G_equivalence}
Acyclic cofibrations between complicial sets are $\Db$-equivalences.
\end{lemma}
\begin{proof}
Let $i:A\to B$ be an acyclic cofibration. The morphism $i$ admits a retraction $r:B\to A$: 
\[\begin{tikzcd}
	A & A \\
	B.
	\arrow["id", from=1-1, to=1-2]
	\arrow["i"', from=1-1, to=2-1]
	\arrow["r"', from=2-1, to=1-2]
\end{tikzcd}\]
and a homotopy $\psi$ between $id_B$ and $ir$ which is constant on the image of $i$, obtained as the lift in the following diagram:
\[\begin{tikzcd}
	{B\times\{0\}\coprod _{A\times\{0\}}A\times[1]_t} & B \\
	{B\times[1]_t}
	\arrow[from=1-1, to=2-1]
	\arrow["\phi"', dashed, from=2-1, to=1-2]
	\arrow[from=1-1, to=1-2]
\end{tikzcd}\]
Let $n>0$ be an integer, and $s$, $t$ be two $(n-1)$-cells of $C$. The retraction implies that $i_!$ is an injection on morphisms. For any $n$-cell $y:i(s)\to i(t)$ in $B$, the homotopy $\psi$ induces a marked cell $y\to ir(y)$ which corresponds to an isomorphism in $\pi_n(is,it,B)$ according to proposition \ref{prop:in_the_homotopy_category_marked_is_iso}. The functor $i_!$ is then essentially surjective. For any $(n+1)$-cell $f:i(x)\to i(y)$, the homotopy $\psi$ induces an equivalence $[ir(f)]\sim [f]$. The morphism $i_!$ is a surjection on morphisms and detects marked $(n+1)$-cells. All put together, $i_!$ is fully faithfull and essentially surjective, and is then a marked equivalence. We proceed similarly to show that $i_!:\pi_0(A)\to \pi_0(B)$ is a marked equivalence.
\end{proof}

\begin{lemma}
\label{lemma:2_out_of_3_for_G_equivalence}
Suppose given a commutative triangle between complicial sets
\[\begin{tikzcd}
	& B \\
	A && C
	\arrow["g"', from=2-1, to=2-3]
	\arrow["f", from=1-2, to=2-3]
	\arrow["i", from=2-1, to=1-2]
\end{tikzcd}\]
If $i$ is an acyclic cofibration, and $g$ is a $\Db$-equivalence, then $f$ is a $\Db$-equivalence.
\end{lemma}
\begin{proof}
Let $s,t$ be any pair of parallel arrows in $B$. There exists a pair of parallel arrows $s',t'$ in $A$ such that $s\cup t$ and $is'\cup it'$ correspond to the same element in $[\partial\Db_n,B]$. We then have a diagram:
\[\begin{tikzcd}
	& { \pi(s,t,B)} & {\pi(fs,ft,C)} \\
	{\pi(s,t,B)} & { \pi(is,it,B)} & {\pi(gs,gt,C).}
	\arrow["\sim", from=2-1, to=2-2]
	\arrow[from=2-2, to=2-3]
	\arrow["\sim", from=1-3, to=2-3]
	\arrow["\sim", from=1-2, to=2-2]
	\arrow[from=1-2, to=1-3]
	\arrow["\sim"', curve={height=18pt}, from=2-1, to=2-3]
\end{tikzcd}\]
where arrows labeled by $\sim$ are marked equivalences according to lemmas \ref{lemma:homotopycategory_are_idenpendant_of} and \ref{lemma:acyclic_cofibration_are_G_equivalence}. By two out of three, this shows that $ \pi(s,t,B)\to \pi(fs,ft,C)$ is an marked equivalence, and $f$ is then a $\Db$-equivalence.
\end{proof}

\begin{prop}
\label{prop:caracterisation_of_G_fibration}
Let $p:C\to D$ be a fibration between complicial sets.
The morphism $p$ is a $\Db$-trivial fibration if and only if it is a $\Db$-equivalence.
\end{prop}
\begin{proof}
If $p$ is a $\Db$-trivial fibration, it is obvious that it is a $\Db$-equivalence. For the converse, suppose $p$ is a fibration and a $\Db$-equivalence, and consider a diagram 
\[\begin{tikzcd}
	{\partial\Db_n} & C \\
	{\Db_n} & D
	\arrow[from=1-1, to=2-1]
	\arrow[from=1-1, to=1-2]
	\arrow["x"', from=2-1, to=2-2]
	\arrow["p", from=1-2, to=2-2]
\end{tikzcd}\]
As $p$ is a $\Db$-equivalence this implies that there exists a cell $\overline{x}:\Db_n\to C$ together with a marked $(n+1)$-cell $y:p(\overline{x})\to y$. All this data corresponds to a diagram:
\[\begin{tikzcd}
	{\Db_n} & C \\
	{(\Db_{n+1})_t} & D
	\arrow["p", from=1-2, to=2-2]
	\arrow["{\delta^0_{n+1}}"', from=1-1, to=2-1]
	\arrow["{\bar{x}}", from=1-1, to=1-2]
	\arrow["y"', from=2-1, to=2-2]
\end{tikzcd}\]
The left hand morphism being an acyclic cofibration according to \ref{lemma:fibration_are_isofibration}, this diagram admits a lift $h:(\Db_{n+1})_t\to C$. The restriction of $h$ to $i^+_{n+1}$ provides a lift in the first diagram. Now, we consider a diagram of shape: 
\[\begin{tikzcd}
	{\Db_n} & C \\
	{(\Db_n)_t} & D
	\arrow[from=1-1, to=2-1]
	\arrow["g", from=1-1, to=1-2]
	\arrow[from=2-1, to=2-2]
	\arrow["p", from=1-2, to=2-2]
\end{tikzcd}\]
with $n>1$.
Let $s,t$ be respectively the $(n-1)$-source and the $(n-1)$-target of $g$. Hypotheses imply that $[p(g)]$ is marked in $\pi_n(s,t,D)$ and because $p$ is a $\Db$-equivalence, so is $[g]$. The case $n=1$ is similar.
 The morphism $f$ is then a $\Db$-trivial fibration.
\end{proof}

\begin{lemma}
\label{lemma:slice_G_fibrations}
Let $p:X\to Y$ be a $\Db$-trivial fibration between complicial sets. Then for any $x\in X_0$, the induced fibrations
$$X_{/x}\to X\times_Y Y_{/p(x)} ~~\mbox{and}~~ X_{x/}\to X\times_Y Y_{p(x)/}$$
are $\Db$-trivial fibrations.
\end{lemma}
\begin{proof}
We define $\mathbb{P}(p,n)$ to be the statement that $p$ has the right lifting property against 
$$ \Db_n\cup \partial\Db_n\star [0]\to \Db_{n+1}\star[0] \mbox{ and }(\Db_n)_t\cup \Db_n\star [0]\to (\Db_{n})_t\star[0]$$
and against
$$[0]\costar \partial \Db_n\cup \Db_n\to [0]\costar \Db_{n+1} \mbox{ and }[0]\star \Db_n\cup (\Db_n)_t\to [0]\costar (\Db_n)_t$$
We then have to show that for any $n$, $\mathbb{P}(p,n)$ holds.

First, it is obvious that each $\Db$-equivalence $p$ satisfies $\mathbb{P}(p,0)$. As $p$ is a fibration, the corollaries \ref{cor:star and zigzag} and \ref{cor:costar and zigzag} then imply that $\mathbb{P}(p,n+1)$ is equivalent to $\mathbb{P}(p(a,b),n)$ for any $a,b\in X_0$, where $p(a,b)$ is the induced morphism: $X(a,b)\to Y(p(a),p(b))$. 

Using the fact that $p(a,b)$ is a $\Db$-trivial fibration as soon as $p$ is, this shows the desired result.
\end{proof}

\begin{lemma}
\label{lemma:G_fibration_right lifting property_against_partial}
 $\Db$-Trivial fibrations between complicial sets have the right lifting property against $\partial[n]\to [n]$.
\end{lemma}
\begin{proof}
Let $C$ be the class of cofibrations having the right lifting property against $\Db$-equivalences. The lemma \ref{lemma:slice_G_fibrations} implies that for any 
 $K\to L$ in $C$, the induced morphism:
$$L\cup K\star[0]\to L\star[0]$$
is in $C$. 
The class $C$ is then closed under Leibniz join. Furthermore, it includes $\partial[1]\to [1]$, and then, by induction, it includes $\partial[n]\to[n]$ for any integer $n$.
\end{proof}

\begin{lemma}
\label{lemma:G_fibration_right lifting property_against_sat}
$\Db$-Trivial fibrations between complicial sets have the right lifting property against $[n]\to [n]_t$.
\end{lemma}
\begin{proof}
Let $p$ be $\Db$-trivial fibrations between complicial sets, and
 $C_{n,p}$ be the set of objects $A$ such that $p$ has the right lifting property against:
$$A\to \tau^i_{n-1}(A).$$
This set is then closed under colimits, and by zigzags of acyclic cofibrations.
Let $k\leq n$ be two integers. We define $\mathbb{P}(k,n,p)$ to be the statement that 
$$ \Sigma [n-k]_{\circ}\star[k-1]~~~\mbox{and}~~~ [k-1]_{\circ}\costar\Sigma [n-k] $$
are in $C_{n+1,p}$.
The statement $\mathbb{P}(0,0,f)$ corresponds to the belonging of $\Db_1$ to $C_{1,p}$, which is obviously true. Suppose that $0<k$ and $\mathbb{P}(k-1,n,p)$. 
According to theorem \ref{theo:cyl_formula}, the object $\Sigma [n-k]_{\circ}\star[k-1]$ is linked by a zigzag of acyclic cofibrations to the colimit of 
$$
(\Sigma [n-k]_{\circ} \fwedge [1])\star [k-2] \leftarrow (\Sigma [n-k]_{\circ})\star [k-2] \to (\Sigma [n-k+1]_{\circ})\star [k-2]
$$
The center object and the left hand object are in $C_{n+1,p}$ because there are invariant under $\tau^i_{n}$, and the 
 right hand object is in $C_{n+1,p}$ by induction hypothesis. The object $\Sigma [n-k]_{\circ}\star[k-1]$ is then in $C_{n+1,p}$.
We demonstrate similarly that $[k-1]_{\circ}\costar\Sigma [n-k]$ is in $C_{n+1,p}$.

This then implies $\mathbb{P}(k,n,p)$. Eventually, $\mathbb{P}(0,n+1,p)$ is equivalent to $\mathbb{P}(n,n,p(a,b))$ for any pair of objects $(a,b)\in X_0$.
The statement $\mathbb{P}(k,n,p)$ is then true for any $k,n$ and $\Db$-trivial fibrations between complicial sets $p$. The statment $\mathbb{P}{n,n,p}$ implies that $p$ has the right lifting property against $[n]\to [n]_t$.
\end{proof}

\begin{theorem}
\label{theo:f_weak_equivalence_ssi_f_G_equivalence}
Let $p$ be a map between complicial sets. Then $p$ is a weak equivalence if and only if it is a $\Db$-equivalence.
\end{theorem}
\begin{proof}
According to lemmas \ref{lemma:acyclic_cofibration_are_G_equivalence} and \ref{lemma:2_out_of_3_for_G_equivalence} we can restrict ourselves to the case where $p$ is a fibration. If it is a weak equivalence, $p$ is then a trivial fibration and is then a $\Db$-equivalence. Suppose now that $p$ is a $\Db$-equivalence. According to proposition \ref{prop:caracterisation_of_G_fibration}, $p$ is then a $\Db$-trivial fibration. Lemmas \ref{lemma:G_fibration_right lifting property_against_partial} and \ref{lemma:G_fibration_right lifting property_against_sat} imply that $p$ is a trivial fibration.
\end{proof}

\begin{definition}
Let $p:X\to Y$ be a morphism between complicial sets. The morphism $p$ is \snotion{essentially surjective}{for marked simplicial sets} if for any $x\in Y_0$, there exists $\bar{x}\in X_0$ together with a marked cell $\bar{x}\to x$. It detects marked $1$-cells if the functor $\pi_0(X)\to \pi_0(Y)$ detects marked $1$-cells.
Eventually, the morphism $f$ is \snotion{fully faithful}{for marked simplicial sets} if the induced morphisms: 
$$X(a,b)\to Y(pa,pb)$$
are weak equivalences for any $a,b\in X_0$. 
\end{definition}

\begin{cor}
Let $p$ be a map between complicial sets. Then $p$ is a weak equivalence if and only if it is fully faithful, essentially surjective, and detects marked $1$-cells.
\end{cor}

\begin{proof}
If $p$ is a weak equivalence, it is then fully faithful, essentially surjective, and detects marked $1$-cells. Conversely, suppose $p$ is fully faithful, essentially surjective, and detects marked $1$-cells. The morphism $\pi_0(X)\to \pi_0(Y)$ is fully faithful, essentially surjective, and detects marked arrows; it is then an equivalence of marked categories. For $(a,b)$ a pair of $0$-cells, we have equalities:
\[\begin{tikzcd}
	{\pi_1(a,b,X)} & {\pi_0(X(a,b))} \\
	{\pi_1(pa,pb,Y)} & {\pi_0(Y(pa,pb)).}
	\arrow["{\pi_0p(a,b)}", from=1-2, to=2-2]
	\arrow["{\pi_1p}"', from=1-1, to=2-1]
	\arrow[Rightarrow, no head, from=2-1, to=2-2]
	\arrow[Rightarrow, no head, from=1-1, to=1-2]
\end{tikzcd}\]
The morphism $\pi_1(a,b,p)$ is then an equivalence of categories. For $(s,t)$ a pair of parallel arrows of dimension $>1$, if we denote by $a$ and $b$ the $0$-source and the $0$-target of $s$ and $t$, we have a diagram:
\[\begin{tikzcd}
	{\pi_n(s,t,X)} & {\pi_{n-1}(s,t,X(a,b))} \\
	{\pi_n(pa,pb,Y)} & {\pi_{n-1}(s,t,Y(pa,pb)).}
	\arrow["{\pi_{n-1}(s,t ,p(a,b))}", from=1-2, to=2-2]
	\arrow["{\pi_np}"', from=1-1, to=2-1]
	\arrow[Rightarrow, no head, from=2-1, to=2-2]
	\arrow[Rightarrow, no head, from=1-1, to=1-2]
\end{tikzcd}\]
The morphism $\pi_n(a,b,p)$ is then an equivalence of categories. The morphism $p$ is then a $\Db$-equivalence, and according to \ref{theo:f_weak_equivalence_ssi_f_G_equivalence}, a weak equivalence.
\end{proof}

\subsection{A criterion to be a weakly invertible transformation}
\label{section:A criterion to be a weakly invertible transformation}
The purpose of this section is to show the following proposition:
\begin{prop}
\label{prop:criterimu_to_be_an_weak_equivalence}
Let $i:\mSset\to \mSset$ and $j:\mSset\to \mSset$ be two left Quillen functors and $\psi:i\to j$ a natural transformation. If 
$\psi(\Db_n):i (\Db_n) \to j (\Db_n)$ is a weak equivalence for any $n$, then $\psi(X):i(X)\to j(X)$ is a weak equivalence for any $X$.
\end{prop}
For the remaining of this section, we fix two left Quillen functors $i$, $j$ and a natural transformation $\psi:i\to j$ satisfying the previous hypothesis. We denote by $N_i$ and $N_j$ the right adjoints of $i$ and $j$.

\begin{lemma}
\label{lemma:psipartial_is_a_weak_equivalence}
Morphisms $\psi(\partial\Db_n):i(\partial\Db_n)\to j(\partial\Db_n)$ are weak equivalences. 
\end{lemma}
\begin{proof}
We proceed by induction on $n$. The case $n=0$ is trivial. Suppose then the result true at the stage $n-1$. Remark then that $\partial \Db_n$ is the colimit and the homotopy colimit of the span
$$\Db_{n-1}\leftarrow \partial\Db_{n-1}\to \Db_{n-1}$$
As $i$ and $j$ are left Quillen functors, the induction hypothesis implies that $\psi(\partial\Db_n):i(\partial\Db_n)\to j(\partial\Db_n)$ is a weak equivalence.
\end{proof}

\begin{lemma}
\label{lemma:psisat_is_a_weak_equivalence}
Morphisms $\psi((\Db_n)_t):i((\Db_n)_t)\to j((\Db_n)_t)$ are weak equivalences. 
\end{lemma}
\begin{proof}
There is a diagram:
\[\begin{tikzcd}
	{i\Db_{n-1}} && {j\Db_{n-1}} \\
	{i(\Db_n)_t} && {j(\Db_n)_t}
	\arrow["{\psi(\Db_n)}", from=1-1, to=1-3]
	\arrow["{i(i^-_n)}"', from=1-1, to=2-1]
	\arrow["{j(i^-_n)}", from=1-3, to=2-3]
	\arrow["{\psi((\Db_n)_t)}"', from=2-1, to=2-3]
	\arrow["\sim"', draw=none, from=1-1, to=1-3]
	\arrow["\sim", draw=none, from=1-1, to=2-1]
	\arrow["\sim"', draw=none, from=1-3, to=2-3]
\end{tikzcd}\]
By two out of three, this shows that $\psi((\Db_n)_t)$ is a weak equivalence.
\end{proof}

\begin{lemma}
\label{lemma:j*is_a_trivial_fibration}
For any complicial set $Y$, the canonical morphism $N_jY\to N_i Y$ is a weak equivalence.
\end{lemma}
\begin{proof}
Let $Y$ be a complicial set. For any integer $n$, we have by adjunction a bijection
$$\Hom_{ho(\mSset)}(\Db_n, N_jY)\cong \Hom_{ho(\mSset)}(\Db_n, N_iY)$$
and according to lemmas \ref{lemma:psipartial_is_a_weak_equivalence} and \ref{lemma:psisat_is_a_weak_equivalence}, we have bijections
$$\Hom_{ho(\mSset)}(\partial \Db_n, N_jY)\cong \Hom_{ho(\mSset)}(\partial\Db_n, N_iY)$$
$$\Hom_{ho(\mSset)}((\Db_n)_t, N_jY)\cong \Hom_{ho(\mSset)}((\Db_n)_t, N_iY).$$
Let $a$ be an element of $\Hom_{ho(\mSset)}(\partial \Db_n, N_jY)$. We recall that the category $\pi_n(a,N_jY)$ is defined in \ref{eq:def of pi a}. The previous equivalences implies that we have an isomorphism of category
$$\pi_n(a,N_jY)\cong \pi_n(a,N_jY).$$
which concludes the proof according to theorem \ref{theo:f_weak_equivalence_ssi_f_G_equivalence}.
\end{proof}

\begin{proof}[Proof of the proposition \ref{prop:criterimu_to_be_an_weak_equivalence}]
Let $X$ be any marked simplicial set and $Y$ a complicial set. We have equalities:
\[\begin{tikzcd}
	{\Hom_{ho(\mSset)}(j_!X,Y)} & {\Hom_{ho(\mSset)}(X,j^*Y)} \\
	{\Hom_{ho(\mSset)}(i_!X,Y)} & {\Hom_{ho(\mSset)}(X,i^*Y)}
	\arrow[from=1-2, to=2-2]
	\arrow[from=1-1, to=2-1]
	\arrow[Rightarrow, no head, from=1-1, to=1-2]
	\arrow[Rightarrow, no head, from=2-1, to=2-2]
\end{tikzcd}\]
Lemma \ref{lemma:j*is_a_trivial_fibration} implies that the right hand morphism is a bijection, and so is the left hand morphism. 
For any $X$, $\psi(X)$ is then a weak equivalence.
\end{proof}

\subsection{Weak characterization of the identity}
 For the rest of this section, we fix a left Quillen functor $i:\mSset\to \mSset$ such that there exists a zigzag of weakly invertible natural transformations:
$$i(\Db_{\uvar}) \leftrightsquigarrow \Db_{\uvar}.$$

\begin{lemma}
\label{lemma:weak characteroeiation 1}
Let $n$ be any integer, the following natural transformations are pointwise acyclic cofibrations:
$$i\tau^i_n\to \tau^i_{n}i\tau^i_n \leftarrow \tau^i_{n}i.$$
\end{lemma}
\begin{proof}
These are natural transformations between left Quillen functors. The hypothesis implies that they induce weak equivalences on globes of dimension inferior or equal to $n$. Remark that for any $k> n$, as $i_{k-1}^-:\Db_{k-1}\to (\Db_{k})_t$ is an acyclic cofibration and $\tau^i_n$ preserves them, 
$\tau^i_n\Db_{k-1}\to \tau^i_n\Db_{k}$ is an acyclic cofibration. A direct induction implies that $\Db_{n}= \tau^i_n\Db_{n}\to \tau^i_n\Db_{k}$ is an acyclic cofibration.
We then have a commutative diagram: 
\[\begin{tikzcd}
	{i{\tau^i_n}(\Db_k)} & {{\tau^i_n}i {\tau^i_n}(\Db_k)} & {{\tau^i_n}i(\Db_k)} \\
	& {i(\Db_{n})}
	\arrow["\sim"', from=2-2, to=1-1]
	\arrow["\sim"', from=2-2, to=1-2]
	\arrow["\sim", from=2-2, to=1-3]
	\arrow[from=1-1, to=1-2]
	\arrow[from=1-3, to=1-2]
	\arrow[draw=none, from=2-2, to=1-1]
	\arrow[draw=none, from=2-2, to=1-2]
	\arrow[draw=none, from=2-2, to=1-3]
\end{tikzcd}\]
where all morphisms labelled by $\sim$ are weak equivalences.

By two out of three, this implies that theses natural transformations induce weak equivalences on all globes, and proposition \ref{prop:criterimu_to_be_an_weak_equivalence} concludes the proof.
\end{proof}

\begin{prop}
\label{prop:modification_of_the_value_on_marked_representables}
 There exists a zigzag of weakly invertible natural transformations 
$$i\leftrightsquigarrow j$$
where $j$ is a left Quillen functor such that $j([n])=i([n])$ and $j([n]_t)=\tau^i_{n-1}i([n])$, and such that the image of $[n]\to [n]_t$ by $j$ is induced by the canonical morphism $id\to \tau^i_{n-1}(id)$.
\end{prop} 
\begin{proof}
We define $\tilde{i}$ (resp. $j$) to be the colimit preserving functor defined on representables by $\tilde{i}([n]):=i([n])$ and $\tilde{i}:=([n]_t)=\tau^i_{n-1}i([n]_t)$ (resp. $j([n]):=i([n])$ and $j([n]_t):=\tau^i_{n-1}i([n])$). We then have a zigzag of natural transformations 
$$i\xrightarrow{\sim} \tilde{i}\xleftarrow{\sim} j.$$
that are pointwise acyclic cofibrations according to \ref{lemma:weak characteroeiation 1}.
This implies that both $\tilde{i}$ and $j$ are left Quillen functors.
\end{proof}

\vspace{1cm}

In the following lemmas, we use the Steiner theory recalled in section \ref{section:Steiner thery}. 
\begin{lemma}
\label{lemma:unicity_of_composition}
Let $m$ be an integer and $X$ and $Y$ be two $\omega$-categories admitting a loop free and atomic basis. We denote by $0$, $1$ and $t$ the three points of $\Sigma X\vee[1]$.
 Let $$f: \Sigma^m ([X,1]\star Y)\to \Sigma^m( ([X,1]\vee[1])\star Y)$$ be a morphism fitting in the following diagram:
\[\begin{tikzcd}
	{\Sigma^m((\{0\}\coprod\{1\})\star Y)} && {\Sigma^m( ([X,1]\vee[1])\star Y)} \\
	{\Sigma^m([X,1]\star Y)} && {\Sigma^m([X,1]\star Y)}
	\arrow["f", from=2-1, to=1-3]
	\arrow["{\Sigma^m(g\star Y)}", from=1-1, to=1-3]
	\arrow[from=1-3, to=2-3]
	\arrow["id"', from=2-1, to=2-3]
	\arrow[from=1-1, to=2-1]
\end{tikzcd}\]
where $g$ sends $0$ on $0$, and sends $1$ on $t$ and the right vertical morphism induced by the retraction $[X,1] \vee[1]\to [X,1]$.

Then $f$ is $\Sigma^m(\triangledown\star Y)$. 
\end{lemma}
\begin{proof}
All these categories admit loop free and atomic basis. We can then show this lemma in the category of augmented directed complexes. Furthermore, in this category, the suspension only makes an index shift, so we can assume without loss of generality that $m=0$.

The commutativity of the diagram implies that 
$$
\begin{array}{rclr}
f(0\star x)&=& 0\star x\\
f(1\star x)&=&t\star x\\
f([x,1] \star y)&=& [x,1] \star y +r_{x ,y} &\\
\end{array}
$$
where $r_{x,y}$ is a positive sum of elements of $(B_{[1]\star Y})_{|x|+|y|+1}$.
We show by induction on $|x|+|y|$ that: 
$$\begin{array}{rcll}
r_{x,y}&=& [1]\star y&\mbox{ if $|x|= 0$}\\
&=&0&\mbox{ if $|x|> 0$} .\\
\end{array}$$

Suppose the result true when the sum of dimensions of $x$ and $y$ is $(k-1)$. Let $x, y$ be two cells such that $|x|+|y|=k$.
\textbf{Case $|x|=0$.} The commutativity of $f$ with $\partial$ and the induction hypothesis imply that 
$$\begin{array}{rcl}
\partial r_{x,y} &=& f(\partial ([x,1]\star y)) - \partial ([x,1]\star y)\\
 &=& \{t\}\star y - \{0\}\star y + f([x,1]\star \partial y) - \{1\}\star y + \{0\}\star y - [x,1]\star \partial y\\
 &=& \{t\}\star y - \{1\}\star y + [1]\star \partial y\\
\end{array}$$
and $r_{x,y}$ is then equal to $[1]\star y$. \textbf{Case $|x|>0$.} The commutativity of $f$ with $\partial$ implies that 
$$ \partial r_{x,y} = 0$$
and $r_{x,y}$ is then equal to $0$.
\end{proof}

\begin{lemma}
\label{lemma:unicity_of_composition 2}
Let $m$ be an integer and $X$ and $Y$ be two $\omega$-categories admitting a loop free and atomic basis. We denote by $0$, $1$ and $t$ the three points of $\Sigma X\vee[1]$. 
 Let $$f: \Sigma^m ([X,1]\star Y)\to \Sigma^m( ([X,1]\vee[1])\star Y)$$ be a morphism fitting in the following diagram:
\[\begin{tikzcd}
	{\Sigma^m(\{t\}\star Y)} & {\Sigma^m(\{1\}\star Y)} \\
	& {\Sigma^m(([X,1]\vee[1])\star Y)} & {\Sigma^m([X,1]\star Y)} \\
	{\Sigma^m([X,1]\star Y)}
	\arrow["f", from=2-2, to=2-3]
	\arrow["id"', curve={height=12pt}, from=3-1, to=2-3]
	\arrow[hook, from=3-1, to=2-2]
	\arrow["\cong", from=1-1, to=1-2]
	\arrow[from=1-2, to=2-3]
	\arrow[from=1-1, to=2-2]
\end{tikzcd}\]
Then $f$ is the morphism induced by the retraction $[X,1]\vee[1]\to [X,1]$.
\end{lemma}
\begin{proof}
The proof is an easy computation using Steiner theory, similar to the one done in lemma \ref{lemma:unicity_of_composition}, and left to the reader.
\end{proof}

\begin{definition}
\label{defi:def of C}
Let $C$ be the subcategory of marked simplicial sets whose
\begin{enumerate}
\item[$-$] objects are the marked simplicial sets $X$ such that $\R(X)$ has no non-trivial automorphisms, and such that there exists a (necessary unique) isomorphism $$\phi_X:\R(iX) \to\R(X),$$
\item[$-$] morphisms are the maps $f:X\to Y$ making the induced diagram
\[\begin{tikzcd}
	{\R(i(X))} & {\R(X)} \\
	{\R(i(Y))} & {\R(Y)}
	\arrow["{\phi_X}", from=1-1, to=1-2]
	\arrow["{\phi_Y}"', from=2-1, to=2-2]
	\arrow["{\R(f)}", from=1-2, to=2-2]
	\arrow["{\R(i(f))}"', from=1-1, to=2-1]
\end{tikzcd}\]
commutative.
\end{enumerate}
We recall that the functor $R:\mSset\to \ocat$ is defined in construction \ref{cons:Street nerve}.
\end{definition}

\begin{remark}
\label{rem:about_P_modified_3}
As $\R$ sends acyclic cofibrations to isomorphisms, $C$ is stable by zigzags of acyclic cofibrations. Moreover, as $\R$ and $i$ preserve colimits, for any diagram $F:I\to C$ such that the $\omega$-category $\R(\colim_IF)$ has no non-trivial automorphisms, $\colim_IF$ is in $C$. Eventually, the colimit of any natural transformation between two such diagrams is in $C$.
\end{remark}

\begin{lemma}
\label{lemma:about_P_modified_2}
Let $(k,n)$ be a couple of integers such that $k\leq n$.  We set the convention $[-1]:=\emptyset$.
For any integer $m$, the following assertion holds:
\begin{enumerate}
\item $\Sigma^m( \Sigma [n-k]_{\circ}\star[k-1])$ and $\Sigma^m( [k-1]_{\circ}\costar \Sigma [n-k])$ are in $C$.
\item For any $-1\leq l\leq k-1$ and $0\leq p\leq n-k$, and any monomorphisms $[l]\to [k-1]$ and $[p]\to [n-k]$, the morphisms
$$\Sigma^m( \Sigma [p]_{\circ}\star[l]) \to \Sigma^m( \Sigma [n-k]_{\circ}\star[k-1])~~\mbox{and}~~
\Sigma^m( [l]_{\circ} \costar \Sigma[p])\to \Sigma^m( [k-1]_{\circ} \costar \Sigma[n-k])$$
are in $C$.
\item For any $\epsilon\in \{0,1\}$, the morphisms
$$\Sigma^m( \{\epsilon\}\star[k-1]) \to \Sigma^m( \Sigma [n-k]_{\circ}\star[k-1])~~\mbox{and}~~
\Sigma^m( [k-1]_{\circ} \costar \{\epsilon\})\to \Sigma^m( [k-1]_{\circ} \costar \Sigma[n-k])$$
are in $C$.
\item If $k>0$, the morphisms
$$\Sigma^m( \emptyset\star[k-1]) \to \Sigma^m( \Sigma [n-k]_{\circ}\star[k-1])~~\mbox{and}~~
\Sigma^m( [k-1]_{\circ} \costar \emptyset)\to \Sigma^m( [k-1]_{\circ} \costar \Sigma[n-k])$$
are in $C$.
\end{enumerate}
\end{lemma}
\begin{proof}
We will proceed by induction on $(k,n)$.

- The case $(0,0)$ corresponds to the belonging of globes to $C$, which is true by the assumptions we made on the functor $i$ and by the proposition \ref{prop:the globes a non non trivial automorphisms} that assert that the globes have no non-trivial automorphisms.

- We now suppose that the case $(n-1,n-1)$ holds and we are willing to show the case $(0,n)$. The assertions $(1)$ and $(2)$ are direct consequences of the case $(n-1,n-1)$ after remarking the isomorphisms:
$$\Sigma^m \Sigma [n]\cong \Sigma^{m+1}((\Sigma[0]_{\circ}) \star[n-2])~~~~~
\Sigma^m \Sigma [n]_{\circ}\cong \Sigma^{m+1}([n-2]_{\circ}\costar(\Sigma[0]))
$$
It remains to show the third assertion.  Let $m$ be any integer and $\epsilon\in \{0,1\}$. 
By induction hypothesis and by the belonging of globes to $C$, the following morphism
$$\Sigma^m( \{\epsilon\})\to \Sigma^m(\Sigma \{0\})\cong \Sigma^{m+1}\{0\}\to \Sigma^{m+1}((\Sigma[0]_{\circ}) \star[n-2])\cong\Sigma^m \Sigma [n]$$
is in $C$. As the morphism $\Sigma^m( \{\epsilon\})\to \Sigma^m \Sigma [n]$ is their composite, it belongs to $C$. We proceed similarly to show that
$\Sigma^m( \{\epsilon\})\to \Sigma^m \Sigma [n]_{\circ}$ belongs to $C$. This concludes the proof of the case  $(0,n)$.

- Suppose the result true for the couples $(k-1,n)$, $(k-1,n-1)$ and $(k-1,k-1)$ for an integer $k$ strictly superior to $0$ and inferior or equal to $n$. We are willing to show the case $(k,n)$. Let $m$ be any integer.

As $R$ commutes with Gray operations and pushouts, the lemma \ref{lemma:non trivial automorphisme 4} implies that 
$\Sigma^m ( (\Sigma[n-k]_{\circ}\coprod_{[0]}[1])\star [k-2])$
together with all the objects appearing in the statement of this lemma are sent by $\R$ to $\omega$-categories with loop free and atomic basis and admitting no non-trivial automorphisms. Remark \ref{rem:about_P_modified_3} implies that for one of these objects (resp. a morphism between them) to belong to $C$, it is sufficient to show that it is linked by a zigzag of acyclic cofibrations to the colimit, computed in $\mSset$, of a diagram with value in $C$ (resp. in the arrow category of $C$).

As $\Sigma [0]_\circ=[1]$, the case $(k-1,k-1)$ implies that the morphism 
$$\Sigma^m(\{0\}\star [k-1])\to \Sigma^m([1]\star [k-1])$$
is in $C$. Combined with the case $(k-1,n-1)$, this implies that the diagram
\[\begin{tikzcd}
	{\Sigma^m ( (\Sigma[n-k]_{\circ})\star [k-2])} & {\Sigma^m ( (\Sigma[n-k]_{\circ})\star [k-2])} \\
	{\Sigma^m ( [0]\star [k-2])} & {\Sigma^m ( [0]\star [k-2])} \\
	{\Sigma^m ( [0]\star [k-2])} & {\Sigma^m ( [1]\star [k-2])}
	\arrow[from=2-2, to=1-2]
	\arrow[from=2-2, to=3-2]
	\arrow[from=1-1, to=1-2]
	\arrow["id"', from=2-1, to=3-1]
	\arrow[from=2-1, to=1-1]
	\arrow["id", from=2-1, to=2-2]
	\arrow[from=3-1, to=3-2]
\end{tikzcd}\]
is in $C$, and so is it's vertical colimits. 
As the codomain is weakly equivalent to $\Sigma^m ( (\Sigma[n-k]_{\circ}\fwedge[1])\star [k-2])$, this implies that $C$ includes the canonical morphism
\begin{equation}
\label{eq:eqlemmaunicity_of_composition1}
\Sigma^m ( (\Sigma[n-k]_{\circ})\star [k-2]) \hookrightarrow \Sigma^m ( (\Sigma[n-k]_{\circ}\fwedge[1])\star [k-2]).
\end{equation}
We can show similarly that the canonical morphism
\begin{equation}
\label{eq:eqlemmaunicity_of_composition1.5}
\Sigma^m ( [1]\star [k-2]) \hookrightarrow \Sigma^m ( (\Sigma[n-k]_{\circ}\fwedge[1])\star [k-2]).
\end{equation}
is in $C$.

The image by $\R$ of the canonical morphism 
$$\Sigma^m ( (\Sigma[n-k]_{\circ}\fwedge[1])\star [k-2])\to \Sigma^m ( (\Sigma[n-k]_{\circ})\star [k-2]) $$
induced by the retraction $\Sigma[n-k]_{\circ}\fwedge[1]\to \Sigma[n-k]_{\circ}$ fulfills the condition of lemma \ref{lemma:unicity_of_composition 2} and then belongs to $C$. 
The lemma \ref{lemma:unicity_of_composition} then implies that the morphism 
\begin{equation}
\label{eq:eqlemmaunicity_of_composition2}
\Sigma^m ( \triangledown\star [k-2]):\Sigma^m ( (\Sigma[n-k]_{\circ})\star [k-2])\to \Sigma^m ( (\Sigma[n-k]_{\circ}\fwedge[1])\star [k-2])
\end{equation}
is in $C$. 
We will use freely in the rest of the proof that morphisms \eqref{eq:eqlemmaunicity_of_composition1}, \eqref{eq:eqlemmaunicity_of_composition1.5} and \eqref{eq:eqlemmaunicity_of_composition2} are in $C$.
 
Theorem \ref{theo:cyl_formula} implies that the object $\Sigma^m( \Sigma [n-k]_{\circ}\star[k-1])$ is linked by a zigzag of acyclic cofibrations to the colimit of 
$$
\Sigma^m ( (\Sigma[n-k]_{\circ}\fwedge[1])\star [k-2]) \leftarrow
\Sigma^m ( \Sigma[n-k]_{\circ}\star [k-2]) \to
\Sigma^m ( \Sigma[n-k+1]_{\circ}\star [k-2])
$$
and the induction hypothesis implies that it belongs to $C$. 
We proceed similarly to show that $\Sigma^m( [k-1]_{\circ}\costar \Sigma [n-k])$ belongs to $C$.

Let $0\leq l\leq k-1$ and $-1\leq p\leq n-k$ be two integers, and $f:[l]\to [k-1]$ and $g:[p]\to [n-k]$ be two monomorphisms. Suppose first that $f$ is of shape $[0]\star f'$ for $f':[l-1]\to [k-2]$.
In this case, $\Sigma^m( \Sigma [p]_{\circ}\star[l]) \to \Sigma^m( \Sigma [n-k]_{\circ}\star[k-1])$ is linked by a zigzag of acyclic cofibrations to the vertical colimit of the diagram
\[\begin{tikzcd}
	{ \Sigma^m ( (\Sigma[p]_{\circ}\fwedge[1])\star [l-1])} & {\Sigma^m ( (\Sigma[n-k]_{\circ}\fwedge[1])\star [k-2])} \\
	{ \Sigma^m ( \Sigma[p]_{\circ}\star [l-1])} & { \Sigma^m ( \Sigma[n-k]_{\circ}\star [k-2])} \\
	{ \Sigma^m ( \Sigma[p+1]_{\circ}\star [l-1])} & { \Sigma^m ( \Sigma[n-k+1]_{\circ}\star [k-2])}
	\arrow[from=2-2, to=1-2]
	\arrow[from=2-2, to=3-2]
	\arrow[from=2-1, to=2-2]
	\arrow[from=1-1, to=1-2]
	\arrow[from=2-1, to=1-1]
	\arrow[from=2-1, to=3-1]
	\arrow[from=3-1, to=3-2]
\end{tikzcd}\]
and the induction hypothesis implies that it belongs to $C$. 
Suppose now that $f$ avoids the initial object of $[k-1]$. In this case, the morphism $\Sigma^m( \Sigma [p]_{\circ}\star[l]) \to \Sigma^m( \Sigma [n-k]_{\circ}\star[k-1])$
 is linked by a zigzag of acyclic cofibrations to the vertical colimit of the diagram
\[\begin{tikzcd}
	{\Sigma^m( \Sigma [p]_{\circ}\star[l])} & {\Sigma^m ( (\Sigma[n-k]_{\circ})\star [k-2])} & {\Sigma^m ( (\Sigma[n-k]_{\circ}\fwedge[1])\star [k-2])} \\
	&& { \Sigma^m ( \Sigma[n-k]_{\circ}\star [k-2])} \\
	&& { \Sigma^m ( \Sigma[n-k+1]_{\circ}\star [k-2])}
	\arrow[from=2-3, to=1-3]
	\arrow[from=2-3, to=3-3]
	\arrow[hook, from=1-2, to=1-3]
	\arrow[from=1-1, to=1-2]
\end{tikzcd}\]
and the induction hypothesis implies that it belongs to $C$.
We prove similarly that $$\Sigma^m( [l]_{\circ} \costar \Sigma[p])\to \Sigma^m( [k-1]_{\circ} \costar \Sigma[n-k])$$ belongs to $C$.

The morphism $\Sigma^m( \{0\}\star[k-1]) \to \Sigma^m( \Sigma [n-k]_{\circ}\star[k-1])$  is linked by a zigzag of acyclic cofibrations to the vertical colimit of the diagram
\[\begin{tikzcd}
	& {\Sigma^m ( (\Sigma[n-k]_{\circ}\fwedge[1])\star [k-2])} \\
	& { \Sigma^m ( \Sigma[n-k]_{\circ}\star [k-2])} \\
	{\Sigma^m( \{0\}\star[k-1])\cong \Sigma^m ( (\Sigma\{n-k+1\})\star [k-2])} & { \Sigma^m ( \Sigma[n-k+1]_{\circ}\star [k-2])}
	\arrow[from=2-2, to=1-2]
	\arrow[from=2-2, to=3-2]
	\arrow[from=3-1, to=3-2]
\end{tikzcd}\]
and the induction hypothesis implies that it belongs to $C$.  The morphism $\Sigma^m( \{1\}\star[k-1]) \to \Sigma^m( \Sigma [n-k]_{\circ}\star[k-1])$   
is linked by a zigzag of acyclic cofibrations to the vertical colimit of the diagram
\[\begin{tikzcd}
	{\Sigma^m( \{1\}\star[k-1])\cong \Sigma^m ( [1]\star [k-2])} & {\Sigma^m ( (\Sigma[n-k]_{\circ}\fwedge[1])\star [k-2])} \\
	& { \Sigma^m ( \Sigma[n-k]_{\circ}\star [k-2])} \\
	& { \Sigma^m ( \Sigma[n-k+1]_{\circ}\star [k-2])}
	\arrow[from=2-2, to=1-2]
	\arrow[from=2-2, to=3-2]
	\arrow[hook, from=1-1, to=1-2]
\end{tikzcd}\]
and the induction hypothesis implies that it belongs to $C$.
We prove similarly that for any $\epsilon\in \{0,1\}$, $$\Sigma^m( [k-1]_{\circ} \costar \{\epsilon\})\to \Sigma^m( [k-1]_{\circ} \costar \Sigma[n-k])$$ belongs to $C$.

Eventually, the morphism $\Sigma^m( \emptyset\star[k-1]) \to \Sigma^m( \Sigma [n-k]_{\circ}\star[k-1])$  is 
is linked by a zigzag of acyclic cofibrations to the vertical colimit of the diagram
\[\begin{tikzcd}
	{\Sigma^m( \{1\}\star[k-2])} & { \Sigma^m ( [1]\star [k-2])} & {\Sigma^m ( (\Sigma[n-k]_{\circ}\fwedge[1])\star [k-2])} \\
	&& { \Sigma^m ( \Sigma[n-k]_{\circ}\star [k-2])} \\
	&& { \Sigma^m ( \Sigma[n-k+1]_{\circ}\star [k-2])}
	\arrow[from=2-3, to=1-3]
	\arrow[from=2-3, to=3-3]
	\arrow[hook, from=1-2, to=1-3]
	\arrow[from=1-1, to=1-2]
\end{tikzcd}\]
and the induction hypothesis implies that it belongs to $C$. We prove similarly that $$\Sigma^m( [k-1]_{\circ} \costar \emptyset)\to \Sigma^m( [k-1]_{\circ} \costar \Sigma[n-k])$$ belongs to $C$.

We have then proven the case $(k,n)$, and this concludes the proof.
\end{proof}

\begin{lemma}
\label{lemma:about_P_modified}
Let $F:\Delta\to \ocat$ be a functor and $\phi:F\to \R$ be a invertible transformation such that for any monomorphism $i:[k]\to [n]$, the induced square
\[\begin{tikzcd}
	{F([k])} & {\R([k])} \\
	{F([n])} & {\R([n])}
	\arrow["{\phi_{[k]}}", from=1-1, to=1-2]
	\arrow["{\phi_{[n]}}"', from=2-1, to=2-2]
	\arrow["{R(i)}", from=1-2, to=2-2]
	\arrow["{F(i)}"', from=1-1, to=2-1]
\end{tikzcd}\]
commutes. Then $\phi$ is an invertible natural transformation between $F$ and $\R$.
\end{lemma}
\begin{proof}
We can suppose without loss of generality that for all integer $n$, $F([n])=\R([n])$. The hypotheses implies that for any monomorphism $i:[n]\to [m]$, $F(i)=\R(i)$ and it then remains to show that for any degeneracy $p:[n]\to [m]$, $F(p)=\R(p)$.

We proceed by induction and we then suppose that for any $0<k\leq n$ and any degeneracy $s:[k]\to [k-1]$, $F(s)=\R(s)$. As any  morphism of $\Delta$ factors as a degeneracy followed by a monomorphism, the induction hypothesis implies that for any $f:[k]\to [n]$ with $k\leq n$,  $F(f)=\R(f)$.

Let $s:[n+1]\to [n]$ be a degeneracy. 
We have a \textit{a priori} non commutative diagram:
\[\begin{tikzcd}[ampersand replacement=\&]
	{\colim_{[k]\underset{\neq id}{\hookrightarrow}[n+1] }\R([k])} \& {\colim_{[k]\underset{\neq id}{\hookrightarrow}[n+1] }\R([k])} \\
	{{\R}([n+1])} \& {{\R}([n+1])} \\
	{{\R}([n])} \& {{\R}([n])}
	\arrow["{F(s)}"', from=2-1, to=3-1]
	\arrow[from=1-1, to=2-1]
	\arrow[from=1-2, to=2-2]
	\arrow["{\R(s)}", from=2-2, to=3-2]
	\arrow[Rightarrow, no head, from=2-1, to=2-2]
	\arrow[Rightarrow, no head, from=1-1, to=1-2]
	\arrow[Rightarrow, no head, from=3-1, to=3-2]
\end{tikzcd}\]
The induction hypothesis implies that the outer and the upper square commute. As $R$ commutes with colimits, $\colim_{[k]\to\partial [n]}\R([k])$ is equivalent to $\R(\partial[n])$. Moreover, the inclusion $\R(\partial[n])\to \R([n])$ induces an isomorphisms on cells of dimension lower or equal to $n$. 
For the lower square to commutes, we then only have to check that the top cell of $\R([n+1])$ is sent on the same element on ${\R}([n])$. That is the case because the two paths send it to an unity as there is no non trivial $(n+1)$-cells in $\R([n])$. 

We then have $F(s)=\R(s)$, which concludes the induction and then the proof.
\end{proof}

\begin{prop}
\label{prop:existence_of_comparaison_with_street}
There exists an invertible natural transformation $\R i\to \R$.
\end{prop} 
\begin{proof}
As $\Sigma[0]_{\circ}$ is isomorphic to $[1]$, the case $(n,n)$ for any integer $n$ of the lemma \ref{lemma:about_P_modified_2} imply that there exists an invertible transformation $\phi:(\R i)_{|\Delta}\to \R_{|\Delta}$ which is natural when restricted to the full subcategory of $\Delta$ whose morphisms are the monomorphisms. 

The lemma \ref{lemma:about_P_modified} then implies that $\phi:(\R i)_{|\Delta}\to \R_{|\Delta}$ is natural. We can extend it to a natural transformation $\phi':(\R i)_{|t\Delta}\to \R_{|t\Delta}$ thanks to the proposition \ref{prop:modification_of_the_value_on_marked_representables}. 

Eventually, as both $\R i$ and $\R$ preserves colimits, we can extend $\phi'$ to a invertible natural transformation between $\R i$ and $\R$.
\end{proof}

\begin{theorem}
\label{theo:criterion_to_be_linked_to_identity}
Let $i: \mSset\to \mSset$ be a left Quillen functor. Suppose that there exists a zigzag of weakly invertible natural transformations:
$$i(\Db_{\uvar}) \leftrightsquigarrow \Db_{\uvar}.$$
Then, there exists a zigzag of weakly invertible natural transformations between $i$ and $id$. In particular, $i$ is a left Quillen equivalence.
\end{theorem} 
\begin{proof} 
The proposition \ref{prop:existence_of_comparaison_with_street} implies that we have a natural transformation $\psi:i\to i_{str}$. 
Furthermore, hypotheses imply that this natural transformation is a weak equivalence on globes. According to proposition \ref{prop:criterimu_to_be_an_weak_equivalence}, $\psi$ is then a weakly invertible natural transformation.
We then have a zigzag of weakly invertible natural transformations: 
$$i\xrightarrow{\sim} i_{str}\xleftarrow{\sim} id.$$
\end{proof}

\begin{cor}
\label{cor:criterion_to_be_linked_to_identity_case stratified}
Let $i: \stratSset\to \stratSset$ be a left Quillen functor. Suppose that there exists a zigzag of weakly invertible natural transformations:
$$i(\Db_{\uvar}) \leftrightsquigarrow \Db_{\uvar}.$$
Then, there exists a zigzag of weakly invertible natural transformations between $i$ and $id$. In particular, $i$ is a left Quillen equivalence.
\end{cor} 
\begin{proof} 
We recall that the adjunction between stratified and marked simplicial sets is denoted by:
\[\begin{tikzcd}
	{(\uvar)_{\mk}:\stratSset} & {\mSset:\iota}
	\arrow[""{name=0, anchor=center, inner sep=0}, shift left=2, from=1-1, to=1-2]
	\arrow[""{name=1, anchor=center, inner sep=0}, shift left=2, from=1-2, to=1-1]
	\arrow["\dashv"{anchor=center, rotate=-90}, draw=none, from=0, to=1]
\end{tikzcd}\]
The proposition \ref{prop:model structure on marked presheaves} states that this adjunction is a Quillen equivalence and that the functor $\iota$ preserves acyclic cofibrations.

Remark now that the functor $(\uvar)_{\mk}\circ i\circ \iota:\mSset\to \mSset$ verifies the hypothesis of theorem \ref{theo:criterion_to_be_linked_to_identity} and we then have a zigzag of of weakly invertible natural transformations:
$$(\uvar)_{\mk}\circ i\circ \iota  \leftrightsquigarrow id$$
This induces a zigzag of of weakly invertible natural transformations:
$$i\to \iota\circ(\uvar)_{\mk}\circ i\circ \iota \circ (\uvar)_{\mk} \leftrightsquigarrow \iota\circ(\uvar)_{\mk}\leftarrow id$$
\end{proof}


\chapter{Complicial sets as a model of $\io$-categories}	
\label{chapter:complicial set as a model of io categories}

\minitoc
\vspace{1cm}
%
%
%
%
%
%
%
%
%
%
%

Let $n\in \Nb\cup\{\omega\}$. We will denote $\incat{n}$ (resp. $\incat{n}^{\nc}$) the $(\infty,1)$-category of \textit{(resp. non-complete) $(\infty,n)$-categories}, i.e., of complete Segal spaces (resp. Segal spaces) on $\Theta_n$. Following the terminology of Barwick and Schommer-Pries (\cite{Barwick_on_the_unicity_of_the_theory_of_higher_categories}), we call a \textit{model of (resp. non-complete) $(\infty,n)$-categories} any model category whose corresponding $(\infty, 1)$-category is $\incat{n}^{\nc}$ (resp. $\incat{n}$).

There are many model categories corresponding to these $(\infty,1)$-categories; see, for instance, \cite{Ara_Higher_quasi_cat}, \cite{Bergner_Comparison_of_model_of_infini_n_categories}, \cite{Bergner_Comparison_of_model_for_infini_n_categories_II}, \cite{Bergner_reedy_category_and_the_theta_construction} (we refer to \cite{Barwick_on_the_unicity_of_the_theory_of_higher_categories} for a comprehensive presentation of these models and their equivalence). For example, one can mention $n$-fold Segal spaces and Simpson's and Tamsamani's Segal $n$-categories, among others.

It was conjectured (\cite{Street_algebra_of_orianted_simplexes}, \cite{Verity_a_complicial_compendium}, \cite{Barwick_on_the_unicity_of_the_theory_of_higher_categories}) that Verity's $n$-complicial sets were also a model of $(\infty,n)$-categories. This would imply that Campion-Kapulkin-Maehara's $n$-comical sets also are, as they are shown to be Quillen equivalent to $n$-complicial sets in \cite{Doherty_Equivalence_of_cubical_and_simplicial_approaches}.

The case $n=2$ follows from results by Gagna, Harpaz, and Lanari (\cite{Gagna_on_the_equivallence_of_all_model_for_infini2_cat}). The purpose of this chapter is to demonstrate this conjecture for any $n\in \Nb\cup\{\omega\}$.

To this end, we first address the more general problem of finding sufficient conditions on a model category $A$ to build a \textit{Gray cylinder} $C\mapsto I\otimes C$ and a \textit{Gray cone} $C\mapsto e\star C$ on stratified Segal precategories enriched in $A$. These two operations should be linked by the following homotopy cocartesian square
\[\begin{tikzcd}
	{\{0\}\otimes C} & {I\otimes C} \\
	e & {e\star C}
	\arrow[from=1-2, to=2-2]
	\arrow[from=1-1, to=2-1]
	\arrow[from=2-1, to=2-2]
	\arrow[from=1-1, to=1-2]
\end{tikzcd}\]
where $e$ is the terminal object. The conditions that $A$ has to fulfill are encapsulated in the notion of a \textit{Gray module} (paragraph \ref{defi:Gray module}). Thanks to the Gray cylinder and cone, we can show the following theorem:

\begin{itheorem}[\ref{theo:complicialGray module}]
If $A$ is a complicial Gray module, there is a Quillen adjunction between the model structure for complicial sets on stratified simplicial sets and stratified Segal precategories enriched in $A$, where the left adjoint sends $[n]$ to $e \star e \star \ldots \star e \star \emptyset$.
\end{itheorem}

By iterating the previous construction, we will be able to construct a Quillen adjunction between a model of (resp. non-complete) $(\infty,n)$-categories and stratified simplicial sets. We will then demonstrate the conjecture:

\begin{itheorem}[\ref{theo:theorem model 1}]
Let $n\in \mathbb{N}\cup \{\omega\}$. The model structure for $n$-complicial sets is a model of non-complete $(\infty,n)$-categories. The model structure for saturated $n$-complicial sets is a model of $(\infty,n)$-categories.
\end{itheorem}

\section{Preliminaries }

\subsection{Segal $A$-precategories}

\begin{definition}
\label{definition:admissible}
A nice model structure on a category $A$ of stratified presheaves on a category $B$ is \textit{admissible}\index[notion]{admissible model category} if 
\begin{enumerate}
\item $B$ is an elegant Reedy category (definition \ref{defi:reedy})
\item the terminal element $e$ is representable.
\item for any representable $a$, the morphism $a^\sharp\to e$ is a weak equivalence, where $(\_)^{\sharp}$ is the functor defined in construction \ref{construction: definition of sharp}
\end{enumerate}
\end{definition}

Through this section and the one that follows, we fix an admissible category $A$ of stratified presheaves on a category $B$.

\begin{definition}
We have an adjunction 
\begin{equation}
\label{eq:ob adj}
\begin{tikzcd}
	{\iota:\Set} & {A:ob}
	\arrow[""{name=0, anchor=center, inner sep=0}, shift left=2, from=1-1, to=1-2]
	\arrow[""{name=1, anchor=center, inner sep=0}, shift left=2, from=1-2, to=1-1]
	\arrow["\dashv"{anchor=center, rotate=-90}, draw=none, from=0, to=1]
\end{tikzcd}
\end{equation}
where the left adjoint sends a set $S$ onto $\coprod_S e$ and the right adjoint is the evaluation at $e$. The objects lying in the image of $\iota$ are called \notion{discrete objects}.
\end{definition}

\begin{definition}
An object $C$ of $\Fun( \Delta^{op},A)$ is a \notion{Segal $A$-precategory} if $C_0$ is discrete. We denote by \wcnotation{$\Seg(A)$}{(seg@$\Seg(A)$} the full subcategory of $\Fun(\Delta^{op},A)$ spanned by the Segal $A$-precategories.
\end{definition}

\begin{construction}
\label{defi:segal technique}
Let $a$ be an object of $A$ and $n$ an integer. We denote by $|[a,n]|$ the object of $\Fun(\Delta^{op},A)$ whose value on $m$ is $a\times \iota (\Hom_\Delta([m],[n]))$. This assignment defines a functor 
$$\begin{array}{ccc}
A\times \Delta &\to& \Fun(\Delta^{op},A)\\
(a,[n])&\mapsto & |[a,n]|
\end{array}$$
We define the Segal $A$-precategory \wcsnotation{$[a,n]$}{((g10@$[a,n]$}{for $A$-Segal precategories} as the pushout: 
\[\begin{tikzcd}
	{\underset{k\leq n}{\cup}{|[a,\{k\}]|}} & {|[a,n]|} \\
	{|[e,0]|} & {[a,n]}
	\arrow[from=1-1, to=1-2]
	\arrow[from=1-1, to=2-1]
	\arrow[from=2-1, to=2-2]
	\arrow[from=1-2, to=2-2]
	\arrow["\lrcorner"{anchor=center, pos=0.125, rotate=180}, draw=none, from=2-2, to=1-1]
\end{tikzcd}\]
The object $[e,0]$, which is the terminal Segal $A$-precategory, is once again denoted $e$.

The assignment $(a,n)\mapsto [a,n]$ induces by left Kan extension a unique functor 
$$[\uvar,\uvar]:A\times \Sset \to \Seg(A).$$ By construction, for any $K\in \Sset$, the functor $[\uvar,K]\to \Seg(A)_{[\emptyset,K]/}$ preserves colimits.
Finally, note that the image of this functor is dense in $\Seg(A)$.
\end{construction}

\begin{construction}
For $\{n_i\}_{i\leq k}$ and $\{a\to a_i\}_{i\leq k}$ two finite sequences, we denote by \wcsnotation{$[a_0,n_0]\vee[a_1,n_1]\vee...\vee [a_k,n_k]$}{((g20@$[a_0,n_0]\vee[a_1,n_1]\vee...\vee [a_k,n_k]$}{for Segal $A$-precategories} the Segal $A$-precategory fitting in the following pushout:
\[\begin{tikzcd}
	{\amalg_{i\leq k}[a,n_i]} & {[a,\Sigma_{i\leq k}n_i]} \\
	{\amalg_{i\leq k}[a_i,n_i]} & {[a_0,n_0]\vee[a_1,n_1]\vee...[a_k,n_k]}
	\arrow[""{name=0, anchor=center, inner sep=0}, from=1-1, to=1-2]
	\arrow[from=1-2, to=2-2]
	\arrow[from=1-1, to=2-1]
	\arrow[from=2-1, to=2-2]
	\arrow["\lrcorner"{anchor=center, pos=0.125, rotate=180}, draw=none, from=2-2, to=0]
\end{tikzcd}\]

The case we will use the most is that of the Segal $A$-precategories $[e,1]\vee[a,n]$ and $[a,n]\vee[e,1]$ corresponding to the sequence $((1,n),(a\to e,a\to a))$ and $((n,1),(a\to a,a\to e))$.
\end{construction}

\begin{definition}
\label{defi:defi of delta[B]}
Let $B$ be the Reedy category and $M$ the subset of objects of $B$ such that $A$ is the category of $M$-stratified presheaves on $B$. Remark that we have a fully faithful functor $\Delta[B]\to \Seg(A)$ where the first category is defined in \ref{defi: of Delta[..]}. We define $\Delta[M]$ as the set of objects of shape $[b,n]$ for $b\in M$ and $n>0$. 
We can easily check that the category $\Seg(A)$ is the category of $\Delta[M]$-stratified presheaves on $\Delta[B]$.

A cellular model for $\stratSeg(A)$ is given by the set of morphisms $[b,\partial n]\cup [a,n]\to [b,n]$ for $n$ an integer, and $a\to b$ a generating cofibration of $A$.

Eventually, for any Segal $A$-precategory $C$, we have an isomorphism $$C\cong \colim_{\Delta[tB]_{/C}}[b,n].$$

Following the definition \ref{defi:entire}, a morphism between Segal precategories is \textit{entire} if it is the identity on the underlying $\Delta[B]$-presheaves.
\end{definition}

\subsection{Stratified Segal $A$-precategories}

\begin{definition} A \notion{stratified Segal $A$-precategory} is a pair $(C,tC)$ where $C$ is a Segal $A$-precategory and $tC$ is a subset of $ob(C_1)$ that factors $s^0: C_0\to ob(C_1)$. A \textit{morphism of stratified Segal $A$-precategories} $(C,tC)\to (D,tD)$ is the data of a morphism $f:C\to D$ such that $f(tC)\subset tD$. The category of stratified Segal $A$-precategories is denoted by \wcnotation{$\stratSeg(A)$}{(tseg@$\stratSeg(A)$}. 
We have a canonical fully faithful functor $\Seg(A)\to \stratSeg(A)$ sending $C$ onto $(C,Im(s^0))$. We will identify Segal $A$-precategories with their image by this functor.
\end{definition}

\begin{definition}
We define $[e,1]_t:= ([e,1],[e,1]_1)$. The subcategory of objects of shape $[a,n]$ or $[e,1]_t$ is then dense in $\stratSeg(A)$.
\end{definition}

\begin{definition}
Let $B$ be the Reedy category and $M$ the subset of objects of $B$ such that $A$ is the category of $M$-stratified presheaves on $B$. We recall that we defined the category $\Delta[B]$ and the set of morphisms $\Delta[M]$ in definition \ref{defi:defi of delta[B]}. We set $t\Delta[M]$ as the union of $\Delta[M]$ and the singleton $\{[e,1]_t\}$. We can easily check that the category $\stratSeg(A)$ is the category of $t\Delta[M]$-stratified presheaves on $\Delta[B]$.
\end{definition}

\begin{remark}
The set of generating cofibrations for $\stratSeg(A)$ then consists of morphisms of shape $[e,1]\to [e,1]_t$ or $[a,n]\cup[b,\partial n]\to [b,n]$ where $a\to b$ is a generating cofibration of $A$. \sym{((g21@$[e,1]_t$}	
For any stratified Segal $A$-precategory $C$, we then have an isomorphism 
$$C\cong \colim_{t\Delta[tB]_{/C}}\uvar.$$
where $t\Delta[tB]$ is the full subcategory of $\stratSeg(A)$ whose objects are in $\Delta[B]$ or $t\Delta[M]$.
\end{remark}

\begin{remark}
As stratified $A$-segal categories is a category of stratified presheaves, the construction \ref{construction: definition of sharp} induces an endofunctor $$(\uvar)^\sharp:\stratSeg(A)\to \stratSeg(A).$$  \index[notation]{((b61@$(\uvar)^{\sharp}$!\textit{for stratified Segal $A$-precategories}}
\end{remark}

\begin{construction}
\label{defi:segal technique new}
By setting $[a,[1]_t]:=[a,1]^{\sharp}$ and $[a,[n]_t]:=[a,n]$ for $n>1$, we can extend the functor given in the construction \ref{defi:segal technique} to a functor 
$$[\uvar,\uvar]:A\times \tPsh{\Delta}\to\stratSeg(A)$$

Given $a$ in $A$ and $K\in \tPsh{\Delta}$, the two functors $[a,\uvar]:\tPsh{\Delta}\to \stratSeg(A)$ and $[\uvar,K]:\tPsh{\Delta}\to \stratSeg(A)_{[\emptyset,K]/}$ preserve colimits.
\end{construction}

\begin{prop}
\label{prop:model structure on stratified Segal category}
There exists an admissible nice model structure on $\stratSeg(A)$, called the \snotion{Segal model structure}{on $\stratSeg(A)$}, and denoted \wcnotation{$\stratSeg(A^\seg)$}{(tsegseg@$\stratSeg(A)^\seg$}, whose weak equivalences are the smallest precocomplete class containing 
\begin{enumerate}
\item for any weak equivalence $a\to b$ of $A$, the morphism $[a,1]\to [b,1]$,
\item for any object $a$ of $A$, and any integer $n$, the morphism $[a,\Sp_n]\to [a,n]$,
\item the morphism $[e,1]_t\to e$.
\end{enumerate}

The model category $\stratSeg(A)^\seg$ admits a left Bousfield localization along the morphism $[e,E^{eq}]\to [0]$. It is denoted \wcnotation{$\stratSeg(A^\cseg)$}{(tsegcseg@$\stratSeg(A)^\cseg$},, and called the \textit{complete Segal model structure}.
\end{prop}

\begin{proof}
Let $\alpha:\Psh{\Delta}\to \Psh(t\Delta[B])$ be the left adjoint sending $[n]$ onto $[e,n]^\sharp$, and let $J$ be the union of the morphisms in $(1)$, $(2)$, or $(3)$. The model structure $\stratSeg(A)^{\seg}$ is defined as the one obtained by applying the theorem \ref{theo:free_model_structure_on_marking} to $\alpha$ and the set $J$. According to the theorem \textit{op. cit.}, the set of weak equivalences, denoted $W$, is the smallest precocomplete class of morphisms that contains
\begin{enumerate}
\item for any weak equivalence $a\to b$ of $A$, the morphism $[a,1]\to [b,1]$,
\item for any object $a$ of $A$, and any integer $n$, the morphism $[a,\Sp_n]\to [a,n]$,
\item[$(3)'$] for any integer $n$, the morphism $[e,1]_t\times [e,n]^{\sharp}\to [e,1]^{t}$,
\item[$(4)'$] for any integer $n$, and for any $b\in tB$, the morphism
$[b,n]\times [e,n]^{\sharp}\to [b,n]$
to weak equivalences.
\end{enumerate}
We denote by $W'$ the smallest precocomplete class of morphisms that contains $J$. We then directly have $W'\subset W$.

We will now show that the inclusion $W\subset W'$ holds. By construction, $[b,[1]_t]:=[b,1]^{\sharp}$ fits in the pushout
\[\begin{tikzcd}
	{\coprod_{ob(b)}[e,1]} & {[b^\sharp,1]} \\
	{\coprod_{ob(b)}[e,1]_t} & {[b,1]^\sharp}
	\arrow[from=1-1, to=1-2]
	\arrow[from=1-1, to=2-1]
	\arrow[from=1-2, to=2-2]
	\arrow[from=2-1, to=2-2]
\end{tikzcd}\]
As $A$ is admissible (definition \ref{definition:admissible}), the morphism $[b,1]^\sharp \to [e,1]^\sharp\cong[e,1]_t\to [0]$ is then in $W'$. The proposition \ref{prop:characterization of morphism depuis strratsset.} then implies that the functor 
$$[b,\uvar]:=\tPsh{\Delta}^1\to \stratSeg(A)$$
sends weak equivalences to $W'$. In particular, the morphism $[b,[n]\times [m]^\sharp]\to [b,[n]]$ is in it. 

Remark now that we have a pushout 
\[\begin{tikzcd}
	{\coprod_{k\leq n}[b,\{k\}\times [m] ^\sharp]} & {[b,[n]\times [m]^\sharp]} \\
	{\coprod_{k\leq n}[e,\{k\}\times [m] ^\sharp]} & {[b,n]\times [e,m]^\sharp}
	\arrow[from=1-1, to=1-2]
	\arrow[from=1-1, to=2-1]
	\arrow[from=1-2, to=2-2]
	\arrow[from=2-1, to=2-2]
\end{tikzcd}\]
The top horizontal morphism is a monomorphism, and the left vertical morphism is in $W'$, so is the right vertical one. By two out of three, this implies that $[b,n]\times [e,n]^{\sharp}\to [b,n]$ is in $W'$. We show similarly that $[e,1]_t\times [e,n]^{\sharp}\to [e,1]_{t}$ is in $W'$, and we then have $W\subset W'$ and so $W=W'$. 

It remains to check that this nice model structure is admissible. The condition $(1)$ of definition \ref{definition:admissible} follows from proposition \ref{prop:delta[B] is reedy}, we already showed that during the proof that the condition $(2)$ was fulfilled, and the last one is true as $[0]$ is both representable and the terminal object of $\stratSeg(A)$.

The case of the model structure $\stratSeg(A)^{\seg}$ is handled similarly by replacing $J$ with $J\cup \{ \{0\}\to [e,E^{eq}]\}$.
\end{proof}

\begin{notation}
When speaking about \textit{acyclic cofibrations} or \textit{weak equivalences} of stratified Segal $A$-precategories, we will always refer implicitly to the Segal model structure. We will speak of \textit{complete weak equivalences} and \textit{complete acyclic cofibrations} when we are working in the complete Segal model structure.
\end{notation}

\subsection{Models of $(\infty,n)$-categories}

\begin{construction}
Given two small categories $A$, $B$, we will denote 
$$ \uvar\boxdot\uvar:\Psh{A}\times \Psh{B}\to \Psh{A\times B}$$
the functor sending two presheaves $X,Y$ to the presheaf satisfying:
$$( X\boxdot Y)(a,b):= X(a)\times Y(b).$$
\end{construction}

\begin{prop}
\label{prop:model structure on theta delta}
Let $n\in \Nb\cup \{\omega\}$. There exists a nice model structure on $\Psh{\Theta_n\times\Delta}$, called the \snotion{Segal model structure}{on $\Psh{\Theta_n\times\Delta}$}, and denoted \wcnotation{${\Psh{\Theta_n\times\Delta}}^{\seg}$}{(pshThetanseg@${\Psh{\Theta_n\times\Delta}}^{\seg}$}, whose weak equivalences are the smallest precocomplete class containing
\begin{enumerate}
\item for an $a\in \Theta_n$ and integer $n$, the morphism $a\boxdot [n]\to a\boxdot[0]$,
\item for any $a\to b$ in $\W_n$, the morphism $a\boxdot [0]\to b\boxdot [0]$.
\end{enumerate} 

The model category $\Psh{\Theta_n\times\Delta}^{\seg}$ admits a left Bousfield localization along the set of morphisms 
$\{\Sigma^kE^{eq}\to \Sigma^k[0],~k<n\}$. It is called the \snotion{complete Segal model structure}{on $\Psh{\Theta_n\times\Delta}$}, and denoted \wcnotation{${\Psh{\Theta_n\times\Delta}}^{\cseg}$}{(pshThetancseg@${\Psh{\Theta_n\times\Delta}}^{\cseg}$}.
\end{prop}

\begin{proof}
The inclusion $\{[0]\}\times \Delta\to \Theta\times \Delta$ induces a left Quillen functor $\alpha:\Psh{\Delta}\to \Psh{\Theta\times \Delta}$. The (resp. complete) Segal model structure is obtained by applying theorem \ref{theo:free_model_structure} to $\alpha$ and the set of morphisms $\W_n$ (resp. $\W_n\cup\W^\sat_n$).
\end{proof}

\begin{definition}
\label{defi: of pi0 of theta}
We have an adjunction
\[\begin{tikzcd}
	{\pi_0:\Psh{\Theta_n\times \Delta}} & {\ncat{n}:\iota\boxdot[0]}
	\arrow[""{name=0, anchor=center, inner sep=0}, shift left=2, from=1-1, to=1-2]
	\arrow[""{name=1, anchor=center, inner sep=0}, shift left=2, from=1-2, to=1-1]
	\arrow["\dashv"{anchor=center, rotate=-90}, draw=none, from=0, to=1]
\end{tikzcd}\]
where $\pi_0$ sends $a\boxdot[n]$ to $a$, and $\iota:\ncat{n}\to \Psh{\Theta_n}$ is the right adjoint of the localization functors given in theorem \ref{theo:theta and ocat}. Moreover, note that the functor $\pi_0$ sends weak equivalences of $\Psh{\Theta_n\times \Delta}^\seg$ to isomorphisms.

Composing with the localization $\ncat{n}\to \ncat{n}^{\comp}$, this induces another adjunction
\[\begin{tikzcd}
	{\Psh{\Theta_n\times \Delta}} & {\ncat{n}^{\comp}}
	\arrow[""{name=0, anchor=center, inner sep=0}, shift left=2, from=1-1, to=1-2]
	\arrow[""{name=1, anchor=center, inner sep=0}, shift left=2, from=1-2, to=1-1]
	\arrow["\dashv"{anchor=center, rotate=-90}, draw=none, from=0, to=1]
\end{tikzcd}\]
where the left adjoint sends weak equivalences of $\Psh{\Theta_n\times \Delta}^\cseg$ to isomorphisms. We recall that the category of complete $n$-categories $\ncat{n}^{\comp}$ is defined in \ref{defi:of complete cat}.
\end{definition}

\begin{definition}
Let $n\in \Nb\cup\{\omega\}$. A \textit{premodel of non-complete $(\infty,n)$-categories} (resp. a \textit{premodel of $(\infty,n)$-categories}) is a adjunction
$$\begin{tikzcd}
	{\pi_0:A} & {\ncat{n}:\pi_0^*}
	\arrow[""{name=0, anchor=center, inner sep=0}, shift left=2, from=1-1, to=1-2]
	\arrow[""{name=1, anchor=center, inner sep=0}, shift left=2, from=1-2, to=1-1]
	\arrow["\dashv"{anchor=center, rotate=-90}, draw=none, from=0, to=1]
\end{tikzcd}~~~~~\mbox{(resp.} \begin{tikzcd}
	{\pi_0:A} & {\ncat{n}^{\comp}:\pi_0^*}
	\arrow[""{name=0, anchor=center, inner sep=0}, shift left=2, from=1-1, to=1-2]
	\arrow[""{name=1, anchor=center, inner sep=0}, shift left=2, from=1-2, to=1-1]
	\arrow["\dashv"{anchor=center, rotate=-90}, draw=none, from=0, to=1]
\end{tikzcd}\mbox{)}$$
where $A$ admits a nice model structure and $\pi_0$ sends weak equivalences to isomorphisms. It is a \notion{model of non-complete $(\infty,n)$-categories} (resp. a \notion{model of $(\infty,n)$-categories}) if there exists a zigzag of Quillen equivalences over $\ncat{n}$ (resp. $\ncat{n}^\comp$) between $M$ and $\Psh{\Theta_n\times \Delta}^{\seg}$ (resp. $\Psh{\Theta_n\times \Delta}^\cseg$).

By abuse of language, we will often leave the functor $\pi_0$ implicit and just say that $A$ is a model of (resp. non-complete) $(\infty,n)$-categories.
\end{definition}

\begin{definition}
\label{defi:globular object in a model}
Let $\pi_0:A\to \ncat{n}$ be a premodel of non-complete $(\infty,n)$-categories. 
A \notion{globular object} is a functor $\alpha:\Db_{\uvar}:\Gb_{\leq n}\to M$, where $\Gb$ is the category of globes, and such that there exists a natural equivalence $\pi_0\alpha(\Db_n)\cong\Db_n$ so that the induced morphism $\alpha(\Db_n)\to\pi_0^*\Db_n$ is a weak equivalence.
\end{definition}

\begin{example}
A globular object for $\Psh{\Theta_n\times \Delta}$ is given by the functor $\Gb_{\leq n}\to \Theta_n\to \Psh{\Theta_n\times \Delta}$.
\end{example}

We now present a result strongly inspired by proposition 15.10 of \cite{Barwick_on_the_unicity_of_the_theory_of_higher_categories}.

\begin{prop}[Barwick--Schommer-Pries]
\label{prop:Shommer preis}
Let $n\leq m\leq \omega$.
Let $A$ be a model of (resp. non-complete) $(\infty,n)$-categories, $B$ a model of (resp. non-complete) $(\infty,m)$-categories, $\Gb^A_{\leq n}\to A$ and $\Gb^B_{\leq m}\to B$ two globular objects, and $j:A\to B$ a left adjoint that preserves, up to a zig-zag of weak equivalence, globes of dimension less than or equal to $n$.

Then, if the left Bousfield localization $B^{n}$ of $B$ along $\{\Gb^B_{m+1}\to \Gb^B_m,~m>n\}$ exists, the left Quillen adjoint $i:A\to B^n$ is a left Quillen equivalence. In particular, $B$ is a model of $(\infty,n)$-categories.
\end{prop}

\begin{proof}
We will use the theory of $(\infty,1)$-categories for this proof. We denote $\incat{m}^{\nc}$ the $(\infty,1)$-categories corresponding to any non-complete model of $(\infty,m)$-categories. This $(\infty,1)$-category then corresponds to the localization of $\iPsh{\Theta_m}$ along the set $\W_m$, where $\mathrm{Psh}^\infty$ denotes the $(\infty,1)$-categorical presheaves construction. We then have an adjunction 
\[\begin{tikzcd}
	{\pi_0:\incat{n}^{\nc}} & {\ncat{n}:\iota}
	\arrow[""{name=0, anchor=center, inner sep=0}, shift left=2, from=1-1, to=1-2]
	\arrow[""{name=1, anchor=center, inner sep=0}, shift left=2, from=1-2, to=1-1]
	\arrow["\dashv"{anchor=center, rotate=-90}, draw=none, from=0, to=1]
\end{tikzcd}\]
whose right adjoint is fully faithful. We will denote $i_n:\incat{n}\to \incat{m}$ the left Kan extension of the functor
$$\Theta_n\to \Theta_m\to \incat{m}.$$
The functor of the statement corresponds to a left adjoint $j:\incat{n}^{\nc}\to \incat{m}^{\nc}$ that preserves globes. The composite functor 
 $$\Theta_n\to  \incat{n}^{\nc}\to  \incat{m}^{\nc}\to  \ncat{m}$$
 then also preserves globes and \cite[lemma 4.10]{Barwick_on_the_unicity_of_the_theory_of_higher_categories} implies that it corresponds to the morphism 
 $$\Theta_n\to  \Theta_m\to   \ncat{m}.$$ 
 By adjunction, this induces a transformation $j(a)\to i_n(a)$ natural in $a:\Theta_n$. The hypothesis implies that it is an equivalence when evaluated in globes, and as $j$ sends spine inclusion to equivalence, this natural transformation is invertible when evaluated at any globular sum. This natural transformation is then invertible, and this  implies the result.

The proof for (complete) models of $(\infty,m)$-categories is similar.
\end{proof}

\vspace{1cm}
We will now introduce another model of (non-complete) $(\infty,n)$-categories, using the family of categories $\Xi_n$ defined in section \ref{section:Xi}, and the notion of marked presheaves exposed in section \ref{section:Marked and stratified presheaves}.

\begin{definition}
Let $n\in \Nb\cup\{\omega\}$. Let $M$ be the set of globes in $\Xi_n$, as defined in \ref{defi:globes for xi}. We will simply denote $\tPsh{\Xi_n}$ and $t\Xi_n$ the categories $\tPshM{\Xi_n}$ and $t_M\Xi_n$ from constructions \ref{defi:of pshMB} and \ref{construction: of tB}. 

Following definition \ref{defi:of pshMB}, the elements of $\Xi_n$ are either of the shape $a$ for $a\in \Xi_n$, or $(\Db_k)_t$ for $k\leq n$.
\end{definition}

\begin{remark}
\label{rem:xi and seg}
Remark that we have a canonical identification $\stratSeg(\tPsh{\Xi_n})\cong \tPsh{\Xi_{n+1}}$.
\end{remark}

\begin{prop}
\label{prop:model structure on Xiset}
Let $n\in \Nb\cup \{\omega\}$. There exists a nice model structure on $\tPsh{\Xi_n}$, called the \snotion{Segal model structure}{on $\tPsh{\Xi_n}$}, and denoted \wcnotation{${\tPsh{\Xi_n}}^{\seg}$}{(tsegseh@${\tPsh{\Xi_n}}^{\seg}$}, whose weak equivalences are the smallest precocomplete class containing
\begin{enumerate}
\item  the set $\M_n$ defined in \ref{definition: of Mn},
\item the morphisms $(\Db_{k+1})_t\to \Db_k$ for any $k< n$.
\end{enumerate}

The model category $\tPsh{\Xi_n}^{\seg}$ admits a left Bousfield localization along the set of morphisms 
$\{\Sigma^kE^{eq}\to \Sigma^k[0],~k<n\}$. It is called the \snotion{complete Segal model structure}{on $\tPsh{\Xi_n}$}, and denoted \wcnotation{${\tPsh{\Xi_n}}^{\cseg}$}{(tsegsei@${\tPsh{\Xi_n}}^{\cseg}$}.
\end{prop}

\begin{proof}
Remark that the objects of $\tPsh{\Xi_1}$ correspond to pairs $(X,tX)$ where $X$ is a simplicial set and $tX$ is a set of $1$-simplices containing degeneracies. By remark \ref{rem:on forgetting some marking}, the desired model structures on $\tPsh{\Xi_1}$ are the ones induced by $\tPsh{\Delta}^1$ and $\tPsh{\Delta}^1_\sat$, and proposition \ref{prop:characterization of morphism depuis strratsset.} implies that they enjoy the desired characterization. Using the identification given in remark \ref{rem:xi and seg}, the result for any $n<\omega$ follows by induction using proposition \ref{prop:model structure on stratified Segal category}.

Eventually, for the case $n=\omega$, we endow $\tPsh{\Xi}$ with the two model structures making the inclusion $\tPsh{\Xi_n}\to \tPsh{\Xi}$ left Quillen functors for the Segal and complete Segal model structures.
\end{proof}

\begin{remark}
The (resp. complete) Segal model structure on $\stratSeg(\tPsh{\Xi_n})$ corresponds through the identification given in remark \ref{rem:xi and seg} to the (resp. complete) Segal model structure on $\tPsh{\Xi_{n+1}}$.
\end{remark}

\begin{remark}
As the model structure $\tPsh{\Delta}^1$ is cartesian, and by the construction of the model structure on stratified Segal precategories, we have, for any object $X$ of $\tPsh{\Xi_n}$ and any integer $k$, a weak equivalence $X \times [k]^\sharp \to X$. Here, the functor $(\uvar)^\sharp$ is the one introduced in construction \ref{construction: definition of sharp}.
\end{remark}

\begin{definition}
\label{defi of pi0 for xi}
We will consider the adjunction
\[\begin{tikzcd}
	{\pi^\Xi_0:\tPsh{\Xi_n}} & {\ncat{n}:\iota^{\Xi}}
	\arrow[""{name=0, anchor=center, inner sep=0}, shift left=2, from=1-1, to=1-2]
	\arrow[""{name=1, anchor=center, inner sep=0}, shift left=2, from=1-2, to=1-1]
	\arrow["\dashv"{anchor=center, rotate=-90}, draw=none, from=0, to=1]
\end{tikzcd}\]
where $\pi_0^\Xi$ sends $a$ onto $j(a)$ and $(\Db_k)_t$ onto $\Db_{k-1}$ for any $k<n$, with $j$ being the functor defined in construction \ref{cons: of the functor $j$}. Remarks moreover that the functor $\pi_0^\Xi$ sends weak equivalences of $\tPsh{\Xi}^\seg$ to isomorphisms.

Composing with the localization $\ncat{n}\to \ncat{n}^{\comp}$, this induces another adjunction
\[\begin{tikzcd}
	{\tPsh{\Xi_n}} & {\ncat{n}^{\comp}}
	\arrow[""{name=0, anchor=center, inner sep=0}, shift left=2, from=1-1, to=1-2]
	\arrow[""{name=1, anchor=center, inner sep=0}, shift left=2, from=1-2, to=1-1]
	\arrow["\dashv"{anchor=center, rotate=-90}, draw=none, from=0, to=1]
\end{tikzcd}\]
where the left adjoint sends weak equivalences of $\tPsh{\Xi}^\cseg$ to isomorphisms.
\end{definition}

\begin{remark}
A globular object for $\tPsh{\Xi_n}$ is given by the functor $\Gb_{\leq n} \to \Xi_n \to \tPsh{\Xi_n}$ where the first functor is defined in \ref{defi:globes for xi}.
\end{remark}

\begin{construction}
\label{cons:between xi et theta delta}
Let $n\in \Nb\cup \{\omega\}$. 
We recall that the functor $j:\Xi_n\to \Theta_n$ is defined in \ref{cons: of the functor $j$}. We have an adjunction
\[\begin{tikzcd}
	{p:\Psh{\Theta_n\times \Delta}} & {\tPsh{\Xi_n}}
	\arrow[""{name=0, anchor=center, inner sep=0}, shift left=2, from=1-1, to=1-2]
	\arrow[""{name=1, anchor=center, inner sep=0}, shift left=2, from=1-2, to=1-1]
	\arrow["\dashv"{anchor=center, rotate=-90}, draw=none, from=0, to=1]
\end{tikzcd}\]
where the left adjoint sends $a\boxdot [n]$ to $j^*a\times[n]^\sharp$. 

Remark that the functor $p$ induces a commutative triangle:
\[\begin{tikzcd}
	& {\tPsh{\Xi_n}} \\
	{\Psh{\Theta_n\times \Delta}} & {\zncat^\nc}
	\arrow["{\pi_0^\Xi}", from=1-2, to=2-2]
	\arrow["p", from=2-1, to=1-2]
	\arrow["{\pi_0}"', from=2-1, to=2-2]
\end{tikzcd}\]
\end{construction}

\begin{prop}
\label{prop:xi is a model}
The adjunction given in construction \ref{cons:between xi et theta delta} is a Quillen equivalence between the (resp. complete) segal model structures. As a consequence, $\tPsh{\Xi_n}^{\seg}$ (resp. $\tPsh{\Xi_n}^{\cseg}$) is a model of non-complete $(\infty,n)$-categories (resp. of $(\infty,n)$-categories).
\end{prop}

\begin{proof}
We will show only the Segal case, as the complete Segal one directly follows. We can adapt the argument of proposition \ref{prop:model structure on Xiset} to obtain a model structure on $\Psh{t\Xi_n}$ whose weak equivalences are the smallest precocomplete class of morphisms containing $\M_n$ and the set $\{(\Db_{k+1})_t\to \Db_k,~k<n\}$. Moreover, the canonical functor $\Psh{t\Xi_n}\to \tPsh{\Xi_n}$ is a left Quillen equivalence as explained in remark \ref{rem: free model structure on marking}. The functor $p$ of the statement factors through a functor $p': \Psh{\Theta_n\times \Delta}\to \Psh{t\Xi_n}$ that preserves monomorphisms. By construction of the model structure on $\Psh{t\Xi_n}$, and using theorem \ref{theo:unit and counit are in Sigma} and proposition \ref{prop:weak equivalence are precocomplete}, the functor $p'$ preserves weak equivalences, and is then a left Quillen functor. We moreover have another functor $q:\Psh{t\Xi_n}\to \Psh{\Theta_n\times \Delta}$ sending $a\in \Theta$ to $a\boxdot [0]$ and $(\Db_n)_t$ onto the pushout:
\[\begin{tikzcd}
	{\Db_k\boxdot\{1\}} & {\Db_k\boxdot[1]} \\
	{\Db_{k-1}} & {q((\Db_k)_t)}
	\arrow[from=1-1, to=1-2]
	\arrow[from=1-1, to=2-1]
	\arrow[from=1-2, to=2-2]
	\arrow[from=2-1, to=2-2]
\end{tikzcd}\]
This functor preserves monomorphisms, and we can easily check that it preserves weak equivalences. It is then a left Quillen adjoint. Remark now that we have a natural transformation $p'q\to id$. It is the identity when evaluated on elements of $\Xi$, and by two out of three, it is a weak equivalence when evaluated on elements of shape $(\Db_n)_t$. By theorem \ref{theo:hom colimi}, this implies that it is a weak equivalence when evaluated on any element of $\Psh{t\Xi_n}$. Conversely, the left Quillen functor $qp'$ obviously preserves the globules, and then is weakly equivalent to the identity by proposition \ref{prop:Shommer preis}. The functors $p'$ and $q$ are then homotopical inverses, and $p'$ is then a left Quillen equivalence, which concludes the proof.
\end{proof}

\subsection{Gray module}


\begin{definition}
A family  of \snotionsym{intelligent $n$-truncation}{(taui@$\tau^i_n$}{for stratified Segal $A$-precategories}\textit{s} for  $n\in \Nb\cup\{\omega\}$ for a model category $A$ is a family of left Quillen functors $\tau^i_{\uvar}:(\mathbb{N}\cup\{\omega\})^{op}\to \End(A)$ such that
\begin{enumerate}[itemsep=0mm]
\item[$-$] $\tau^i_\omega = id$,
\item[$-$] for any $n\leq m$, $\tau^i_n\tau^i_m=\tau^i_n$,
\item[$-$] for any $n\leq m$, the natural transformation $\tau^i_m\to \tau^i_n$ is an entire monomorphism,
\end{enumerate}
\end{definition}

\begin{definition}
\label{defi:Gray module}
Let $A$ be an admissible category  of stratified presheaves on a category $B$ (definition \ref{definition:admissible}) .
A \notion{Gray module} \textit{structure} for the model category $A$ is the data of 
\begin{enumerate}
\item[$-$] a family of {intelligent $n$-truncation} for any $n \in \Nb \cup \{\omega\}$.
\item[$-$] a left Quillen functor $\uvar \otimes \uvar: \stratSset^1 \times A \to A$,
\item[$-$] for any $a$ in $A$, and any pair of stratified simplicial sets $K,L$, a natural morphism $K \otimes (L \otimes a) \to (K \times L) \otimes a$.
\end{enumerate}
such that 
\begin{enumerate}
\item for any stratified simplicial set $M$, the following square commutes
\[\begin{tikzcd}
	{K\otimes (L\otimes (M\otimes a))} & {(K\times L)\otimes (M\otimes a)} \\
	{K\otimes ((L\times M)\otimes a)} & {(K\times L\times M) \otimes a}
	\arrow[from=1-1, to=1-2]
	\arrow[from=1-2, to=2-2]
	\arrow[from=2-1, to=2-2]
	\arrow[from=1-1, to=2-1]
\end{tikzcd}\]
\item The functor $[0] \otimes \uvar: A \to A$ is the identity.
\item For any integer $n$, for any object $a$  such that $\tau^i_n(a)=a$ and  for any stratified simplicial set $K$, we have  $\tau^i_{n+1}(K \otimes a)=K\otimes a$.

Such a Gray module is \wcnotion{saturated}{saturated Gray module} if the functor $\otimes$ induces a left Quillen functor $\uvar \otimes \uvar: \stratSset^1_\seg \times A \to A$.
\end{enumerate}
\end{definition}
\begin{remark} 
We wish to emphasize the fact that in the previous definition, $\stratSset$ is endowed with the (resp. saturated) $1$-complicial model structure.
\end{remark}

\begin{example}
\label{example:of Gray module}

For any $d\in \mathbb{N}\cup \{\omega\}$, the model category $\stratSset^d$ (resp. $\stratSset^d_\sat$), corresponding to the (resp. saturated) $d$-complicial model structure, and where $K\otimes L := \tau^i_1(K)\boxtimes L$, is an example of (resp. saturated) Gray module.
\end{example}

\begin{example}
\label{example:Xi is Gray module}
 
We recall from \ref{rem:on forgetting some marking} that the functor $u:\tPsh{\Delta}\to\mathrm{t}_{\leq 1}\Psh{\Delta}\cong\tPsh{\Xi_1}$ is a left Quillen equivalence between the (resp. saturated) $1$-complicial model structure and the (resp. complete) Segal model structure. The model category $\tPsh{\Xi_1}^\seg$ (resp. $\tPsh{\Xi_1}^\cseg$) then admits a (resp. saturated) Gray module structure where $K\otimes L:= u(K)\times L$.
\end{example}

\begin{construction}
\label{cons:truncation for segals}

Let $A$ be a nice model category of stratified presheaves on an elegant Reedy category, endowed with {intelligent $n$-truncation} for $n \in \Nb \cup \{\omega\}$. We now construct a family of {intelligent $n$-truncation} for $n \in \Nb \cup \{\omega\}$ for $\stratSeg(A)$.

Let $k$ be any non-negative integer. The \textit{intelligent $k$-truncation functor}, denoted by $\tau^i_k$, is the colimit-preserving functor such that $\tau^i_k([a,n]) = [\tau^i_{k-1}(a),n]$ and $\tau^i_k[e,1]_t = [e,1]_t$. The intelligent \textit{$0$-truncation functor}, denoted by $\tau^i_0$, is the colimit-preserving functor such that $\tau^i_0([a,n])$ fits in the following pushout 
\[\begin{tikzcd}
	{\underset{ob(a)\times \Hom([1],[n])}{\coprod}[e,1]} & {[\tau^i_0(a),n]} \\
	{\underset{ob(a)\times \Hom([1],[n])}{\coprod}[e,1]_t} & {\tau^i_0([a,n])}
	\arrow[from=1-1, to=1-2]
	\arrow[from=1-2, to=2-2]
	\arrow[from=1-1, to=2-1]
	\arrow[from=2-1, to=2-2]
\end{tikzcd}\]
and such that $\tau^i_0[e,1]_t = [e,1]_t$. Using proposition \ref{prop:model structure on stratified Segal category} and remark \ref{rem:mapping from a free model structure}, we can easily check that the intelligent $k$-truncation are left Quillen functors for the Segal and complete Segal model structure.
\end{construction}

\begin{construction}
\label{cons:otimes pour segal}
We consider the colimit-preserving functor 
$$\uvar \otimes \uvar:\Psh{\Delta} \times \Seg(A) \to \Seg(A)$$
whose value on $([n],[a,m])$ fits in the pushout
\[\begin{tikzcd}
	{\coprod_{l\leq m}\colim_{[k_0,k_1]\to [n]\otimes\{l\}}[[k_0]\otimes a, k_1]} & {\colim_{[k_0,k_1]\to [n]\otimes[m]}[[k_0]\otimes a, k_1]} \\
	{\coprod_{l\leq m}\colim_{[k_0,k_1]\to [n]\otimes\{l\}}[ e, k_1]} & {[n]\otimes[a,m]}
	\arrow[from=1-1, to=2-1]
	\arrow[from=2-1, to=2-2]
	\arrow[from=1-2, to=2-2]
	\arrow[""{name=0, anchor=center, inner sep=0}, from=1-1, to=1-2]
	\arrow["\lrcorner"{anchor=center, pos=0.125, rotate=180}, draw=none, from=2-2, to=0]
\end{tikzcd}\]
where $\uvar\otimes\uvar:\ncat{1}\times\ncat{1}\to \ncat{2}$ is the Gray tensor product defined in theorem \ref{theo:otimes in zocat}.
We extend $\uvar \otimes \uvar$ to a functor 
$$\uvar \otimes \uvar:\stratSset \times \stratSeg(A) \to \stratSeg(A)$$
by setting $[1]_t \otimes[a,m]$ as the colimit
\[\begin{tikzcd}
	{\coprod_{l\leq m}\colim_{[k_0,k_1]\to [1]\otimes\{l\}}[[k_0]\otimes a, k_1]} & {\colim_{[k_0,k_1]\to [1]\otimes[m]}[[k_0]^\sharp\otimes a, k_1]} \\
	{\coprod_{l\leq m}\colim_{[k_0,k_1]\to [1]\otimes\{l\}}\tau^i_0[e, k_1]} & {[1]_t\otimes[a,m]}
	\arrow[from=1-1, to=2-1]
	\arrow[from=2-1, to=2-2]
	\arrow[from=1-2, to=2-2]
	\arrow[""{name=0, anchor=center, inner sep=0}, from=1-1, to=1-2]
	\arrow["\lrcorner"{anchor=center, pos=0.125, rotate=180}, draw=none, from=2-2, to=0]
\end{tikzcd}\]
and for any integer $k>1$, $$[k]_t \otimes [a,n] := [k]\otimes[a,n],$$
and  eventually, for any stratified simplicial set $K$, by setting $K\otimes[e,1]_t$ as the pushout
\[\begin{tikzcd}
	{\coprod_{c\in ob(K)}\tau^i_1(\{c\}\otimes[e,1])} & {\tau^i_1(K\otimes[e,1])} \\
	{\coprod_{c\in ob(K)}\{c\}\otimes[e,1]_t} & {K\otimes[e,1]_t}
	\arrow[from=1-1, to=2-1]
	\arrow[from=2-1, to=2-2]
	\arrow[from=1-2, to=2-2]
	\arrow[""{name=0, anchor=center, inner sep=0}, from=1-1, to=1-2]
	\arrow["\lrcorner"{anchor=center, pos=0.125, rotate=180}, draw=none, from=2-2, to=0]
\end{tikzcd}\]
\end{construction}

\begin{notation}
We will denote by $K_1\otimes ... \otimes K_n\otimes C$ the object $(K_1\otimes ( .... \otimes (K_n\otimes C)...))$
\end{notation}

\begin{prop}
\label{prop:otimes est quillen dans Segal}
The functor $\otimes: \stratSset \times \stratSeg(A) \to \stratSeg(A)$ is a left Quillen functor when $\stratSset$ is endowed with the (resp. saturated) $1$-complicial model structure and $\stratSeg(A)$ with the (resp. complete) Segal model structure.
\end{prop}

\begin{proof} 
We first fix an object $a$ in $A$. The functor $\uvar \otimes [a,\uvar]: \Sset \times \Sset \to \stratSeg(A)$ is the composite
\[\begin{tikzcd}
	{\Sset \times \Sset} & {\Psh{\Theta_2}} & {\Psh{\Delta[\Delta]} \cong \Seg(\Sset)} & {\stratSeg(A)}
	\arrow["\uvar \otimes \uvar", from=1-1, to=1-2]
	\arrow["{i^*}", from=1-2, to=1-3]
	\arrow["{\Seg(\uvar \otimes a)}", from=1-3, to=1-4]
\end{tikzcd}\]
According to proposition \ref{prop:weak equivalence are precocomplete} and theorems \ref{theo:unit and counit are in Sigma} and \ref{theo:otimes presserves W}, this functor then sends $\W_1 \times \W_1$ to weak equivalences of $\stratSeg(A)^\seg$, and $\W_1^\sat \times \W_1^\sat$ to weak equivalences of $\stratSeg(A)^\cseg$.
We can show similarly that $\uvar \otimes [e,1]_t: \Sset \to \stratSeg(A)$ and $[1]_t \otimes [a,\uvar]: \Sset \to \stratSeg(A)$ send $\W_1$ to weak equivalences of $\stratSeg(A)^\seg$ and $\W_1^\sat$ to weak equivalences of $\stratSeg(A)^\cseg$.

We now fix a marked simplicial set $K$ and an integer $n$. Let $i:a \to b$ be a weak equivalence of $A$. The morphism $K \otimes [a,n] \to K \otimes [b,n]$ is a colimit of natural transformations that are pointwise weak equivalences. As this colimit is indexed by the elegant Reedy category $\Theta_{/K \otimes [n]}$ and verifies the condition of theorem \ref{theo:hom colimi}, the morphism $K \otimes [i,n]: K \otimes [a,n] \to K \otimes [b,n]$ is a weak equivalence.

We are now willing to show that for all stratified Segal $A$-precategories $C$, the morphism $[1]_t \otimes C \to C$ is a weak equivalence.
As $[1]_t\otimes\uvar$ preserves monomorphisms, the functor $t\Delta[tB]_{/C}\to \Arr(\stratSeg(A))$ whose value on $x$ is $[1]_t \otimes x \to x$ is Reedy cofibrant. The theorem \ref{theo:hom colimi} then implies that it is sufficient to show that $[1]_t \otimes C \to C$ is a weak equivalence when $C$ is representable, i.e., when it is of shape $[a,n]$ or $[e,1]_t$. Moreover, as we already showed that $[1]_t \otimes [a,\uvar]$ sends spine inclusions to weak equivalences, we can reduce to the case where $C$ is either $[a,1]$ or $[e,1]_t$. 
The proposition \ref{prop:explicit Gray} then implies that $[1]_t \otimes [a,1]$ is the colimit of the diagram
\[\begin{tikzcd}
	{[e,1]_t \vee [a,1]} & {[a,1]} & {[[1]_t \otimes a,1]} & {[a,1]} & {[a,1] \vee [e,1]_t}
	\arrow["{[a,d^1]}"', from=1-2, to=1-1]
	\arrow["{[a,d^1]}", from=1-4, to=1-5]
	\arrow["{[d^1,1]}"', from=1-4, to=1-3]
	\arrow["{[d^0,1]}", from=1-2, to=1-3]
\end{tikzcd}\]
We then have cocartesian squares 
\[\begin{tikzcd}
	{[e,1]_t\vee[a,1]} & {[1]_t\otimes[a,1]} & {[a,1]\vee[e,1]_t} & \bullet \\
	{[a,1]} & \bullet & {[a,1]} & {[[1]_t\otimes a,1]}
	\arrow[from=1-1, to=1-2]
	\arrow[from=1-1, to=2-1]
	\arrow[from=2-1, to=2-2]
	\arrow[from=1-2, to=2-2]
	\arrow[from=1-3, to=2-3]
	\arrow[from=2-3, to=2-4]
	\arrow[from=1-3, to=1-4]
	\arrow[from=1-4, to=2-4]
	\arrow["\lrcorner"{anchor=center, pos=0.125, rotate=180}, draw=none, from=2-4, to=1-3]
	\arrow["\lrcorner"{anchor=center, pos=0.125, rotate=180}, draw=none, from=2-2, to=1-1]
\end{tikzcd}\]
whose left vertical morphisms are weak equivalences. As weak equivalences are stable under pushouts along cofibrations and by composition, the canonical morphism $[1]_t\otimes[a,1]\to [[1]_t\otimes a,1]$ is a weak equivalence. As the canonical morphism $[1]_t \otimes [a,1] \to [a,1]$ is the composite of $[1]_t \otimes [a,1] \to [[1]_t \otimes a,1]$ with the weak equivalence $[[1]_t \otimes a,1] \to [a,1]$, it is a weak equivalence. 

We proceed similarly to demonstrate that for all marked complicial sets $K$, $K \otimes [e,1]_t \to K \otimes [0]$ is a weak equivalence. 

The propositions \ref{prop:model structure on stratified Segal category} and \ref{prop:characterization of morphism depuis strratsset.} and the remark \ref{rem:mapping from a free model structure} then imply that the functor $\otimes: \stratSset^1 \times \stratSeg(A) \to \stratSeg(A)$ is a left Quillen functor between the appropriate model structures.
\end{proof}

\begin{construction}
\label{cons:liens entre otimes et times pour segal}
Let $a$ be an object of $A$ and $l,m,n$ three integers. By construction, $[l]\otimes[m]\otimes[a,n]$ is a quotient of 
$$P_{a,l,m,n}:=\colim_{[[k_0],k_1]\to [m]\otimes[n]}\colim_{[[k_2],k_3]\to [l]\otimes[k_1]}[[k_2]\otimes[k_0]\otimes a, k_3]$$
while $([l]\times [m])\otimes [a,n]$ is a quotient of
$$Q_{a,l,m,n}:=\colim_{[[k_4],k_3]\to ([l]\times[n])\otimes [m]}[[k_4]\otimes a, k_3].$$
Lemma \ref{lemma:technical steiner} and the Gray module structure on $A$ then induce a morphism $$P_{a,l,m,n}\to Q_{a,l,m,n}.$$ 
We can check that this morphism passes to the quotient and then induces a natural morphism 
$$[l]\otimes[m]\otimes[a,n]\to ([l]\times [m])\otimes [a,n].$$
By extension by colimit, this induces, for any Segal $A$-category $C$, and any pair of simplicial sets $K,L$, a morphism 
$$K\otimes L\otimes C\to (K\times L)\otimes C.$$
Moreover, we can check that this natural transformation between $\uvar\otimes \uvar\otimes \uvar$ and $(\uvar\times \uvar)\otimes \uvar$ extends to stratified simplicial sets and stratified Segal $A$-categories. Eventually, by construction and using the equality \eqref{eq:alpha}, we get a commutative square
\[\begin{tikzcd}
	{K\otimes L\otimes M\otimes C} & {(K\times L)\otimes M\otimes C} \\
	{K\otimes(L\times M)\otimes  C} & {(K\times L\times M)\otimes C}
	\arrow[from=1-1, to=2-1]
	\arrow[from=2-1, to=2-2]
	\arrow[from=1-2, to=2-2]
	\arrow[from=1-1, to=1-2]
\end{tikzcd}\]
for any stratified Segal $A$-category $C$ and any stratified simplicial sets $K,L,M$.
\end{construction}

\begin{theorem}
\label{theo:Gray structure on seg}
A (resp. saturated) Gray  module structure  on $A$ induces a (resp. saturated) Gray module  structure on $\stratSeg(A)$. The family of intelligent truncations is defined in \ref{cons:truncation for segals}, and the tensoring by $\stratSset$ is defined in \ref{cons:otimes pour segal}. The natural comparison maps between $K\otimes (L\otimes C)$ and $(K\times L)\otimes C$ are provided by the construction \ref{cons:liens entre otimes et times pour segal}.
\end{theorem}
\begin{proof}
The proposition \ref{prop:otimes est quillen dans Segal} states that the functor $\uvar\otimes\uvar$ constructed in \ref{cons:otimes pour segal} is a left Quillen functor between the appropriate model structure. The first condition of the definition \ref{defi:Gray module} follows from construction \ref{cons:liens entre otimes et times pour segal}, and the two other are obviously fulfilled.	
\end{proof}

\begin{cor}

\label{cor:Gray module on Xi}
For any $n\in  \Nb\cup \{\omega\}$, $\tPsh{\Xi_n}^{\seg}$ is a Gray module and $\tPsh{\Xi_n}^{\cseg}$ is a saturated Gray module.
\end{cor}

\begin{proof}
We will only show that $\tPsh{\Xi_n}^{\seg}$ is a Gray module, the other statement being proven similarly. We can first show by induction on $n$ that $\tPsh{\Xi_n}^{\seg}$ is a Gray module. The base case is given in example \ref{example:Xi is Gray module} and the inductive step by theorem \ref{theo:Gray structure on seg} where we use the identification given in remark \ref{rem:xi and seg}.

For the case $n=\omega$, we define $\otimes$ as the unique left adjoint such that for any integer $n$, the following diagram commutes:
\[\begin{tikzcd}
	{\tPsh{\Delta}\times \tPsh{\Theta_n}} & {\tPsh{\Delta}\times \tPsh{\Theta_{n+1}}} & {\tPsh{\Theta_{n+1}}} \\
	{\tPsh{\Delta}\times \tPsh{\Theta}} && {\tPsh{\Theta}}
	\arrow["{(id,\iota)}", from=1-1, to=1-2]
	\arrow[from=1-1, to=2-1]
	\arrow["\otimes", from=1-2, to=1-3]
	\arrow[from=1-3, to=2-3]
	\arrow["\otimes"', from=2-1, to=2-3]
\end{tikzcd}\]
and the family of $n$-truncations is given by the functor marking any cell of dimension higher or equal to $n$. We can easily check that this endows $\Psh{\Theta_n}$ with a structure of Gray module.
\end{proof}

\subsection{Complicial Gray module}
\begin{construction}
\label{cons:joint in gray module}
Let $A$ be a Gray module and $a$ an object of $A$.
We define $e\star a$ as the pushout:
\[\begin{tikzcd}
	{\{0\}\times a} & {[1]\otimes a} \\
	e & {e\star a}
	\arrow[from=1-1, to=2-1]
	\arrow[from=1-1, to=1-2]
	\arrow[from=1-2, to=2-2]
	\arrow[from=2-1, to=2-2]
	\arrow["\lrcorner"{anchor=center, pos=0.125, rotate=180}, draw=none, from=2-2, to=1-1]
\end{tikzcd}\]
We consider the natural transformations $s^0\star a:e\star e\star a\to e\star a$ and $d^0\star a:a\to e\star a$,
induced respectively by the morphism
$$
\begin{array}{cclcccccccc}
[1]\otimes [1]\otimes a&\to & ([1]\times [1])\otimes a &\to & [1]\otimes a\\
&&(\{i\}\times \{j\})\otimes a&\mapsto & \{i\wedge j\}\otimes a.
\end{array}$$
and the morphism 
$$\{1\}\otimes a \to [1]\otimes a.$$
These natural transformations induce commutative diagrams:
\[\begin{tikzcd}
	{e\star e\star e\star a } & { e\star e\star a } && {e\star a} & { e\star e\star a } & {e\star a} \\
	{ e\star e\star a } & {e\star a} &&& {e\star a}
	\arrow["{s^0\star a}", from=1-2, to=2-2]
	\arrow["{s^0\star a}"', from=2-1, to=2-2]
	\arrow["{s^0\star (e\star a)}", from=1-1, to=1-2]
	\arrow["{e\star (s^0\star a)}"', from=1-1, to=2-1]
	\arrow["{e\star d^0}", from=1-4, to=1-5]
	\arrow["{d^0\star (e\star a)}", from=1-5, to=1-6]
	\arrow["{s^0\star a}", from=1-5, to=2-5]
	\arrow["id", curve={height=-6pt}, from=1-6, to=2-5]
	\arrow["id"', curve={height=6pt}, from=1-4, to=2-5]
\end{tikzcd}\]
The (inverted) composition $g,f\mapsto g\circ f$ is a monoidal structure on the category of endomorphisms of $A$ and the natural transformation $s^0:e\star e \star\uvar \to e\star \uvar$ defines a structure of monoid for $e\star\uvar$.
This induces a functor $\Delta\times A\to A$ sending $([n],a)$ to $e\star e\star ....\star a$. We extend this to a functor $\Delta_t\times A\to A$ in defining $[n]_t\star a$ as the pushout:
\[\begin{tikzcd}
	{\underset{k\geq -1}{\coprod}~~\underset{b,~\tau^i_k(b)=b}{\coprod}~~\underset{b\to a}{\coprod}[n]\star b} & {[n]\star a} \\
	{\underset{k\geq -1}{\coprod}~~\underset{b,~\tau^i_k(b)=b}{\coprod}~~\underset{b\to a}{\coprod}\tau^i_{n+k}([n]\star b)} & {[n]_t\star a}
	\arrow[from=2-1, to=2-2]
	\arrow[""{name=0, anchor=center, inner sep=0}, from=1-1, to=1-2]
	\arrow[from=1-2, to=2-2]
	\arrow[from=1-1, to=2-1]
	\arrow["\lrcorner"{anchor=center, pos=0.125, rotate=180}, draw=none, from=2-2, to=0]
\end{tikzcd}\]
where $\tau^i_{-1}$ is the constant functor with value $\emptyset$.

By left Kan extention, this gives a colimit preserving functor 
\begin{equation}
\stratSset\times  \stratSeg(A)\to \stratSeg(A).
\end{equation}
and evaluated on the empty Segal $A$-category, a colimit preserving functor 
\begin{equation}
\label{eq:def of the first adjonctionwith stratified simplicials et}
\stratSset\to \stratSeg(A).
\end{equation}
\end{construction}

\begin{definition}
\label{defi:complicial Gray module}
A  (resp. satutared) Gray module $A$ is a \wcnotion{complicial}{complicial Gray module} if
\begin{enumerate}
\item For any  $a$, the morphisms $\Lambda^1[2]\star a\to [2]_t\star a$ and  $\{\epsilon\}\star a\to [1]_t\star a$ with $\epsilon\in \{-,+\}$ are acyclic cofibrations. 
\item The functor $\stratSset\to \stratSeg(A)$ defined in \eqref{eq:def of the first adjonctionwith stratified simplicials et} is a left Quillen functor between the (resp. saturated) complicial model structure and the (resp. complete) model structure.
\end{enumerate}
\end{definition}

\begin{remark}
In general, $[n]\otimes e$ and $[n]\star \emptyset$ are two very different objects. Indeed $[n]\otimes e$ has to be invariant up to homotopy under $\tau^i_1$ which is not the case for $[n]\star \emptyset$. Analogously $[k]\otimes ([l]\otimes [a])$ and $([k]\otimes [l])\otimes [a]$ have \textit{a priori} no links.
 \end{remark}
 
 \begin{notation}
We will denote by $[n_0]\otimes[n_1]\otimes..[n_k]\otimes a$ the object $[n_0]\otimes([n_1]\otimes..([n_k]\otimes a))$.
\end{notation}

\begin{example}
\label{example:stratsset is complicial Gray module}

The (resp. saturated) Gray module structure on $\stratSset^d$ given in example \ref{example:of Gray module} is complicial. 
Indeed, if $n$ is any integer, we define $[n]^{\diamond}:=[0]\diamond [0]\diamond...\diamond [0]$ and $[n]_t^{\diamond}:= \tau^i_n([n]^{\diamond})$. This induces a colimit preserving functor $K\mapsto K^{\diamond}$. The join coming from $\tau^i_1(\uvar)\boxtimes \uvar$ then corresponds to the functor $(K,L)\mapsto K^\diamond\diamond L$. The proposition \ref{prop:equivalence between diamond and join product} provides a natural transformation $K^{\diamond}\diamond L\to K\star L$, which implies that the first functor is left Quillen.
\end{example}

\begin{example}
\label{example:Xi is complicial Gray module}

The functor $u:\tPsh{\Delta}\to\mathrm{t}_{\leq 1}\Psh{\Delta}\cong\tPsh{\Xi_1}$ preserves the Gray module structure. The previous example then implies that the Gray module structure of $\tPsh{\Xi_1}^{\seg}$ and the saturated Gray module structure of $\tPsh{\Xi_1}^{\cseg}$ are complicial.
\end{example}

\section{Complicial Gray module structure on $\stratSeg(A)$}
The purpose of this section is to show that for any complicial Gray module $A$, the Gray module structure on $\stratSeg(A)$ constructed in \ref{theo:Gray structure on seg} is complicial. This is achieved in theorem \ref{theo:complicialGray module}.

\vspace{1cm}

We fix a complicial Gray module $A$ until the end of this section. When the model structure is not specified, we will always assume that $\tPsh{\Delta}$ is endowed with the complicial model structure and $\stratSeg(A)$ is endowed with the Segal model structure. The weak equivalences of $\tPsh{\Delta}_\sat$ will be called saturated, and the ones of $\stratSeg(A)^\cseg$ complete.

\subsection{$\circ$-cone in $\stratSeg(A)$ }

To show that the Gray module $\stratSeg(A)$ is complicial, we need to demonstrate that the adjunction with marked simplicial sets constructed in \ref{cons:joint in gray module} is a Quillen adjunction. This adjunction is constructed using an op-cone $e\star\uvar:\stratSeg(A)\to \stratSeg(A)$ arising from the Gray module structure of $\stratSeg(A)$. However, for technical reasons, it will be useful to work with another op-cone that is  constructed in \ref{cons:star pour segal}.
We have chosen to also denote this op-cone on $\stratSeg(A)$ by $e\star\uvar$, as it is the only one we will use from now on.

Proposition \ref{prop:comparisons between star and otimes in segal} shows that these two op-cones are weakly equivalent, implying that the two adjunctions with stratified simplicial sets they induce are weakly equivalent.

\begin{construction}
\label{cons:star pour segal0}
We consider the colimit-preserving functor 
$$e\star\uvar: \Seg(A)\to \Seg(A)$$
whose value on $[a,m]$ fits in the pushout
\[\begin{tikzcd}
	{\coprod_{l\leq m}\colim_{[k_0,k_1]\to 1\star\{l\}}[[k_0]\otimes a, k_1]} & {\colim_{[k_0,k_1]\to 1\star[m]}[[k_0]\otimes a, k_1]} \\
	{\coprod_{l\leq m}\colim_{[k_0,k_1]\to 1\star\{l\}}[e, k_1]} & {e\star[a,m]}
	\arrow[from=1-1, to=2-1]
	\arrow[from=2-1, to=2-2]
	\arrow[from=1-2, to=2-2]
	\arrow[""{name=0, anchor=center, inner sep=0}, from=1-1, to=1-2]
	\arrow["\lrcorner"{anchor=center, pos=0.125, rotate=180}, draw=none, from=2-2, to=0]
\end{tikzcd}\]
This functor is called the \snotionsym{Gray $\circ$-cylinder}{((d31@$[1]\otimes \uvar$}{for stratified Segal $A$-precategories}.
where $1\star\uvar:\ncat{1}\to \ncat{2}$ denotes the Gray $\circ$-cone defined in \ref{cons:Gray cone for omega cat}.
The  morphism $d^0:[m]\to 1\star[m]$ induces a  morphism
$$d^0\star [a,m]: [a,m]\cong \colim_{[k_1]\to [m]}[a,k_1]\to e\star[a,m].$$
By left Kan extension, this induces a transformation 
$$d^0\star C: C\to e\star C$$
natural in $C:\Seg(A)$.
\end{construction}

\begin{construction}
\label{cons:star pour segal}
We extend $e\star\uvar$ as a functor 
$$e\star\uvar:\stratSeg(A)\to \stratSeg(A)$$
by setting $e\star[e,1]_t$ as the colimit
\[\begin{tikzcd}
	{[e,1]} & {\tau^i_1(e\star[e,1])} \\
	{[e,1]_t} & {e\star[e,1]_t}
	\arrow[from=1-1, to=2-1]
	\arrow[from=2-1, to=2-2]
	\arrow[from=1-2, to=2-2]
	\arrow[""{name=0, anchor=center, inner sep=0}, "{d^0\star[e,1]}", from=1-1, to=1-2]
	\arrow["\lrcorner"{anchor=center, pos=0.125, rotate=180}, draw=none, from=2-2, to=0]
\end{tikzcd}\]
The natural transformation $d^0\star\uvar$ extends to a transformation 
$$d^0\star C: C\to e\star C$$
natural in $C:\stratSeg(A)$.
\end{construction}

\begin{prop}
\label{prop:comparisons between star and otimes in segal}
For any stratified Segal $A$-precategory $X$, there exists a weak equivalence 
$$\{0\}\coprod_{\{0\}\otimes X} [1]\otimes X\to e\star X$$
natural in $X$.
\end{prop}
\begin{proof}
As the two functors $\{0\}\coprod_{\{0\}\otimes\uvar} [1]\otimes \uvar$ and $ e\star \uvar$ are left Quillen functors, it is sufficient to construct this comparison  when $C$ is of shape $[a,n]$ or $[e,1]_t$. In this case, the canonical morphism $[1]\otimes [n]\to 1\star [n]$ of $\zo$-categories induces comparison morphisms
$$[1]\otimes [a,n]\to e\star [a,n]~~~~~[1]\otimes[e,1]_t\to e\star[e,1]_t$$
that respectively send $\{0\}\otimes[a,n]$ and $\{0\}\otimes[e,1]_t$ to $e\star\emptyset$. The two previous morphisms then induce natural morphisms
$$\{0\}\coprod_{\{0\}\otimes[a,n]}[1]\otimes [a,n]\to e\star [a,n]~~~~~~~~\{0\}\coprod_{\{0\}\otimes[e,1]_t}[1]\otimes[e,1]_t\to e\star[e,1]_t$$
Now, remark that these two morphisms fit in the following cocartesian squares:
\[\begin{tikzcd}
	{\colim_{[[k_0],k_1]\to \{0\}\coprod_{\{0\}\otimes[n]}[1]\otimes[n]}[[k_0]\otimes a,k_1]} & {e\coprod_{\{0\}\otimes[a,n]}[1]\otimes [a,n]} \\
	{\colim_{[[k_0],k_1]\to 1\star[n]}[[k_0]\otimes a,k_1]} & { e\star [a,n]} \\
	{\colim_{[[k_0],k_1]\to \{0\}\coprod_{\{0\}\otimes[1]}[1]\otimes[1]}[[k_0]\otimes a,k_1]} & {e\coprod_{\{0\}\otimes[a,n]}[1]\otimes [e,1]_t} \\
	{\colim_{[[k_0],k_1]\to 1\star[1]}[[k_0]\otimes a,k_1]} & {e\star[e,1]_t}
	\arrow[""{name=0, anchor=center, inner sep=0}, from=1-1, to=1-2]
	\arrow[from=2-1, to=2-2]
	\arrow[from=1-2, to=2-2]
	\arrow[from=1-1, to=2-1]
	\arrow[from=3-1, to=4-1]
	\arrow[from=3-2, to=4-2]
	\arrow[from=4-1, to=4-2]
	\arrow[""{name=1, anchor=center, inner sep=0}, from=3-1, to=3-2]
	\arrow["\lrcorner"{anchor=center, pos=0.125, rotate=180}, draw=none, from=2-2, to=0]
	\arrow["\lrcorner"{anchor=center, pos=0.125, rotate=180}, draw=none, from=4-2, to=1]
\end{tikzcd}\]
We claim that the functor whose value on a $\Theta_2$-set $X$ is $\colim_{[[k_0],k_1]\to X}[[k_0]\otimes a,k_1]$ sends $\overline{\W_2}$ to weak equivalences. Combined with proposition \ref{prop:comparaison in theta set between otimes and star}, it will conclude the proof. 

To show the desired claim, remark that this functor is the composite
\[\begin{tikzcd}
	{\Psh{\Theta_2}} & {\Psh{\Delta[\Delta]}\cong \Seg(\Sset)} & {\Seg(A)}
	\arrow["{i^*}", from=1-1, to=1-2]
	\arrow["{\Seg(\uvar\otimes a)}", from=1-2, to=1-3]
\end{tikzcd}\]
and the results follow from theorem \ref{theo:unit and counit are in Sigma} and proposition \ref{prop:weak equivalence are precocomplete}.
\end{proof}

\begin{prop}
\label{prop:star est quillen dans Segal}
The functor $e\star\uvar: \stratSeg(A)\to \stratSeg(A)$ is a left Quillen functor for the Segal and complete Segal space model structure. 
\end{prop}
\begin{proof}
The proposition \ref{prop:comparisons between star and otimes in segal} implies that $e\star\uvar$ is pointwise weakly equivalent to 
the functor $\{0\}\coprod_{\{0\}\otimes \uvar}[1]\otimes \uvar$. As this last functor is a homotopy colimit of functors preserving (resp. complete) weak equivalence, the  functor $e\star\uvar$ also preserves them. As $e\star\uvar$ also preserves monomorphisms, this concludes the proof.
\end{proof}

\begin{construction}
\label{cons:liens entre star et starstar pour segal}
Let $a$ be an object of $A$ and $l,m$ two integers. By construction, $e\star e[a,m]$ is a quotient of 
$$P_{a,l,m}:=\colim_{[k_0,k_1]\to 1\star [m]}\colim_{[k_2,k_3]\to [l]\otimes[k_1]}[[k_2]\otimes[k_0]\otimes a, k_3]$$
while $e\star [a,m]$ is a quotient of
$$Q_{a,l,m}:=\colim_{[k_4,k_3]\to 1\star [m]}[[k_4]\otimes a, k_3].$$

Lemma \ref{lemma:techiniqcal steiner2version star} and the Gray module structure on $A$ then induce a morphism $$P_{a,l,m}\to Q_{a,l,m}.$$ 
We can check that this morphism passes to the quotient and then induces a natural morphism 
$$s^0\star [a,n]:e\star e\star [a,n]\to e\star [a,n].$$
By extension by colimit, this induces, for any Segal $A$-category $C$, a morphism 
$$s^0\star C:e\star e\star C\to e\star C.$$
We can moreover check that this natural transformation between $e\star e \star \uvar$ and $e\star \uvar$ extends to stratified Segal $A$-categories. Finally, by construction and using the equality \eqref{eq:beta}, we get a commutative square
\[\begin{tikzcd}
	{e\star e\star e\star C} & { e\star e\star C} \\
	{ e\star e\star C} & {e\star C}
	\arrow["{e\star s^0\star C}"', from=1-1, to=2-1]
	\arrow["{s^0\star C}"', from=2-1, to=2-2]
	\arrow["{s^0\star C}", from=1-2, to=2-2]
	\arrow["{s^0\star e\star C}", from=1-1, to=1-2]
\end{tikzcd}\]
for any stratified Segal $A$-category $C$.
\end{construction}

\begin{prop}
\label{prop:explicit expression of e star a,1}
The stratified Segal $A$-precategory $e\star[a,1]$ is the colimit of the diagram
\[\begin{tikzcd}
	{[e\star a,1]} & {[a,1]} & {[e,1]\vee[a,1]}
	\arrow["{[d^0\star a,1]}"', from=1-2, to=1-1]
	\arrow["{[a,d^1]}", from=1-2, to=1-3]
\end{tikzcd}\]
and the stratified Segal $A$-precategory $e\star[e,1]_t$ is the colimit of the diagram
\[\begin{tikzcd}
	{[[1]_t,1]} & {[e,1]} & {[e,1]\vee[e,1]_t}
	\arrow["{[d^0\star e,1]}"', from=1-2, to=1-1]
	\arrow["{[e,d^1]}", from=1-2, to=1-3]
\end{tikzcd}\]
\end{prop}
\begin{proof}
We recall that $e\star a$ is the object of $A$ fitting in the following cocartesian square
\[\begin{tikzcd}
	{\{0\}\otimes a} & {[1]\otimes a} \\
	e & {e\star a}
	\arrow[from=1-1, to=2-1]
	\arrow[from=2-1, to=2-2]
	\arrow[from=1-1, to=1-2]
	\arrow[from=1-2, to=2-2]
\end{tikzcd}\]
The results then directly follow from the construction of the functor $e\star\uvar:\stratSeg(A)\to \stratSeg(A)$ and from proposition \ref{prop:exppliocit Graycone}.
\end{proof}
\begin{remark}
The last proposition can be seen as an analogue in stratified simplicial sets of the third formula of theorem \ref{theo:appendice formula for star}.
\end{remark}

\begin{prop}
\label{prop:explicit expression of e star a,2}
The stratified  Segal $A$-precategory $e\star[a,2]$ is the colimit of the diagram
\[\begin{tikzcd}
	{[[2]\botimes a,1]} & {[e\star a,1]} & {[e\star a,1]\vee[a,1]} & {[e,1]\vee[a,1]} & {[e,1]\vee[a,2]}
	\arrow["{[d^0\otimes a,1]}"', from=1-2, to=1-1]
	\arrow["{[[1]\otimes a,d^1]}", from=1-2, to=1-3]
	\arrow["{[d^0\otimes a,2]}"', from=1-4, to=1-3]
	\arrow["{[a,d^1]}", from=1-4, to=1-5]
\end{tikzcd}\]
where $[2]\botimes a$ and $[e\star a,1]\vee[a,1]$ are the pushouts:
\[\begin{tikzcd}
	{[1]\otimes a\amalg [1]\otimes a} && {[2]\otimes a} & {[[1]\otimes a,1]\amalg [[1]\otimes a,2]} & {[[1]\otimes a,2]} \\
	{e\star a\amalg e\star a} && {[2]\botimes a} & {[e\star a,1]\amalg[a,1]} & {[e\star a,1]\vee[a,1]}
	\arrow[""{name=0, anchor=center, inner sep=0}, "{d^1\otimes a\amalg d^2\otimes a}", from=1-1, to=1-3]
	\arrow["{d^1\botimes a\amalg d^2\botimes a}"', from=2-1, to=2-3]
	\arrow[from=1-1, to=2-1]
	\arrow[from=1-3, to=2-3]
	\arrow[""{name=1, anchor=center, inner sep=0}, "{[[1]\otimes a,d^2\amalg d^1]}", from=1-4, to=1-5]
	\arrow[from=1-4, to=2-4]
	\arrow[from=2-4, to=2-5]
	\arrow[from=1-5, to=2-5]
	\arrow["\lrcorner"{anchor=center, pos=0.125, rotate=180}, draw=none, from=2-3, to=0]
	\arrow["\lrcorner"{anchor=center, pos=0.125, rotate=180}, draw=none, from=2-5, to=1]
\end{tikzcd}\]
\end{prop}
\begin{proof}
The result directly follows from the construction of the functor $e\star\uvar:\stratSeg(A)\to \stratSeg(A)$ and of proposition \ref{prop:exppliocit Graycone2}.
\end{proof}

\begin{prop}
\label{prop:explicit expression of e star e star a,1}
The stratified Segal $A$-precategory
 $e\star e\star[a,1]$ is the colimit of the diagram
$$
\begin{tikzcd}[column sep=0.7 cm]
	{[[2]\botimes a,1]} & {[[1]\otimes a,1]} & {[[1],1]\vee[a,1]} & {[e,1]\vee[a,1]} & {[e,2]\vee[a,1]} \\
	{[e\star a,1]} &&& {[a,1]} & {[e,1]\vee[a,1]} \\
	{[[1]\star a,1]} &&& {[e\star a,1]} & {[e,1]\vee[e\star a,1]}
	\arrow["{[d^0\otimes a,1]}"', from=1-2, to=1-1]
	\arrow["{[[1]\otimes a,d^1]}", from=1-2, to=1-3]
	\arrow["{[d^0\otimes a,2]}"', from=1-4, to=1-3]
	\arrow["{[a,d^1]}", from=1-4, to=1-5]
	\arrow["{[a,d^1]}", from=2-4, to=2-5]
	\arrow["{[d^1\botimes a,1]}", from=2-1, to=1-1]
	\arrow["{[a,d^1]}", from=2-4, to=1-4]
	\arrow["{[a,d^2]}"', from=2-5, to=1-5]
	\arrow["{[d^1\star a,1]}"', from=2-1, to=3-1]
	\arrow["{[d^{0}\star a,2]}", from=2-5, to=3-5]
	\arrow["{[d^{0}\star a,1]}"', from=2-4, to=2-1]
	\arrow["{[d^0\star a,1]}", from=3-4, to=3-1]
	\arrow["{[d^{0}\star a,1]}"', from=2-4, to=3-4]
	\arrow["{[e\star a,d^1]}"', from=3-4, to=3-5]
\end{tikzcd}$$
where $[2]\botimes a$ and $[[1],1]\vee[a,1]$ are the pushouts:
\[\begin{tikzcd}
	{[1]\otimes a\amalg [1]\otimes a} && {[2]\otimes a} & {[[1]\otimes a,1]\amalg [[1]\otimes a,2]} & {[[1]\otimes a,2]} \\
	{e\star a\amalg e\star a} && {[2]\botimes a} & {[[1],1]\amalg[a,1]} & {[[1],1]\vee[a,1]}
	\arrow[""{name=0, anchor=center, inner sep=0}, "{d^1\otimes a\amalg d^2\otimes a}", from=1-1, to=1-3]
	\arrow["{d^1\botimes a\amalg d^2\botimes a}"', from=2-1, to=2-3]
	\arrow[from=1-1, to=2-1]
	\arrow[from=1-3, to=2-3]
	\arrow["{[[1]\otimes a,d^2\amalg d^1]}", from=1-4, to=1-5]
	\arrow[from=1-4, to=2-4]
	\arrow[from=2-4, to=2-5]
	\arrow[from=1-5, to=2-5]
	\arrow["\lrcorner"{anchor=center, pos=0.125, rotate=180}, draw=none, from=2-5, to=1-4]
	\arrow["\lrcorner"{anchor=center, pos=0.125, rotate=180}, draw=none, from=2-3, to=0]
\end{tikzcd}\]
\end{prop}
\begin{proof}
The proposition \ref{prop:explicit expression of e star a,2} implies that 
the Segal $A$-precategory $e\star([e,1]\vee[a,1])$ is  the colimit of the  diagram
\[\begin{tikzcd}[column sep=0.7 cm]
	{[[2]\botimes a,1]} & {[[1]\otimes a,1]} & {[[1],1]\vee[a,1]} & {[e,1]\vee[a,1]} & {[e,2]\vee[a,1]}
	\arrow["{[d^0\otimes a,1]}"', from=1-2, to=1-1]
	\arrow["{[[1]\otimes a,d^1]}", from=1-2, to=1-3]
	\arrow["{[d^0\otimes a,2]}"', from=1-4, to=1-3]
	\arrow["{[a,d^1]}", from=1-4, to=1-5]
\end{tikzcd}\]
The fact that $e\star e \star[a,1]$ is the colimit of the given diagram then follows from the explicit expression of 
$e\star [\uvar a,1]$ as a colimit given in proposition \ref{prop:explicit expression of e star a,1}.
\end{proof}

\subsection{Adjunction with $\stratSset$}
\begin{construction}
\label{cons: of the adjunction with simplcii}
The (inverted) composition $g,f\mapsto g\circ f$ is a monoidal structure on the category of endomorphisms of $\stratSeg(A)$. The construction \ref{cons:liens entre star et starstar pour segal} shows that $e\star \uvar$ is a monoid for this monoidal structure. This induces a cosimplicial object: 
$$
\begin{array}{rcl}
\Delta &\to & \End(\stratSeg(A))\\
~[n] &\mapsto & [n]\star\uvar :=\underbrace{e\star e\star...\star e}_{n+1}\star \uvar
\end{array}
$$
We extend this functor to $\Delta_t$ by setting, for a stratified Segal $A$-precategory $C$ and an integer $n>0$:
\[\begin{tikzcd}
	{\underset{k\geq -1}{\coprod}~~\underset{D,~\tau^i_k(D)=D}{\coprod}~~\underset{D\to C}{\coprod}[n]\star D} & {[n]\star C} \\
	{\underset{k\geq -1}{\coprod}~~\underset{D,~\tau^i_k(D)=D}{\coprod}~~\underset{D\to C}{\coprod}\tau^i_{n+k}([n]\star D)} & {[n]_t\star C}
	\arrow[from=2-1, to=2-2]
	\arrow[""{name=0, anchor=center, inner sep=0}, from=1-1, to=1-2]
	\arrow[from=1-2, to=2-2]
	\arrow[from=1-1, to=2-1]
	\arrow["\lrcorner"{anchor=center, pos=0.125, rotate=180}, draw=none, from=2-2, to=0]
\end{tikzcd}
\]
where $\tau^i_{-1}$ is the constant functor with value $\emptyset$.
By left Kan extention, this gives a colimit preserving functor 
\begin{equation}
\stratSset\times  \stratSeg(A)\to \stratSeg(A).
\end{equation}
and evaluated on the empty Segal $A$-category, a colimit preserving functor 
\begin{equation}
\label{eq:def of the fundamental adjunction}
\stratSset\to \stratSeg(A).
\end{equation}
The image of $([n],\emptyset)$ (resp. $([n]_t,\emptyset)$) is  noted as $[n]$ (resp. $[n]_t$).

By construction, for $K,L$ two stratified sets and $D$ a stratified Segal $A$-precategory, we have $K\star (L\star C)\cong (K\star L)\star C$.
\end{construction}

\begin{remark}
We now have two functors from stratified simplicial sets to stratified Segal $A$-precategories. The one constructed in \ref{cons:star pour segal0}, and the one coming from the Gray module structure  of $\stratSeg(A)$ and constructed in \ref{cons:joint in gray module}. Moreover, Proposition \ref{prop:comparisons between star and otimes in segal} induces a weakly invertible natural transformation between them.

Both are denoted in the same way, but this should not create confusion because we will only consider the one constructed in \ref{cons:star pour segal0}.
\end{remark}

\begin{prop}
\label{prop:leibnizt joint is Quillen}
Let $K$ be a stratified simplicial set. The morphism $K\star\uvar:\stratSeg(A)\to \stratSeg(A)$ is a left Quillen endfunctor for the Segal and complete Segal space model structure. 
Moreover, if $i$ is a cofibration of stratified simplicial sets and $g$ an (resp. complete) acyclic cofibration of stratified Segal $A$-precategories, the morphism $i\hstar g$ is an (resp. complete) acyclic cofibration. 
\end{prop}
\begin{proof}
Since $\star$ preserves monomorphisms, the functor $\uvar\star\uvar:\Delta_{/K}\to \End( \stratSeg(A))$ is Reedy cofibrant. The theorem \ref{theo:hom colimi} then implies that it is sufficient to show that  for any integer $n$, $[n]\star\uvar $ is a left Quillen functor. In this case, this is a repeated application of proposition \ref{prop:star est quillen dans Segal}.
By diagram chasing and the use of two out of three, this implies the second assertion.
\end{proof}

\subsection{Complicial horn inclusions}
\label{section:Complicial horn inclusion}

\begin{notation*}
In this section, we will often consider morphisms $\tilde{a}\to \tilde{b}$ that fit into cocartesian squares:
\[\begin{tikzcd}
	a & b \\
	{\tilde{a}} & {\tilde{b}}
	\arrow[from=1-1, to=2-1]
	\arrow["i", from=1-1, to=1-2]
	\arrow[from=2-1, to=2-2]
	\arrow[from=1-2, to=2-2]
	\arrow["\lrcorner"{anchor=center, pos=0.125}, draw=none, from=1-1, to=2-2]
\end{tikzcd}\]
where $a\to \tilde{a}$ and $b\to \tilde{b}$ are epimorphisms.
To avoid complicating the notations unnecessarily, the induced morphism $\tilde{a}\to \tilde{b}$ will just be denoted $i$.
\end{notation*}

\begin{definition}
\label{defi:marked segal}
A \notion{marked Segal $A$-precategory} is a stratified Segal $A$-precategory having the right lifting property against all entire acyclic cofibrations. We denote by \wcnotation{$\mSeg(A)$}{(mseg@$\mSeg(A)$} the full subcategory of marked Segal $A$-precategory. We then have an adjunction: 
\[\begin{tikzcd}
	{(\uvar)_{\mk}:\stratSeg(A)} & {\mSeg(A):\iota}
	\arrow[""{name=0, anchor=center, inner sep=0}, shift left=2, from=1-2, to=1-1]
	\arrow[""{name=1, anchor=center, inner sep=0}, shift left=2, from=1-1, to=1-2]
	\arrow["\dashv"{anchor=center, rotate=-90}, draw=none, from=1, to=0]
\end{tikzcd}\]
where the left adjoint $(\uvar)_{\mk}$ sends a stratified Segal $A$-precategory $(C,tC)$ to the marked Segal $A$-precategory $(C,\overline{tC})$, where $\overline{tC}$ is the smaller stratification that includes $tC$ and makes $(C,\overline{tC})$ a marked Segal $A$-precategory, and where the right adjoint is a fully faithful inclusion.
Remark furthermore that at the level of preshaves, these two adjoints are the identity. We denote by\wcnotation{$r_C:C\to C_{\mk}$}{(rc@$r_C:C\to C_{\mk}$} the canonical inclusion. The proposition \ref{prop:X to Xmk is acycli cof} states that $r_C$ is an entire acyclic cofibration.

There is an isomorphism $(e\star C_{\mk})_{\mk}\cong (e\star C)_{\mk}$. Indeed $e\star\uvar$ preserves both entire cofibrations and weak equivalences, we have two entire acyclic cofibration $e\star C\to (e\star C)_{\mk}$ and $e \star C\to (e\star C_{\mk})_{\mk}$.  As the two codomain are marked, they are isomorphic. 
\end{definition}
The fact that will be used the most with the marked Segal $A$-precategory is their right lifting property with respect to morphisms of shape $[\tau^i_n(a),\Lambda^1[2]]\cup [a,2]\to [\tau^i_n(a),2]$. This fact will  be used freely.

\vspace{1cm}

 We recall that $[2]\botimes a$ is the following pushout:
\[\begin{tikzcd}
	{[1]\otimes a\amalg [1]\otimes a} && {[2]\otimes a} \\
	{e\star a\amalg e\star a} && {[2]\botimes a}
	\arrow[""{name=0, anchor=center, inner sep=0}, "{d^1\otimes a\amalg d^2\otimes a}", from=1-1, to=1-3]
	\arrow["{d^1\botimes a\amalg d^2\botimes a}"', from=2-1, to=2-3]
	\arrow[from=1-1, to=2-1]
	\arrow[from=1-3, to=2-3]
	\arrow["\lrcorner"{anchor=center, pos=0.125, rotate=180}, draw=none, from=2-3, to=0]
\end{tikzcd}\]
\begin{definition}
We define $[e,1]\vee(e\star[a,1])$ as the colimit of the following diagram
\[\begin{tikzcd}
	{[e,1]\vee[e\star a,1]} & {[e,1]\vee[a,1]} & {[e,2]\vee[a,1]}
	\arrow["{[d^0\star a,2]}"', from=1-2, to=1-1]
	\arrow["{[a,d^2]}", from=1-2, to=1-3]
\end{tikzcd}\]
The canonical composite morphism 
$$[e\star a,1]\xrightarrow{[e\star a,d^1]}[e,1]\vee[e\star a,1]\to [e,1]\vee(e\star[a,1])$$ 
is also denoted by $[e\star a,d^1]$. Eventually, we define $\overline{[1]\star[a,1]}$ as the following pushout
\[\begin{tikzcd}
	{[1]\star\{0\}} & {{[1]\star[a,1]}} \\
	{[2]_t} & {\overline{[1]\star[a,1]}}
	\arrow[from=1-1, to=2-1]
	\arrow[from=2-1, to=2-2]
	\arrow[from=1-2, to=2-2]
	\arrow[from=1-1, to=1-2]
	\arrow["\lrcorner"{anchor=center, pos=0.125, rotate=180}, draw=none, from=2-2, to=1-1]
\end{tikzcd}\] 
\end{definition}
\begin{lemma}
\label{lemma:le lemme quon voulais pas faire}
There is a weak equivalence from $\overline{[1]\star[a,1]}$ to the colimit of the diagram
\[\begin{tikzcd}
	{[[1]\star a,1]} & {[e\star a,1]} & {[e,1]\vee(e\star[ a,1])}
	\arrow["{[e\star a,d^1]}", from=1-2, to=1-3]
	\arrow["{[d^0\star a,1]}"', from=1-2, to=1-1]
\end{tikzcd}\]
making $\overline{[1]\star[a,1]}$ the homotopy colimit of the previous diagram.
\end{lemma}
\begin{proof}
The proposition \ref{prop:explicit expression of e star e star a,1} implies that $(\overline{[1]\star [a,1]})_{\mk}$ is the colimit of the diagram
\begin{equation}
\label{eq:changemeet markage2}
\begin{tikzcd}[column sep=0.7cm]
	{[[2]^2\botimes a,1]} & {[[1]_t\otimes a,1]} & {[[1]_t,1]\vee[a,1]} & {[e,1]\vee[a,1]} & {[e,2]\vee[a,1]} \\
	{[e\star a,1]} &&& {[a,1]} & {[e,1]\vee[a,1]} \\
	{[[1]\star a,1]} &&& {[e\star a,1]} & {[e,1]\vee[e\star a,1]}
	\arrow["{[d^0\otimes a,1]}"', from=1-2, to=1-1]
	\arrow["{[[1]\otimes a,d^1]}", from=1-2, to=1-3]
	\arrow["{[d^0\otimes a,2]}"', from=1-4, to=1-3]
	\arrow["{[a,d^1]}", from=1-4, to=1-5]
	\arrow["{[a,d^1]}", from=2-4, to=2-5]
	\arrow["{[d^1\botimes a,1]}", from=2-1, to=1-1]
	\arrow["{[a,d^1]}", from=2-4, to=1-4]
	\arrow["{[a,d^2]}"', from=2-5, to=1-5]
	\arrow["{[d^{0}\star a,2]}", from=2-5, to=3-5]
	\arrow["{[d^{0}\star a,1]}"', from=2-4, to=2-1]
	\arrow["{[d^{0}\star a,1]}"', from=2-4, to=3-4]
	\arrow["{[e\star a,d^1]}"', from=3-4, to=3-5]
	\arrow["{[d^1\star a,1]}"', from=2-1, to=3-1]
	\arrow["{[d^0\star a,1]}", from=3-4, to=3-1]
\end{tikzcd}
\end{equation}
In the previous diagram, the fact that we have $[[1]_t\otimes a,1]$  instead of $[[1]\otimes a,1]$ comes from the fact that we have considered $(\overline{[1]\star [a,1]})_{\mk}$ instead of $\overline{[1]\star [a,1]}$.

Consider now the morphism
\begin{equation}
\label{eq:changemeet markage}
[[2]^2\botimes a,1]\coprod_{[[1]_t\otimes a,1]}   [[1]_t,1]\vee[a,1]\to e\star[a,1]
\end{equation}
induces by the vertical colimit of the diagram
\begin{equation}
\label{eq:changemeet markage5}
\begin{tikzcd}
	{[[2]^2\botimes a,1]} & {[[1]_t\otimes a,1]} & {[[1]_t,1]\vee[a,1]} \\
	{[e\star a,1]} & {[a,1]} & {[e,1]\vee[a,1]}
	\arrow["{[d^0\otimes a,1]}"', from=1-2, to=1-1]
	\arrow["{[[1]\otimes a,d^1]}", from=1-2, to=1-3]
	\arrow["{[s^0,1]\vee[a,1]}", from=1-3, to=2-3]
	\arrow[from=2-2, to=2-3]
	\arrow[from=2-2, to=2-1]
	\arrow["{[s^0\otimes a,1]}", from=1-2, to=2-2]
	\arrow["{[s^0\botimes a,1]}"', from=1-1, to=2-1]
\end{tikzcd}
\end{equation}
As all the horizontal morphisms of \eqref{eq:changemeet markage5} are cofibrations, the colimit of each line is a homotopy colimit. As all the vertical morphisms of \eqref{eq:changemeet markage5} are weak equivalences, the morphism \eqref{eq:changemeet markage} also is a weak equivalence.  

Consider now the span
\begin{equation}
\label{eq:deojfoizejfgorezj}
 e\star[a,1]\xleftarrow{\eqref{eq:changemeet markage}} [[2]^2\botimes a,1]\coprod_{[[1]_t\otimes a,1]}   [[1]_t,1]\vee[a,1]\to (\overline{[1]\star [a,1]})_{\mk}
 \end{equation}
As the right hand morphism is a cofibration, and as  \eqref{eq:changemeet markage} is a weak equivalence, the canonical morphism from 
$(\overline{[1]\star [a,1]})_{\mk}$ to the colimit of \eqref{eq:deojfoizejfgorezj} is a weak equivalence.
Using the diagram \eqref{eq:changemeet markage2}, the colimit of \eqref{eq:deojfoizejfgorezj} is also the colimit of the following diagram
\begin{equation}
\label{eq:deojfoizejfgorezj3}
\begin{tikzcd}
	{e\star[a,1]} & {[e,1]\vee[a,1]} & {[e,2]\vee[a,1]} \\
	{[e\star a,1]} & {[a,1]} & {[e,1]\vee[a,1]} \\
	{[[1]\star a,1]} & {[e\star a,1]} & {[e,1]\vee[e\star a,1]}
	\arrow["{[a,d^1]}", from=1-2, to=1-3]
	\arrow["{[a,d^1]}", from=2-2, to=2-3]
	\arrow["{[a,d^1]}", from=2-2, to=1-2]
	\arrow["{[a,d^2]}"', from=2-3, to=1-3]
	\arrow["{[d^{0}\star a,2]}", from=2-3, to=3-3]
	\arrow["{[d^{0}\star a,1]}"', from=2-2, to=2-1]
	\arrow["{[d^{0}\star a,1]}"', from=2-2, to=3-2]
	\arrow["{[e\star a,d^1]}"', from=3-2, to=3-3]
	\arrow["{[d^1\star a,1]}"', from=2-1, to=3-1]
	\arrow["{[d^0\star a,1]}", from=3-2, to=3-1]
	\arrow[from=1-2, to=1-1]
	\arrow[from=2-1, to=1-1]
	\arrow["\lrcorner"{anchor=center, pos=0.125}, draw=none, from=1-1, to=2-2]
\end{tikzcd}
\end{equation}
As the  upper left square is cocartesian, the colimit of the diagram \ref{eq:deojfoizejfgorezj3} is equivalent to the colimit of
the diagram
\begin{equation}
\label{eq:deojfoizejfgorezj4}
\begin{tikzcd}
	&& {[e,2]\vee[a,1]} \\
	&& {[e,1]\vee[a,1]} \\
	{[[1]\star a,1]} & {[e\star a,1]} & {[e,1]\vee[e\star a,1]}
	\arrow["{[a,d^2]}"', from=2-3, to=1-3]
	\arrow["{[d^{0}\star a,2]}", from=2-3, to=3-3]
	\arrow["{[e\star a,d^1]}"', from=3-2, to=3-3]
	\arrow["{[d^0\star a,1]}", from=3-2, to=3-1]
\end{tikzcd}
\end{equation}
As the proposition \ref{prop:explicit expression of e star a,1} implies that  the colimit of the the diagram \ref{eq:deojfoizejfgorezj4} is equivalent to the one of the diagram given in the statement, this concludes the proof.
\end{proof}

\begin{lemma}
\label{lemma:le lemme quon voulais pas faire2}
The morphism 
$$[e,1]\vee(e\star[a,1])\cup e\star [e\star a,1]\to[e,1]\vee(e\star[e\star a,1])$$
is a weak equivalence. 
\end{lemma}
\begin{proof}
We have a cocartesian square
\begin{equation}
\label{eq:lemma:le lemme quon voulais pas faire2}
\begin{tikzcd}
	{[e,1]\cup e\star[ a,1]} & {[e,1]\cup e\star[e\star a,1]} \\
	{[e,1]\vee(e\star[a,1])} & {[e,1]\vee(e\star[a,1])\cup e\star [e\star a,1]}
	\arrow[from=1-1, to=2-1]
	\arrow[from=2-1, to=2-2]
	\arrow["{[e,1]\cup e\star[d^0\star a,1]}", from=1-1, to=1-2]
	\arrow[from=1-2, to=2-2]
\end{tikzcd}
\end{equation}
Remark that the left vertical morphism is the vertical colimit and homotopy colimit of the diagram
\[\begin{tikzcd}
	{[e,1]\cup[e\star a,1]} & {[e,1]\cup[a,1]} & {[e,1]\cup[e,1]\vee[a,1]} \\
	{[e,1]\vee[e\star a,1]} & {[e,1]\vee[a,1]} & {[e,2]\vee[a,1]}
	\arrow[from=2-2, to=2-1]
	\arrow[from=2-2, to=2-3]
	\arrow[from=1-2, to=1-1]
	\arrow[from=1-2, to=1-3]
	\arrow[from=1-2, to=2-2]
	\arrow[from=1-3, to=2-3]
	\arrow[from=1-1, to=2-1]
\end{tikzcd}\]
and is then a weak equivalence. 
This implies that the right vertical morphism of \eqref{eq:lemma:le lemme quon voulais pas faire2} is a weak equivalence. 
Similarly, $[e,1]\cup e\star[e\star a,1]\to [e,1]\vee(e\star[e\star a,1])$ is a weak equivalence. 
By two out of three this concludes the proof.
\end{proof}

\begin{lemma}
\label{lem:outer horn inclusion2}
The morphism $\{1\}\star [0]\to [1]_t\star [0]$ is an acyclic cofibration.
\end{lemma}
\begin{proof}
Using proposition \ref{prop:explicit expression of e star a,1} we deduce that $[1]_t\star [0]$ is the colimit of the diagram 
\[\begin{tikzcd}
	{[[1]_t,1]} & {[e,1]} & {[e,1]_t\vee[e,1]}
	\arrow[from=1-2, to=1-3]
	\arrow[from=1-2, to=1-1]
\end{tikzcd}\]
The inclusion $\{1\}\star [0]\to [1]_t\star [0]$ is then the composite of the following sequence
\[\begin{tikzcd}
	& {[e,1]} & {[[1]_t,1]} \\
	{[e,1]} & {[e,1]_t\vee[e,1]} & {[1]_t\star [0]}
	\arrow["{[e,d^0]}", from=2-1, to=2-2]
	\arrow[from=1-2, to=2-2]
	\arrow[from=1-3, to=2-3]
	\arrow["{[d^0,1]}", from=1-2, to=1-3]
	\arrow[from=2-2, to=2-3]
	\arrow["\lrcorner"{anchor=center, pos=0.125, rotate=180}, draw=none, from=2-3, to=1-2]
\end{tikzcd}\]
As the morphism $[e,d^0]$ and $[d^0,1]$ are acyclic cofibrations, this concludes the proof.
\end{proof}

\begin{lemma}
\label{lem:outer horn inclusion1}
The morphism $\{1\}\star[a,1]\to [1]_t\star[a,1]$ is an acyclic cofibration.
\end{lemma}
\begin{proof}
The Segal $A$-precategory $[1]_t\star [a,1]$ is the colimit  and the homotopy colimit of the diagram
\[\begin{tikzcd}
	{[1]\star\emptyset} && {[[1]\star a,1]} \\
	{[1]_t\star\emptyset} & {\overline{[1]\star[a,1]}} & {[[1]_t\star a,1]}
	\arrow[from=1-3, to=2-2]
	\arrow[from=1-1, to=2-1]
	\arrow[from=1-1, to=2-2]
	\arrow[from=1-3, to=2-3]
\end{tikzcd}\]
The lemma \ref{lemma:le lemme quon voulais pas faire} then implies that we have a weak equivalence from $[1]_t\star [a,1]$ to the colimit, denoted by $K$, of the diagram 
\[\begin{tikzcd}
	{[[1]_t\star a,1]} & {[e\star a,1]} & {[e,1]_t\vee(e\star [ a,1])}
	\arrow["{[e\star a,d^1]}", from=1-2, to=1-3]
	\arrow["{[d^0\star a,1]}"', from=1-2, to=1-1]
\end{tikzcd}\]
As the left hand morphism is a weak acyclic cofibration, so is the canonical morphism 
$$[e,1]_t\vee(e\star [ a,1])\to K.$$
We then have a commutative square
\[\begin{tikzcd}
	{\{1\}\star[a,1]} & {[1]_t\star[a,1]} \\
	{[e,1]_t\vee(e\star[a,1])} & K
	\arrow[from=2-1, to=2-2]
	\arrow[from=1-1, to=2-1]
	\arrow[from=1-1, to=1-2]
	\arrow[from=1-2, to=2-2]
\end{tikzcd}\]
where the two horizontal morphisms and the right vertical morphism are weak equivalences. The result the follows by two out of three.
\end{proof}

\begin{lemma}
\label{lem:horn_inclusion_2}
The morphism 
$\Lambda^1[2]\star [0]\to [2]_t\star [0]$
is an acyclic cofibration.
\end{lemma}
\begin{proof}
The Segal $A$-precategory $[2]_t\star [0]$ is the colimit of the following diagram
\[\begin{tikzcd}
	{[[2]_t,1]} & {[[2],1]} & {\overline{[1]\star[1]}}
	\arrow[from=1-2, to=1-1]
	\arrow[from=1-2, to=1-3]
\end{tikzcd}\]
The lemma \ref{lemma:le lemme quon voulais pas faire} then implies that we have a weak equivalence from $[2]_t\star [0]$ to the colimit, denoted by $K$, of the diagram 
\[\begin{tikzcd}
	{[[2]_t,1]} & {[[1],1]} & {[e,1]\vee(e\star[e,1])}
	\arrow["{[[1],d^1]}", from=1-2, to=1-3]
	\arrow["{[d^{0},1]}"', from=1-2, to=1-1]
\end{tikzcd}\]
On the other side, $\Lambda^1[2]\star [0]$ is the colimit of the diagram 
\[\begin{tikzcd}
	&&& {[e,1]} \\
	{[[1],1]} & {[e,1]} & {[e,2]} && {[[1],1]} & {[e,1]} & {[e,2]}
	\arrow["{[d^0,1]}"', from=2-6, to=2-5]
	\arrow["{[d^1,1]}", from=1-4, to=2-5]
	\arrow["{[e,d^1]}", from=2-3, to=2-2]
	\arrow["{[d^0,1]}", from=2-2, to=2-1]
	\arrow["{[e,d^0]}"', from=1-4, to=2-3]
	\arrow["{[e,d^1]}", from=2-6, to=2-7]
\end{tikzcd}\]
The composite $\Lambda^1[2]\star [0]\to [2]_t\star [0]\to K$ fits in the sequence of acyclic cofibrations
\[\begin{tikzcd}[column sep =0.3cm]
	{[e,d^0]\cup[e,d^2]} & {[e,3]} & {[\Lambda^1[2],1]} & {[[2]_t,1]} \\
	{\Lambda^1[2]\star [0]} & \bullet & \bullet & K \\
	& {([e,1]\cup[[1],1])\cup[e,1]\vee[\partial[1],1]} & {[e,1]\vee[[1],1]}
	\arrow[from=1-1, to=2-1]
	\arrow[""{name=0, anchor=center, inner sep=0}, from=1-1, to=1-2]
	\arrow[from=1-2, to=2-2]
	\arrow[from=2-1, to=2-2]
	\arrow[from=1-3, to=1-4]
	\arrow["\lrcorner"{anchor=center, pos=0.125, rotate=180}, draw=none, from=2-4, to=1-3]
	\arrow[from=2-3, to=2-4]
	\arrow[""{name=1, anchor=center, inner sep=0}, from=2-2, to=2-3]
	\arrow[from=1-4, to=2-4]
	\arrow[from=1-3, to=2-3]
	\arrow[from=3-2, to=2-2]
	\arrow[from=3-2, to=3-3]
	\arrow[from=3-3, to=2-3]
	\arrow["\lrcorner"{anchor=center, pos=0.125, rotate=180}, draw=none, from=2-2, to=0]
	\arrow["\lrcorner"{anchor=center, pos=0.125, rotate=180}, draw=none, from=3-3, to=1]
\end{tikzcd}\]
and is then a weak equivalence.
By two out of three, this concludes the proof.
 \end{proof}

\begin{lemma}
\label{lem:horn_inclusion_3}
The morphism 
$\Lambda^1[2]\star [a,1]\to [2]_t\star [a,1]$
is an acyclic cofibration.
\end{lemma}
\begin{proof}
The lemma \ref{lem:horn_inclusion_2} implies that the inclusion $\Lambda^1[2]\star [a,1]\to \Lambda^1[2]\star [a,1]\cup [2]_t\star \{0\}$ is an acyclic cofibration.
Using proposition \ref{prop:explicit expression of e star a,1}, we deduce that the
 Segal $A$-precategory $[2]_t\star[a,1]$ is the colimit of the diagram
\[\begin{tikzcd}
	{[1]\star[e,1]} && {[1]\star[a,1]} \\
	{\overline{\overline{[1]\star[e,1]}}} & {[1]\star([e,1]\vee[a,1])} & {\overline{\overline{[1]\star[e\star a,1]}}}
	\arrow["{[1]\star[d^0\star a,1]}", from=1-3, to=2-3]
	\arrow["{[1]\star[a,d^1]}"', from=1-3, to=2-2]
	\arrow[from=1-1, to=2-1]
	\arrow[from=1-1, to=2-2]
\end{tikzcd}\]
while $ \Lambda^1[2]\star [a,1]\cup [2]_t\star \{0\}$ is the colimit  of the diagram
\[\begin{tikzcd}
	{\{1\}\star[e,1]} && {\{1\}\star[a,1]} \\
	{\overline{\overline{[1]\star[e,1]}}} & {\{1\}\star([e,1]\vee[a,1])\cup[1]\star[a,1]} & {\{1\}\star[e\star a,1]}
	\arrow["{\{1\}\star[a,d^1]}"', from=1-3, to=2-2]
	\arrow["{\{1\}\star[d^0\star a,1]}", from=1-3, to=2-3]
	\arrow[from=1-1, to=2-1]
	\arrow[from=1-1, to=2-2]
\end{tikzcd}\]
where $\overline{\overline{[1]\star[e,1]}}:=[2]_t\star [0]$ and where $\overline{\overline{[1]\star[e\star a,1]}}$ is the following pushout:
\[\begin{tikzcd}
	{ [[2]\star a,1]} & {e\star [[1]\star a,1]} & {\overline{[1]\star [e\star a,1]}} \\
	{ [[2]_t\star a,1]} && {\overline{\overline{[1]\star[e\star a,1]}}}
	\arrow[from=1-1, to=1-2]
	\arrow[from=1-2, to=1-3]
	\arrow[from=1-1, to=2-1]
	\arrow[from=2-1, to=2-3]
	\arrow[from=1-3, to=2-3]
	\arrow["\lrcorner"{anchor=center, pos=0.125, rotate=180}, draw=none, from=2-3, to=1-2]
\end{tikzcd}\]
Let $K_1$ be the following pushout:
\[\begin{tikzcd}
	{\{1\}\star ([e,1]\vee [a,1]) \cup [1]\star ([e,1]\cup [a,1]) } & {\Lambda^1[2]\star[a,1]\cup [2]_t\star\{0\}} \\
	{[1]\star ([e,1]\vee[a,1])} & {K_1}
	\arrow[from=1-1, to=2-1]
	\arrow[""{name=0, anchor=center, inner sep=0}, from=1-1, to=1-2]
	\arrow[from=2-1, to=2-2]
	\arrow[from=1-2, to=2-2]
	\arrow["\lrcorner"{anchor=center, pos=0.125, rotate=180}, draw=none, from=2-2, to=0]
\end{tikzcd}\]
The left-hand morphism is equal to $(d^0:[0]\to [1])\hstar ([e,1]\cup[a,1]\to [e,1]\vee[a,1])$ which is an acyclic cofibration according to proposition \ref{prop:leibnizt joint is Quillen}.
Furthermore, the morphism $K_1\to [2]_t\star [a,1]$ fits in the following pushout:
\[\begin{tikzcd}
	{\overline{[1]\star[ a,1]}\cup\{1\}\star[e\star a,1]} & {K_1} \\
	{\overline{\overline{[1]\star[e\star a,1]}}} & {[2]_t\star[a,1]}
	\arrow[from=1-2, to=2-2]
	\arrow[""{name=0, anchor=center, inner sep=0}, from=1-1, to=1-2]
	\arrow[from=1-1, to=2-1]
	\arrow[from=2-1, to=2-2]
	\arrow["\lrcorner"{anchor=center, pos=0.125, rotate=180}, draw=none, from=2-2, to=0]
\end{tikzcd}\]
To conclude, we will prove that the left vertical morphism is a weak equivalence.

The lemma \ref{lemma:le lemme quon voulais pas faire} implies that we have a weak equivalence from $\overline{[1]\star[ a,1]}\cup\{1\}\star[e\star a,1]$ to the colimit, denoted by $K_2$, of the diagram 
\[\begin{tikzcd}
	{[[1]\star a,1]} & {[e\star a,1]} & {[e,1]\vee(e\star[a,1])\cup \{1\}\star[e\star a,1]}
	\arrow["{[e\star a,d^1]}", from=1-2, to=1-3]
	\arrow["{[d^0\star a,1]}"', from=1-2, to=1-1]
\end{tikzcd}\]
We now define $K_3$ as the colimit of the diagram
\[\begin{tikzcd}
	{[\Lambda^1[2]\star  a,1]} & {[[1]\star a,1]} & {[e,1]\vee(e\star[e\star a,1])}
	\arrow["{[d^0\star a,1]}"', from=1-2, to=1-1]
	\arrow["{[[1]\star a,d^1]}", from=1-2, to=1-3]
\end{tikzcd}\]
The canonical morphism $K_2\to K_3$ fits in the cocartesian square
\[\begin{tikzcd}
	{[e,1]\vee(e\star[a,1])\cup \{1\}\star[e\star a,1]} & {K_2} \\
	{[e,1]\vee(e\star[e\star a,1])} & {K_3}
	\arrow[from=1-1, to=2-1]
	\arrow[from=2-1, to=2-2]
	\arrow[""{name=0, anchor=center, inner sep=0}, from=1-1, to=1-2]
	\arrow[from=1-2, to=2-2]
	\arrow["\lrcorner"{anchor=center, pos=0.125, rotate=180}, draw=none, from=2-2, to=0]
\end{tikzcd}\]
and is then a weak equivalence according to the lemma \ref{lemma:le lemme quon voulais pas faire2}.

On the other side, the lemma \ref{lemma:le lemme quon voulais pas faire} also implies that we have a weak equivalence from $\overline{\overline{[1]\star[e\star a,1]}}$ to the colimit, denoted by $K_4$, of the diagram
\[\begin{tikzcd}
	{[[2]_t\star  a,1]} & {[[1]\star a,1]} & {[e,1]\vee(e\star[e\star a,1])}
	\arrow["{[d^0\star a,1]}"', from=1-2, to=1-1]
	\arrow["{[[1]\star a,d^1]}", from=1-2, to=1-3]
\end{tikzcd}\]

Remark now that all the morphisms appearing in the diagrams that define $K_3$ and $K_4$ are cofibrations.
As $\Lambda^1[2]\star a \to [2]_t\star a$ is a weak equivalence in $A$, this implies that the canonical morphism $K_3\to K_4$ is also a weak equivalence. We then have  commutative diagram:
\[\begin{tikzcd}
	{[1]\star[a,1]\cup \{1\}\star[e\star a,1]} && {\overline{[1]\star[e\star a,1]}} \\
	{K_2} & {K_3} & {K_4}
	\arrow["\sim"', from=2-2, to=2-3]
	\arrow[from=1-1, to=1-3]
	\arrow["\sim", from=1-3, to=2-3]
	\arrow["\sim"', from=1-1, to=2-1]
	\arrow["\sim"', from=2-1, to=2-2]
\end{tikzcd}\]
where all arrows labelled by $\sim$ are weak equivalences. By two out of three, this implies the result.
\end{proof}

\begin{prop}
\label{prop:horn_inclusion_4}
For any stratified Segal $A$-precategory $C$, the morphisms $\Lambda^1[2]\star C\to [2]_t\star C$ and $\{\epsilon\}\star C\to [1]_t\star C$ with $\epsilon\in\{0,1\}$ are
 acyclic cofibrations.
Moreover, for any cofibration of stratified Segal $A$-precategory $i$, and $j$ being either $\{1\}\to [1]_t$ or $\Lambda^1[2]\to [2]_t$, the morphism $j\hstar i$ is an acyclic cofibration.
\end{prop}
\begin{proof}
We begin with the first assertion. By two out of three, we can suppose that $\epsilon:=1$.
The proposition \ref{prop:leibnizt joint is Quillen} implies that  $\Lambda^1[2]\star\uvar$ and $[2]_t\star\uvar$ are left Quillen functors.  As every object is a homotopy colimit of objects of shape $[a,n]$ or $[e,1]_t$, we can reduce to the case where $C$ is of this shape.
Using Segal extensions, we can reduce to the case where $C$ is $[a,1]$, $[0]$ or $[e,1]_t$.

If $C$ is $[a,1]$ or $[0]$, the result follows from lemmas \ref{lem:outer horn inclusion2}, \ref{lem:outer horn inclusion1}, \ref{lem:horn_inclusion_2} and \ref{lem:horn_inclusion_3}.
Eventually, for $C := [e,1]_t$, we have a diagram:
\[\begin{tikzcd}
	{\{1\}\star[e,1]_t} & {\{0\}\star [0]} && {\Lambda^1[2]\star [e,1]_t} & {\Lambda^1[2]\star [0]} \\
	{[1]_t\star[e,1]_t} & {[1]_t\star [0]} && {[2]_t\star[e,1]_t} & {[2]_t\star [0]}
	\arrow[from=1-4, to=2-4]
	\arrow[from=1-4, to=1-5]
	\arrow[from=2-4, to=2-5]
	\arrow[from=1-5, to=2-5]
	\arrow[from=2-1, to=2-2]
	\arrow[from=1-1, to=2-1]
	\arrow[from=1-2, to=2-2]
	\arrow[from=1-1, to=1-2]
\end{tikzcd}\]
The proposition \ref{prop:leibnizt joint is Quillen} and the lemmas \ref{lem:outer horn inclusion2} and \ref{lem:horn_inclusion_2} imply that all
 horizontal morphisms and right vertical morphisms are weak equivalences. By two out of three, this implies that the left vertical morphisms are weak equivalences.

This concludes the proof of the first assertion. The second one is obtained with some diagram chasing.
\end{proof}

\begin{prop}
\label{prop:horn_inclusion}
The functor $\stratSset\to \stratSeg(A)$ sends complicial horn inclusions to weak equivalences.
\end{prop}
\begin{proof}
Let $k\leq n$ be two integers. First, we suppose that $0<k<n$. We then have an equality $$(\Lambda^k[n]\to [n]^k)= (\partial[k-2]\to [k-2])\hstar (\Lambda^1[2]\to [2]_t)\hstar (\partial [n-k-2]\to [n-k-2]).$$ This is an acyclic cofibration according to propositions \ref{prop:leibnizt joint is Quillen} and  \ref{prop:horn_inclusion_4}. If $k=0$, we have an equality $$(\Lambda^0[n]\to [n]^0) = (\{1\}\to [e,1]_t)\hstar (\partial[n-2]\to [n-2])$$
and the right hand morphism is an acyclic cofibration again thanks to proposition \ref{prop:horn_inclusion_4}. Eventually, for $k=n$, note that $$(\Lambda^n[n]\to [n]^n) = (\partial[n-2]\to [n-2])\hstar (\{0\}\to [e,1]_t).$$ This morphism is an acyclic cofibration according to proposition \ref{prop:leibnizt joint is Quillen}.
\end{proof}

\subsection{Complicial thinness extensions}
\label{section:Complicial thinness extensions}
\begin{notation*}
In this section, we will often consider morphisms $\tilde{a}\to \tilde{b}$ that fit into cocartesian squares:
\[\begin{tikzcd}
	a & b \\
	{\tilde{a}} & {\tilde{b}}
	\arrow["i", from=1-1, to=1-2]
	\arrow[from=1-1, to=2-1]
	\arrow[from=1-2, to=2-2]
	\arrow[from=2-1, to=2-2]
	\arrow["\lrcorner"{anchor=center, pos=0.125, rotate=180}, draw=none, from=2-2, to=1-1]
\end{tikzcd}\]
where $a\to \tilde{a}$ and $b\to \tilde{b}$ are epimorphisms.
To avoid complicating the notations unnecessarily, the induced morphism $\tilde{a}\to \tilde{b}$ will just be denoted $i$.
\end{notation*}

\begin{lemma}
\label{lemma:thinnes extension case 0 and n}
Morphisms $([n]^0)'\to ([n]^0)''$ and $([n]^n)' \to ([n]^n)''$ are acyclic cofibrations.
\end{lemma}
\begin{proof}
For $k$ equal to $0$ or $n$, we have pushout diagrams: 
\[\begin{tikzcd}
	{[n]^k} & {([n]^k)'} & {([n]^k)''} \\
	{[n-1]} & {[n-1]_t} & {[n-1]_t}
	\arrow[from=1-1, to=2-1]
	\arrow[from=1-1, to=1-2]
	\arrow[from=1-2, to=2-2]
	\arrow[from=2-1, to=2-2]
	\arrow[from=1-2, to=1-3]
	\arrow[from=1-3, to=2-3]
	\arrow["id"', from=2-2, to=2-3]
	\arrow["\lrcorner"{anchor=center, pos=0.125, rotate=180}, draw=none, from=2-2, to=1-1]
	\arrow["\lrcorner"{anchor=center, pos=0.125, rotate=180}, draw=none, from=2-3, to=1-2]
\end{tikzcd}\]
Propositions \ref{prop:leibnizt joint is Quillen} and \ref{prop:horn_inclusion_4} imply that both $s^0:[n]^0\to [n-1]$ and $s^{n-1}:[n]^{n-1}\to [n-1]$ are weak equivalences. As horizontal morphisms are cofibrations, the left properness imply that all the vertical morphisms are weak equivalences.
 By two out of three, this shows that $([n]^k)' \to ([n]^k)''$ is a weak equivalence. 
\end{proof}

\begin{construction}
\label{cons:the big construction}
The propositions \ref{prop:explicit expression of e star a,1} and \ref{prop:explicit expression of e star a,2} provide canonical morphisms:
$$\begin{array}{cc}
\alpha_a:[e\star a,1]\to e\star [a,1]&\beta_a:[e,1]\vee [ a,1]\to e\star [a,1] \\
 \delta_a:[e\star a,1]\vee[a,1]\to e\star [a,2] & \epsilon_a:[[2]\botimes a,1]\to e\star[a,2] \end{array}$$
where $[2]\botimes a$ and $[e\star a,1]\vee[a,1]$ are the following pushouts:
\[\begin{tikzcd}
	{[1]\otimes a\amalg [1]\otimes a} && {[2]\otimes a} && {[[1]\otimes a,1]\amalg[[1]\otimes a,1]} & {[[1]\otimes a,2]} \\
	{e\star a\amalg e\star a} && {[2]\botimes a} && {[e\star a,1]\amalg [a,1]} & {[e\star a,1]\vee[a,1]}
	\arrow[""{name=0, anchor=center, inner sep=0}, "{d^1\otimes a\amalg d^2\otimes a}", from=1-1, to=1-3]
	\arrow["{d^1\botimes a\amalg d^2\botimes a}"', from=2-1, to=2-3]
	\arrow[from=1-1, to=2-1]
	\arrow[from=1-3, to=2-3]
	\arrow[from=1-5, to=2-5]
	\arrow[""{name=1, anchor=center, inner sep=0}, "{[[1]\otimes a,d^2\amalg d^0]}", from=1-5, to=1-6]
	\arrow[from=2-5, to=2-6]
	\arrow[from=1-6, to=2-6]
	\arrow["\lrcorner"{anchor=center, pos=0.125, rotate=180}, draw=none, from=2-3, to=0]
	\arrow["\lrcorner"{anchor=center, pos=0.125, rotate=180}, draw=none, from=2-6, to=1]
\end{tikzcd}\]
Moreover they fit in the following commutative diagram:
\[\begin{tikzcd}
	{[a,1]} & {[e,1]\vee[a,1]} & {[a,1]} & {[e\star a,1]} \\
	\textcolor{white}{e\star [a,1]} & {e\star [a,1]} & {[e,1]\vee[a,1]} & {e\star [a,1]} \\
	{[e\star a,1]} & {[e\star a,1]\vee[a,1]} & {[e\star a,1]} & {[[2]\botimes a,1]} \\
	{e\star[ a,1]} & {e\star[a,2]} & {e\star[ a,1]} & {e\star[a,2]} \\
	{[[1]\otimes a,1]} & {[[2]\botimes a,1]} & {[e\star a,1]} & {[[2]\botimes a,1]} \\
	{[e\star a,1]\vee[a,1]} & {e\star[a,2]} & {e\star[ a,1]} & {e\star[a,2]}
	\arrow["{\delta_a}", from=3-2, to=4-2]
	\arrow["{\alpha_a}"', from=3-3, to=4-3]
	\arrow["{[d^1\botimes a,1]}", from=3-3, to=3-4]
	\arrow["{\epsilon_a}", from=3-4, to=4-4]
	\arrow["{e\star [a,d^1]}"', from=4-3, to=4-4]
	\arrow["{\alpha_a}"', from=3-1, to=4-1]
	\arrow["{e\star[a,d^2]}"', from=4-1, to=4-2]
	\arrow["{[e\star a,d^2]}", from=3-1, to=3-2]
	\arrow["{[a,d^1]}"', from=1-3, to=2-3]
	\arrow["{[d^{0}\star a,1]}", from=1-3, to=1-4]
	\arrow["{\alpha_a}", from=1-4, to=2-4]
	\arrow["{\beta_a}", from=2-3, to=2-4]
	\arrow["{\beta_a}", from=1-2, to=2-2]
	\arrow["{[a,d^0]}", from=1-1, to=1-2]
	\arrow["{(3):}"{description}, shift right=5, curve={height=30pt}, draw=none, from=3-1, to=4-1]
	\arrow["{[[1]\otimes a,d^1]}"', from=5-1, to=6-1]
	\arrow["{\epsilon_a}", from=5-2, to=6-2]
	\arrow["{[d^0\otimes a,1]}", from=5-1, to=5-2]
	\arrow["{(5):}"{description}, shift right=5, curve={height=30pt}, draw=none, from=5-1, to=6-1]
	\arrow["{\delta_a}"', from=6-1, to=6-2]
	\arrow["{:(4)}"{description}, shift left=5, curve={height=-30pt}, draw=none, from=3-4, to=4-4]
	\arrow["{d^{0}\star {[a,1]}}"', from=1-1, to=2-2]
	\arrow["{(1):}"{description}, shift right=5, curve={height=30pt}, draw=none, from=1-1, to=2-1]
	\arrow["{[d^2\botimes a,1]}", from=5-3, to=5-4]
	\arrow["{e\star[a,d^0]}"', from=6-3, to=6-4]
	\arrow["{\alpha_a}"', from=5-3, to=6-3]
	\arrow["{\epsilon_a}", from=5-4, to=6-4]
	\arrow["{:(6)}"{description}, shift left=5, curve={height=-30pt}, draw=none, from=5-4, to=6-4]
	\arrow["{:(2)}"{description}, shift left=5, curve={height=-30pt}, draw=none, from=1-4, to=2-4]
\end{tikzcd}\]
\end{construction}
\begin{definition}
Let $b$ be an object of $A$ and $x:a\to b,~x':a'\to b$ two morphisms. The element $b$ is \wcnotion{$n$-relying on $x$}{relying on $x$@$n$-relying on $x$} if for any $k\geq -1$, the following square is homotopy cocartesian:
\[\begin{tikzcd}
	{[k]\star a} & {[k]\star b} \\
	{\tau^i_{n+k+1}([k]\star a)} & {\tau^i_{n+k+1}([k]\star b)}
	\arrow[from=1-1, to=1-2]
	\arrow[from=2-1, to=2-2]
	\arrow[from=1-1, to=2-1]
	\arrow[from=1-2, to=2-2]
\end{tikzcd}\]
The element $b$ is \wcnotion{$n$-relying on $x$ and $x'$}{relying on $x$ and $x'$@$n$-relying on $x$ and $x'$} if for any $k\geq -1$, the following square is homotopy cocartesian:
\[\begin{tikzcd}
	{[k]\star a\amalg [k]\star a'} & {[k]\star b} \\
	{\tau^i_{n+k+1}([k]\star a)\amalg \tau^i_{n+k+1}([k]\star a')} & {\tau^i_{n+k+1}([k]\star b)}
	\arrow[tail reversed, no head, from=2-1, to=1-1]
	\arrow[from=1-1, to=1-2]
	\arrow[from=2-1, to=2-2]
	\arrow[tail reversed, no head, from=2-2, to=1-2]
\end{tikzcd}\]
\end{definition}

\begin{remark}
We recall that we denote by $C_{\mk}$ the marked Segal $A$-precategory associated to a stratified Segal $A$-precategory $C$. 
The canonical inclusion $C\to C_{\mk}$ is denoted $r_C$ and is an acyclic cofibration according to the proposition \ref{prop:X to Xmk is acycli cof}. These notions and notations are defined in definition \ref{defi:marked segal}.
The fact that will be used the most with the marked Segal $A$-precategory is their right lifting property with respect to morphisms of shape $[\tau^i_n(a),\Lambda^1[2]]\cup [a,2]\to [\tau^i_n(a),2]$. This fact will  be used freely.
\end{remark}
\begin{definition}
\label{def:order relation case 1}
Let $C$ be a Segal $A$-precategory. We define the relation \wcnotation{$\geq_{n}$}{((g37@$\geq_{n}$} on morphisms of shape $[a,1]\to C$ for $a$ verifying $\tau^i_{n}a = a$, as the smallest reflexive and transitive relation such that $(x:[a,1]\to C) \geq_n (x':[a',1]\to C)$ whenever one of the three following conditions is verified:
\begin{enumerate}
\item The elements $a$ and $a'$ are equal and there exists a lifting the following diagram:
\[\begin{tikzcd}
	{[a,1]} \\
	& {[a,1]\vee[\tau^i_{n-1}a,1]} & C \\
	{[a,1]}
	\arrow["{[a,d^2]}"', from=1-1, to=2-2]
	\arrow["x", curve={height=-12pt}, from=1-1, to=2-3]
	\arrow["{[a,d^1]}", from=3-1, to=2-2]
	\arrow[dotted, from=2-2, to=2-3]
	\arrow["{x'}"', curve={height=12pt}, from=3-1, to=2-3]
\end{tikzcd}\]
\item The elements $a$ and $a'$ are equal and there exists a lifting in the following diagram:
\[\begin{tikzcd}
	{[a,1]} \\
	& {[\tau^i_{n-1}a,1]\vee[a,1]} & C \\
	{[a,1]}
	\arrow["{[a,d^0]}"', from=1-1, to=2-2]
	\arrow["x", curve={height=-12pt}, from=1-1, to=2-3]
	\arrow["{[a,d^1]}", from=3-1, to=2-2]
	\arrow[dotted, from=2-2, to=2-3]
	\arrow["{x'}"', curve={height=12pt}, from=3-1, to=2-3]
\end{tikzcd}\]
\item There exists an element $b$ which is $(n-1)$-relying on $a\to b$ and dotted arrows in the following diagram:
\[\begin{tikzcd}
	{[a,1]} \\
	& {[b,1]} & {C_{\mk}} \\
	{[a',1]}
	\arrow["{ }"', from=1-1, to=2-2]
	\arrow[dotted, from=3-1, to=2-2]
	\arrow["{r_C\circ x}", curve={height=-12pt}, from=1-1, to=2-3]
	\arrow["{r_C\circ x'}"', curve={height=12pt}, from=3-1, to=2-3]
	\arrow[dotted, from=2-2, to=2-3]
\end{tikzcd}\]
\end{enumerate}
\end{definition}

\begin{definition}
\label{def:order relation case 2}
We also set $(\bar{x}:[\bar{a},1]\to C,\bar{x}':[\bar{a}',1]\to C)\geq_n \bar{x}'':[\bar{a}'',1]\to C$ if there exists three elements $x:[a,1]\to C$, $x':[a',1]\to C$ and $x'':[a'',1]\to C$ such that $\bar{x}\geq_n x$, $\bar{x}'\geq_n x'$, $x''\geq_n \bar{x}''$ and one of the two following conditions is verified:
\begin{enumerate}
\item The elements $a$, $a'$ and $a''$ are equal and there exists a dotted arrow:
\[\begin{tikzcd}
	{[a,1]\cup[a,1]} \\
	& {[a,2]} & C \\
	{[a,1]}
	\arrow["{[a,d^2\cup d^0]}"'{pos=0.1}, from=1-1, to=2-2]
	\arrow["{[a,d^1]}"{pos=0.4}, from=3-1, to=2-2]
	\arrow["{x\cup x'}", curve={height=-12pt}, from=1-1, to=2-3]
	\arrow["{x''}"', curve={height=12pt}, from=3-1, to=2-3]
	\arrow[dotted, from=2-2, to=2-3]
\end{tikzcd}\]
\item There exists an element $b$ which is $(n-1)$-relying on $a\to b$ and $a'\to b$ and dotted arrows in the following diagram:
\[\begin{tikzcd}
	{[a,1]\amalg[a',1]} \\
	& {[b,1]} & {C_{mk}} \\
	{[a'',1]}
	\arrow[from=1-1, to=2-2]
	\arrow[dotted, from=3-1, to=2-2]
	\arrow["{r_C\circ x\amalg r_C\circ x'}", curve={height=-12pt}, from=1-1, to=2-3]
	\arrow["{r_C\circ x''}"', curve={height=12pt}, from=3-1, to=2-3]
	\arrow[dotted, from=2-2, to=2-3]
\end{tikzcd}\]
\end{enumerate}
\end{definition}

\begin{prop}
\label{prop:meaning of geq case 1}
Let $C$ be a stratified Segal $A$-precategory and $x:[a,1]\to C$, $y:[a',1]\to C$ two morphisms such that $x\geq_n y$. The morphism 
$$ C\coprod_{[a,1]} \tau^i_n( [a,1])\to \tau^i_n( [a',1])\coprod_{[a',1]}C\coprod_{[a,1]} \tau^i_n( [a,1])$$
is an acyclic cofibration. 
\end{prop}
\begin{proof}
By two out of three, we can suppose without loss of generality that $C$ is already a marked Segal $A$-precategory. We suppose first that $x$ and $y$ fulfill one of the three cases of definition \ref{def:order relation case 1}. The following square is then homotopy cartesian: 
\[\begin{tikzcd}
	{[a,1]} & C \\
	{\tau^i_n[a,1]} & {\tau^i_n[a,1]\amalg_{[a,1]} C\coprod_{[a',1]} \tau^i_n[a',1]}
	\arrow["x", from=1-1, to=1-2]
	\arrow[from=1-1, to=2-1]
	\arrow[from=1-2, to=2-2]
	\arrow[from=2-1, to=2-2]
\end{tikzcd}\]
As the cocartesian square:
\[\begin{tikzcd}
	{[a,1]} & C \\
	{\tau^i_n[a,1]} & {\tau^i_n[a,1]\amalg_{[a,1]} C}
	\arrow["x", from=1-1, to=1-2]
	\arrow[from=1-1, to=2-1]
	\arrow[from=1-2, to=2-2]
	\arrow[from=2-1, to=2-2]
	\arrow["\lrcorner"{anchor=center, pos=0.125, rotate=180}, draw=none, from=2-2, to=1-1]
\end{tikzcd}\]
is also homotopy cocartesian, this
 implies that $$C\coprod_{[a,1]} \tau^i_n( [a,1])\to \tau^i_n( [a',1])\coprod_{[a',1]}C\coprod_{[a,1]} \tau^i_n( [a,1])$$ is an acyclic cofibration. Suppose now that there exists a familly of morphisms $(x_k:[a_k,1])_{k\leq m}\to C$ such that $x_0=x$, $x_m=y$ and for any $k$, $x_k$ and $x_{k+1}$ fullfill one of the three cases of definition \ref{def:order relation case 1}. We then have two homotopy cocartesian squares:
\[\begin{tikzcd}
	{C\coprod_{[a',1]} \tau^i_n[a',1]} & {[a,1]} & C \\
	{C\coprod_{\coprod_{k\leq m}[a_k,1]}\coprod_{k\leq m}\tau^i_n[a_k,1]} & {\tau^i_n[a,1]} & {C\coprod_{\coprod_{k\leq m}[a_k,1]}\coprod_{k\leq m}\tau^i_n[a_k,1]}
	\arrow[from=1-2, to=1-3]
	\arrow[from=1-2, to=2-2]
	\arrow[from=1-3, to=2-3]
	\arrow[from=2-2, to=2-3]
	\arrow[from=1-2, to=1-1]
	\arrow[from=2-2, to=2-1]
	\arrow[from=1-1, to=2-1]
\end{tikzcd}\]
As before, this implies that $$C\coprod_{[a,1]} \tau^i_n( [a,1])\to C\coprod_{\coprod_{k\leq m}[a_k,1]}\coprod_{k\leq m}\tau^i_n[a_k,1]$$ and $$\tau^i_n( [a',1])\coprod_{[a',1]}C\coprod_{[a,1]} \tau^i_n( [a,1])\to C\coprod_{\coprod_{k\leq m}[a_k,1]}\coprod_{k\leq m}\tau^i_n[a_k,1]$$ are acyclic cofibrations. 
By two out of three, this implies the result.
\end{proof}

One can show similarly:

\begin{prop}
\label{prop:meaning of geq case 2}
Let $C$ be a stratified Segal $A$-precategory, and $x:[a,1]\to C$, $y:[a',1]\to C$ and $z:[a'',1]\to C$ three morphisms such that $(x,y)\geq_n z$. The morphism 
$$ \tau^i_n( [a',1])\coprod_{[a',1]}C\coprod_{[a,1]} \tau^i_n( [a,1])\to \tau^i_n( [a',1])\coprod_{[a',1]}C\coprod_{[a,1]} \tau^i_n( [a,1]) \coprod_{[a'',1]} \tau^i_n( [a'',1])$$
is an acyclic cofibration. 
\end{prop}

\begin{lemma}
\label{lem:abstract thinness 0}
Let $n$ be a non null integer and $a$ an element such that $\tau^i_{n}(a)=a$. The object $[2]^2\otimes a$ is $n$-relying on $d^1\botimes a:e\star a\to [2]^2\botimes a$.
\end{lemma}
\begin{proof}
As the morphism $d^1\botimes a:e\star a\to [2]^2\botimes a$ is a weak equivalence, so are the horizontal morphisms of the following diagram:
\[\begin{tikzcd}
	{[k]\star e\star a} & {[k]\star([2]^2\botimes a)} \\
	{\tau^i_{n+k+1}([k]\star e\star a)} & {\tau^i_{n+k+1}([k]\star([2]^2\botimes a))}
	\arrow["\sim", from=1-1, to=1-2]
	\arrow[from=1-1, to=2-1]
	\arrow[from=1-2, to=2-2]
	\arrow["\sim"', from=2-1, to=2-2]
\end{tikzcd}\]
As the vertical morphisms are cofibrations, this implies that this square is homotopy cocartesian.
\end{proof}

\begin{lemma}
\label{lem:abstract thinness 1}
Let $n$ be a non null integer and $a$ an element such that $\tau^i_{n}(a)=a$. The object $[2]\botimes a$ is $n$-relying on $d^0\otimes a:[1]\otimes a\to [2]\botimes a$ and $d^2\otimes a:e\star a\to [2]\otimes a$. Moreover, $[2]\botimes a\coprod_{d^0\otimes a} \tau^i_{n}([1]\otimes a)$ (resp. $[2]\botimes a\coprod_{d^2\botimes a} \tau^i_{n}(e\star a)$) is $n$-relying on $d^{2}\otimes a$ (resp. $d^0\botimes a$).
\end{lemma}
\begin{proof}
Consider the following diagram:
\[\begin{tikzcd}[column sep=0.3cm]
	{[k]\star([1]\otimes a)\amalg [k]\star([1]\otimes a)} & {[k]\star(\Lambda^1[2]\otimes a)} & {[k]\star([2]\otimes a)} \\
	{\tau^i_{n+k+1}([k]\star([1]\otimes a))\amalg \tau^i_{n+k+1}([k]\star([1]\otimes a))} & {\tau^i_{n+k+1}([k]\star(\Lambda^1[2]\otimes a))} & {\tau^i_{n+k+1}([k]\star([2]\otimes a))}
	\arrow[from=1-1, to=2-1]
	\arrow[""{name=0, anchor=center, inner sep=0}, from=1-1, to=1-2]
	\arrow[from=1-2, to=2-2]
	\arrow[from=2-1, to=2-2]
	\arrow["\sim", from=1-2, to=1-3]
	\arrow["\sim"', from=2-2, to=2-3]
	\arrow[from=1-3, to=2-3]
	\arrow["\lrcorner"{anchor=center, pos=0.125, rotate=180}, draw=none, from=2-2, to=0]
\end{tikzcd}\]
The left square is cocartesian and so homotopy cocartesian. Horizontal morphisms of the right square are weak equivalences, so this square is also homotopy cocartesian.
The outer square is then homotopy cocartesian and this implies that 
$[[2]\otimes a,1]$ is $n$-relying on $d^0\otimes a$ and $d^2\otimes a$. We then have a diagram:
\[\begin{tikzcd}[column sep=0.3cm]
	{[k]\star([1]\otimes a)\amalg [k]\star([1]\otimes a)} & {[k]\star([2]\otimes a)} & {[k]\star([2]\botimes a)} \\
	{\tau^i_{n+k+1}([k]\star([1]\otimes a))\amalg \tau^i_{n+k+1}([k]\star([1]\otimes a))} & {\tau^i_{n+k+1}([k]\star([2]\otimes a))} & {\tau^i_{n+k+1}([k]\star([2]\botimes a))}
	\arrow[""{name=0, anchor=center, inner sep=0}, from=1-2, to=1-3]
	\arrow[from=1-2, to=2-2]
	\arrow[from=2-2, to=2-3]
	\arrow[from=1-3, to=2-3]
	\arrow[from=2-1, to=2-2]
	\arrow[from=1-1, to=1-2]
	\arrow[from=1-1, to=2-1]
	\arrow["\lrcorner"{anchor=center, pos=0.125, rotate=180}, draw=none, from=2-3, to=0]
\end{tikzcd}\]
where the two squares are homotopy cocartesian and so is the outer one. This implies the first assertion and the two others follow easily.
\end{proof}

\begin{lemma}
\label{lem:abstract thinness 2}
Let $n$ be an integer strictly superior to $1$ and $a$ such that $\tau^i_{n}(a) = a$.
We consider the projection $\pi:[a,2]\to [a,1]\vee[\tau^i_{n-1}(a),1]$ and $\pi':[a,2]\to [\tau^i_{n-1}(a),1]\vee[a,1]$. We then have inequalities
$$ e\star\pi\circ \epsilon_a\circ [d^0\otimes a, 1]\geq_{n+1} e\star\pi\circ \epsilon_a\circ [d^1\botimes a, 1]$$ and 
$$ e\star\pi'\circ \epsilon_a\circ [d^2\botimes a, 1]\geq_{n+1}e\star \pi\circ \epsilon_a\circ [d^1\botimes a, 1].$$
\end{lemma}
\begin{proof}
Using the diagram $(6).\ref{cons:the big construction}$ we get a diagram
\[\begin{tikzcd}
	{[e\star a,1]} & {[[2]\botimes a,1]} \\
	{[\tau^i_{n}(e\star a),1]} & {e\star[ a,1]} & {e\star[a,2]} \\
	& {e\star[\tau^i_{n-1}(a),1]} & {e\star([a,1]\vee[\tau^i_{n-1}(a),1])}
	\arrow["{[d^2\botimes a,1]}", from=1-1, to=1-2]
	\arrow["{e\star[a,d^0]}"', from=2-2, to=2-3]
	\arrow["{\alpha_a}", from=1-1, to=2-2]
	\arrow["{\epsilon_a}", from=1-2, to=2-3]
	\arrow["e\star\pi", from=2-3, to=3-3]
	\arrow[from=2-2, to=3-2]
	\arrow[from=3-2, to=3-3]
	\arrow[from=1-1, to=2-1]
	\arrow[from=2-1, to=3-2]
\end{tikzcd}\]
The morphism $r_{e\star([a,1]\vee[\tau^i_{n-1}(a),1])}\circ e\star\pi\circ \epsilon_a$ then factors through $[[2]\botimes a\coprod_{d^2\botimes a} \tau^i_{n}(e\star a),1]$. According to lemma \ref{lem:abstract thinness 1}, we then get the first inequalities.

For the second inequality, using the diagrams $(3).\ref{cons:the big construction}$ and $(5).\ref{cons:the big construction}$, we have a diagram:
\[\begin{tikzcd}
	& {[[1]\otimes a,1]} & {[[2]\botimes a,1]} \\
	& {[e\star a,1]\vee[a,1]} & {e\star[a,2]} \\
	{[e\star a,1]} & {e\star[ a,1]} & {e\star([\tau^i_{n-1}(a),1]\vee[a,1])} \\
	{[\tau^i_{n}(e\star a),1]} & {e\star[ \tau^i_{n-1}(a),1]}
	\arrow["{\alpha_a}"', from=3-1, to=3-2]
	\arrow["{[[1]\otimes a,d^1]}"', from=1-2, to=2-2]
	\arrow["{\epsilon_a}", from=1-3, to=2-3]
	\arrow["{[d^0\otimes a,1]}", from=1-2, to=1-3]
	\arrow["{\delta_a}", from=2-2, to=2-3]
	\arrow["{e\star \pi'}", from=2-3, to=3-3]
	\arrow["{[e\star a,d^2]}", from=3-1, to=2-2]
	\arrow["{e\star[a,d^2]}"{description}, from=3-2, to=2-3]
	\arrow[from=3-2, to=4-2]
	\arrow[from=4-2, to=3-3]
	\arrow["{\alpha_{\tau^i_{n-1}(a)}}"', from=4-1, to=4-2]
	\arrow[from=3-1, to=4-1]
\end{tikzcd}\]
This implies that $r_{e\star([\tau^i_{n-1}(a),1]\vee[a,1])}\circ e\star\pi'\circ e\star[a,d^2]\circ \alpha_a$ factors through $[\tau^i_n(e\star a),1]$. The morphism $r_{e\star([\tau^i_{n-1}(a),1]\vee[a,1])}\circ e\star\pi\circ \epsilon_a$ then factors through $[[2]\otimes a\coprod_{d^0\otimes a} \tau^i_{n}([1]\otimes a),1]$. According to lemma \ref{lem:abstract thinness 1}, we then get the second inequality.

\end{proof}

\begin{lemma}
\label{lem:abstract thinness 3}
Let $n$ be an integer strictly superior to $1$ and $a$ such that $\tau^i_{n}(a) = a$.
We then have $\delta_a\circ [e\star a,d^2]\geq_{n+1} \delta_a\circ [[1]\otimes a,d^1]$. 
\end{lemma}
\begin{proof}
There is a diagram:
\[\begin{tikzcd}
	{[e\star a,1]} & {[e\star a,1]} & {[[1]\otimes a,1]} \\
	{e\star[a,2]} & {[e\star a,1]\vee[a,1]} & {[[1]\otimes a,1]\vee[a,1]} \\
	& {[[1]\otimes a,1]} & {[[1]\otimes a,1]}
	\arrow["{\delta_a}", from=2-2, to=2-1]
	\arrow["{[e\star a,d^2]}", from=1-2, to=2-2]
	\arrow["{[[1]\otimes a,d^1]}"', from=3-2, to=2-2]
	\arrow["{[[1]\otimes a,d^2]}", from=1-3, to=2-3]
	\arrow["id", from=3-3, to=3-2]
	\arrow["{[[1]\otimes a,d^1]}"', from=3-3, to=2-3]
	\arrow[from=1-3, to=1-2]
	\arrow[from=2-3, to=2-2]
	\arrow["id", from=1-1, to=1-2]
\end{tikzcd}\]
As the morphism $[[1]\otimes a,1]\vee[a,1]\to [e\star a,1]\vee[a,1]$ factors through $[[1]\otimes a,1]\vee[\tau^i_{n}([1]\otimes a),1]$, we get the desired inequality.
\end{proof}

\begin{prop}
\label{prop:geqn stable by star}
Let $a$ be an object such that $\tau^i_{n}(a)=a$.
Let $x:[a,1]\to C,y:[a',1]\to C$ be two morphisms, such that $x\geq_ny$, then if we denote by $\bar{x} := e\star x\circ \alpha_a$ and $\bar{y} :=e\star y\circ \alpha_{a'} $, we have $\bar{x}\geq_{n+1}\bar{y}$.
\end{prop}
\begin{proof}
First, we suppose that we are in the first case of the definition \ref{def:order relation case 1}. We can then suppose without loss of generality that $C= [a,1]\vee[\tau^i_{n-1}(a),1]$. We denote by $\pi$ the projection of $[a,2]$ on $[a,1]\vee[\tau^i_{n-1}(a),1]$.
 Using the diagrams $(3).\ref{cons:the big construction}$, $(4).\ref{cons:the big construction}$ and $(5).\ref{cons:the big construction}$, we have a diagram:
\[\begin{tikzcd}
	{[[1]\otimes a,1]} & {[[2]\botimes a,1]} & {[e\star a,1]} \\
	{[e\star a,1]\vee[a,1]} & {e\star[a,2]} & {e\star[ a,1]} \\
	{[e\star a,1]} & {e\star[ a,1]} & {e\star ( [a,1]\vee[\tau^i_{n-1}(a),1])}
	\arrow["{\delta_a}", from=2-1, to=2-2]
	\arrow["{\alpha_a}", from=1-3, to=2-3]
	\arrow["{\alpha_a}"', from=3-1, to=3-2]
	\arrow["{e\star[a,d^2]}", from=3-2, to=2-2]
	\arrow["{[e\star a,d^2]}", from=3-1, to=2-1]
	\arrow["{[d^0\otimes a,1]}", from=1-1, to=1-2]
	\arrow["e\star\pi", from=2-2, to=3-3]
	\arrow["{[[1]\otimes a,d^1]}"', from=1-1, to=2-1]
	\arrow["{\epsilon_a}", from=1-2, to=2-2]
	\arrow["{[d^1\botimes a,1]}"', from=1-3, to=1-2]
	\arrow["{e\star [a,d^1]}"', from=2-3, to=2-2]
\end{tikzcd}\]
Thanks to lemmas \ref{lem:abstract thinness 2} and \ref{lem:abstract thinness 3}, this implies the result.

If we are in the second case of \ref{def:order relation case 1}, we can suppose that $C = [\tau^i_{n-1}(a),1]\vee[a,1]$, and we note by $\pi'$ the projection from $[a,2]\to[\tau^i_{n-1}(a),1]\vee[a,1]$ . Using the diagrams $(4).\ref{cons:the big construction}$ and $(6).\ref{cons:the big construction}$, we have a diagram:
\[\begin{tikzcd}
	{[e\star a,1]} & {[[2]\botimes a,1]} & {[e\star a,1]} \\
	{e\star[ a,1]} & {e\star[a,2]} & {e\star[ a,1]} \\
	& {e\star([\tau^i_{n-1}(a),1]\vee[a,1])}
	\arrow["{\alpha_a}", from=1-3, to=2-3]
	\arrow["{\epsilon_a}", from=1-2, to=2-2]
	\arrow["{[d^2\botimes a,1]}", from=1-1, to=1-2]
	\arrow["{e\star[a,d^0]}"', from=2-1, to=2-2]
	\arrow["{\alpha_a}"', from=1-1, to=2-1]
	\arrow["{e\star\pi'}", from=2-2, to=3-2]
	\arrow["{e\star [a,d^1]}", from=2-3, to=2-2]
	\arrow["{[d^1\botimes a,1]}"', from=1-3, to=1-2]
\end{tikzcd}\]
Thanks to lemmas \ref{lem:abstract thinness 2}, this implies the result.

If we are in the third case, it is a direct consequence of the naturality of $\alpha$, of the definition of $n$-reliability and of the fact that $(e\star C)_{\mk}\cong (e\star C_{\mk})_{\mk}$ as remarked in \ref{defi:marked segal}.
\end{proof}

\begin{prop}
\label{prop:2geqn stable by star}
Let $x:[a,1]\to C$, $y:[a',1]\to C$ and $z:[a'',1]$ be three morphisms, such that $(x,y)\geq_nz$, then if we denote by $\bar{x} := e\star x\circ \alpha_a$, $\bar{y} :=e\star y\circ \alpha_{a'} $ and $\bar{z} :=e\star z\circ \alpha_{a''} $, we have $(\bar{x},\bar{y})\geq_{n+1}\bar{z}$.
\end{prop}
\begin{proof}
Suppose first that we are in the first case of the definition \ref{def:order relation case 2}. We can then suppose without loss of generality that $C= [a,2]$.
We define $\tilde{x}:=\epsilon_a\circ [d^0\otimes a,1]$. Diagram $(6).\ref{cons:the big construction}$ and
lemma \ref{lem:abstract thinness 2} imply that $(\tilde{x},\bar{y})\geq_{n+1} \bar{z}$. Eventually, diagrams $(3).\ref{cons:the big construction}$ and $(5).\ref{cons:the big construction}$ induce a diagram:
\[\begin{tikzcd}
	{[e\star a,1]} & {[e\star a,1]\vee[a,1]} & {[[1]\otimes a,1]} \\
	{e\star[ a,1]} & {e\star[a,2]} & {[[2]\botimes a,1]}
	\arrow["{\delta_a}", from=1-2, to=2-2]
	\arrow["{\alpha_a}"', from=1-1, to=2-1]
	\arrow["{e\star[a,d^2]}"', from=2-1, to=2-2]
	\arrow["{[e\star a,d^2]}", from=1-1, to=1-2]
	\arrow["{[d^0\otimes a,1]}", from=1-3, to=2-3]
	\arrow["{[[1]\otimes a,d^1]}"', from=1-3, to=1-2]
	\arrow["{\epsilon_a}", from=2-3, to=2-2]
\end{tikzcd}\]
wich implies that $\bar{x}\geq_{n+1}\tilde{x}$.

If we are in the second case of the definition, it is a direct consequence of the naturality of $\alpha$, of the definition of $n$-reliability and of the fact that $(e\star C)_{\mk}\cong (e\star C_{\mk})_{\mk}$ as remarked in definition \ref{defi:marked segal}.
\end{proof}

\begin{lemma}
\label{lemma:case k 0}
For any $a$ such that $\tau^i_na=a$ and $x:[a,1]\to C$, if we denote by $\bar{x} := e\star x \circ d^0\star[a,1]$ and $\tilde{x}:= e\star x \circ \alpha_a\circ [d^0\star a,1]$, then $\bar{x}\geq_{n+1}\tilde{x}$.
\end{lemma}
\begin{proof}
Using the diagrams $(1).\ref{cons:the big construction}$ and $(2).\ref{cons:the big construction}$, we have a diagram:
\[\begin{tikzcd}
	{[a,1]} & {[e\star a,1]} \\
	{[e,1]\vee[a,1]} & {e\star [a,1]} & C \\
	{[a,1]}
	\arrow["{[a,d^0]}", from=3-1, to=2-1]
	\arrow["{d^{0}\star {[a,1]}}"', from=3-1, to=2-2]
	\arrow["{\beta_a}", from=2-1, to=2-2]
	\arrow["{[d^{0}\star a,1]}", from=1-1, to=1-2]
	\arrow["{[a,d^1]}"', from=1-1, to=2-1]
	\arrow["{\alpha_a}", from=1-2, to=2-2]
	\arrow["{e\star x}", from=2-2, to=2-3]
\end{tikzcd}\]
which implies the desired inequality.
\end{proof}

\vspace{1cm}

We now use these results to show that the thinness extensions are weak equivalences.
\begin{definition}
\label{defi:a cocartesian square for intelingent truncation} 
 We define by induction on $n\geq 2$ the morphism $\iota_n:[[n-1],1]\to[n]$ where $\iota_2:= \alpha_{[0]}$ and $\iota_{n+1} := e\star\iota_n\circ \alpha_{[n-1]}$.
\end{definition}

 We can easily show by induction that $[n]$ is a colimit of terms which are all invariant under $\tau^i_{n-1}$ except the one corresponding to $\iota_n$.
 For any $n$ we then have a pushout square:
\[\begin{tikzcd}
	{[[n-1],1]} & {[n]} \\
	{[[n-1]_t,1]} & {[n]_t}
	\arrow["{\iota_n}", from=1-1, to=1-2]
	\arrow[from=1-1, to=2-1]
	\arrow[from=1-2, to=2-2]
	\arrow[from=2-1, to=2-2]
	\arrow["\lrcorner"{anchor=center, pos=0.125, rotate=180}, draw=none, from=2-2, to=1-1]
\end{tikzcd}\]

\begin{lemma}
\label{lem:thinness extension last step0}
For any $n$ and for any $k<n$, such that $k\neq n-2$, we have inequalities
$d^k\circ\iota_{n-1}\geq_{n-1}\iota_n \circ[d^k,1]$ and $(d^n\circ\iota_{n-1},d^{n-2}\circ \iota_{n-1})\geq_{n-1}\iota_n\circ [d^{n-2},1]$
\end{lemma}
\begin{proof}
We start by showing the first inequality by induction on $n$. If $n=2$, the only case is $k=1$, and the two morphisms are equal.

Suppose now the result true at the stage $n$. 
If $k>0$, we have
$$
\begin{array}{rllc}
d^{k}\circ\iota_{n}&=& e\star d^{k-1}\circ e\star \iota_{n-1}\circ \alpha_{[n-2]}\\
&\geq_{n}&e\star \iota_n \circ e\star [d^{k-1},1] \circ \alpha_{[n-2]}&\mbox{(induction hypothesis and \ref{prop:geqn stable by star})}\\
&=& e\star \iota_n \circ \alpha_{[n-1]} \circ [e\star d^{k-1},1]\\
&=& \iota_{n+1} \circ \alpha_{[n-1]} \circ [d^{k},1]
\end{array}
$$
We still have to deal with the case $k=0$. As $d^0:[n]\to [n+1]$ (resp $[d^0,1]:[[n-1],1]\to [[n],1]$) is equal to $d^0\star [n]$ (resp. $[d^0\star [n-1],1]$), this is exactly the content of lemma \ref{lemma:case k 0}.

For the second inequality, we proceed again by induction. We remark that this is true for $n=2$. Suppose now the result true at the stage $n$. We have
$$
\begin{array}{rllc}
(d^{n+1}\circ\iota_{n},d^{n-1}\iota_{n})&=& (e\star d^{n}\circ e\star \iota_{n-1}\circ\alpha_{[n-2]},e\star d^{n-2}\circ e\star \iota_{n-1}\circ\alpha_{[n-2]})\\
&\geq_{n-1}& e\star \iota_n\circ e\star[d^{n-2},1]\circ \alpha_{[n-2]}~~~~~\mbox{(induction hypothesis and \ref{prop:2geqn stable by star})}\\
&=& e\star \iota_n\circ e\star \alpha_{[n-1]}\circ[e\star d^{n-2},1]\\ 
&=& \iota_{n+1}\circ [d^{n-1},1]
\end{array}
$$
\end{proof}

\begin{lemma}
\label{lem:thinness extension last step1}
Let $0<k<n$ be two integers.
We denote by $\tau^k$ the projection $[n]\to [n]^k$. We then have $$\tau^k\circ \iota_n\circ [d^k,1]\geq_{n-1}\tau^k\circ d^k\circ\iota_{n-1}.$$
\end{lemma}
\begin{proof}
We demonstrate the result by induction on $n$. For the initialization, the only case is $n=2$ and $k=1$, and is obvious. 
Suppose now the result true at the stage $n$, and let $k>1$. 
We have inequalities:
$$\begin{array}{rlll}
\tau^k\circ \iota_{n+1}\circ [d^k,1] &=& e\star \tau^k\circ e\star \iota_n \circ \alpha_{[n-1]}\circ [d^k,1]\\
&=&e\star \tau^k\circ e\star \iota_n \circ e\star [d^{k-1},1]\circ \alpha_{[n-2]}\\
&\geq_n & e\star \tau_k\circ e\star d^{k-1}\circ e\star \iota_{n-1}\circ \alpha_{[n-2]}& \mbox{(induction hypothesis and \ref{prop:geqn stable by star})}\\
&=& \tau_k\circ d^k\circ \iota_n
\end{array}$$
We still have to deal with the case $k=1$. Using diagrams $(1)$, $(2)$, $(4)$ and $(5)$, of construction \ref{cons:the big construction}, we get a diagram:
\[\begin{tikzcd}
	{[[n-1],1]} & {e\star[[n-2],1]} & {[n]} \\
	{[[2]\botimes [n-2],1]} & {e\star([e,1]\vee[[n-2],1])} & {[n+1]} & {[n+1]^1} \\
	{[[n-1],1]} & {e\star[[n-2],1]} & {e\star[[n-1],1]}
	\arrow["{d^1}", from=1-3, to=2-3]
	\arrow["{e\star\iota_n}"', from=3-3, to=2-3]
	\arrow["{e\star[d^0,1]}"', from=3-2, to=3-3]
	\arrow["{e\star [[n-1],d^1]}"', from=3-2, to=2-2]
	\arrow["{e\star [[n-1],d^0]}", from=1-2, to=2-2]
	\arrow["{e\star \iota_{n-1}}", from=1-2, to=1-3]
	\arrow["{e\star\beta_{[n-1]}}", from=2-2, to=2-3]
	\arrow["{\alpha_{[n-2]}}"', from=3-1, to=3-2]
	\arrow["{\alpha_{[n-2]}}", from=1-1, to=1-2]
	\arrow["{e\star\pi\circ \epsilon_{[n-2]}}", from=2-1, to=2-2]
	\arrow["{[d^2\botimes [n-2],1]}"', from=1-1, to=2-1]
	\arrow["{[d^1\botimes [n-2],1]}", from=3-1, to=2-1]
	\arrow["{\tau^1}", from=2-3, to=2-4]
\end{tikzcd}\]
where $\pi$ is the projection $[[n-2],2]\to [e,1]\vee[[n-2],1]$.
However, according to the diagrams $(5)$ and $(3)$ of \ref{cons:the big construction}, there is a diagram:
\[\begin{tikzcd}
	{[[1]\otimes [n-2],1]} & {[e\star [n-2],1]\vee[[n-2],1]} & {[e\star[n-2],1]} \\
	{[[2]\botimes [n-2],1]} & {[[n-2],2]} & {e\star[[n-2],1]} \\
	& {e\star([e,1]\vee[[n-2],1])} & {e\star[e,1]} \\
	& {[n+1]^1} & {[2]_t}
	\arrow["{[d^0\otimes [n-2],1]}"', from=1-1, to=2-1]
	\arrow["{[[1]\otimes[n-2],d^1]}", from=1-1, to=1-2]
	\arrow["{d^3\circ...\circ d^{n+1}}", from=4-3, to=4-2]
	\arrow["{\delta_{[n-2]}}", from=1-2, to=2-2]
	\arrow["{ \epsilon_{[n-2]}}"', from=2-1, to=2-2]
	\arrow["{e\star \pi}", from=2-2, to=3-2]
	\arrow["{[e\star[n-2],d^2]}"', from=1-3, to=1-2]
	\arrow["{\alpha_{[n-2]}}", from=1-3, to=2-3]
	\arrow[from=2-3, to=3-3]
	\arrow[from=3-3, to=3-2]
	\arrow[from=2-3, to=2-2]
	\arrow["{\tau_1\circ e\star\beta_{[n-1]}}"', from=3-2, to=4-2]
	\arrow[from=3-3, to=4-3]
\end{tikzcd}\]
This implies that $[[2]\botimes [n-2],1]\to [n+1]^{k}\to ([n+1]^{k})_{\mk} $ factors through $[[2]\botimes[n-2]\coprod_{d^0\otimes a}\tau^i_{n-1}([1]\otimes [n-2]),1]$. We can then apply lemma \ref{lem:abstract thinness 1}. 	
\end{proof}

\begin{lemma}
\label{lem:thinness extension last step2}
Let $0<k<n-1$ be two integers.
We denote by $\tau^k$ the projection $[n]\to [n]^k$. We then have $$(\tau^k\circ \iota_n\circ [d^{k-1},1], \tau^k\circ \iota_n\circ [d^{k+1},1])\geq_{n-1}\tau^k\circ \iota_n\circ [d^{k},1]$$ and $$\tau^{n-1}\circ \iota_n\circ [d^{n-2},1]\geq_{n-1}\tau^k\circ \iota_n\circ [d^{n-1},1].$$
\end{lemma}
\begin{proof}
By construction, for any $a$, the morphism $[[2]\star a,1]\to [2]\star[a,1]\to [2]_t\star[a,1]$ factors through $[[2]_t\star a,1]$. By induction, this implies that the composite morphism $[[n-1],1]\xrightarrow{\iota_n}[n]\to [n]^k$ factors through $[[n-1]^k,1]$ for any $k<n-1$. This implies the first assertion. 

For the second one, note that $[[1],e]\to [2]\to [2]_t$ factors through $[[1]_t,e]$. By induction, this implies that the composite morphism $[[n-1],1]\xrightarrow{\iota_n}[n]\to [n]^{n-1}$ factors through $[[n-1]^{n-2},1]$ which gives the second one.
\end{proof}

\begin{prop}
\label{prop:thinness extension}
For any $0\leq k\leq n$, the morphism $([n]^k)' \to ([n]^k)''$ is a weak equivalence.
\end{prop}
\begin{proof}
The case $k=0$ and $k=n$ are demonstrated in lemma \ref{lemma:thinnes extension case 0 and n}. For the case $0<k<n$, lemmas \ref{lem:thinness extension last step0}, \ref{lem:thinness extension last step1} and \ref{lem:thinness extension last step2} imply that if we denote by $\tau_k$ the projection $[n]\to [n]^k$, we have an inequality: $(\tau_k\circ d^{k-1}\circ \iota_{n-1},\tau_k\circ d^{k+1}\circ \iota_{n-1})\geq_{n-1}\tau_k\circ d^k\circ \iota_{n-1}$. Together with the proposition \ref{prop:meaning of geq case 2}, this implies that the following square is homotopy cartesian:
\[\begin{tikzcd}
	{[n-1]\cup[n-1]} & {[n]^k} \\
	{[n-1]_t\cup[n-1]_t} & {([n]^k)''}
	\arrow["{d^{k+1}\cup d^{k-1}}", from=1-1, to=1-2]
	\arrow[from=1-1, to=2-1]
	\arrow[from=2-1, to=2-2]
	\arrow[from=1-2, to=2-2]
\end{tikzcd}\]
The morphism $([n]^k)' \to ([n]^k)''$ is then a weak equivalence.
\end{proof}

\subsection{Saturation extensions}
\label{section:Saturation extensions}

\begin{prop}
\label{prop:saturation extension}
Suppose that the Gray module structure on $A$ is saturated.
For any $n\geq -1$, the morphism $[n]\star [3]^{eq}\to [n]\star [3]^{\sharp}$ is an saturated complete acyclic cofibration.
\end{prop}
\begin{proof}
Let $\Lambda[3]^{eq}\to [3]^{eq}$ be the entire inclusion generated by $Im(d^3)\cup Im(d^0)\subset [3]$. This inclusion fits in the following sequence:
\[\begin{tikzcd}
	{\Lambda^1[2]} & {[2]_t} & {([3]^1)'} & {([3]^1)''} \\
	{\Lambda[3]^{eq}} & \bullet & \bullet & {[3]^{eq}} \\
	& {\Lambda^1[3]} & {[3]}
	\arrow[from=1-1, to=2-1]
	\arrow["{d^2}", from=1-2, to=2-2]
	\arrow[from=1-1, to=1-2]
	\arrow[from=2-1, to=2-2]
	\arrow["\lrcorner"{anchor=center, pos=0.125, rotate=180}, draw=none, from=2-2, to=1-1]
	\arrow[from=3-2, to=2-2]
	\arrow[from=3-3, to=2-3]
	\arrow[from=3-2, to=3-3]
	\arrow[from=2-2, to=2-3]
	\arrow["\lrcorner"{anchor=center, pos=0.125, rotate=-90}, draw=none, from=2-3, to=3-2]
	\arrow[from=1-3, to=2-3]
	\arrow[from=1-4, to=2-4]
	\arrow[from=1-3, to=1-4]
	\arrow[from=2-3, to=2-4]
	\arrow["\lrcorner"{anchor=center, pos=0.125, rotate=180}, draw=none, from=2-4, to=1-3]
\end{tikzcd}\]
This inclusion is then a weak equivalence according to propositions \ref{prop:horn_inclusion} and \ref{prop:thinness extension}.
Now, note that we have a pushout:
\[\begin{tikzcd}
	{[2]_t\amalg [2]_t} & {\Lambda[3]^{eq}} \\
	{[e,2]\coprod [e,2]} & {[e,[3]^{eq}]}
	\arrow[from=1-1, to=2-1]
	\arrow[from=1-1, to=1-2]
	\arrow[from=1-2, to=2-2]
	\arrow[from=2-1, to=2-2]
	\arrow["\lrcorner"{anchor=center, pos=0.125, rotate=180}, draw=none, from=2-2, to=1-1]
\end{tikzcd}\]
As the left vertical morphism is a weak equivalence, so is the right one. 
Let $\Lambda[3]^{\sharp}\to [3]^{\sharp}$ be the entire inclusion generated by $Im(d^3)\cup Im(d^0)\subset [3]$.
Using the same reasoning, we show that this cofibration is acyclic and that there is a weak equivalence $\Lambda[3]^\sharp \to [e,[3]^\sharp]$. We then have a commutative square:
\[\begin{tikzcd}
	{[e,[3]^{eq}]} & {\Lambda[3]^{eq}} & {[3]^{eq}} \\
	{[e,[3]^{\sharp}]} & {\Lambda[3]^{{\sharp}}} & {[3]^{{\sharp}}}
	\arrow["\sim"', from=1-2, to=1-1]
	\arrow["\sim", from=1-2, to=1-3]
	\arrow[from=1-3, to=2-3]
	\arrow["\sim", from=2-2, to=2-1]
	\arrow["\sim"', from=2-2, to=2-3]
	\arrow["\sim"', from=1-1, to=2-1]
	\arrow[from=1-2, to=2-2]
\end{tikzcd}\]
where all arrows labelled by $\sim$ are saturated weak equivalences. By two out of three, this implies that $[3]^{eq}\to [3]^\sharp$ is a saturated weak equivalence. Combined with the proposition \ref{prop:leibnizt joint is Quillen}, this concludes the proof.
\end{proof}

\subsection{Conclusion}

\begin{prop}
\label{prop:Quillen adjunction}
Let $A$ be a complicial (resp. saturated) Gray module.
The stratified cosimplicial object $\stratSset\to \stratSeg(A)$ constructed in \ref{cons: of the adjunction with simplcii} induces a Quillen adjunction between the (resp. saturated) complicial model structure and the (resp. complete) model structure.
\end{prop}

\begin{proof}
It is a direct consequence of theorem \ref{theo:model structure on complicial set} and propositions \ref{prop:horn_inclusion}, \ref{prop:thinness extension}, and \ref{prop:saturation extension}.
\end{proof}

\begin{theorem}
\label{theo:complicialGray module}
Let $A$ be a complicial (resp. saturated) Gray module. The (resp. saturated) Gray module structure on $\stratSeg(A)^\seg$ (resp. $\stratSeg(A)^\cseg$) given by theorem \ref{theo:Gray structure on seg} is complicial.
\end{theorem}

\begin{proof}
The constructions \ref{cons:joint in gray module} and \ref{cons:star pour segal0} provide two functors $\stratSset\times \stratSeg(A)\to \stratSeg(A)$. Moreover, the proposition \ref{prop:comparisons between star and otimes in segal} implies that they are weakly equivalent. By propositions \ref{prop:horn_inclusion_4} and \ref{prop:Quillen adjunction}, the functor of construction \ref{cons:star pour segal0} fulfills all the conditions of the definition \ref{defi:complicial Gray module}, and so does the one of construction \ref{cons:joint in gray module}.
\end{proof}

\begin{cor}
\label{cor:Gray module on xi is complicial}
Let $n\in \Nb\cup \{\omega\}$. The (resp. saturated) Gray module structure on $\tPsh{\Xi_n}^{\seg}$ (resp. on $\tPsh{\Xi_n}^{\cseg}$) given in corollary \ref{cor:Gray module on Xi} is complicial.
\end{cor}

\begin{proof}
The case $n<\omega$ follows from a direct induction using example \ref{example:Xi is complicial Gray module} for the base case, and \ref{theo:complicialGray module} for the inductive hypothesis.

To demonstrate the case $n=\omega$, note that for any integer $k$ we have by construction a commutative square:
\[\begin{tikzcd}
	{\tPsh{\Delta_{\leq k}}} & {\tPsh{\Xi_{k}}} \\
	{\tPsh{\Delta}} & {\tPsh{\Xi}}
	\arrow[from=1-1, to=1-2]
	\arrow[from=1-2, to=2-2]
	\arrow[from=2-1, to=1-1]
	\arrow[from=2-1, to=2-2]
\end{tikzcd}\]
As the functor $\tPsh{\Xi_{k}}\to \tPsh{\Xi}$ between the segal and complete segal model structures is left Quillen, and as each (resp. saturated) complicial elementary anodyne extension belongs to one of the $\tPsh{\Delta_{\leq k}}$, this then implies that the functor $\tPsh{\Delta}\to \tPsh{\Xi}$ is a left Quillen functor between the (resp. saturated) complicial model structure and the (resp. complete) model structure.
\end{proof}

\section{The homotopical Robert-Street conjecture}

\subsection{Complicial sets as a model of $(\infty,n)$-categories}

\begin{construction}
\label{cons: alpha and beta}
Let $n\in \Nb\cup \{\omega\}$.
By corollary \ref{cor:Gray module on xi is complicial}, $\tPsh{\Xi_n}^{\seg}$ is a complicial Gray module. We denote by
$$\alpha_n:\tPsh{\Delta}\to \tPsh{\Xi_n}$$ the functor constructed in \ref{cons:joint in gray module} which is then a left Quillen adjoint between the (resp. saturated) complicial model structure and the (resp. complete) Segal model structure by corollary \ref{cor:Gray module on xi is complicial}. We will denote simply $\alpha$ the functor $\alpha_{\omega}$.

Conversely, we can define by induction a left adjoint 
$$\beta:\tPsh{t\Xi} \to \stratSset$$
such that $\beta([0]):=[0]$, $\beta ([\textbf{a},n])$ fits in the pushout:
\[\begin{tikzcd}
	{\coprod_{k\leq n}\beta(a) \boxtimes\{k\}} & {\beta(a) \boxtimes\tau_1^i([n])} \\
	{\coprod_{k\leq n}\{k\}} & {\beta([a,n])}
	\arrow[from=1-1, to=1-2]
	\arrow[from=1-1, to=2-1]
	\arrow[from=1-2, to=2-2]
	\arrow[from=2-1, to=2-2]
	\arrow["\lrcorner"{anchor=center, pos=0.125, rotate=180}, draw=none, from=2-2, to=1-1]
\end{tikzcd}\]
and $\beta((\Db_n)_t):=\tau^i_{k-1}(\beta(\Db_n))$. By proposition \ref{prop:model structure on Xiset} and remark \ref{rem:mapping from a free model structure}, the functor $\beta$ is a left Quillen adjoint between the (resp. complete) Segal model structure and the (resp. saturated) complicial model structure.
\end{construction}

\begin{lemma}
\label{lemma:alpha preserves globes}
Let $n\in\Nb\cup \{\omega\}$.
The functor $\Gb_{\leq n}\to \tPsh{\Xi_n}$ sending $k$ to $\alpha_n(\Db_k)$ where $\Db_k$ is the complicial set defined in \ref{defi:definition of simplicial globes} is a globular object in the sense of definition \ref{defi:globular object in a model}.
\end{lemma}

\begin{proof}
For any $X$ in $\tPsh{\Delta}$, and $K$ in $\tPsh{\Xi_n}$, we denote $\tilde{\Sigma}(X)$ and $\tilde{\Sigma}(K)$ the objects fitting in the pushout square
\[\begin{tikzcd}
	{ X} & {[0]\star X} && {\{0\}\otimes K} & {[1]\otimes K} \\
	{[0]} & {\tilde{\Sigma}X} && {[0]} & {\tilde{\Sigma}K}
	\arrow[from=1-1, to=1-2]
	\arrow[from=1-1, to=2-1]
	\arrow[from=1-2, to=2-2]
	\arrow[from=1-4, to=1-5]
	\arrow[from=1-4, to=2-4]
	\arrow[from=1-5, to=2-5]
	\arrow[from=2-1, to=2-2]
	\arrow["\lrcorner"{anchor=center, pos=0.125, rotate=180}, draw=none, from=2-2, to=1-1]
	\arrow[from=2-4, to=2-5]
	\arrow["\lrcorner"{anchor=center, pos=0.125, rotate=180}, draw=none, from=2-5, to=1-4]
\end{tikzcd}\]
This defines two left Quillen functors 
$$\tilde{\Sigma}:\tPsh{\Delta}\to \tPsh{\Delta}_{\partial[1]/}~~~~\tilde{\Sigma}:\tPsh{\Xi}\to \tPsh{\Xi}_{\partial/[1]}$$
and note that by construction, the functor $\alpha$ and $\tilde{\Sigma}$ commute. 

The construction \ref{cons:sigma star} provides a weakly invertible natural transformation $(\Sigma X^{op})^{op}\to \tilde{\Sigma}(X)$. As the functor $\R:\stratSset\to \ocat$ preserves suspension and the op duality, we have $\R(\Db_{n})\cong \R(\tilde{\Sigma}^n[0])$, and the two induced morphisms $\Db_{n}\to \N(\Db_{n})\leftarrow \tilde{\Sigma}^n[0]$ are weak equivalences by a direct induction using theorem \ref{theo:strict susension}.

Conversely, in the category $\tPsh{\Xi}_{\partial/[1]}$, the proposition \ref{prop:explicit Gray} implies that the functor $\tilde{\Sigma}([\uvar,1])$ is the (homotopy) colimit of the diagram of left adjoints:
\[\begin{tikzcd}
	{[\uvar,1]} & {[\uvar,1]} && {[\uvar,1]} & {[\uvar,1]} \\
	{[0]} & {[1]_t\vee[\uvar,1]} & {[[1]\otimes \uvar,1]} & {[\uvar,1]\vee[1]_t} & {[0]}
	\arrow[from=1-1, to=2-1]
	\arrow["{[\uvar,d^0]}", from=1-1, to=2-2]
	\arrow["{[\uvar,d^1]}", from=1-2, to=2-2]
	\arrow["{[d^0,1]}", from=1-2, to=2-3]
	\arrow["{[d^1,1]}"', from=1-4, to=2-3]
	\arrow["{[\uvar,d^1]}"', from=1-4, to=2-4]
	\arrow["{[\uvar,d^2]}"', from=1-5, to=2-4]
	\arrow[from=1-5, to=2-5]
\end{tikzcd}\]
We then have a weakly invertible natural transformation $\tilde{\Sigma}([\uvar,1])\to [\tilde{\Sigma}(\uvar),1]$ as this last functor is the (homotopy) colimit of the diagram of functors
\[\begin{tikzcd}
	{[\uvar,1]} && {[\uvar,1]} \\
	{[0]} & {[[1]\otimes \uvar,1]} & {[0]}
	\arrow[from=1-1, to=2-1]
	\arrow[from=1-1, to=2-2]
	\arrow[from=1-3, to=2-2]
	\arrow[from=1-3, to=2-3]
\end{tikzcd}\]
Combining all the ingredients, we then obtain a natural family of weak equivalences:
$$\alpha(\Db_{k})\to \alpha(\N(\Db_k)) \leftarrow \alpha(\tilde{\Sigma}^k[0])\to \tilde{\Sigma}^k[0]\to \Db_k$$
which concludes the proof.
\end{proof}

\begin{definition}
We recall the adjunction
\[\begin{tikzcd}
	{R:\tPsh{\Delta}} & {\ocat:\N}
	\arrow[""{name=0, anchor=center, inner sep=0}, shift left=2, from=1-1, to=1-2]
	\arrow[""{name=1, anchor=center, inner sep=0}, shift left=2, from=1-2, to=1-1]
	\arrow["\dashv"{anchor=center, rotate=-90}, draw=none, from=0, to=1]
\end{tikzcd}\]
constructed in \ref{cons:Street nerve}. The left adjoint sends every weak equivalences of the complicial model structure to isomorphism. It is then a premodel of non-complete $(\infty,\omega)$-categories. Composing with the adjunction 
$$\begin{tikzcd}
	\ocat & {\ncat{n}}
	\arrow[""{name=0, anchor=center, inner sep=0}, shift left=2, from=1-1, to=1-2]
	\arrow[""{name=1, anchor=center, inner sep=0}, shift left=2, from=1-2, to=1-1]
	\arrow["\dashv"{anchor=center, rotate=-90}, draw=none, from=0, to=1]
\end{tikzcd}~~~~~~~\mbox{and}~~~~~~~
\begin{tikzcd}
	\ocat & {\ncat{n}^\comp}
	\arrow[""{name=0, anchor=center, inner sep=0}, shift left=2, from=1-1, to=1-2]
	\arrow[""{name=1, anchor=center, inner sep=0}, shift left=2, from=1-2, to=1-1]
	\arrow["\dashv"{anchor=center, rotate=-90}, draw=none, from=0, to=1]
\end{tikzcd},$$
this endows $\tPsh{\Delta}^n_\sat$ (resp. $\tPsh{\Delta}^n$) with the structure of a premodel of (resp. complete) $(\infty,n)$-categories.
\end{definition}
\begin{lemma}
\label{lemma:intermediary}
There is a natural isomorphism $R\beta \to \pi_0^\Xi$, where $\pi_0^\Xi$ is defined in \ref{defi of pi0 for xi}.
\end{lemma}

\begin{proof}
We can reduce to the case $n=\omega$. By construction, $R\beta$ sends all weak equivalences of $\tPsh{\Xi}$ to isomorphisms. It then factors through the localization of $\tPsh{\Xi}$ at the set of maps $\{\Db_{n+1}\to \Db_n,~ n>0\}$, which is nothing more than $\Psh{\Xi}$. We denote $f':\Psh{\Xi}\to \ocat$ the induced cocontinuous functor. As the class of morphisms sent by $f'$ to isomorphisms forms a precocomplete class and contains $\M$, we can deduce from theorem \ref{theo:unit and counit are in Sigma} that this localization is isomorphic to $\ocat$. The functor $f'$ then factors as:
$$\Psh{\Xi}\to  \ocat\xrightarrow{F} \ocat$$
where $F$ is a left adjoint. According to theorem \ref{theo:appendince unicity of operation}, it is sufficient to show that $F$ preserves globes to conclude. This follows from the construction of $\beta$ and proposition \ref{prop:suspensino and gray}.
\end{proof}

\begin{prop}
\label{prop: alpha and beta are adjoint equivalence}
The functor $$\alpha:\tPsh{\Delta}\to \tPsh{\Xi}~~~~~\mbox{and}~~~~~\beta:\tPsh{\Xi}\to \tPsh{\Delta}$$
are left Quillen equivalences between the (resp. saturated) complicial model structure and the (resp. complete) model structure.
\end{prop}

\begin{proof}
The lemma \ref{lemma:intermediary} implies that we have  $\R\beta\cong \pi_0$. Remark now that $\beta$ preserves globes by construction. Combined with proposition \ref{lemma:alpha preserves globes}, this implies that $\beta\alpha$ preserves globes up to a zigzag of weak equivalence, and so is a left Quillen equivalence by corollary \ref{cor:criterion_to_be_linked_to_identity_case stratified}. The proposition \ref{prop:existence_of_comparaison_with_street} then implies that we have $R\beta\alpha\cong R$, and so $\pi_0\alpha\cong \R\beta\alpha\cong R$.

Now, the proposition \ref{prop:Shommer preis} together with the fact that $\beta$ and $\alpha$ preserve globes up to a zigzag of weak equivalence implies that $\alpha\beta$ is a left Quillen equivalence. 
\end{proof}

\begin{lemma}
The model category $\tPsh{\Delta}^n$ (resp. $\tPsh{\Delta}^n$) is the left Bousfield localization of $\tPsh{\Delta}^n_\sat$ (resp. $\tPsh{\Delta}^n$) by the set of morphisms $\{\Db_{k+1}\to \Db_k,~ k\geq n\}$ where $\Db_n$ is the complicial set defined in \ref{defi:definition of simplicial globes}.
\end{lemma}

\begin{proof}
The statement for the saturated case follows from the non-saturated one. We then focus on the second case. Suppose we are given a left Quillen functor $F:\tPsh{\Delta}^n\to D$. Suppose first that $F$ sends $[k+1]\to [k+1]_t$ to a weak equivalence for any $k\geq n$. It then implies that for any stratified simplicial set, the morphism $A\to \tau_{k}^iA$ is sent to an equivalence, and so in particular, so is $\Db_{k+1}\to (\Db_{k+1})_t$, for any $k\leq n$. 

Conversely, suppose that $F$ sends $\Db_{k+1}\to \Db_k$ to a weak equivalence for any $k\geq n$. Let $k\geq n$ and let $C$ be the set of complicial sets $A$ such that $F$ sends $A\to \tau_k^i(A)$ to a weak equivalence. By the same proof as that of lemma \ref{lemma:G_fibration_right lifting property_against_sat}, we can show that $[k]$ belongs to $C$. The functor $F$ then sends $[k+1]\to [k+1]_t$ to a weak equivalence for any $k\geq n$.
\end{proof}

\begin{theorem}
\label{theo:theorem model 1}
Let $n\in \Nb\cup \{\omega\}$. The $n$-complicial model structure $\tPsh{\Delta}^{n}$ is a model of non-complete $(\infty,n)$-categories. The saturated $n$-complicial model structure $\tPsh{\Delta}^{n}_\sat$ is a model of $(\infty,n)$-categories.
\end{theorem}

\begin{proof}
The proposition \ref{prop: alpha and beta are adjoint equivalence} implies that $\tPsh{\Delta}$ (resp. $\tPsh{\Delta}_\sat$) is a model of non-complete $(\infty,n)$-categories (resp. of $(\infty,n)$-categories). As the proposition \ref{lemma:alpha preserves globes} implies that the functors $\alpha_n$ preserve globes, the result follows from proposition \ref{prop:Shommer preis}.
\end{proof}

\begin{cor}
\label{cor:fundamental adj}
Let $n\in \Nb$.
The adjunction between $\Psh{\Theta_n\times \Delta}^{\cseg}$ and $\stratSset^n_\sat$ constructed in \cite{Ozornova_a_quillen_adjunction_between_globular_and_complicial} is a Quillen equivalence.
\end{cor}

\begin{proof}
A direct induction using \cite[theorem 3.22]{Ozornova_a_quillen_adjunction_between_globular_and_complicial} implies that the left adjoint preserves globes. The results then follow from the fact that these two categories are models of $(\infty,n)$-categories and from proposition \ref{prop:Shommer preis}.
\end{proof}

\subsection{Complicial spaces as a model of $(\infty,\omega)$-categories V2}

\begin{prop}[Ozornova-Rovelli]
\label{prop:complicial model structure}
There exists a nice model structure on $\Psh{t\Delta\times \Delta}$, called the \emph{complicial model structure}, and denoted $\Psh{t\Delta\times \Delta}^k$. Its weak equivalences are the smallest precocomplete class of morphisms that contain
\begin{enumerate}
\item For any integers $n,m$, the morphisms $[n]\boxdot [m]\to [n]\boxdot[0]$ and $[n]_t\boxdot [m]\to [n]_t\boxdot[0]$,
\item For any integer $n$, the morphism $[n]_t\coprod_{[n]}[n]_t\boxdot [0]\to [n]_t\boxdot [0]$,
\item For any morphism anodyne extension $K\to L$ given in definition \ref{defi:anodyne extension complicial}, the morphism $K\boxdot [0]\to L\boxdot [0]$.
\end{enumerate}

The model category $\Psh{t\Delta\times \Delta}$ admits a left Bousfield localization along the morphism saturation extension. It is denoted $\Psh{t\Delta\times \Delta}^k_\sat$ and called the \emph{saturated complicial space model structure}.
\end{prop}

\begin{theorem}[Ozornova-Rovelli]
\label{theo:complicial model structure}
Let $n\in \Nb\cup \{\infty\}$.
The adjunction
\[\begin{tikzcd}
	{\pi:\Psh{t\Delta\times\Delta}} & {\tPsh{\Delta}}
	\arrow[""{name=0, anchor=center, inner sep=0}, shift left=2, from=1-1, to=1-2]
	\arrow[""{name=1, anchor=center, inner sep=0}, shift left=2, from=1-2, to=1-1]
	\arrow["\dashv"{anchor=center, rotate=-90}, draw=none, from=0, to=1]
\end{tikzcd}\]
where $\pi$ sends $[k]\boxdot [m]$ onto $[k]\times [m]^\sharp$ and $[k]_t\boxdot [m]$ onto $[k]_t\times [m]^\sharp$, induces a Quillen equivalence between the (resp. saturated) $n$-complicial space model structure and the (resp. saturated) $k$-complicial model structure.
\end{theorem}

\begin{proof}
This is \cite[Theorem 2.14]{Ozornova_model_structure_for_infini_n_categories}
\end{proof}

\begin{lemma}
\label{lemma:alpha and times}
There is a canonical natural transformation
$$\alpha(K \times [n]^\sharp) \to \alpha_n(K) \times [n]^\sharp$$
that is pointwise a weak equivalence of the Segal model structure on $\tPsh{\Xi}$.
\end{lemma}

\begin{proof}
Let's first deal with the case where $K$ is the terminal stratified simplicial set. Remark that the two functors $\uvar\otimes [0]:\tPsh{\Delta}\to \tPsh{\Xi}$ and  $\iota:\tPsh{\Delta}\to \tPsh{\Xi}$ are equivalent, where the first one is induced by the Gray module structure on $\tPsh{\Xi}$ given in corollary \ref{cor:Gray module on Xi} and the second is induced by the canonical functor $\tPsh{\Delta}\to \tPsh{\Xi_1}\hookrightarrow \tPsh{\Xi}$. As a consequence, we obtain a natural transformation between $\uvar\otimes \uvar:\tPsh{\Delta}\times  \tPsh{\Xi}\to \tPsh{\Xi}$ and $\iota(\uvar)\times \uvar:\tPsh{\Delta}\times  \tPsh{\Xi}\to \tPsh{\Xi_n}$. Remark now that this second functor induces another Gray module structure on $ \tPsh{\Xi}$, and the associated functor $\tPsh{t\Delta}\to \tPsh{\Xi_n}$ from construction \ref{cons:joint in gray module} sends $[n]^\sharp$ to $[n]^\sharp$. As a consequence, we obtain a natural transformation $\alpha([n]^\sharp)\to [n]^\sharp$. This induces a comparison
$$\alpha(K\times [n]^\sharp)\to \alpha(K)\times \alpha([n]^\sharp)\to \alpha(K)\times [n]^\sharp.$$
 To show that it is a weak equivalence, remark that we have a commutative square
\[\begin{tikzcd}
	{\alpha(K\times[n]^\sharp)} & {\alpha(K)\times [n]^\sharp} \\
	{\alpha(K\times [0])} & {\alpha(K)\times[0]}
	\arrow[from=1-1, to=1-2]
	\arrow[from=1-1, to=2-1]
	\arrow[from=1-2, to=2-2]
	\arrow[equals, from=2-1, to=2-2]
\end{tikzcd}\]
where vertical morphisms are weak equivalences. By two out of three, this implies the results.
\end{proof}

\begin{construction}
\label{cons: of realization infini}
We consider the left adjoint:
$$\R_\infty:\Psh{t\Delta\times \Delta}\to \tPsh{\Theta\times \Delta}$$
sending $[n]\boxdot[m]$ to $\iota(\O_n)\boxdot [m]$ and $[n]_t\boxdot[m]$ to the pushout:
\[\begin{tikzcd}
	\begin{array}{c} \\Db_n\boxdot([m]\times \{0\}) \end{array} && {\Db_n\boxdot([m]\times \{1\})} \\
	{\iota(O_n)\boxdot [m]} & {\Db_n\boxdot([m]\times [1])} & {\Db_{n-1}\boxdot([m]\times \{1\})}
	\arrow[from=1-1, to=2-1]
	\arrow[from=1-1, to=2-2]
	\arrow[from=1-3, to=2-2]
	\arrow[from=1-3, to=2-3]
\end{tikzcd}\]
where the morphism $\Db_n\to \iota(\O_n)$ corresponds to the unique non identity $n$-cell of $\O_n$.
\end{construction}

\begin{remark}
\label{rem:explanation of realization infini}
The value of the previous functor $[n]_t\boxdot[m]$ could deserve some explanation. We can in fact see that $\R_\infty([n]_t\boxdot[m])$ is nothing more than a homotopy pushout of the cospan
\[\begin{tikzcd}
	{\iota(\O_n)\boxdot [m]} & {\Db_{n}\boxdot [m]} & {\Db_{n-1}\boxdot [m]}
	\arrow[from=1-2, to=1-1]
	\arrow[from=1-2, to=1-3]
\end{tikzcd}\]
As the left leg is a monomorphism, we then have a weak equivalence
$$\R_\infty([n]_t\boxdot[m])\to(\iota(\O_n)\coprod_{\Db_{n-1}} \Db_{n})\boxdot [m]$$
\end{remark}

\begin{remark}
\label{rem:equivalence of model and iota}
Suppose we are given two models of non-complete $(\infty,\omega)$-categories $\pi_0:M\to \ocat$ and $\pi_0':N\to \ocat$. Let $L:M\to N$ be a left Quillen equivalence over $\ocat$. We denote $R:\N\to M$, $\pi_0^*:\ocat\to M$, and $(\pi_0')^*:\ocat\to N$ as the associated right adjoints. The equivalence $\pi_0^*\to R(\pi_0')^*$ induces a natural transformation $L\pi_0^*\to (\pi_0)^*$ that is pointwise weakly invertible as $L$ is a left Quillen equivalence.
\end{remark}

\begin{notation}
We will denote $\iota:\ocat\to \Psh{\Theta}$ the canonical inclusion.
\end{notation}

\begin{lemma}
\label{lemma: all th functor}
We have a diagram
\[\begin{tikzcd}
	{\Psh{t\Delta\times \Delta}} & {\tPsh{\Delta}} \\
	{\Psh{\Theta\times \Delta}} & {\tPsh{\Xi_k}^\seg}
	\arrow["\pi", from=1-1, to=1-2]
	\arrow["{\R_\infty}"', from=1-1, to=2-1]
	\arrow["\alpha", from=1-2, to=2-2]
	\arrow["p"', from=2-1, to=2-2]
\end{tikzcd}\]
that commutes up to a zigzag of weakly invertible transformations between left Quillen adjoints, where $\alpha$ and $p$ are respectively defined in constructions \ref{cons: alpha and beta} and \ref{cons:between xi et theta delta}.
\end{lemma}

\begin{proof}
We recall that the functor $\iota^\Xi:\ocat\to \tPsh{\Xi}$ is defined in \ref{defi of pi0 for xi}. We will now define several left adjoint:
$$F_0,F_1,F_2,F_3:\Psh{t\Delta\times \Delta}\to \tPsh{\Xi}^\seg$$
Their values on $[n]\boxdot [m]$ are given by 
$$
\begin{array}{rclrrcl}
F_0([n]\boxdot [m])&:=&p(\iota(\O_n))\times [m]^\sharp &&F_1([n]\boxdot [m])&:=&\iota^\Xi(\O_n)\times [m]^\sharp\\
F_2([n]\boxdot [m])&:=&\alpha(\N(\O_n))\times [m]^\sharp &&F_3([n]\boxdot [m])&:=&\alpha([n])\times [m]^\sharp
\end{array}
$$
and for any $i\leq 3$, 
$$F_i([n]_t\boxdot[m]):=\tau_{n-1}^iF_i([n]_t\boxdot [m])$$
Remark now that we have a zigzag natural transformation
$$ F_0\to F_1\leftarrow F_2\leftarrow F_3\leftarrow \alpha\pi$$
induced by remark \ref{rem:equivalence of model and iota} except the first one that comes from the lemma \ref{lemma:alpha and times}. They are by construction weak equivalences once evaluated on $[n]\boxdot [m]$, but also when evaluated on $[n]_t\boxdot [m]$ as the functor $\tau_i^{n-1}$ is a left Quillen functor. Remark now that we have a canonical natural transformation $p\R_\infty\to F_0$ that is the identity when evaluated on $[n]\boxdot[m]$. It then remains to show that it is an equivalence when evaluated on $[n]_t\boxdot[m]$. By the construction of the model structure on $\tPsh{\Xi_k}$, we can reduce to the case where $m=0$. In this case, remark that there exists a unique morphism $\Db_n\to p(\iota(\O_n)\boxdot[0])$ that does not factor through $\Db_{n-1}$. We then have $$\tau_{n-1}^i(p(\iota(\O_n)\boxdot[0])):=(\Db_n)_t\coprod_{\Db_n}p(\iota(\O_n)\boxdot[0])$$
which implies by remark \ref{rem:explanation of realization infini} that $p\R_\infty\to F_0$ is also an equivalence when computed on $[n]_t\boxdot [m]$.

Eventually, as the functors $F_i$ for $i\leq 3$ preserve monomorphisms, and as $\alpha \pi$ is a left Quillen adjoint, they all are left Quillen adjoints.
\end{proof}

\vspace{1cm}
The next theorem is inspired by \cite{gepner2026oriented}. It corresponds to the "model-categorical translation" of \cite[theorem 3.4.10]{gepner2026oriented}. Note that the proof of the later theorem relies on a conjecture (\cite[conjecture 3.4.2]{gepner2026oriented}) and that translating the theorem \ref{theorem:homotopical robert street} into the $(\infty,1)$-categorical setting provides a proof of both \cite[theorem 3.4.10]{gepner2026oriented} and conjecture \cite[conjecture 3.4.2]{gepner2026oriented}.

\begin{theorem}
\label{theorem:homotopical robert street}
Let $n\in \Nb\cup \{\omega\}$. The adjunction
\[\begin{tikzcd}
	{\R_\infty:\Psh{t\Delta\times \Delta}} & {\Psh{\Theta\times \Delta}:\N_\infty}
	\arrow[""{name=0, anchor=center, inner sep=0}, shift left=2, from=1-1, to=1-2]
	\arrow[""{name=1, anchor=center, inner sep=0}, shift left=2, from=1-2, to=1-1]
	\arrow["\dashv"{anchor=center, rotate=-90}, draw=none, from=0, to=1]
\end{tikzcd}\]
induces a Quillen equivalence between the (resp. saturated) $n$-complicial space model structure and the (resp. complete) Segal model structure. In particular, (resp. saturated) $n$-complicial spaces form a model of non-complete $(\infty,n)$-categories (resp. a model of $(\infty,n)$-categories).
\end{theorem}

\begin{proof}
This follows directly from theorem \ref{theo:complicial model structure} and from lemma \ref{lemma: all th functor}.
\end{proof}

\begin{remark}
Using the theory of $(\infty,1)$-categories, we can interpret the previous result as follows. Let $\iocat^\nc$ be the $\iun$-category of non-complete $(\infty,\omega)$-categories, i.e., the localization of the $\infty$-presheaves category $\iPsh{\Theta}$ along $\W$. Let $\iota:\ocat\to \iocat^\nc$ be the canonical inclusion from strict $\omega$-categories to non-complete $(\infty,\omega)$-categories. 
We have a functor $t\Delta\to  \iocat^{\nc}$ sending $[n]$ onto $\N(\iota[n])$ and $[n]_t$ onto $\tau_{n-1}^i(\N(\iota[n]))$, where $\tau_{n-1}^i$ is the functor that localizes cells of dimension greater than or equal to $n$. This induces an adjunction
\[\begin{tikzcd}
	{\iPsh{t\Delta}} & {\ocat^{\nc}}
	\arrow[""{name=0, anchor=center, inner sep=0}, shift left=2, from=1-1, to=1-2]
	\arrow[""{name=1, anchor=center, inner sep=0}, shift left=2, from=1-2, to=1-1]
	\arrow["\dashv"{anchor=center, rotate=-90}, draw=none, from=0, to=1]
\end{tikzcd}\]
As $\N(\iota[n])\to \tau_{n-1}^i(\N(\iota[n]))$ is an epimorphism, this adjunction specializes into an adjunction
\[\begin{tikzcd}
	{\R_\infty:\mathrm{tPsh}^{\infty}(\Delta)} & {\ocat^{\nc}:\N_\infty}
	\arrow[""{name=0, anchor=center, inner sep=0}, shift left=2, from=1-1, to=1-2]
	\arrow[""{name=1, anchor=center, inner sep=0}, shift left=2, from=1-2, to=1-1]
	\arrow["\dashv"{anchor=center, rotate=-90}, draw=none, from=0, to=1]
\end{tikzcd}\]
where $\mathrm{tPsh}^{\infty}(\Delta)$ is the sub $(\infty,1)$-category of $\iPsh{t\Delta}$ whose objects correspond to $\infty$-presheaves $X$ such that $X([n]_t)\to X([n])$ is a monomorphism. This later adjunction can be seen as a $(\infty,1)$-categorical analogue of the adjunction given in construction \ref{cons:Street nerve}. The theorem \ref{theorem:homotopical robert street} then implies that this adjunction induces an equivalence between $\iocat^{\nc}$ (resp. $\iocat$) and the localization of $\iPsh{t\Delta}$ at the set of (resp. saturated) generating complicial extensions. 
Therefore, theorem \ref{theorem:homotopical robert street} can be seen as an $\infty$-categorical generalization of the Robert-Street conjecture.
\end{remark}


\cleardoublepage
\phantomsection
\addcontentsline{toc}{part}{Index of symbols} 
\printindex[notation]
\clearpage
\phantomsection
\addcontentsline{toc}{part}{Index of notions} 
\printindex[notion]

\cleardoublepage
\phantomsection
\addcontentsline{toc}{part}{Bibliography} 
\bibliography{biblio}{}
\bibliographystyle{alpha}

\end{document}